\documentclass{amsart}

\usepackage[scr=rsfs]{mathalpha}

\usepackage{bm}

\usepackage{rotfloat}

\usepackage{lscape}

\usepackage{tcolorbox}
\usepackage{subcaption}
\usepackage{graphicx}
\usepackage{enumitem}   

\usepackage{cleveref}
\usepackage{mathtools}
\usepackage{geometry}   
\usepackage{framed,multirow,booktabs}

\usepackage{array}

\newtheorem{theorem}{Theorem}[section]
\newtheorem{lemma}[theorem]{Lemma}

\newtheorem{conjecture}{Conjecture}[section]

\theoremstyle{remark}
\newtheorem{remark}[theorem]{Remark}

\theoremstyle{problem}
\newtheorem{problem}[theorem]{Problem}

\newtheorem{algorithm}[theorem]{Algorithm}

\theoremstyle{definition}
\newtheorem{definition}[theorem]{Definition}

\numberwithin{equation}{section}

\newcommand{\Wmu}{{\overline \omega}}

\newcommand{\dt}{\Delta t}
\newcommand{\dx}{\Delta x}
\newcommand{\dy}{\Delta y}

\begin{document}
\global\pdfpageattr\expandafter{\the\pdfpageattr/Rotate 0}

\title[Optimal Cell Average Decomposition for Bound-Preserving Schemes]{On Optimal Cell Average Decomposition for High-Order Bound-Preserving Schemes of Hyperbolic Conservation Laws}

\author{Shumo Cui}
\address{Department of Mathematics, Southern University of Science and Technology, Shenzhen, Guangdong 518055, China}
\email{cuism@sustech.edu.cn}

\author{Shengrong Ding}
\address{Department of Mathematics, Southern University of Science and Technology, Shenzhen, Guangdong 518055, China}
\email{dingsr@sustech.edu.cn}

\author{Kailiang Wu}
\address{Corresponding author. Department of Mathematics \& SUSTech International Center for Mathematics, Southern University of Science and Technology, and National Center for Applied Mathematics Shenzhen (NCAMS), Shenzhen, Guangdong 518055, China}
\email{wukl@sustech.edu.cn}
\thanks{The third author is the corresponding author. This work is supported in part by NSFC grant 12171227.}

\subjclass[2020]{Primary 35L65, 65M12, 65M60, 65M08} 



\keywords{Discontinuous Galerkin methods, hyperbolic conservation laws, bound-preserving schemes,  cell average decomposition}

\begin{abstract}
This paper presents the first systematic study on the fundamental problem of seeking optimal cell average decomposition (OCAD), which arises from constructing efficient high-order bound-preserving (BP) numerical methods within Zhang--Shu framework. 
Since proposed in 2010, Zhang--Shu framework has attracted extensive attention and been applied to developing many 
 high-order BP discontinuous Galerkin and finite volume schemes for various hyperbolic equations. 
An essential ingredient in the framework is the decomposition of the cell averages of the numerical solution into a convex combination of the solution values at certain quadrature points.  
The classic CAD originally proposed by Zhang and Shu has been widely used in the past decade. However, the feasible CADs are not unique, and different CAD would affect the theoretical BP CFL condition and thus the computational costs. 
Zhang and Shu only checked, for the 1D $\mathbb P^2$ and $\mathbb P^3$ spaces, that their  classic CAD based on the Gauss--Lobatto quadrature is {\em optimal} 
in the sense of achieving the mildest BP CFL conditions. 
However, it was recently discovered that the classic CAD is generally not optimal for the multidimensional $\mathbb P^2$ and $\mathbb P^3$ spaces. 
It remained unclear what CAD is optimal for general polynomial spaces, especially in the multiple dimensions.  
In this paper, we establish the general theory for studying the OCAD problem on Cartesian meshes in 1D and 2D. 
We rigorously prove that the classic CAD is optimal for general 1D $\mathbb P^k$ spaces 
and general 2D $\mathbb Q^k$ spaces of an 
arbitrary $k$. 
For the widely used 2D $\mathbb P^k$ spaces, 
the classic CAD is not optimal, and we establish the general approach to find out the genuine OCAD and propose a more practical quasi-optimal CAD, 
both of which provide much milder BP CFL conditions than the classic CAD. 
As a result, our OCAD and quasi-optimal CAD notably improve the efficiency of high-order BP schemes for a large class of hyperbolic or convection-dominated equations, at the little cost of only a slight and local modification to the implementation code. 
The remarkable advantages in efficiency 
are further confirmed by 
several numerical examples covering four hyperbolic partial differential equations.

The proposed analysis and theory of OCADs are highly nontrivial and involve novel techniques from several branches of mathematics. We prove 
several key properties of the OCAD problem, including 
the existence of OCAD by using Carath\'eodory's theorem from convex geometry. 
Through transformation onto a reference cell, 
we simplify the 2D OCAD problem to a symmetric OCAD problem based on the invariant theory of symmetric group in abstract algebra. 
 Most notably, we discover that the symmetric OCAD problem is closely related to 
 polynomial optimization of a positive linear functional on the positive polynomial cone, thereby establishing four useful criteria for examining the optimality of a feasible CAD. Some geometric insights are also provided to interpret our critical findings.   


\end{abstract}

\maketitle



\section{Introduction}
This paper is concerned with robust and efficient high-order numerical methods for hyperbolic conservation laws 
\begin{equation}\label{HCL}
	\begin{cases}
	u_t + \nabla \cdot {\bm f} (u) =0,  \quad  & ({\bm x},t) \in \mathbb R^d \times \mathbb R^+,
	\\
	u({\bm x},0) = u_0({\bm x}), \quad & x \in \mathbb R^d,
\end{cases}
\end{equation}
where ${\bm x}$ denotes the spatial coordinate variable(s) in $d$-dimensional space, $t$ denotes the time, 
the conservative variable(s) $u$ takes values in $\mathbb R^m$, and the flux $\bm f$ takes values in 
$(\mathbb R^m)^d$. Our discussions in this paper can also be applicable to 
other related hyperbolic or convection dominated equations. 

Solutions to the hyperbolic equations \eqref{HCL} typically satisfy certain bounds, which define a convex invariant region $G \subset \mathbb R^m$. For example, the entropy solution to scalar conservation laws ($m=1$)  
satisfies the maximum principle \cite{zhang2010}: 
\begin{equation}\label{scalar_bounds}
	u({\bm x},t)  \in G:= [ U_{\rm min}, U_{\rm max} ]  \qquad \forall t\ge 0 
\end{equation}
with $U_{\rm min} := \min_{{\bm x}} u_0 ({\bm x})$ and $U_{\rm max} := \max_{{\bm x}} u_0 ({\bm x})$. 	
Other important examples include but are not limited to: 
\begin{itemize}
	\item the positivity of water height in shallow water equations \cite{xing2010positivity};
	\item the positivity of density and pressure in compressible Euler equations \cite{zhang2010b};
	\item the subluminal constraint on fluid velocity, the positivity of pressure, and the positivity of rest-mass density in relativistic hydrodynamics \cite{WuTang2015,QinShu2016,Wu2017}.
\end{itemize}
When numerically solving such hyperbolic equations, it is highly desirable or even essential to 
preserve the intrinsic bounds, namely, to preserve the numerical solutions in the region $G$. 
In fact, if the numerical solutions go outside the bounds, for example, negative density or negative pressure is produced 
when solving the Euler equations, the discrete problem would become ill-posed due to the loss of hyperbolicity of the system, 
leading to the instability or breakdown of the numerical computation. 

As well known  for scalar conservation laws,  the first-order monotone schemes
\begin{equation}\label{scalar-1st-order}
	\bar {u}_j^{n+1} = \bar u_{j}^n - \lambda  
\left(  \hat f( \bar u_{j}^n, \bar u_{j+1}^n  ) - \hat f( \bar u_{j-1}^n, \bar u_{j}^n  )  \right) =: H_\lambda(\bar u_{j-1}^n,\bar u_{j}^n, \bar u_{j+1}^n )  
\end{equation}
with $H_\lambda$ monotonically increasing in all of its arguments, 
have been proved to preserve the bounds \eqref{scalar_bounds} under a suitable CFL condition 
\begin{equation}\label{1st_CFL}
	a  \lambda \le c_0,
\end{equation}
where $\lambda = \Delta t/\Delta x$ is the ratio of the temporal and spatial step-sizes, $a$ denotes the maximum characteristic speed,  and $c_0$ is the maximum allowable CFL number. 
Examples of such first-order monotone schemes include the Godunov scheme, the Lax--Friedrichs scheme, and 
the Engquist-Osher scheme, etc. 
These first-order bound-preserving (BP) schemes were also extended to many hyperbolic systems. 

However, constructing high-order accurate BP schemes is rather nontrivial. 
In \cite{zhang2010,zhang2010b}, Zhang and Shu proposed a general framework of designing high-order BP discontinuous Galerkin (DG) and  finite volume (FV) schemes for hyperbolic conservation laws  on rectangular meshes. 
Later, Zhang, Xia, and Shu further extended the framework to unstructured triangular meshes in \cite{zhang2012maximum}. 
Over the past decade, their framework has attracted extensive attention and been generalized to various hyperbolic or convection dominated equations; see, for example, \cite{xing2010positivity,wang2012robust,zhang2013maximum,WuTang2015,QinShu2016,ZHANG2017301,Wu2017,du2019high,du2019third,wu2021minimum}  
and the survey papers \cite{zhang2011b,XuZhang2017}. 
Recently, inspired by a series of BP study on magnetohydrodynamics \cite{Wu2017a,WuShu2018,WuShu2020NumMath}, 
the geometric quasilinearization (GQL) framework was established in  \cite{WuShu2021GQL} for BP problems involving nonlinear constraints.  
The readers are also referred to \cite{Xu2014,xiong2016parametrized,WuTang2015,guermond2017invariant,jiang2018invariant,du2021maximum,campos2019algorithms} 
for some other BP techniques.

The Zhang--Shu approach \cite{zhang2010,zhang2010b} classifies 
the loss of BP property in high-order FV and DG schemes into two cases: The first case is that 
the updated cell averages of the numerical solutions may be outside the set $G$, 
while the second is that the point values of the piecewise polynomial solutions, 
either reconstructed in a FV scheme or evolved by a DG scheme, 
 may be out of the region $G$. 
As long as the BP property of the updated cell averages is guaranteed, 
then a simple scaling BP limiter can be employed to enforce the pointwise bounds 
of the piecewise polynomial solutions without affecting the high-order accuracy \cite{zhang2010,zhang2010b}. 
Therefore, the key task is to ensure that the cell averages are always preserved within the region $G$ during the updating process. 
Let us consider the evolution equation of cell averages for a high-order FV or DG scheme for the one-dimensional (1D) scalar conservation laws, 
which can be written in a unified form as 
\begin{equation}\label{ZS-1D-cell-ave-evo}
		\bar {u}_j^{n+1} = \bar u_{j}^n - \lambda 
	\left(  \hat f( u_{j+\frac12}^-, u_{j+\frac12}^+  ) - \hat f( u_{j-\frac12}^-, u_{j-\frac12}^+  )  \right). 
\end{equation} 
Here $\bar u_{j}^n$ denotes the average on cell $\Omega_{j}:=[x_{j-\frac12}, x_{j+\frac12}]$ at time level $n$, 
and the limiting values at the cell interfaces are computed by  
$$
u_{j-\frac12}^+ = p_j( x_{j-\frac12} ), \qquad  u_{j+\frac12}^- = p_j( x_{j+\frac12} ) \qquad \forall j,
$$
where the polynomial $p_j(x)$ of degree $k$ is either evolved in a DG scheme or reconstructed in a FV scheme on $\Omega_{j}$ with its cell average on $\Omega_{j}$ equaling $\bar u_{j}^n$. Assume the numerical flux $\hat f ( \cdot, \cdot )$ is monotone, so that the corresponding first-order scheme \eqref{scalar-1st-order} is BP under the CFL condition \eqref{1st_CFL}. 
It can be verified that the high-order scheme \eqref{ZS-1D-cell-ave-evo} is increasing 
only with $\bar u_{j}^n$ and  the exterior limiting values $\{u_{j+\frac12}^+, u_{j-\frac12}^-\}$, but decreasing 
with the interior limiting values $\{u_{j+\frac12}^-, u_{j-\frac12}^+\}$. 
Therefore, the monotonically increasing property is invalid to achieve the BP property for the high-order scheme \eqref{ZS-1D-cell-ave-evo}. 
To address this issue, Zhang and Shu \cite{zhang2010} proposed a novel strategy by using 
the cell average $\bar u_{j}^n$ to control the effect of the 
interior limiting values $\{u_{j+\frac12}^-, u_{j-\frac12}^+\}$. They decomposed the cell average $\bar u_{j}^n$ into a convex combination of some point values via the $L$--point Gauss--Lobatto quadrature with $L=\lceil \frac{k+3}{2}\rceil$, which is exact for polynomials of degree up to $k$. 
This implies 
\begin{equation}\label{eq:GL1D}
	\bar u_j^n = \frac{1}{\Delta x} \int_{ \Omega_{j} } p_j(x) {\rm d} x = \sum_{\ell = 1}^L {\omega}_\ell^{{\tt GL}} p_j( x_{j,\ell}^{{\tt GL}} ) =  \sum_{\ell = 2}^{L-1}  {\omega}_\ell^{{\tt GL}} p_j(  x_{j,\ell}^{{\tt GL}} )  
	+ {\omega}_1^{{\tt GL}}  u_{j-\frac12}^+  +  {\omega}_L^{{\tt GL}}  u_{j+\frac12}^-,
\end{equation}
where  $\{ {\omega}_\ell^{{\tt GL}} \}$ are the Gauss--Lobatto quadrature weights (which are all positive), and $\{  x_{j,\ell}^{{\tt GL}} \}$ are the quadrature nodes with $x_{j,1}^{{\tt GL}} = x_{j-\frac12}$ and $x_{j,L}^{{\tt GL}} = x_{j+\frac12}$. 
Based on the cell average decomposition \eqref{eq:GL1D}, Zhang and Shu rewrote the scheme \eqref{ZS-1D-cell-ave-evo} equivalently as 
$$
\bar u_j^{n+1} =  {\omega}_1^{{\tt GL}}   H_{ \frac{ \lambda}{  {\omega}_1^{{\tt GL}} } } \left( u_{j-\frac12}^-, u_{j-\frac12}^+, u_{j+\frac12}^- \right) +   {\omega}_L^{{\tt GL}} H_{ \frac{ \lambda}{  {\omega}_L^{{\tt GL}} } } \left( u_{j-\frac12}^+, u_{j+\frac12}^-, u_{j+\frac12}^+ \right)  + \sum_{\ell = 2}^{L-1}  {\omega}_\ell^{{\tt GL}} p_j(  x_{j,\ell}^{{\tt GL}} ),
$$
which is a convex combination form of the formally first-order scheme. If 
we use a simple accuracy-maintaining BP limiter \cite{zhang2010} to enforce  
\begin{equation}\label{eq:PPcond}
	p_{j}(  x_{j,\ell}^{{\tt GL}} ) \in G  \qquad \forall j, \ell, 
\end{equation}
then by the convexity of $G$, the high-order scheme \eqref{ZS-1D-cell-ave-evo} preserves $\bar u_j^{n+1} \in G$ under the CFL condition 
\begin{equation}\label{1Dhst_CFL}
	a  \lambda \le c_0 \min\{  \omega_1^{{\tt GL}},  \omega_L^{{\tt GL}} \}.
\end{equation}
For the $L$-point Gauss--Lobatto quadrature, $\omega_1^{{\tt GL}} = \omega_L^{{\tt GL}} = \frac{1}{L(L-1)}$.

As we have seen, the cell average decomposition (CAD) in \eqref{eq:GL1D} plays a critical role in constructing high-order BP schemes in Zhang--Shu framework. 
The classic decomposition \eqref{eq:GL1D} originally proposed by Zhang and Shu has been widely used over the past decade. Clearly, the {\em feasible} decomposition strategies, as defined below, are not unique. The classic CAD \eqref{eq:GL1D} is obviously feasible.

\begin{definition}[1D Feasible CAD]
	Let $\mathbb P^k$ denote the space of 1D polynomials of degree up to $k$. 
Let $\langle p \rangle_{ \Omega_j }$ denote the average of a polynomial $p$ over a closed interval $\Omega_j=[x_{j-\frac12},x_{j+\frac12}]$.  
A 1D cell average decomposition   
	\begin{equation}\label{1Ddecomp}
		\langle p \rangle_{ \Omega_j } := \frac{1}{\Delta x} \int_{ \Omega_j } p(x) {\rm d} x =
		\omega^- p( x_{j-\frac12} ) + \omega^+ p( x_{j+\frac12} ) + 
		 \sum_{s=1}^S {\omega}_{s} p(  x_{j}^{(s)} )
	\end{equation}
is said to be {\em feasible} for the space $\mathbb P^k$, if it simultaneously satisfies 
the following three conditions:
\begin{enumerate}[label=(\roman*)]
\item  the identity \eqref{1Ddecomp} exactly holds for all $p \in \mathbb P^k$;
\item  the weights $\{\omega^\pm,{\omega}_{s}\}$ are all positive (their summation equals one);
\item  the internal node set $\mathbb S_j := \{ x_{j}^{(s)} \}_{s=1}^S \subset \Omega_j$. 
\end{enumerate}
\end{definition}

Note that different decomposition 
strategies would give 
different values of $\min\{ \omega^-, \omega^+ \}$ and lead to 
different BP CFL condition  
\begin{equation}\label{1Dhst_CFL_all}
	a  \lambda \le c_0 \min\{ \omega^-, \omega^+ \},
\end{equation}
which affects the computational costs of the overall scheme. 
In view of the efficiency, it is natural to ask a fundamental and important question:
\\[0.01mm]
\begin{center} 
	\parbox{0.8\textwidth}{\em {
			What is the optimal CAD (OCAD) such that the maximum BP CFL number $c_0 \min\{ \omega^-, \omega^+ \}$ is largest?}\\[0.01mm]}
\end{center}
\begin{problem}[1D OCAD Problem]\label{prob:1DOCAD}
Given $k \in \mathbb N_+$, 
find the optimal decomposition in the form of \eqref{1Ddecomp} that is feasible and maximizes 
$${\mathcal G}_1 ( \omega^-, \omega^+ ) := \min\{ \omega^-, \omega^+ \}$$ 
among all feasible CADs, or equivalently, find the 1D positive quadrature rule\footnote{A quadrature rule is said to be positive, if its weights are all positive.} on $\Omega_j$, with the degree of algebraic accuracy 
being at least $k$ and the quadrature nodes including the two endpoints of $\Omega_j$, that provides the maximum ${\mathcal G}_1 ( \omega^-, \omega^+ )$. 
\end{problem}
In the trivial case of $k=1$, 
the classic decomposition \eqref{eq:GL1D} based on the Gauss--Lobatto quadrature is evidently optimal, because $\omega_1^{{\tt GL}}= \omega_L^{{\tt GL}}=\frac12$. 
Zhang and Shu mentioned in \cite[Remark 2.7]{zhang2010} that they had checked that
the decomposition \eqref{eq:GL1D} is optimal in the special cases of $2\le k\le3$. 
\\[0.01mm]
\begin{center} 
	\parbox{0.8\textwidth}{\em {
			 It remains unclear whether the CAD \eqref{eq:GL1D} is also optimal for general $k\ge 4$.}\\[0.01mm]}
\end{center}

Compared to the 1D OCAD problem, the multi-dimensional OCAD problems \cite{cui2022classic} are much more complicated and challenging. 
In this paper, we focus on two classes of two-dimensional (2D) polynomial spaces: 
the polynomial space of total degree up to $k$ (denoted by $\mathbb P^k$) and 
the tensor-product polynomial space of degree up to $k$ (denoted by $\mathbb Q^k$).
On 2D rectangular meshes, the feasible CAD 
for high-order BP schemes can be defined as follows.

\begin{definition}[2D Feasible CAD]\label{def:2D_FCAD}
	Let the 2D polynomial space $\mathbb V^k$ be either $\mathbb P^k$ or $\mathbb Q^k$. 
	Let  
	$$
	\langle p \rangle_{ \Omega_{ij} } : = \frac{1}{\Delta x \Delta y} \int_{ x_{i-\frac12} }^{ x_{i+\frac12} } \int_{ y_{j-\frac12} }^{ y_{j+\frac12} } p(x,y) {\rm d} y {\rm d} x
	$$
	denote the average of a 2D polynomial $p$ over a rectangular cell $\Omega_{ij}=[x_{i-\frac12},x_{i+\frac12}] \times [y_{j-\frac12},y_{j+\frac12}]$. 
	A 2D CAD    
	\begin{equation}\label{2Ddecom}
	\begin{aligned}
			\langle p \rangle_{ \Omega_{ij} } &=  
		\frac{1}{ \Delta y } \int_{ y_{j-\frac12} }^{ y_{j+\frac12} } 
		 \Big( \omega_1^-  p(x_{i-\frac12},y) + \omega_1^+  p(x_{i+\frac12},y)   \Big) 
		 {\rm d} y
		\\
		& + 
		\frac{1}{ \Delta x } \int_{ x_{i-\frac12} }^{ x_{i+\frac12} } 
	\Big(	\omega_2^- p(x,y_{j-\frac12}) + \omega_2^+ p(x,y_{j+\frac12})   \Big)
		{\rm d} x
		 + \sum_{s=1}^S \omega_s p ( x_{ij}^{(s)}, y_{ij}^{(s)} )
	\end{aligned}
	\end{equation}
	is said to be {\em feasible} for the space $\mathbb V^k$, if it simultaneously satisfies 
	the following three conditions:
	\begin{enumerate}[label=(\roman*)]
		\item  the identity \eqref{2Ddecom} exactly holds for all $p \in \mathbb V^k$;
		\item  the weights $\{{\omega}_1^\pm, {\omega}_2^\pm, {\omega}_s\}$ are all positive (their summation equals one);
		\item  the point set $\mathbb S_{ij} := \{ ( x_{ij}^{(s)}, y_{ij}^{(s)}  ) \}_{s=1}^S \subset \Omega_{ij}$. 
	\end{enumerate}
For convenience, we refer to $\{{\omega}_1^\pm, {\omega}_2^\pm\}$ as the {\bf boundary weights} and $\omega_s$ as the {\bf internal weights}, and 
refer to the node set $\mathbb S_{ij}$ as the set of {\bf internal nodes}. 
\end{definition}

In 2D high-order FV or DG methods, the integration of numerical fluxes on the cell interface should be discretized by a suitable 1D Gauss quadrature rule, which is exact for polynomials up to degree $k$. 
Applying this Gauss quadrature to the line integrals in \eqref{2Ddecom}, we can 
equivalently rewrite 
 the feasible CAD \eqref{2Ddecom} as  
	\begin{equation}\label{2DdecomGS}
	\begin{aligned}
		\langle p \rangle_{ \Omega_{ij} } &=  
		\sum_{q=1}^Q  \omega_{q}^{{\tt G}} 
		\Big( \omega_1^-  p(x_{i-\frac12}, y_{j,q}^{{\tt G}}  ) + \omega_1^+  p(x_{i+\frac12},   y_{j,q}^{{\tt G}}  )   \Big) 
		\\
		& + 
		\sum_{q=1}^Q   \omega_{q}^{{\tt G}}  
		\Big(	\omega_2^- p(  x_{i,q}^{{\tt G}}  ,y_{j-\frac12}) + \omega_2^+ p(  x_{i,q}^{{\tt G}}, y_{j+\frac12})   \Big)
		+\sum_{s=1}^S \omega_s p( x_{ij}^{(s)}, y_{ij}^{(s)}  ),
	\end{aligned}
\end{equation}
where $\{ x_{i,q}^{{\tt G}}  \}$ and $\{ y_{j,q}^{{\tt G}}  \}$ are the Gauss quadrature nodes in the intervals $[x_{i-\frac12},x_{i+\frac12}]$ and $[y_{j-\frac12},y_{j+\frac12}]$, respectively, 
and $\{\omega_{q}^{{\tt G}}  \}$ are the Gauss weights.   
As long as the feasible decomposition \eqref{2Ddecom} or \eqref{2DdecomGS} is available, one can construct high-order BP schemes on 2D rectangular meshes under the following CFL condition  (see \cite{cui2022classic})
\begin{equation}\label{2Dhst_CFL_all}
	\Delta t \le c_0 \min \left\{ \frac{\omega_1^- \Delta x}{a_1}, \frac{\omega_1^+ \Delta x}{a_1} , \frac{\omega_2^- \Delta y}{a_2}, \frac{\omega_2^+ \Delta y}{a_2} \right\}=: c_0 \, {\mathcal G}_2( \omega_1^-, \omega_1^+, \omega_2^-, \omega_2^+  ),
\end{equation}
where  $a_1$ and $a_2$ are the maximum characteristic speeds in the $x$- and $y$-directions, respectively. 
It is natural to seek the optimal decomposition such that the CFL condition \eqref{2Dhst_CFL_all} is mildest. 

\begin{problem}[2D OCAD Problem]\label{prob:2DOCAD}
Given $k \in \mathbb N_+$, the space $\mathbb V^k$ and $\{a_1>0, a_2>0, \Delta x>0, \Delta y>0 \}$, 
find the optimal feasible decomposition \eqref{2Ddecom} that maximizes ${\mathcal G}_2( \omega_1^-, \omega_1^+, \omega_2^-, \omega_2^+  )$. This is equivalent to 
find the 2D positive quadrature rule on $\Omega_{ij}$, which is exact for all $p \in \mathbb V^k$ with the quadrature nodes including all the Gauss nodes on $\partial \Omega_{ij}$ and maximizes ${\mathcal G}_2( \omega_1^-, \omega_1^+, \omega_2^-, \omega_2^+  )$. 
\end{problem}

The 2D OCAD problem is much more challenging than the 1D case. Based on the tensor product of 
the $Q$-point Gauss quadrature and the $L$-point Gauss--Lobatto quadrature, 
Zhang and Shu proposed the following feasible CAD  in \cite{zhang2010,zhang2010b}:
\begin{equation} \label{eq:U2Dsplit}
	\begin{split}
		\langle p \rangle_{ \Omega_{ij} }
		&	= \frac{ a_1 /\Delta x    }{  a_1 /\Delta x + a_2 /\Delta y } \omega_1^{{\tt GL}}
			\sum_{q=1}^Q   \omega_{q}^{{\tt G}}  
		\Big(  p(x_{i-\frac12},  y_{j,q}^{{\tt G}}  ) +   p(x_{i+\frac12},   y_{j,q}^{{\tt G}}  )   \Big) 
		\\
		& + 
	\frac{ a_2 /\Delta y    }{  a_1 /\Delta x + a_2 /\Delta y } \omega_1^{{\tt GL}}
	\sum_{q=1}^Q    \omega_{q}^{{\tt G}}  
	\Big(	 p(  x_{i,q}^{{\tt G}}  ,y_{j-\frac12}) +  p( x_{i,q}^{{\tt G}}  ,y_{j+\frac12})   \Big)
	\\
	& + \sum_{\ell=2}^{L-1} \sum_{q=1}^Q 
	 \omega_\ell^{{\tt GL}}   \omega_q^{{\tt G}}  \left(
	\frac{a_1/\Delta x}{a_1/\Delta x + a_2/\Delta y} p(x_{i,\ell}^{{\tt GL}} , y_{j,q}^{{\tt G}}) + 
	\frac{a_2/\Delta y}{a_1/\Delta x + a_2/\Delta y} p(x_{i,q}^{{\tt G}},  y_{j,\ell}^{{\tt GL}}) 
	\right),
	\end{split}
\end{equation}
which corresponds to ${\omega}_1^\pm = \frac{ a_1 \Delta y  \omega_1^{{\tt GL}} }{  a_1 \Delta y + a_2 \Delta x } $ and ${\omega}_2^\pm = \frac{ a_2 \Delta x  \omega_1^{{\tt GL}} }{  a_1 \Delta y + a_2 \Delta x } $. As a special case of \eqref{2Dhst_CFL_all}, the CAD \eqref{eq:U2Dsplit} leads to the following BP CFL condition \cite{zhang2010,zhang2010b}: 
\begin{equation}\label{eq:ZS-2D-CFL}
	\Delta t \le  \frac{  c_0  \omega_1^{{\tt GL}} }{  a_1/\Delta x + a_2 /\Delta y },
\end{equation}
namely, the corresponding maximum CFL number is $c_0  \omega_1^{{\tt GL}}$. 
It is natural to ask: Is the classic CAD \eqref{eq:U2Dsplit} optimal in the 2D case? 
This question had been open until our recent work in \cite{cui2022classic}. 
We found that the classic Zhang--Shu CAD \eqref{eq:U2Dsplit} is generally not optimal for $\mathbb P^k$ spaces in multiple dimensions, 
and we successfully constructed the OCAD for the multidimensional $\mathbb P^2$ and $\mathbb P^3$ spaces; see \cite{cui2022classic} for details. 
\\[0.01mm]
\begin{center} 
	\parbox{0.8\textwidth}{\em {
			 It is still unclear what CAD is optimal for general polynomial spaces, including 2D $\mathbb P^k$ spaces with $k \ge 4$ and 2D $\mathbb Q^k$ spaces.  
			}\\[0.01mm]}
\end{center}


 \begin{table}[b!] 
	\renewcommand\arraystretch{1.8}
	\centering
	\caption{Comparison of  
		optimal BP CFL number $\Wmu_\star c_0$, classic BP CFL number $\omega_1^{{\tt GL}} c_0$, and 
		standard linearly stable CFL number $\frac{1}{2k+1}$, in the case of  $c_0=1$ and $a_1 \Delta y = a_2 \Delta x$ for $\mathbb P^k$-based DG methods with $1\le k \le 9$.}
	\label{tab:CFL}
	\setlength{\tabcolsep}{4mm}{
		\begin{tabular}{cllll}
			\toprule[1.5pt]
			$k$ & standard ($\frac{1}{2k+1}$) & classic ($\omega_1^{{\tt GL}}$) & {optimal} ($\Wmu_{\star}$)   \\
			
			\midrule[1.5pt]
			1 & $\frac{1}{3} = 0.3333$  & $\frac{1}{2} = 0.5$       &  $\frac{1}{2} = 0.5$ \\
			2 & $\frac{1}{5} = 0.2$      & $\frac{1}{6} = 0.1667$   &  $\frac{1}{4} = 0.25$       \\
			3 & $\frac{1}{7} = 0.1429$   & $\frac{1}{6} = 0.1667$   &  $\frac{1}{4} = 0.25$       \\
			4 & $\frac{1}{9} = 0.1111$  & $\frac{1}{12} = 0.08333$  &  $2-\frac{\sqrt{14}}{2} = 0.1292$ \\
			5 & $\frac{1}{11} = 0.09091$ & $\frac{1}{12} = 0.08333$ &  $2-\frac{\sqrt{14}}{2} = 0.1292$ \\
			6 & $\frac{1}{13} = 0.07692$ & $\frac{1}{20} = 0.05$                &  $1-\frac{\sqrt{30}}{6} = 0.08713$ \\
			7 & $\frac{1}{15} = 0.06667$ & $\frac{1}{20} = 0.05$                &  $1-\frac{\sqrt{30}}{6} = 0.08713$ \\
			8 & $\frac{1}{17} = 0.05882$  & $\frac{1}{30} = 0.03333$  &  0.05767 \\
			9 & $\frac{1}{19} = 0.05263$  & $\frac{1}{30} = 0.03333$  &  0.05767 \\
			\bottomrule
			
		\end{tabular}
	}
\end{table}

This paper makes the first attempt to systematically study the OCAD problems for general 1D and 2D polynomial spaces. 
Our efforts and findings are summarized as follows. 
\begin{itemize}
	\item  
	We rigorously prove that, in the 1D case, the classic CAD \eqref{eq:GL1D} is optimal for general $\mathbb P^k$ spaces of an arbitrary $k \in \mathbb N_+$. 
	The key point of the proof is to first relate the OCAD problem to a new optimization problem (\Cref{prob:1Dnew}) with infinitely many linear constraints described by non-negative polynomials, and then to construct a non-negative polynomial in $\mathbb P^k$ vanishing at all the internal nodes of \eqref{eq:GL1D}. 
	\item We establish the general theory for studying the 2D OCAD problem on Cartesian meshes. 
	This is highly nontrivial and involves novel techniques from several branches of mathematics. We prove 
	several key properties of the OCAD problem, including 
	the existence of OCAD by using Carath\'eodory's theorem from convex geometry. 
	Through transformation onto a reference cell, 
	we simplify the 2D OCAD problem to a symmetric OCAD problem (\Cref{prb:4.4}) based on the invariant theory of symmetric group. 
	Most notably, we discover that the symmetric OCAD problem is closely related to 
	polynomial optimization of a positive linear functional on the positive polynomial cone, thereby establishing four useful criteria for examining the optimality of a feasible CAD. Some geometric insights are also provided to interpret our critical findings.
	\item Based on the proposed theory, 
	we rigorously prove that, in the 2D case, the classic CAD \eqref{eq:U2Dsplit} is optimal for general $\mathbb Q^k$ spaces of an arbitrary $k \in \mathbb N_+$. 
	\item 
	It is observed that the classic CAD \eqref{eq:U2Dsplit} is {\em not} optimal for 2D $\mathbb P^k$ spaces. 
	As the polynomial degree $k$ increases, seeking the genuine OCAD becomes more and more difficult. 
	We develop a systematical approach to 
	find the genuinely optimal CADs for the 2D $\mathbb P^k$ spaces, which is a highly nontrivial task. 
	We derive the analytical formulas of OCADs for $\mathbb P^k$ spaces with $k \le 7$. 
	A general algorithm is also proposed to construct the OCADs for $\mathbb P^k$ spaces with $k\ge 8$. 
	The discovery of OCAD is highly nontrivial yet meaningful, as it leads to an improvement of high-order BP schemes for a large class of hyperbolic or convection-dominated equations, at the little cost of only a slight and local modification to the implementation code; see \Cref{tab:CFL} for a comparison. 
	\item Based on geometric insights, we also propose a more practical quasi-optimal CAD, which can be easily constructed via a convex combination of the OCADs in three special cases. It is demonstrated that the quasi-optimal CAD can achieve a near-optimal BP CFL condition, which is very close to the optimal one.
	\item We apply the proposed OCAD and quasi-optimal CAD to designing more efficient BP high-order schemes with milder CFL condition for hyperbolic conservation laws. The notable advantages in efficiency are also demonstrated by several examples covering four hyperbolic partial differential equations, including the convection equation, the inviscid Burgers' equation, the compressible Euler equations, and the relativistic hydrodynamic equations.   
\end{itemize}

The paper is organized as illustrated in \Cref{fig:paper}. 
We analyze the 1D OCAD problem in \Cref{sec:1D_OCAD} and prove that classic CAD \eqref{eq:GL1D} is optimal in the 1D case. 
\Cref{sec:2Dtheory} proposes the general theory for the 2D OCAD problem. 
The OCADs for 2D $\mathbb Q^k$ spaces and 2D $\mathbb P^k$ spaces are studied in \Cref{sec:2DQk} and \Cref{sec:2DPk}, respectively. \Cref{sec:quasi-optimal} discusses the 2D quasi-optimal CAD for $\mathbb{P}^{k}$ spaces. 
We apply the OCAD and quasi-optimal CAD to design efficient BP high-order schemes in \Cref{sec:BPscheme}. 
Several numerical tests are presented in \Cref{sec:examples}, before concluding the paper in \Cref{sec:conclusion}.

\begin{figure}[htbp]
	\centering
	\includegraphics[width=0.999\textwidth]{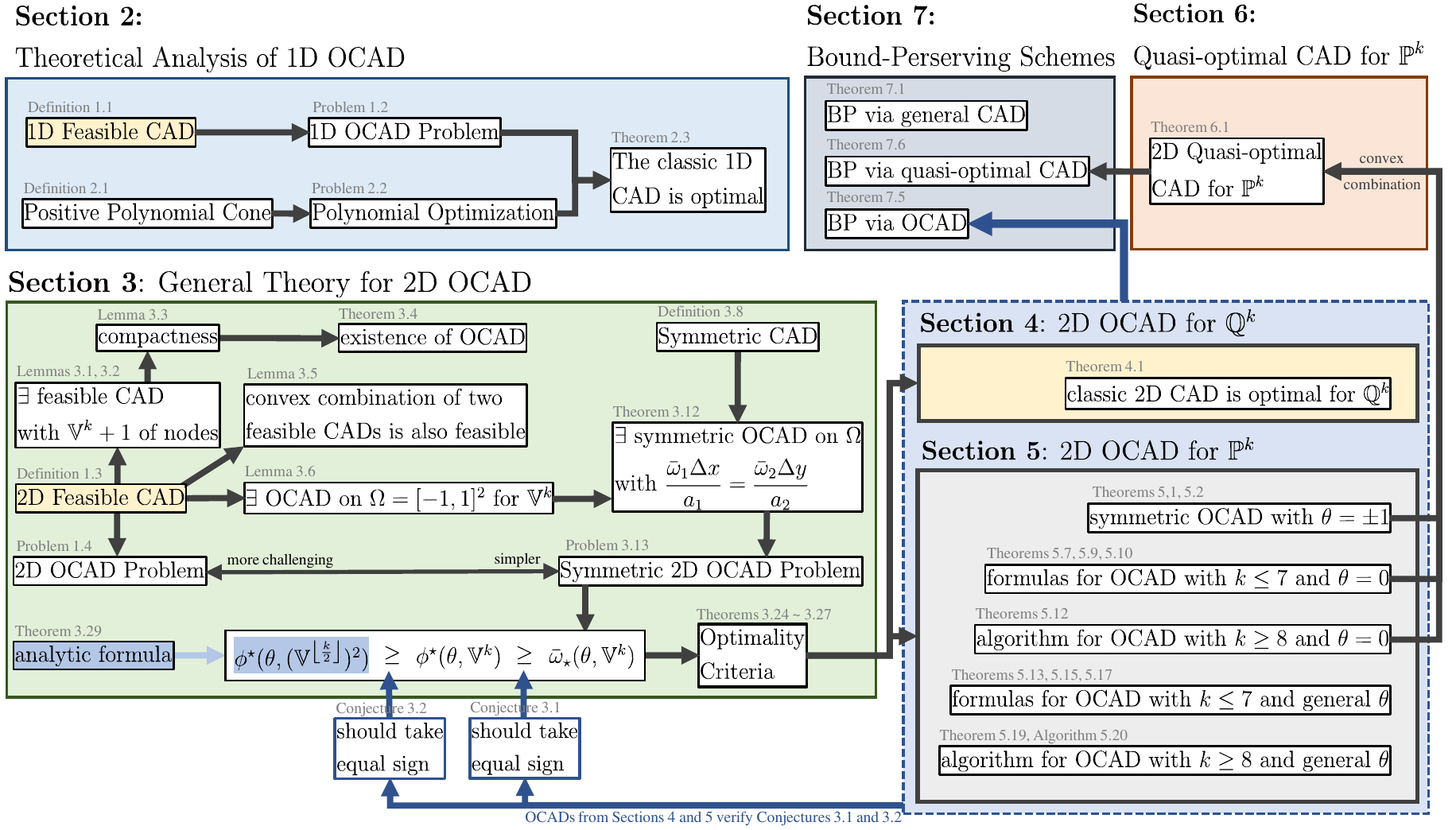}
	\caption{Structure of the present paper}\label{fig:paper}
\end{figure}
 


 




\section{Theoretical analysis of 1D OCAD}\label{sec:1D_OCAD}
In this section, we rigorously prove that the classic 1D CAD \eqref{eq:GL1D} is optimal for general polynomial spaces $\mathbb P^k$ of an arbitrary degree $k \in \mathbb N_+$. 

\begin{definition}[Positive Polynomial Cone]
	Let $\mathbb P^k_+(\Omega_j)$ denote the set of 1D polynomials of degree at most $k$ and nonnegative on $\Omega_j$, namely,     
	\begin{equation}\label{eq:PPC}
		\mathbb P^k_+(\Omega_j):= \left\{ p \in \mathbb P^k:~p(x) \ge 0~\forall x\in \Omega_j  \right\}.
	\end{equation}
	Since the set $\mathbb P^k_+(\Omega_j)$ is a closed convex cone satisfying  
	$$
	t_1 p_1 + t_2 p_2 \in  \mathbb P^k_+(\Omega_j), \qquad \forall p_1,p_2 \in \mathbb P^k_+(\Omega_j),~\forall t_1,t_2\ge 0, 
	$$
	we call $\mathbb P^k_+(\Omega_j)$ a {\em positive polynomial cone} on $\Omega_j$. 
\end{definition}



We observe that the 1D OCAD problem is closely related to the following optimization problem, as shown in \Cref{thm:1D}. 

\begin{problem}\label{prob:1DOCAD-functional}\label{prob:1Dnew} 
	For a given $k \in \mathbb N_+$, 
	 maximize ${\mathcal G}_1 ( \omega^-, \omega^+ )$ 
	subject to the constraint 
	$$
	\langle p \rangle_{ \Omega_j }  - 
	\omega^- p( x_{j-\frac12} ) - \omega^+ p( x_{j+\frac12} ) \ge 0, \qquad \forall p \in \mathbb P^k_+(\Omega_j). 
	$$
\end{problem}

\begin{theorem}\label{thm:1D}
  \Cref{prob:1Dnew} has a unique optimal solution.   
	The boundary weights in the 1D OCAD, denoted by $(\omega^-_\star, \omega^+_\star)$, are unique and given by 
		$$ \omega^-_\star = \omega^+_\star = \omega_1^{{\tt GL}} = \frac{ 1 }{ L(L-1) }, \qquad  L = \left \lceil \frac{k+3}2 \right \rceil,  $$  
	which is also the optimal solution to \Cref{prob:1Dnew}. This means the classic CAD \eqref{eq:GL1D} based on the $L$-point Gauss--Lobatto quadrature is the OCAD in the 1D case.
\end{theorem}

\begin{proof}
	Recall the Gauss--Lobatto quadrature weights satisfy $ \omega_1^{{\tt GL}} =  \omega_L^{{\tt GL}} = \frac{ 1 }{ L(L-1) }$. 
	Since the CAD \eqref{eq:GL1D} is feasible, we have 
\begin{equation}\label{eq:wkl31211}
		 \omega_1^{{\tt GL}} =  \omega_L^{{\tt GL}}  =   {\mathcal G}_1 (  \omega_1^{{\tt GL}} , \omega_L^{{\tt GL}} ) 
	\le {\mathcal G}_1 (\omega^-_\star, \omega^+_\star). 
\end{equation}
	Let $(\omega^-_{\square}, \omega^+_{\square} )$ be an optimal solution to \Cref{prob:1Dnew} and satisfy 
	 	\begin{equation}\label{eq:wkl23s1}
	 		\langle p \rangle_{ \Omega_j }  - 
	 		\omega^-_{\square} p( x_{j-\frac12} ) - \omega^+_{\square} p( x_{j+\frac12} ) \ge 0 \qquad   \forall p \in \mathbb P^k_+(\Omega_j). 
	\end{equation}
	Using \eqref{1Ddecomp} gives 
	$
	\langle p \rangle_{ \Omega_j }  - 
	\omega^-_\star p( x_{j-\frac12} ) - \omega^+_\star p( x_{j+\frac12} )  =  \sum_{s=1}^S {\omega}_s p(  x_{j}^{(s)} ) \ge 0$ for all $p\in \mathbb P^k_+(\Omega_j),   
	$
	which implies that $(\omega^-_\star, \omega^+_\star)$ is a feasible solution to \Cref{prob:1Dnew}. This yields 
\begin{equation}\label{eq:wkl31210}
		{\mathcal G}_1 (\omega^-_\star, \omega^+_\star) \le {\mathcal G}_1 (\omega^-_{\square}, \omega^+_{\square} ). 
\end{equation} 
	Define 
	\begin{equation}\label{eq:Ps1D}
	p^\star (x) := \prod_{\ell = 2}^{L-1} \left( x -  x_{j,\ell} ^{{\tt GL}}  \right)^2,
	\end{equation}
	where $\{ x_{j,\ell}^{{\tt GL}} \}$ are the Gauss--Lobatto nodes.  Clearly, $p^\star (x) \ge 0$ for all $x$, and the degree of $p^\star$ is 
	$
	2(L-2) = 2 \left(  \left \lceil \frac{k+3}2 \right \rceil -2 \right) \le k. 
	$ 
	Hence $p^\star \in \mathbb P_+^k$. 
	Noting that 
	$$
	p^\star( x_{j,\ell}^{{\tt GL}}  )=0, \quad 2 \le \ell \le L-1; \qquad  p^\star( x_{j,1}^{{\tt GL}}  ) > 0, \quad  p^\star( x_{j,L}^{{\tt GL}}  ) > 0, 
	$$
	and using \eqref{eq:wkl23s1} for $p=p^\star$, we obtain  
	\begin{align} \nonumber
		0 & \le  
		\langle p^\star \rangle_{ \Omega_j }  - 
		\omega^-_{\square} p^\star( x_{j-\frac12} ) - \omega^+_{\square} p^\star( x_{j+\frac12} )
		\\ \nonumber
		& = \sum_{\ell =1}^L   \omega_\ell^{{\tt GL}}   p^\star ( x_{j,\ell}^{{\tt GL}}   ) 
		- 
		\omega^-_{\square} p^\star( x_{j-\frac12} ) - \omega^+_{\square} p^\star( x_{j+\frac12} ) 
		\\ \label{wkle101}
		& = (   \omega_1^{{\tt GL}} - \omega^-_{\square}  ) p^\star( x_{j-\frac12} ) 
		+ (   \omega_1^{{\tt GL}} - \omega^+_{\square} ) p^\star( x_{j+\frac12} ) 
		\\ \nonumber
		& \le \left(   \omega_1^{{\tt GL}} - \min\{ \omega^-_{\square}, \omega^+_{\square}\}  \right) \left( p^\star( x_{j-\frac12} ) + p^\star( x_{j+\frac12} ) \right).
	\end{align}
Since $p^\star( x_{j-\frac12} ) + p^\star( x_{j+\frac12} )>0$, we therefore have 
\begin{equation}\label{eq:wkl3121}
	 \omega_1^{{\tt GL}} \ge \min\{ \omega^-_{\square}, \omega^+_{\square}\} = {\mathcal G}_1  (\omega^-_{\square}, \omega^+_{\square} ). 
\end{equation}
	Combining \eqref{eq:wkl31211}--\eqref{eq:wkl31210} with \eqref{eq:wkl3121} gives 
	${\mathcal G}_1 (\omega^-_\star, \omega^+_\star) = {\mathcal G}_1 (\omega^-_{\square}, \omega^+_{\square} )=  \omega_1^{{\tt GL}}$. 
	Without loss of generality, assume that $\omega^-_{\square} \ge \omega^+_{\square}$, then ${\mathcal G}_1 (\omega^-_{\square}, \omega^+_{\square} )=\omega^+_{\square} = \omega_1^{{\tt GL}}$ and using \eqref{wkle101} gives 
	$(  \omega_1^{{\tt GL}} - \omega^-_{\square}  ) p^\star( x_{j-\frac12} ) \ge 0$. 
	This leads to 
	$ \omega_1^{{\tt GL}} \ge \omega^-_{\square} \ge  \omega^+_{\square} =   \omega_1^{{\tt GL}}$. 
	Therefore, $\omega^-_{\square} = \omega^+_{\square} =  \omega_1^{{\tt GL}}$, which demonstrates the uniqueness of the optimal solution to \Cref{prob:1Dnew}. 
	 Because ${\mathcal G}_1 (\omega^-_\star, \omega^+_\star) = {\mathcal G}_1 (\omega^-_{\square}, \omega^+_{\square} )$, 
	 we have that $(\omega^-_\star, \omega^+_\star)$ is an optimal solution to \Cref{prob:1Dnew}. 
	 By the uniqueness, we have 
	$\omega^-_\star = \omega^+_\star = \omega^-_{\square} = \omega^+_{\square} =  \omega_1^{{\tt GL}}$. 
	This completes the proof. 
\end{proof}

As proved in \Cref{thm:1D},  the classic 1D CAD \eqref{eq:GL1D} proposed in \cite{zhang2010} is optimal for all $k \in \mathbb N_+$.   
This confirms 
the conjecture of Zhang and Shu in \cite[Remark 2.7]{zhang2010}.



\section{General theory for 2D OCAD}\label{sec:2Dtheory}

In this section, we establish the general theory for studying the 2D OCAD problem (\Cref{prob:2DOCAD})  
on an arbitrary 
rectangular cell $\Omega_{ij}:=[x_{i-\frac12},x_{i+\frac12}] \times [y_{j-\frac12},y_{j+\frac12}]$ for 2D polynomial space $\mathbb V^k$, which can be $\mathbb P^k$ or $\mathbb Q^k$ or other suitable subspaces of $\mathbb Q^k$.

\subsection{Existence}
We first prove the existence of the 2D OCAD, i.e., \Cref{prob:2DOCAD} always has at least one OCAD solution. 

For convenience, we introduce the following notations 
\[
\langle p \rangle_{\Omega_{ij}}^{\pm x}:=
\frac{1}{\Delta y} \int_{y_{j-\frac12}}^{y_{j+\frac12}} p(x_{i\pm\frac12},y) {\rm d}y
\quad
\text{and}
\quad
\langle p \rangle_{\Omega_{ij}}^{\pm y}:=
\frac{1}{\Delta x} \int_{x_{i-\frac12}}^{x_{i+\frac12}}
p(x,y_{j\pm\frac12}) {\rm d}x.
\]
Then a 2D feasible CAD \eqref{2Ddecom} can be written as 
\begin{equation}\label{eq:6521}
	\langle p \rangle_{\Omega_{ij}} = 
	\omega_1^{-} \langle p \rangle^{-x}_{\Omega_{ij}}+
	\omega_1^{+} \langle p \rangle^{+x}_{\Omega_{ij}}+
	\omega_2^{-} \langle p \rangle^{-y}_{\Omega_{ij}}+
	\omega_2^{+} \langle p \rangle^{+y}_{\Omega_{ij}}+
	\sum_{s=1}^S  {\omega}_s p( {x}_{ij}^{(s)}, {y}_{ij}^{(s)} ).
\end{equation}
For ease of the following discussions in this subsection, we may relax the condition (ii) in \Cref{def:2D_FCAD} to a milder one---``the weights $\{{\omega}_1^\pm, {\omega}_2^\pm, {\omega}_s\}$ are all {\em nonnegative}''. Note that such a relaxation does not affect the OCAD problem, because the zero boundary weights lead to $\mathcal G_2 =0$ and obviously do not correspond to an OCAD, while the zero weights in $\{{\omega}_s\}$ can be safely removed in the CAD.

Notice that only the boundary weights $(\omega_1^{-},\omega_1^{+},\omega_2^{-},\omega_2^{+})$ are  involved in the objective function  
${\mathcal G}_2( \omega_1^-, \omega_1^+, \omega_2^-, \omega_2^+  )$. 
Let $\mathbb A_\omega \subset \mathbb R^4$ denote the set of the boundary weights of all feasible CADs. Evidently, the set $\mathbb A_\omega \subset [0,1]^4$ is bounded. 
We would like to show that the 
set $\mathbb A_\omega$ is both {\em compact} and {\em convex} in $\mathbb R^4$.


\begin{lemma}[]\label{lem:node_bnd}
	Given a feasible CAD \eqref{eq:6521} for $\mathbb V^k$, 
	there always exists a ``conciser'' feasible CAD 
	with the same boundary weights as \eqref{eq:6521} and with only $(\dim \mathbb{V}^k+1)$ or less internal nodes, namely, there exists a feasible CAD
	\begin{equation}\label{613}
		\langle p \rangle_{\Omega_{ij}} = 
		\omega_1^{-} \langle p \rangle^{-x}_{\Omega_{ij}}+
		\omega_1^{+} \langle p \rangle^{+x}_{\Omega_{ij}}+
		\omega_2^{-} \langle p \rangle^{-y}_{\Omega_{ij}}+
		\omega_2^{+} \langle p \rangle^{+y}_{\Omega_{ij}}+
		\sum_{s\in\mathbb{S}}  \widetilde{\omega}_s p( {x}_{ij}^{(s)}, {y}_{ij}^{(s)} )
	\end{equation}
	with 
	\begin{equation}\label{617}
		\mathbb{S} \subseteq\{1,2,\dots,S\}\quad \text{and} \quad
		\# \mathbb{S} \le \dim \mathbb{V}^k+1.
	\end{equation}
\end{lemma}

\begin{proof}
	If $S \le \dim \mathbb{V}^k+1$, the claim (\ref{617}) is obviously true with $\mathbb{S} = \{1,2,\dots,S\}$. 
	Let 
	\[
	{\omega}_0 
	:= 1-\omega_1^{-}-\omega_1^{+}-\omega_2^{-}-\omega_2^{+} 
	= \sum_{s=1}^S {\omega}_s \in [0,1].
	\]
	If ${\omega}_0 =0$, then we know $\omega_s = 0$ for all $s = 1,\dots,S$. In this case, we can safely remove all the internal nodes in the CAD \eqref{eq:6521} without affecting its feasibility, resulting in a feasible CAD with $\mathbb{S} = \emptyset$ and $\#\mathbb{S} = 0$, implying the claim (\ref{617}) is true.
	
	Now, we consider the case $ S > \dim \mathbb{V}^k+1$ and ${\omega}_0>0$. Let $D = \dim \mathbb{V}^k$ and $\{b_1\equiv 1, b_2, \dots,b_D\}$ be a basis for $\mathbb{V}^k$. Define the vector 
	\[
	\bm{m} := (m_1,m_2,\dots,m_d)^\top \in \mathbb{R}_+^D
	\]	
	with $$	m_i := \langle b_i \rangle_{\Omega_{ij}} -
	\omega_1^{-} \langle b_i \rangle^{-x}_{\Omega_{ij}}-
	\omega_1^{+} \langle b_i \rangle^{+x}_{\Omega_{ij}}-
	\omega_2^{-} \langle b_i \rangle^{-y}_{\Omega_{ij}}-
	\omega_2^{+} \langle b_i \rangle^{+y}_{\Omega_{ij}},$$ 
	and define the following point set
	$$
		\mathcal{B} := \left\{{\omega}_0 \,\bm{b}(x_{ij}^{(1)},y_{ij}^{(1)}) ,\dots,{\omega}_0\,\bm{b}(x_{ij}^{(S)},y_{ij}^{(S)} )\right\} \subseteq \mathbb{R}^D
	$$ 
	with $\bm{b}  :=  ( b_1,b_2,\dots,b_D )^\top$. 
	The feasibility condition (i) of CAD \eqref{eq:6521} implies
	\[
	\bm{m} 
	= \sum_{s=1}^S \omega_s \bm{b}(x_{ij}^{(s)},y_{ij}^{(s)})
	= \sum_{s=1}^S \frac{\omega_s}{{\omega}_0} \cdot 
	{\omega}_0 \, \bm{b}(x_{ij}^{(s)},y_{ij}^{(s)}).
	\]
	Notice that $\bm{m}$ lies in the {\em convex hull} of the set $\mathcal{B}$, because $\sum_{s=1}^S \frac{\omega_s}{{\omega}_0} = 1$ and $\frac{\omega_s}{ {\omega}_0 }>0$. According to {\em Carath\'{e}odory's theorem} \cite{Caratheodory1911}, the vector $\bm{m}$ can be written as the convex combination of at most $(D+1)$ points in $\mathcal{B}$, namely, 
	\begin{equation}\label{650}
		\bm{m} 
		= \sum_{s\in \mathbb{S}} \widehat{\omega}_s \, {\omega}_0 \,\bm{b}(x_{ij}^{(s)},y_{ij}^{(s)}) 
		= \sum_{s\in \mathbb{S}} \widetilde{\omega}_s \, \bm{b} (x_{ij}^{(s)},y_{ij}^{(s)})
	\end{equation}
	with 	
\begin{align*}
		&\widehat{\omega}_s > 0, \quad
	\widetilde{\omega}_s = \widehat{\omega}_s \, {\omega}_0 > 0 \quad \forall s \in 
	\mathbb{S} \subseteq \{1,2,\dots,S\}, 
	\\
	&
	 \sum_{s\in\mathbb{S}} \widehat{\omega}_s = 1,
	\quad
	\sum_{s\in\mathbb{S}} \widetilde{\omega}_s = {\omega}_0, \quad
	\#\mathbb{S} \le D +1.
\end{align*}	
	The equality \eqref{650} implies that the CAD \eqref{613} holds for all the basis functions $\{b_i\}_{i=1}^D$ of $\mathbb V^k$, thus for any $p \in \mathbb V^k$. 
	The feasibility condition (i) is verified. 
	Notice that $\widetilde{\omega}_s > 0$ for all $s \in \mathbb{S}$ and $\omega_1^{-}+\omega_1^{+}+\omega_2^{-}+\omega_2^{+}+\sum_{s\in\mathbb{S}} \widetilde{\omega}_s = 1$, 
	the feasibility condition (ii) is verified.
	Finally, the feasibility condition (iii) is true, because $( {x}_{ij}^{(s)}, {y}_{ij}^{(s)} ) \in \Omega_{ij} \; \forall s \in \mathbb{S}$.
	To conclude, \eqref{613} is a 2D feasible CAD for the space $\mathbb V^k$. The proof is completed. 
\end{proof}

As a direct consequence of \Cref{lem:node_bnd}, we have the following result. 

\begin{lemma}\label{lem:S=dim+1}
	For any  $(\omega_1^{-},\omega_1^{+},\omega_2^{-},\omega_2^{+}) \in \mathbb A_\omega$, 
	there always exists a feasible CAD in the form \eqref{eq:6521}  with 
	$(\omega_1^{-},\omega_1^{+},\omega_2^{-},\omega_2^{+})$ as its
	boundary weights and with $S = \dim \mathbb{V}^k+1$. 
\end{lemma}

\begin{proof}
	Due to \Cref{lem:node_bnd}, there exists a feasible CAD like \eqref{eq:6521} with boundary weights $(\omega_1^{-},\omega_1^{+},\omega_2^{-},\omega_2^{+})$ and 	$S \le \dim \mathbb{V}^k+1$. 
	Since we have relaxed the condition (ii) in \Cref{def:2D_FCAD} to allow $\{{\omega}_s\}$ being zero, we can add some nodes with zero weight such that $S = \dim \mathbb{V}^k+1$. 
\end{proof}

\begin{lemma}\label{lem:compact}
The set $\mathbb A_\omega$ is compact.
\end{lemma}

\begin{proof}
	Since $\mathbb A_\omega \subset [0,1]^4$ is bounded, it remains to prove that $\mathbb A_\omega$ is closed, namely, 
	 the limit of every convergent sequence contained in $\mathbb A_\omega$ is also an element of $\mathbb A_\omega$. 
	 Assume that $\{ {\bm \omega}_n 
	 := (\omega_{1,n}^{-},\omega_{1,n}^{+},\omega_{2,n}^{-},\omega_{2,n}^{+}) \}_{n \ge 1}$ is an arbitrary convergent sequence in $\mathbb A_\omega$, and denote 
	 $$
	 \lim_{n \to \infty} {\bm \omega}_n = {\bm \omega} =: (\omega_1^{-},\omega_1^{+},\omega_2^{-},\omega_2^{+}).
	 $$
	It suffices to show the limit ${\bm \omega} \in \mathbb A_\omega$. 
	For every ${\bm \omega}_n \in \mathbb A_\omega $, \Cref{lem:S=dim+1} tells us that there exists a feasible CAD with ${\bm \omega}_n$ as its boundary weights in the following form 
	\begin{equation}\label{6131}
	\langle p \rangle_{\Omega_{ij}} = 
	\omega_{1,n}^{-} \langle p \rangle^{-x}_{\Omega_{ij}}+
	\omega_{1,n}^{+} \langle p \rangle^{+x}_{\Omega_{ij}}+
	\omega_{2,n}^{-} \langle p \rangle^{-y}_{\Omega_{ij}}+
	\omega_{2,n}^{+} \langle p \rangle^{+y}_{\Omega_{ij}}+
	\sum_{s=1}^S  {\omega}_{s,n} p\left( {x}_{ij}^{(s,n)}, {y}_{ij}^{(s,n)} \right)
\end{equation}	
with $S = \dim \mathbb V^k + 1$. Note that for all $n\ge 1$, 
${\omega}_{s,n}\in [0,1]$ and $( {x}_{ij}^{(s,n)}, {y}_{ij}^{(s,n)} ) \in \Omega_{ij}$. According to the {\em Bolzano–Weierstrass theorem}, 
the bounded sequence 
$$
	\left \{ ({\omega}_{1,n},\cdots,{\omega}_{S,n}, {x}_{ij}^{(1,n)}, \cdots, {x}_{ij}^{(S,n)}, {y}_{ij}^{(1,n)}, \cdots, {y}_{ij}^{(S,n)}) \right\}_{n\ge 1} 
	\subseteq [0,1]^S \times [x_{i-\frac12},x_{i+\frac12}]^S \times [y_{j-\frac12},y_{j+\frac12}]^S
$$
has a convergent subsequence, denoted by 
\begin{equation}\label{363}
	\lim_{\ell \to \infty} {\omega}_{s,n_{\ell}} = {\omega}_{s} \in [0,1], \quad 
	\lim_{\ell \to \infty}  \left( {x}_{ij}^{(s,n_\ell)}, {y}_{ij}^{(s,n_\ell) } \right)
	= \left( {x}_{ij}^{(s)}, {y}_{ij}^{(s)} \right) \in \Omega_{ij}, \quad 
	 s=1,2,\dots, S. 
\end{equation}
Taking $n=n_\ell$ in \eqref{6131} and letting $\ell \to +\infty$, we obtain for any $p \in \mathbb V^k$ that 
 	\begin{equation}\label{6132}
 	\langle p \rangle_{\Omega_{ij}} = 
 	\omega_{1}^{-} \langle p \rangle^{-x}_{\Omega_{ij}}+
 	\omega_{1}^{+} \langle p \rangle^{+x}_{\Omega_{ij}}+
 	\omega_{2}^{-} \langle p \rangle^{-y}_{\Omega_{ij}}+
 	\omega_{2}^{+} \langle p \rangle^{+y}_{\Omega_{ij}}+
 	\sum_{s=1}^S  {\omega}_{s} p\left( {x}_{ij}^{(s)}, {y}_{ij}^{(s)} \right),
 \end{equation}	
which is a feasible CAD, because $\omega_{1}^{\pm}\ge 0$, $\omega_{2}^{\pm} \ge 0$, ${\omega}_{s} \ge 0$, and $ ( {x}_{ij}^{(s)}, {y}_{ij}^{(s)} ) \in \Omega_{ij}$. 
Therefore, ${\bm \omega} = (\omega_1^{-},\omega_1^{+},\omega_2^{-},\omega_2^{+}) \in \mathbb A_\omega$. 
In conclusion, the set $\mathbb A_\omega$ is closed and thus compact. 
\end{proof}

\begin{theorem}\label{lem:existence2D}
	The 2D OCAD always exists, i.e., \Cref{prob:2DOCAD} has at least one OCAD solution.  
	Moreover, there exists an OCAD whose boundary weights $( \omega_{1,\star}^-, \omega_{1,\star}^+, \omega_{2,\star}^-, \omega_{2,\star}^+ )$
	maximize  ${\mathcal G}_2$ and also satisfy 
	\begin{equation}\label{eq:2123}
		\frac{\omega_{1,\star}^- \Delta x}{a_1} = \frac{\omega_{1,\star}^+ \Delta x}{a_1} = \frac{\omega_{2,\star}^- \Delta y}{a_2}= \frac{\omega_{2,\star}^+ \Delta y}{a_2} = \max {\mathcal G}_2.
	\end{equation} 
\end{theorem}

\begin{proof}
	According to \Cref{lem:compact}, the feasible region $\mathbb A_\omega$ is compact. 
	Since the objective function $\mathcal{G}_2$ is continuous, 
	by the 
	{\em Weierstrass extreme value theorem}, 
	\Cref{prob:2DOCAD} has at least one OCAD solution. 

	

	Consider an OCAD in the form \eqref{eq:6521} which attains   
	$\max {\mathcal G}_2$. Define 
	$$
	\omega_{1,\star}^- = \omega_{1,\star}^+= a_1  \max {\mathcal G}_2 /\Delta x,  \qquad 
	\omega_{2,\star}^- = \omega_{2,\star}^+= a_2 \max {\mathcal G}_2/\Delta y,
	$$ 
	which satisfy \eqref{eq:2123}. From the definition of ${\mathcal G}_2$, one can observe that 
	$$
	\omega_{1}^\pm - \omega_{1,\star}^\pm \ge 0 , \qquad \omega_{2}^\pm - \omega_{2,\star}^\pm \ge 0. 
	$$
	It follows that 
	$$
	\begin{aligned}
		\langle p \rangle_{\Omega_{ij}} &= 
		\omega_{1,\star}^- \langle p \rangle^{-x}_{\Omega_{ij}}+
		\omega_{1,\star}^+ \langle p \rangle^{+x}_{\Omega_{ij}}+
		\omega_{2,\star}^- \langle p \rangle^{-y}_{\Omega_{ij}}+
		\omega_{2,\star}^+ \langle p \rangle^{+y}_{\Omega_{ij}}+
		\sum_{s=1}^S  {\omega}_s p({x}_{ij}^{(s)}, {y}_{ij}^{(s)}) 
		\\
		& \quad + 
		\sum_{q=1}^Q  \omega_{q}^{{\tt G}} 
		\Big( (\omega_1^--\omega_{1,\star}^-)  p(x_{i-\frac12}, y_{j,q}^{{\tt G}}  ) + (\omega_1^+-\omega_{1,\star}^+)  p(x_{i+\frac12},   y_{j,q}^{{\tt G}}  )   \Big) 
		\\
		& \quad + 
		\sum_{q=1}^Q   \omega_{q}^{{\tt G}}  
		\Big(	(\omega_2^- - \omega_{1,\star}^-) p(  x_{i,q}^{{\tt G}}  ,y_{j-\frac12}) + (\omega_2^+ - \omega_{1,\star}^+) p(  x_{i,q}^{{\tt G}}, y_{j+\frac12})   \Big),
	\end{aligned}
	$$
	which is also an OCAD and satisfies \eqref{eq:2123}. The proof is completed. 
\end{proof}

\subsection{Convexity} 

\begin{lemma}\label{lem:convex}
	A convex combination of any two feasible CADs is also a feasible CAD. Furthermore,  $\mathbb A_\omega$ is a  convex set. 
\end{lemma}

\begin{proof}
Consider two arbitrary 2D feasible CADs 
\begin{align}\label{931}
	\langle p \rangle_{\Omega_{ij}} &= 
	\omega_1^{-} \langle p \rangle^{-x}_{\Omega_{ij}}+
	\omega_1^{+} \langle p \rangle^{+x}_{\Omega_{ij}}+
	\omega_2^{-} \langle p \rangle^{-y}_{\Omega_{ij}}+
	\omega_2^{+} \langle p \rangle^{+y}_{\Omega_{ij}}+
	\sum_{s=1}^S  {\omega}_s p( {x}_{ij}^{(s)}, {y}_{ij}^{(s)} ),
	\\ \label{939}
	\langle p \rangle_{\Omega_{ij}} &= 
	\widehat{\omega}_1^{-} \langle p \rangle^{-x}_{\Omega_{ij}}+
	\widehat{\omega}_1^{+} \langle p \rangle^{+x}_{\Omega_{ij}}+
	\widehat{\omega}_2^{-} \langle p \rangle^{-y}_{\Omega_{ij}}+
	\widehat{\omega}_2^{+} \langle p \rangle^{+y}_{\Omega_{ij}}+
	\sum_{s=1}^{\widehat{S}}  \widehat{\omega}_s p( \widehat{x}_{ij}^{(s)}, \widehat{y}_{ij}^{(s)} ).
\end{align}
Their convex combination
\begin{equation}\label{948}
	\begin{aligned}
		\langle p \rangle_{\Omega_{ij}} & = 
		\big[ \lambda \, \omega_1^{-} + (1-\lambda) \, \widehat{\omega}_1^{-}\big]
		\langle p \rangle^{-x}_{\Omega_{ij}} 
		+ \big[ \lambda \, \omega_1^{+} + (1-\lambda) \, \widehat{\omega}_1^{+}\big]
		\langle p \rangle^{+x}_{\Omega_{ij}} \\
		& + \big[ \lambda \, \omega_2^{-} + (1-\lambda) \, \widehat{\omega}_2^{-}\big]
		\langle p \rangle^{-y}_{\Omega_{ij}} 
		+ \big[ \lambda \, \omega_2^{+} + (1-\lambda) \, \widehat{\omega}_2^{+}\big]
		\langle p \rangle^{+y}_{\Omega_{ij}} \\
		& + \lambda \, \sum_{s=1}^S  {\omega}_s p( {x}_{ij}^{(s)}, {y}_{ij}^{(s)} )
		+ (1-\lambda) \, \sum_{s=1}^{\widehat{S}}  \widehat{\omega}_s p ( \widehat{x}_{ij}^{(s)}, \widehat{y}_{ij}^{(s)} )
	\end{aligned}
\end{equation}
is also a 2D feasible CAD for any $\lambda \in [0,1]$. Thus, the set $\mathbb A_\omega$ is convex.
\end{proof}

\subsection{Transformation to a reference cell $\Omega = [-1,1]^2$}
For convenience, we propose to transform the 2D OCAD problem on an arbitrary rectangular cell $\Omega_{ij}$ into the OCAD problem on a reference cell $\Omega = [-1,1]^2$. 
It should be noted that 
the objective function ${\mathcal G}_2( \omega_1^-, \omega_1^+, \omega_2^-, \omega_2^+  )$ depends on the cell size $\{\Delta x,\Delta y\}$. Hence, 
such a transformation is {\bf not} exactly equivalent to directly considering the OCAD problem on the reference cell $\Omega = [-1,1]^2$, because $\Delta x$ and $\Delta y$ in ${\mathcal G}_2$ are generally not equal to $2$.  

\begin{lemma}\label{lem:C2}
	The existence of a feasible CAD on 
	$\Omega_{ij} = [x_{i-\frac12},x_{i+\frac12}]\times[y_{j-\frac12},y_{j+\frac12}]$ for
	$\mathbb{V}^k$ of the form 
	\begin{equation}\label{eq:652}
		\langle p \rangle_{\Omega_{ij}} = 
		\omega_1^{-} \langle p \rangle^{-x}_{\Omega_{ij}}+
		\omega_1^{+} \langle p \rangle^{+x}_{\Omega_{ij}}+
		\omega_2^{-} \langle p \rangle^{-y}_{\Omega_{ij}}+
		\omega_2^{+} \langle p \rangle^{+y}_{\Omega_{ij}}+
		\sum_{s=1}^S  {\omega}_s p({x}_{ij}^{(s)}, {y}_{ij}^{(s)}) \quad \forall p \in \mathbb{V}^k
	\end{equation}
	is equivalent to the existence of a feasible CAD on $\Omega = [-1,1]^2$ for
	$\mathbb{V}^k$ of the form
	\begin{equation}\label{eq:662}
		\langle q \rangle_{ \Omega } = 
		\omega_1^{-} \langle q \rangle^{-x}_{ \Omega }+
		\omega_1^{+} \langle q \rangle^{+x}_{ \Omega }+
		\omega_2^{-} \langle q \rangle^{-y}_{ \Omega }+
		\omega_2^{+} \langle q \rangle^{+y}_{ \Omega }+
		\sum_{s=1}^S {\omega}_s p({x}^{(s)}, {y}^{(s)}), \quad \forall q \in \mathbb{V}^k,
	\end{equation}
	where
	\begin{equation}\label{eq:trans}
		{x}^{(s)} = \frac{ {x}_{ij}^{(s)}-x_i}{\Delta x/2}, \quad
		{y}^{(s)} = \frac{ {y}_{ij}^{(s)}-y_j}{\Delta y/2},  \qquad  x_i:=\frac{x_{i-\frac12}+x_{i+\frac12}}2, \quad y_j:=\frac{ y_{j-\frac12} + y_{j+\frac12} }2. 
	\end{equation}
\end{lemma}

\begin{proof}
	For any $p \in \mathbb{V}^k$, we define another polynomial $q \in \mathbb{V}^k$ as 
	\begin{equation}\label{eq:680}
		q(x,y) := p\left(
		x_i+\frac{\dx}{2} x ,~ y_j+\frac{\dy}{2} y \right).
	\end{equation}
	The polynomials $p$ and $q$ have the following connections
	\begin{equation}\label{eq:685}
		\langle p \rangle_{\Omega_{ij}} = 
		\langle q \rangle_{\Omega}, \quad
		\langle p \rangle_{\Omega_{ij}}^{\pm x} = 
		\langle q \rangle_{\Omega}^{\pm x}, \quad
		\langle p \rangle_{\Omega_{ij}}^{\pm y} = 
		\langle q \rangle_{\Omega}^{\pm y}, \quad
		p({x}^{(s)}_{ij},{y}^{(s)}_{ij}) = 
		q( {x}^{(s)}, {y}^{(s)}) ~~ \forall s.
	\end{equation}
	Thus, (\ref{eq:662}) implies (\ref{eq:652}). 	
	Conversely, given an arbitrary $q \in \mathbb{V}^k$, we can construct a polynomial $p \in \mathbb{V}^k$ similar to (\ref{eq:680}), which satisfies the relations in (\ref{eq:685}). It immediately follows that (\ref{eq:652}) implies (\ref{eq:662}). 
\end{proof}

\subsection{Symmetric CAD}

The reference cell $\Omega=[-1,1]^2$ is symmetric with respect to $x$- and $y$-axes. 
It is natural to seek a feasible CAD with the same symmetry.  

In order to precisely describe such symmetric structures, we can 
invoke the concept of invariance \cite[Section 1.3]{sturmfels2008algorithms} with respect to the following symmetric group of transformations: 
\begin{equation}\label{eq:Gs}
	{\mathscr G}_{s} := 
\left\{
	(x,y)\mapsto( x, y),~
	(x,y)\mapsto(-x, y),~
	(x,y)\mapsto( x,-y),~
	(x,y)\mapsto(-x,-y)
\right\}.
\end{equation}
Clearly, the reference cell $\Omega=[-1,1]^2$ is ${\mathscr G}_{s}$-invariant, namely, 
$g(\Omega) = \Omega$ for all $g \in {\mathscr G}_{s}$. 
In addition, the space 
$\mathbb{V}^k$ (either $\mathbb{P}^k$ or $\mathbb{Q}^k$) is 
${\mathscr G}_{s}$-invariant, namely, 
\begin{equation}\label{eq:Vk-invar}
g(p):=p(g(x,y)) \in \mathbb{V}^k \qquad \forall g \in {\mathscr G}_{s}.	
\end{equation}

\begin{definition}(${\mathscr G}_s$-invariant subspace)
	The polynomial space 
	$$\mathbb{V}^k({\mathscr G}_{s}) := \{p \in \mathbb{V}^k : 
	g(p) = p ~~~~~ \forall g \in {\mathscr G}_{s} \}
	$$ 
	is called the ${\mathscr G}_s$-invariant subspace of $\mathbb{V}^k$.  
\end{definition}


\begin{definition}[Symmetric CAD]\label{def:310}
	A feasible CAD on $\Omega$ is called a symmetric CAD, if it can be written as the following form
	\begin{equation}\label{eq:715}
		\begin{aligned}
			\langle p \rangle_{\Omega} = 
			2\overline \omega_1 \langle p \rangle_{\Omega}^x+
			2\overline \omega_2 \langle p \rangle_{\Omega}^y+
			\sum_{s=1}^S \omega_s  \overline{ p( x^{(s)}, y^{(s)}) } \qquad \forall p \in \mathbb V^k, 
		\end{aligned}
	\end{equation}
	where $\overline \omega_1 > 0$, $\overline \omega_2 > 0$, $\omega_s > 0$, and $( x^{(s)}, y^{(s)}) \in [0,1]^2$ for all $s$, and 
	\begin{align}
		&\langle p \rangle_{ \Omega  }^x := \frac{1}{2} 
		(\langle p \rangle_{ \Omega }^{-x} + \langle p \rangle_{\Omega}^{+x}), 
		\qquad 
		\langle p \rangle_{\Omega}^y := \frac{1}{2} 
		(\langle p \rangle_{\Omega}^{-y} + \langle p \rangle_{\Omega}^{+y})
		\\ \label{eq:1270}
		& \overline{ p( x^{(s)}, y^{(s)}) } := \frac{1}{4}
		\Big[
		p( x^{(s)}, y^{(s)})+
		p(-x^{(s)}, y^{(s)})+
		p( x^{(s)},-y^{(s)})+
		p(-x^{(s)},-y^{(s)})
		\Big].
	\end{align}
\end{definition}


\begin{remark}
	All the internal nodes involved in the classic CAD (\ref{eq:U2Dsplit}) are symmetrically distributed in $\Omega_{ij}$, and the weights associated symmetric nodes 
	are equal. Hence the 
	classic CAD (\ref{eq:U2Dsplit}) is symmetric. 
\end{remark}

The symmetry of a CAD is helpful for relaxing the feasibility requirement in condition (i) of \Cref{def:2D_FCAD}, as shown in the following lemma. 


\begin{lemma}\label{lem:Gs-invarint}
	If a symmetric CAD \eqref{eq:715} is feasible for the ${\mathscr G}_s$-invariant subspace $\mathbb{V}^k({\mathscr G}_{s})$, then it is feasible for $\mathbb{V}^k$. 
\end{lemma}
\begin{proof}
	The symmetric CAD \eqref{eq:715} is associated with the following linear functional of $p$ on $\mathbb{V}^k$:
	\[
	\mathscr{F} (p ) = \langle p \rangle_{\Omega} - 
	2\overline \omega_1 \langle p \rangle_{\Omega}^x-
	2\overline \omega_2 \langle p \rangle_{\Omega}^y-
	\sum_{s=1}^S \omega_s  \overline{ p( x^{(s)}, y^{(s)}) }.
	\]
	Since the CAD \eqref{eq:715} is feasible for $\mathbb{V}^k({\mathscr G}_{s})$, 
	we have 
	\begin{equation}\label{eq:3432}
		\mathscr{F} (p ) =0  \quad \forall p \in \mathbb{V}^k({\mathscr G}_{s}).
	\end{equation}
	We aim to show that $\mathscr{F} (p ) =0$ for all $p \in \mathbb{V}^k$. 
	Let us prove this by contradiction. 
	Assume that there exists a $p_0 \in \mathbb{V}^k$ such that $\mathscr{F} (p_0 ) \neq 0$. 
	Notice that for every $g \in {\mathscr G}_s$, 
	$$
	\langle g(p_0) \rangle_{\Omega} = \langle p_0 \rangle_{\Omega}, \quad 
	\langle g( p_0) \rangle_{\Omega}^x =  \langle p_0 \rangle_{\Omega}^x, \quad 
	\langle g( p_0) \rangle_{\Omega}^y =  \langle p_0 \rangle_{\Omega}^y, \quad 
	\overline{ p_0( g( x^{(s)}, y^{(s)} )) } = \overline{ p_0( x^{(s)}, y^{(s)}) },
	$$
	which imply 
	\begin{equation}\label{eq:8331}
		\mathscr{F} ( g ( p_0 ) ) = \mathscr{F} (p_0 ) \neq 0 \qquad \forall g \in {\mathscr G}_s.
	\end{equation}
	Then for every $h \in {\mathscr G}_s$, we have
	$$
	h\left(\sum_{g \in {\mathscr G}_s} g(p_0)\right)=\sum_{g \in {\mathscr G}_s} h(g(p_0))=\sum_{h g \in {\mathscr G}_s} h(g(p_0))=\sum_{g \in {\mathscr G}_s} g(p_0),
	$$
	Thus, $p_1:= \sum_{g\in {\mathscr G}_s} g(p_0) \in \mathbb{V}^k({\mathscr G}_s)$. However,
	\[
	\mathscr{F} (p_1 ) = \mathscr{F} \left( \sum_{g\in {\mathscr G}_s } g(p_0) \right) = \sum_{g\in {\mathscr G}_s} \mathscr{F} (g(p_0) )  \overset{\eqref{eq:8331}}{=} \sum_{g\in {\mathscr G}_s} \mathscr{F} ( p_0 ) = \ |{\mathscr G}_s| \, \mathscr{F} ( p_0 )   \neq 0,
	\]
	which contradicts \eqref{eq:3432}. Hence the assumption is incorrect. We have proved that 
	$\mathscr{F} (p ) =0$ for all $p \in \mathbb{V}^k$, which means the symmetric CAD \eqref{eq:715} is feasible for $\mathbb{V}^k$. 
\end{proof}

\begin{remark}
	\Cref{lem:Gs-invarint} 
	can be considered as a corollary of the Sobolev theorem \cite{Sobolev1962}. 
	Note that 
	the dimension of the subspace $\mathbb{V}^k({\mathscr G}_{s})$ is typically much smaller than the dimension of $\mathbb{V}^k$. 
	For example, for $\mathbb{P}^3 = \operatorname{span}\{1,x,y,x^2,xy,y^2,x^3,x^2y,xy^2,y^3\}$, we have  $\mathbb{P}^3({\mathscr G}_{s}) = \operatorname{span}\{1,x^2,y^2\}$, 
	and $\dim \mathbb{P}^3({\mathscr G}_{s}) =3 \ll  \dim \mathbb{P}^3 = 10$. 
	Therefore, \Cref{lem:Gs-invarint} helps to greatly simplify the feasibility requirement and  
	will be very useful for seeking OCADs in \Cref{sec:2DPk}.
\end{remark}

\subsection{Symmetric OCAD problem} 
After careful observation, we find that there always exists an OCAD that is symmetric. 
This new insight gives significant simplification of the 2D OCAD problem, motivating us to consider 
the simplified symmetric OCAD problem in the next subsection. 

\begin{theorem}\label{thm:2D:symOCAD}
	There exists a symmetric OCAD on $\Omega$ with 
	\begin{equation}\label{eq:weighteq2}
		\frac{ \overline \omega_1 \Delta x}{a_1} = \frac{ \overline \omega_2 \Delta y}{a_2}.
	\end{equation}
\end{theorem}
\begin{proof}
	According to \Cref{lem:existence2D}, there exists an OCAD on $\Omega$ of the form
	\begin{equation}\label{eq:735}
		\langle p \rangle_{\Omega} = 
		\omega_1^{-} \langle p \rangle^{-x}_{\Omega}+
		\omega_1^{+} \langle p \rangle^{+x}_{\Omega}+
		\omega_2^{-} \langle p \rangle^{-y}_{\Omega}+
		\omega_2^{+} \langle p \rangle^{+y}_{\Omega}+
		\sum_{s=1}^S  {\omega}_s p( x^{(s)}, y^{(s)}),
	\end{equation}
	which maximizes ${\mathcal G}_2( \omega_1^-, \omega_1^+, \omega_2^-, \omega_2^+  )$ and satisfies
	\begin{equation}\label{eq:3113}
		\frac{\omega_{1}^- \Delta x}{a_1} = \frac{\omega_{1}^+ \Delta x}{a_1} = \frac{\omega_{2}^- \Delta y}{a_2}= \frac{\omega_{2}^+ \Delta y}{a_2}.
	\end{equation}
	For any $p \in \mathbb{V}^k$, we define $q(x,y) := p(-x,y)$, then $q \in \mathbb V^k$ due to \eqref{eq:Vk-invar} and moreover,  
	\[
	\langle q \rangle_{\Omega} = \langle p \rangle_{\Omega}, \quad 
	\langle q \rangle^{\pm x}_{\Omega}
	= 
	\langle p \rangle^{\mp x}_{\Omega}, \quad
	\langle q \rangle^{\pm y}_{\Omega},
	=
	\langle p \rangle^{\pm y}_{\Omega}, \quad
	q( x^{(s)}, y^{(s)})
	= 
	p( -x^{(s)}, y^{(s)}).
	\]
	Applying the OCAD (\ref{eq:735}) to $q$, we find that the following CAD is feasible:
	\begin{equation}\label{eq:762}
		\langle p \rangle_{\Omega} = 
		\omega_1^{+} \langle p \rangle^{-x}_{\Omega}+
		\omega_1^{-} \langle p \rangle^{+x}_{\Omega}+
		\omega_2^{-} \langle p \rangle^{-y}_{\Omega}+
		\omega_2^{+} \langle p \rangle^{+y}_{\Omega}+
		\sum_{s=1}^S  {\omega}_s p( - x^{(s)}, y^{(s)}).
	\end{equation}
	Similarly, we obtain the following two feasible CADs:
	\begin{align}\label{eq:771}
		\langle p \rangle_{\Omega} &= 
		\omega_1^{-} \langle p \rangle^{-x}_{\Omega}+
		\omega_1^{+} \langle p \rangle^{+x}_{\Omega}+
		\omega_2^{+} \langle p \rangle^{-y}_{\Omega}+
		\omega_2^{-} \langle p \rangle^{+y}_{\Omega}+
		\sum_{s=1}^S  {\omega}_s p( x^{(s)}, -y^{(s)}),\\
		\label{eq:780}
		\langle p \rangle_{\Omega} &= 
		\omega_1^{+} \langle p \rangle^{-x}_{\Omega}+
		\omega_1^{-} \langle p \rangle^{+x}_{\Omega}+
		\omega_2^{+} \langle p \rangle^{-y}_{\Omega}+
		\omega_2^{-} \langle p \rangle^{+y}_{\Omega}+
		\sum_{s=1}^S  {\omega}_s p( -x^{(s)}, -y^{(s)}).
	\end{align}
	Taking an average of the four feasible CADs (\ref{eq:735}) and \eqref{eq:762}--(\ref{eq:780}), we obtain the following feasible CAD
	\begin{equation}\label{eq:789}
		\begin{aligned}
			\langle p \rangle_{\Omega} & = 
			\frac{\omega_1^{-}+\omega_1^{+}}{2} 
			\Big[ \langle p \rangle^{+x}_{\Omega} +
			\langle p \rangle^{-x}_{\Omega} \Big]
			+
			\frac{\omega_2^{-}+\omega_2^{+}}{2}
			\Big[ \langle p \rangle^{+y}_{\Omega} +
			\langle p \rangle^{-y}_{\Omega} \Big] + \sum_{s=1}^S \omega_s  \overline{ p( x^{(s)}, y^{(s)}) }.
		\end{aligned}
	\end{equation}
	or equivalently,
	\begin{equation}\label{eq:810}
		\begin{aligned}
			\langle p \rangle_{\Omega} & = 
			\left({\omega_1^{-}+\omega_1^{+}}\right) 
			\langle p \rangle^{x}_{\Omega} 
			+
			\left( {\omega_2^{-}+\omega_2^{+}} \right)
			\langle p \rangle^{y}_{\Omega} + \sum_{s=1}^S \omega_s  \overline{ p( |x^{(s)}|, |y^{(s)}|) },
		\end{aligned}
	\end{equation}
	which is a symmetric CAD with $( |x^{(s)}|, |y^{(s)}|) \in [0,1]^2$ for all $s$. 
	Moreover, because
	\[
	\overline \omega_1:= \frac{\omega_1^{-}+\omega_1^{+}}{2} \ge 
	\min \big\{\omega_1^{-},\omega_1^{+}\big\}
	\quad
	\text{and}
	\quad
	\overline \omega_2:= \frac{\omega_2^{-}+\omega_2^{+}}{2} \ge 
	\min \big\{\omega_2^{-},\omega_2^{+}\big\},
	\]
	we have	
	\begin{align*}
		& {\mathcal G}_2( \overline \omega_1, \overline \omega_1, \overline \omega_2,  \overline \omega_2  )
		= \min 
		\left\{
		\frac{ \overline \omega_1  \dx}{a_1},
		\frac{ \overline \omega_2 \dy}{a_2}
		\right\}
		\\
		& \quad 
		\ge
		\min \left\{ \frac{\omega_1^- \Delta x}{a_1}, \frac{\omega_1^+ \Delta x}{a_1} , \frac{\omega_2^- \Delta y}{a_2}, \frac{\omega_2^+ \Delta y}{a_2} \right\} ={\mathcal G}_2( \omega_1^-, \omega_1^+, \omega_2^-, \omega_2^+  ),
	\end{align*}
	which implies the symmetric CAD (\ref{eq:810}) is also an OCAD. Moreover, \eqref{eq:3113} implies \eqref{eq:weighteq2}. 
\end{proof}

Let $\mathbb B_\omega \subset \mathbb R^2$ denote the set of the boundary weights $(\overline \omega_1, \overline \omega_2)$ for all feasible symmetric CADs, namely, 
\begin{equation}\label{def:Bw}
	\mathbb{B}_\omega := \left \{( \overline \omega_1, \overline \omega_2) \in [0,1]^2:~ \exists \textrm{ symmetric CAD } 
\langle p \rangle_\Omega = 2\overline \omega_1 \langle p \rangle_{\Omega}^x+
2\overline \omega_2 \langle p \rangle_{\Omega}^y+
\sum_{s} \omega_s  \overline{ p( x^{(s)}, y^{(s)}) } \right\}.
\end{equation}
In fact, the set $\mathbb B_\omega $ is a projection of the set $\mathbb A_\omega$. 
Since $\mathbb A_\omega \subset [0,1]^4$ is compact (\Cref{lem:compact}) and convex (\Cref{lem:convex}), we know that 
$\mathbb B_\omega \subset [0,1]^2$ is also compact and convex. 
\Cref{fig:1363} illustrates the region $\mathbb B_\omega$ (shaded in green) for 
$\mathbb V^k=\mathbb{Q}^2$ or $\mathbb{Q}^3$ and $\mathbb V^k=\mathbb{P}^2$ or $\mathbb{P}^3$, respectively.

\begin{figure}[h!]
	\centerline{
		\begin{subfigure}[t][][t]{0.5\textwidth}
			\centering
			\includegraphics[width=0.96\textwidth]{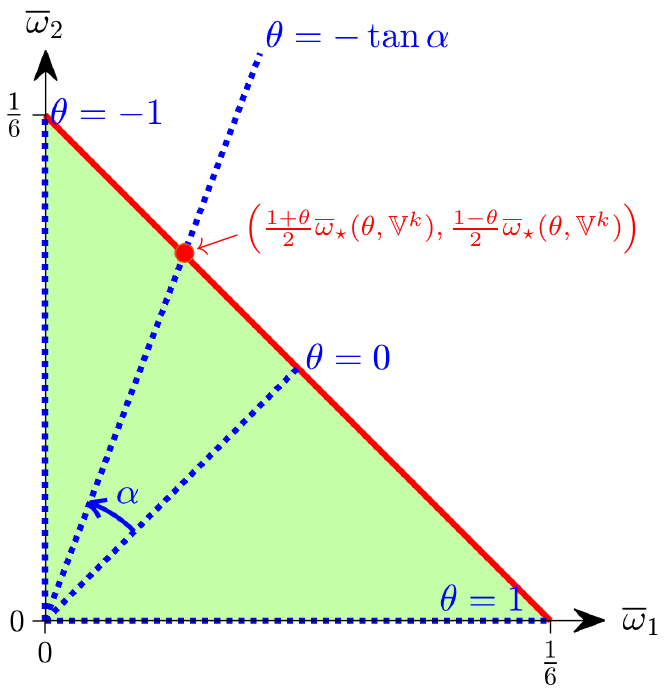}
			\caption{$\mathbb V^k=\mathbb{Q}^2$ or $\mathbb{Q}^3$.}
		\end{subfigure}
		\begin{subfigure}[t][][t]{0.5\textwidth}
			\centering
			\includegraphics[width=0.96\textwidth]{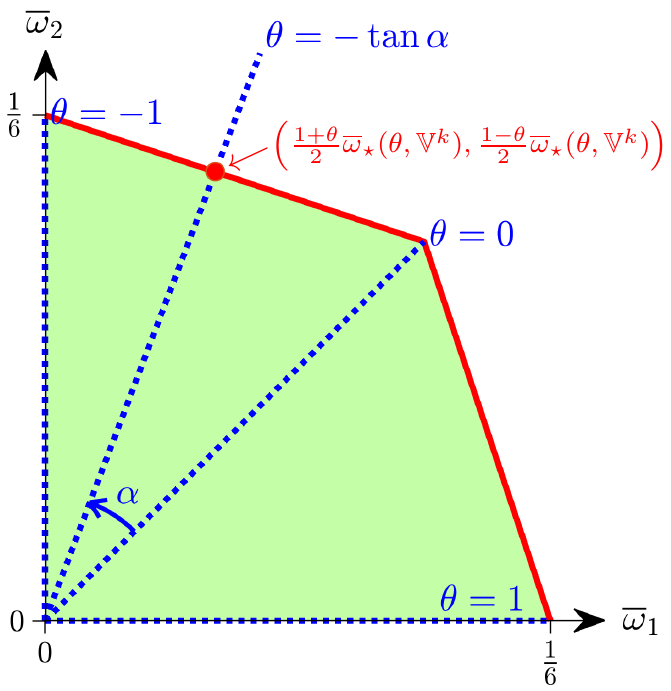}
			\caption{$\mathbb V^k=\mathbb{P}^2$ or $\mathbb{P}^3$.}
			\label{fig:P2P3}
		\end{subfigure}
	}
\caption{The region $\mathbb B_\omega$ (shaded in green) and the geometric interpretation of 
	the transformation \eqref{key112} for \Cref{prb:4.4}.}\label{fig:1363}
\end{figure}

Thanks to \Cref{lem:C2} and \Cref{thm:2D:symOCAD}, we only need to seek the symmetric OCAD on the reference cell $\Omega$ that maximizes 
\begin{equation}\label{obj12}
		\min 
	\left\{
	\frac{ \overline \omega_1  \dx}{a_1},
	\frac{ \overline \omega_2 \dy}{a_2}
	\right\}  \qquad  \mbox{with} \quad \frac{ \overline \omega_1 \Delta x}{a_1} = \frac{ \overline \omega_2 \Delta y}{a_2}.
\end{equation}	
Define 	
\begin{equation}\label{eq:eta}
\theta := \frac{a_1/\Delta x - a_2 / \Delta y }{a_1 / \Delta x + a_2 / \Delta y} \in [-1,1], 
\end{equation}
and introduce a new quantity $\Wmu$ to automatically meet the constraint $\frac{ \overline \omega_1 \Delta x}{a_1} = \frac{ \overline \omega_2 \Delta y}{a_2}$  as follows: 
\begin{equation}\label{key112}
	\overline \omega_1 = \frac12 \Wmu ( 1+ \theta ), \qquad \overline \omega_2 = \frac12 \Wmu (1-\theta). 
\end{equation}
See \Cref{fig:1363} for the geometric interpretation of the transformation \eqref{key112}.  
Then a symmetric CAD \eqref{eq:715} can be rewritten as 
\begin{equation}\label{eq:715b}
	\begin{aligned}
		\langle p \rangle_{\Omega} = 
		\Wmu \left[ ( 1+ \theta ) \langle p \rangle_{\Omega}^x+
	 (1-\theta) \langle p \rangle_{\Omega}^y \right] +
		\sum_{s=1}^S \omega_s  \overline{ p( x^{(s)}, y^{(s)}) } \qquad \forall p \in \mathbb V^k, 
	\end{aligned}
\end{equation}
and 
the objective function \eqref{obj12} is simplified to $ \Wmu / (a_1 / \Delta x + a_2 / \Delta y)$. 
	This procedure greatly simplifies the 2D OCAD problem to the following \Cref{prb:4.4}.

\begin{problem}[Symmetric 2D OCAD Problem]\label{prb:4.4}
	For $k \in \mathbb N_+$ and any given $\theta\in [-1,1]$, 
	find the symmetric OCAD on the reference cell $\Omega$ with the maximum boundary weight $\Wmu_\star(\theta,\mathbb{V}^k)$:
	\begin{equation}\label{eq:C2_CAD}
\begin{aligned}
			\langle p \rangle_{\Omega} = 
		\Wmu_\star(\theta,\mathbb{V}^k) \left[ ( 1+ \theta ) \langle p \rangle_{\Omega}^x+
		(1-\theta) \langle p \rangle_{\Omega}^y \right] +
		\sum_{s=1}^S \omega_s  \overline{ p( x^{(s)}, y^{(s)}) }
		\qquad \forall p \in \mathbb V^k,
\end{aligned}
	\end{equation}
	where $\omega_s > 0$ and $( x^{(s)}, y^{(s)}) \in [0,1]^2$ for all $s$. 
\end{problem}

In the following, we will mainly focus on the simplified \Cref{prb:4.4} on the reference cell $\Omega$. The resulting optimal CFL condition will be
\begin{equation}\label{eq:newCFL2Dsys}
	\Delta t \le c_0 \max{\mathcal G}_2 = \frac{ c_0 \Wmu_\star }{ a_1 / \Delta x + a_2 / \Delta y }. 
\end{equation}
Once \Cref{prb:4.4} is solved, we immediately obtain an OCAD on the rectangular cell $\Omega_{ij}$ with  
the corresponding nodes $\{( x_{ij}^{(s)}, y_{ij}^{(s)} )\}$ given by the inverse transformation of (\ref{eq:trans}).

\begin{lemma}[Monotonicity]\label{lem:monotonic}
	The optimal weight $\Wmu_\star(\theta,\mathbb{V}^{k})$ is monotonically non-increasing with respect to $k$, namely,  
	\begin{equation}\label{eq:monotonic}
		\Wmu_\star(\theta,\mathbb{V}^{k_1}) \ge \Wmu_\star(\theta,\mathbb{V}^{k_2}) \qquad \mbox{if} \quad k_1 < k_2.
	\end{equation} 
\end{lemma}

\begin{proof}
	If $k_1 < k_2$, then $\mathbb{V}^{k_1} \subseteq \mathbb{V}^{k_2}$. Any OCAD for $\mathbb{V}^{k_2}$ is a feasible CAD for $\mathbb{V}^{k_1}$. Hence we have \eqref{eq:monotonic}.
\end{proof}

\begin{lemma}\label{lem:2k2kp1}
	For any $k \in \mathbb N_+$, 
	the symmetric OCAD on $\Omega$ for the space $\mathbb{V}^{2k}$ is also a symmetric OCAD for the space $\mathbb{V}^{2k+1}$. Furthermore, 
\begin{equation}\label{eq:2k=2kp1}
		\Wmu_\star(\theta,\mathbb{V}^{2k+1}) = \Wmu_\star(\theta,\mathbb{V}^{2k}) \qquad \forall \theta \in [-1,1].
\end{equation}
\end{lemma}

\begin{proof}
	Let 
\begin{equation}\label{eq:sCAD2k}
		\langle p \rangle_{\Omega} = 
\Wmu_\star(\theta,\mathbb{V}^{2k}) \left[ ( 1+ \theta ) \langle p \rangle_{\Omega}^x+
(1-\theta) \langle p \rangle_{\Omega}^y \right] +  \sum_{s=1}^S \omega_s  \overline{ p( x^{(s)}, y^{(s)}) } 
\end{equation}
be a symmetric OCAD on the reference cell $\Omega$ for the space $\mathbb{V}^{2k}$. 
For any polynomial $q(x,y) = \sum_{ \| {\bf i} \| \le 2k+1} a_{i_1 i_2} x^{i_1} y^{i_2} \in \mathbb{V}^{2k+1}$, define the truncated polynomial $p(x,y) := \sum_{ \| {\bf i} \| \le 2k} a_{i_1 i_2} x^{i_1} y^{i_2} \in \mathbb{V}^{2k}$, where the norm $\| {\bf i} \| = i_1 + i_2$   for $\mathbb{V}^{2k}=\mathbb{P}^{2k}$ and $\| {\bf i} \| = \max\{ i_1, i_2\}$ for $\mathbb{V}^{2k}=\mathbb{Q}^{2k}$. Note that the polynomial $q(x,y)-p(x,y)$ is odd with respect to either $x$ or $y$. 
Thanks to the symmetry, we have  
$$
\langle q-p \rangle_{\Omega} = \langle  q-p \rangle_{\Omega}^x = \langle  q-p \rangle_{\Omega}^y = \overline{  (q-p) ( x^{(s)}, y^{(s)}) } = 0.  
$$
It then follows from \eqref{eq:sCAD2k} that  
\begin{equation}\label{eq:sCAD2kp1}
	\langle q \rangle_{\Omega} = 
	\Wmu_\star(\theta,\mathbb{V}^{2k}) \left[ ( 1+ \theta ) \langle q \rangle_{\Omega}^x+
	(1-\theta) \langle q \rangle_{\Omega}^y \right]+  \sum_{s=1}^S \omega_s  \overline{ q( x^{(s)}, y^{(s)}) } \quad \forall q \in \mathbb V^{2k+1}.
\end{equation}
This means \eqref{eq:sCAD2kp1} is a symmetric feasible CAD for $\mathbb{V}^{2k+1}$, and thus  $\Wmu_\star(\theta,\mathbb{V}^{2k}) \le \Wmu_\star(\theta,\mathbb{V}^{2k+1})$. 
On the other hand, \Cref{lem:monotonic} yields 
$
\Wmu_\star(\theta,\mathbb{V}^{2k}) \ge \Wmu_\star(\theta,\mathbb{V}^{2k+1}). 
$
Therefore, $\Wmu_\star(\theta,\mathbb{V}^{2k}) = \Wmu_\star(\theta,\mathbb{V}^{2k+1})$, and \eqref{eq:sCAD2kp1} is also a symmetric OCAD for $\mathbb{V}^{2k+1}$. The proof is completed. 
\end{proof}

\begin{remark}
	Thanks to \Cref{lem:2k2kp1}, we only need to seek the symmetric OCAD for $\mathbb{V}^{2k}$ of even degrees. 
\end{remark}

\begin{lemma}\label{lem:1525}
The symmetric CAD
\begin{equation}\label{eq:1527}
	\langle p \rangle_{\Omega} = \Wmu \left[ (1+\theta) \langle p \rangle_{\Omega}^x + (1-\theta) \langle p \rangle_{\Omega}^y \right] + \sum_{s=1}^S \omega_s \overline{p(x^{(s)},y^{(s)})}
\end{equation}
is feasible for $\mathbb{V}^k$ if and only if the CAD
\begin{equation}\label{eq:1530}
	\langle p \rangle_{\Omega} = \Wmu \left[ (1-\theta) \langle p \rangle_{\Omega}^x + (1+\theta) \langle p \rangle_{\Omega}^y \right] + \sum_{s=1}^S \omega_s \overline{p(y^{(s)},x^{(s)})}
\end{equation}
is feasible for $\mathbb{V}^k$. This implies the region $\mathbb B_\omega$ is 
symmetric with respect to the line $\overline \omega_1 = \overline \omega_2$.
\end{lemma}
\begin{proof}
The feasibility of a CAD, as defined in \Cref{def:2D_FCAD}, requires three conditions. The satisfaction of condition (ii) for CADs \eqref{eq:1527} and \eqref{eq:1530} is equivalent. 
The satisfaction of condition (iii) is also equivalent for both CADs, due to the geometric symmetry of $\Omega$. 

Let us verify the equivalence for condition (i). Assume the condition (i) of \eqref{eq:1527} is satisfied. For any $p(x,y) \in \mathbb{V}^k$, $q(x,y) := p(y,x)$ is also a polynomial in $\mathbb{V}^k$. Note that
\[
	\langle p \rangle_{\Omega} = \langle q \rangle_{\Omega}, 		\quad
	\langle p \rangle_{\Omega}^x = \langle q \rangle_{\Omega}^y,		\quad
	\langle p \rangle_{\Omega}^y = \langle q \rangle_{\Omega}^x,		\quad
	\overline{p(y^{(s)},x^{(s)})} = \overline{q(x^{(s)},y^{(s)})},
\]
we have
\begin{align*}
	& \langle p \rangle_{\Omega} - \Wmu \left[ (1-\theta) \langle p \rangle_{\Omega}^x + (1+\theta) \langle p \rangle_{\Omega}^y \right] - \sum_{s=1}^S \omega_s \overline{p(y^{(s)},x^{(s)})} 
	\\
	& \quad = 
	\langle q \rangle_{\Omega} - \Wmu \left[ (1-\theta) \langle q \rangle_{\Omega}^y + (1+\theta)  \langle q \rangle_{\Omega}^x \right] - \sum_{s=1}^S \omega_s \overline{q(x^{(s)},y^{(s)})} = 0.
\end{align*}
Thus, the condition (i) of \eqref{eq:1527} is a sufficient condition of the condition (i) of \eqref{eq:1530}. The necessity can be similarly proved, and the proof of \Cref{lem:1525} is completed.
\end{proof}

\begin{lemma}\label{thm:sym_theta}
The function $\Wmu_\star(\theta,\mathbb{V}^k)$ is even with $\theta$, namely, 
$$
	\Wmu_\star(-\theta,\mathbb{V}^k) = \Wmu_\star(\theta,\mathbb{V}^k) \quad \theta \in [-1,1].
$$
\end{lemma}
\begin{proof}
Consider a symmetric OCAD \eqref{eq:C2_CAD} for $\mathbb{V}^k$. 
By \Cref{lem:1525}, the following symmetric CAD is also feasible for $\mathbb{V}^k$: 
\begin{equation}\label{eq:nTheta}
		\langle p \rangle_{\Omega} = \Wmu_\star(\theta,\mathbb{V}^k) \left[ (1-\theta) \langle p \rangle_{\Omega}^x + (1+\theta) \langle p \rangle_{\Omega}^y \right] + \sum_{s=1}^S \omega_{s} \overline{p(y^{(s)},x^{(s)})},
\end{equation}
which implies $\Wmu_\star(-\theta,\mathbb{V}^k) \ge \Wmu_\star(\theta,\mathbb{V}^k)$. 
Similar, one has $\Wmu_\star(\theta,\mathbb{V}^k) \ge \Wmu_\star(-\theta,\mathbb{V}^k)$. 
Hence $
\Wmu_\star(-\theta,\mathbb{V}^k) = \Wmu_\star(\theta,\mathbb{V}^k)
$, and \eqref{eq:nTheta} is a symmetric OCAD for $(-\theta,\mathbb{V}^k)$.
\end{proof}


\begin{remark}\label{cor:1575}
Thanks to \Cref{thm:sym_theta}, we only need to investigate the symmetric 2D OCAD problem (\Cref{prb:4.4}) for $\theta \in [-1,0]$. For $\theta \in [0,1]$, the corresponding symmetric OCAD can then be directly obtained by 
\eqref{eq:nTheta}.
\end{remark}

\subsection{Optimality criteria}

Let $\mathbb{V}_+^k$ denote the set of all the polynomials in $\mathbb{V}^k$ that is non-negative over $\Omega$. The following lemma establishes the close connection between $\Wmu_\star(\theta,\mathbb{V}^k)$ and $\mathbb{V}_+^k$.
\begin{lemma}\label{lem:430}
	Given $k\in\mathbb{N}_+$, for any $p \in \mathbb{V}_+^k$, define
	\[
	\phi (p;\theta):=
	\begin{dcases}
		\frac{  \langle p\rangle_{\Omega}}
		{ (1+\theta) \langle p \rangle^x_{\Omega}+
			(1-\theta) \langle p \rangle^y_{\Omega}} \quad 
		& \text{ if } (1+\theta) \langle p \rangle^x_{\Omega}+
		(1-\theta) \langle p \rangle^y_{\Omega}\neq 0, \\
		+ \infty \quad 
		& \text{ if }  (1+\theta) \langle p \rangle^x_{\Omega}+
		(1-\theta) \langle p \rangle^y_{\Omega}= 0,
	\end{dcases}
	\]
	then  
	\begin{equation}\label{eq:1579} 
	 \phi^\star (\theta,\mathbb{V}^k_{+}) :=  \inf_{ p \in \mathbb{V}_+^k  }\phi (p;\theta) \ge   \Wmu_\star(\theta,\mathbb{V}^k).
	\end{equation}
\end{lemma}

\begin{proof}
	For a symmetric OCAD \eqref{eq:C2_CAD} on $\Omega$, we have for all $p \in \mathbb{V}_+^k \setminus \{0\}$ that 
	\begin{align*}
		\langle p \rangle_{\Omega} &= 
	\Wmu_\star(\theta,\mathbb{V}^k) \left[ ( 1+ \theta ) \langle p \rangle_{\Omega}^x+
	(1-\theta) \langle p \rangle_{\Omega}^y \right] +  \sum_{s=1}^S \omega_s  \overline{ p( x^{(s)}, y^{(s)}) }
	\\
	& \ge 
	\Wmu_\star(\theta,\mathbb{V}^k) \left[ ( 1+ \theta ) \langle p \rangle_{\Omega}^x+
	(1-\theta) \langle p \rangle_{\Omega}^y \right],
	\end{align*}
	which implies 
	$$
	 \phi (p;\theta) \ge \Wmu_\star(\theta,\mathbb{V}^k) \qquad \forall p \in \mathbb{V}_+^k. 
	$$
	Taking the infimum for $p$ on the above inequality completes the proof.  
\end{proof}

For any $p \in \mathbb{V}_+^k$, $ \phi (p;\theta)$ is an upper bound of $\Wmu_\star(\theta,\mathbb{V}^k)$.
As suggested by \Cref{lem:430}, one may minimize $\phi(p;\theta)$ in $\mathbb{V}_+^k$ and obtain $ \phi^\star (\theta,\mathbb{V}^k_{+})$ as a sharp upper bound for $\Wmu_\star(\theta,\mathbb{V}^k)$. Very interestingly, for all the 2D OCADs found in this paper, we discover that 
$ \phi^\star (\theta,\mathbb{V}^k_{+}) = \Wmu_\star(\theta,\mathbb{V}^k)$, which will be further discussed in \Cref{con:2070}. This motivates us to consider the following problem. 

\begin{problem}\label{prb:2}
	Given $k\in\mathbb{N}^+$ and $\theta \in [-1,1]$, minimize $\phi(p;\theta)$ for $p \in \mathbb{V}_+^k$.
\end{problem}

For given $\theta \in [-1,1]$ and $\mathbb{V}_+^k$, if there exists a $p^\star \in \mathbb{V}_+^k \setminus \{0\}$ such that 
\begin{equation}\label{eq:1650}
p^\star:= \mathop{\arg\min}\limits_{p\in \mathbb{V}_+^k} \phi(p;\theta),
\end{equation}
then we call $p^\star$ a \textit{critical positive polynomial} for $\phi^\star(\theta,\mathbb{V}^k_+)$.

\begin{remark}
It is worth noting that, for given $\theta \in [-1,1]$ and $k\in\mathbb{N}^+$, 
the critical positive polynomial for $\phi^\star(\theta,\mathbb{V}^k_+)$ and the symmetric OCAD for $\mathbb V^k$ can be possibly not unique. 
\end{remark}

\begin{lemma}\label{lem:notUnique}
If two different nonzero nonnegative polynomials, $p_1^\star,p_2^\star  \in \mathbb{V}_+^k \setminus \{0\}$, are both critical positive polynomials for $\phi^\star(\theta,\mathbb{V}^k_+)$, namely, 
\begin{equation}\label{eq:phi366}
	\phi^\star (\theta,\mathbb{V}^k_{+}) =  \phi (p_1^\star;\theta) =  \phi (p_2^\star;\theta), 
\end{equation} 
then for any $\alpha_1,\alpha_2\in \mathbb R_+$, the positively combined polynomial 
$$
\alpha_1 p_1^\star + \alpha_2 p_2^\star  \in \mathbb{V}_+^k \setminus \{0\}
$$
is also a critical positive polynomial for $\phi^\star(\theta,\mathbb{V}^k_+)$.
\end{lemma}

\begin{proof}
We observe for any $\alpha_1,\alpha_2\in \mathbb R_+$ that 
\begin{align*}
	\frac1{ \phi (\alpha_1 p_1^\star + \alpha_2 p_2^\star;\theta) }
	&= \frac
	{ (1+\theta) \langle \alpha_1 p_1^\star + \alpha_2 p_2^\star \rangle^x_{\Omega}+
		(1-\theta) \langle \alpha_1 p_1^\star + \alpha_2 p_2^\star \rangle^y_{\Omega}}{  \langle \alpha_1 p_1^\star + \alpha_2 p_2^\star  \rangle_{\Omega}}
	\\
	&= \frac { \alpha_1  \langle  p_1^\star  \rangle_{\Omega} }  { \alpha_1  \langle  p_1^\star  \rangle_{\Omega} +   \alpha_2 \langle   p_2^\star  \rangle_{\Omega}   }  
	\frac1{ \phi ( p_1^\star;\theta) } 
	+ \frac { \alpha_2  \langle  p_1^\star  \rangle_{\Omega} }  { \alpha_1  \langle  p_1^\star  \rangle_{\Omega} +   \alpha_2 \langle   p_2^\star  \rangle_{\Omega}   }  
	\frac1{ \phi ( p_2^\star;\theta) }
	\\
	& \overset{\eqref{eq:phi366}}{=} \frac1{ \phi^\star (\theta,\mathbb{V}^k_{+}) },
\end{align*}
which yields $\phi (\alpha_1 p_1^\star + \alpha_2 p_2^\star;\theta) = \phi^\star (\theta,\mathbb{V}^k_{+}) $. Hence $\alpha_1 p_1^\star + \alpha_2 p_2^\star$ is also a critical positive polynomial.
\end{proof}


%
%
%
%
%
%
%

We discover that the symmetric OCAD problem (\Cref{prb:4.4}) is strongly connected with 
\Cref{prb:2}. Their connection leads to the following optimality criteria, 
which are very useful for examining the optimality of a feasible symmetric CAD: 
\begin{equation}\label{eq:1652}
	\langle p \rangle_{\Omega} = 
	\Wmu \left[ (1+\theta) \langle p \rangle_{\Omega}^x + (1-\theta) \langle p \rangle_{\Omega}^y \right] +  \sum_{s=1}^S \omega_s   \overline{ p( x^{(s)}, y^{(s)}) } \qquad \forall p \in \mathbb{V}^k.
\end{equation}


\begin{theorem}[Optimality Criterion \#1]\label{thm:1648}
	If $\Wmu =  \phi^\star(\theta,\mathbb{V}^k_+)$, then 
	$\Wmu_\star(\theta,\mathbb{V}^k) =  \phi^\star(\theta,\mathbb{V}^k_+)$ and 
	the CAD \eqref{eq:1652} is a symmetric OCAD for $\mathbb{V}^k$.
\end{theorem}
\begin{proof}
	\Cref{lem:430} tells us that 
		$$
	\phi^\star(\theta,\mathbb{V}^k_+) \ge    \Wmu_\star(\theta,\mathbb{V}^k) \ge \Wmu. 
	$$
	If $\Wmu = \phi^\star(\theta,\mathbb{V}^k)$, then $\Wmu = \Wmu_\star(\theta,\mathbb{V}^k) =  \phi^\star(\theta,\mathbb{V}^k)$ and 
	the symmetric CAD \eqref{eq:1652} is optimal for $\mathbb{V}^k$.
\end{proof}

\begin{theorem}[Optimality Criterion \#2]\label{lem:511}
	For a given $\theta \in [-1,1]$, if there exists a polynomial $p^\star \in \mathbb{V}_+^k \setminus \{0\}$ such that 
	$p^\star ( 
	  x^{(s)}, y^{(s)})=0$ for all $s$, 
	then we have: 
	\begin{itemize}
		\item $\Wmu_\star(\theta,\mathbb{V}^k) =  \phi^\star(\theta,\mathbb{V}^k_+)$;
		\item the CAD \eqref{eq:1652} is a symmetric OCAD for $\mathbb{V}^k$;
		\item the polynomial $p^\star$ is a critical positive polynomial for $\phi^\star(\theta,\mathbb{V}^k_+)$.
	\end{itemize}
\end{theorem}
\begin{proof}
	Taking $p=p^\star$ in \eqref{eq:1652} gives   
	$$
	\langle p^\star \rangle_{\Omega} = 	\Wmu \left[ (1+\theta) \langle p^\star \rangle_{\Omega}^x + (1-\theta) \langle p^\star \rangle_{\Omega}^y \right].
	$$
	Since $p^\star\not\equiv 0$, we have $\langle p^\star \rangle_\Omega \neq 0$ and $\phi(p^\star;\theta) = \Wmu$. 
	Thanks to \Cref{lem:430}, we obtain 
	$$
	\Wmu =  \phi(p^\star;\theta) \ge  \inf_{  p \in \mathbb{V}_+^k  }\phi (p;\theta) \ge   \Wmu_\star(\theta,\mathbb{V}^k) \ge \Wmu, 
	$$
	which implies $\Wmu = \Wmu_\star(\theta,\mathbb{V}^k)=\phi(p^\star;\theta) = \min\limits_{ p \in \mathbb{V}_+^k  }\phi (p;\theta) = \phi^\star(\theta,\mathbb{V}^k_+)$. Hence the CAD \eqref{eq:1652} is optimal, and  $p^\star$ is a critical positive polynomial. 
\end{proof}

\begin{theorem}[Optimality Criterion \#3]\label{lem:512}
	For a given $\theta \in [-1,1]$, if there exists a polynomial $q_\star \in \mathbb{V}^{ \left \lfloor \frac{k}2 \right \rfloor } \setminus \{0\}$ such that $q_\star ( 
	x^{(s)}, y^{(s)})=0$ for all $s$,  
	then we have: 
	\begin{itemize}
		\item $\Wmu_\star(\theta,\mathbb{V}^k) =  \phi^\star(\theta,\mathbb{V}^k_+) = \phi(q_\star^2;\theta) =   \inf_{q \in \mathbb{V}^{  \lfloor \frac{k}2  \rfloor } } \phi(q^2;\theta) $;
		\item the CAD \eqref{eq:1652} is a symmetric OCAD for $\mathbb{V}^k$;
		\item the polynomial $p^\star:=q_\star^2$ is a critical positive polynomial for $\phi^\star(\theta,\mathbb{V}^k_+)$.
	\end{itemize}
\end{theorem}
\begin{proof}
	Since $q_\star \in \mathbb{V}^{ \left \lfloor \frac{k}2 \right \rfloor } \setminus \{0\}$, we have $p^\star:=q_\star^2 \in \mathbb{V}_+^k \setminus \{0\}$. 
	Moreover, $p^\star ( 
	x^{(s)}, y^{(s)})=0$ for all $s$. 
	By \Cref{lem:511}, we have $\Wmu_\star(\theta,\mathbb{V}^k) =  \phi^\star(\theta,\mathbb{V}^k_+) = \phi(q_\star^2;\theta) $, 
	 \eqref{eq:1652} is a symmetric OCAD for $\mathbb{V}^{k}$, and $p^\star$ is a critical positive polynomial. 
	 Noting $q^2 \in \mathbb{V}_+^{k}$ for all $q \in \mathbb{V}^{ \left \lfloor \frac{k}2 \right \rfloor } $, we have 
	 $(\mathbb{V}^{ \left \lfloor \frac{k}2 \right \rfloor } )^2 \subset  \mathbb{V}_+^{k}$, which implies 
	 $$ \phi(q_\star^2;\theta) \ge \inf_{q \in \mathbb{V}^{  \lfloor \frac{k}2  \rfloor } } \phi(q^2,\theta) \ge 
	 \inf_{p \in \mathbb{V}_+^{k}} \phi(p;\theta) = \phi^\star(\theta,\mathbb{V}^k_+).
	 $$
	 Therefore, $\phi(q_\star^2;\theta) = \inf_{q \in \mathbb{V}^{  \lfloor \frac{k}2  \rfloor } } \phi(q^2;\theta) = \phi^\star(\theta,\mathbb{V}^k_+) $. The proof is completed. 
\end{proof}

\begin{theorem}[Optimality Criterion \#4]\label{lem:513}
	Define $\phi^\star(\theta,(\mathbb{V}^{ \left \lfloor \frac{k}2 \right \rfloor })^2) := \inf_{q \in \mathbb{V}^{  \lfloor \frac{k}2  \rfloor } } \phi(q^2;\theta)$. 
	If there is a feasible CAD \eqref{eq:1652} with $\Wmu =  \phi^\star(\theta,(\mathbb{V}^{ \left \lfloor \frac{k}2 \right \rfloor })^2)$,  
	then we have: 
	\begin{itemize}
		\item $\Wmu_\star(\theta,\mathbb{V}^k) =  \phi^\star(\theta,\mathbb{V}^k_+) =   \phi^\star(\theta,(\mathbb{V}^{ \left \lfloor \frac{k}2 \right \rfloor })^2)$;
		\item the CAD \eqref{eq:1652} is a symmetric OCAD for $\mathbb{V}^k$.
	\end{itemize}
\end{theorem}

\begin{proof}
	Recalling $(\mathbb{V}^{ \left \lfloor \frac{k}2 \right \rfloor } )^2 \subset  \mathbb{V}_+^{k}$ 
	and \eqref{eq:1579}, we have 
	$$
	\Wmu \le \Wmu_\star(\theta,\mathbb{V}^k) \le   \phi^\star(\theta,\mathbb{V}^k_+) \le   \phi^\star(\theta,(\mathbb{V}^{ \left \lfloor \frac{k}2 \right \rfloor })^2). 
	$$
	If $\Wmu =  \phi^\star(\theta,(\mathbb{V}^{ \left \lfloor \frac{k}2 \right \rfloor })^2)$, we must have 
	$\Wmu=\Wmu_\star(\theta,\mathbb{V}^k) =  \phi^\star(\theta,\mathbb{V}^k_+) =   \phi^\star(\theta,(\mathbb{V}^{ \left \lfloor \frac{k}2 \right \rfloor })^2)$, 
	which implies the CAD \eqref{eq:1652} is a symmetric OCAD for $\mathbb{V}^k$. 
\end{proof}


\begin{remark}\label{rem:diff}
	Notice that 
	$\phi(p;\theta)$ satisfies $\phi(\alpha p;\theta) =  \phi(p;\theta)$ for any $\alpha > 0$. 
	Thus, by normalization with $\langle p \rangle_{\Omega} = 1$,  
	\Cref{prb:2} 
	can be equivalently cast into 
	\begin{equation}\label{key3331}
		\begin{aligned}
			\mathop{\max}\limits_{p\in \mathbb{V}^k} \quad  &
			(1+\theta) \langle p \rangle^x_{\Omega} +(1-\theta) \langle p \rangle^y_{\Omega}  \\
			\textrm{subject to:} \quad & p(x,y) \ge 0 \quad \forall (x,y) \in \Omega, \\
			& \langle p \rangle_{\Omega} = 1.
		\end{aligned}
	\end{equation}
	Obviously, this problem \eqref{key3331} is a convex semi-infinite optimization problem, which actually 
	falls into the category of polynomial optimization, and 
	its solution is very difficult to obtain. 
	Lasserre \cite{LasserreJeanB2007Aspa} proposed the semidefinite relaxation method to numerically solve polynomial optimization problems like \eqref{key3331}, by relaxing the problem into a sequence of semidefinite programming (SDP) problems with increasing complexity. 
	See \cite{Lasserre2010_MPP, Lasserre2015_IPS} for a comprehensive introduction of polynomial optimization and the semidefinite relaxation method. 
\end{remark}

Due to the challenges in solving \eqref{key3331}, it seems difficult to seek 2D OCADs directly based on the optimality criterion \#1. 
Fortunately, we 
find an explicit analytical formula for $\phi^\star(\theta,(\mathbb{V}^{ \left \lfloor \frac{k}2 \right \rfloor })^2) := \inf_{q \in \mathbb{V}^{  \lfloor \frac{k}2  \rfloor } } \phi(q^2;\theta)$ with $k \in \mathbb N_+$, as derived below. 
This will inspire us to understand, construct, and verify 2D OCADs via the above optimality criteria \#3 and \#4.

Define $D:= \dim  \mathbb{V}^{ \left \lfloor \frac{k}2 \right \rfloor }$. Denote 
the monomial  basis of $\mathbb{V}^{ \left \lfloor \frac{k}2 \right \rfloor }$ by ${\bm  b}(x,y):=(b_1(x,y),\dots,b_D(x,y))^\top$. Define three $D\times D$ symmetric matrices $\mathbf{M}_\Omega  $, $\mathbf{M}_\Omega^x$, and $\mathbf{M}_\Omega^y$ by 
\begin{equation}\label{eq:2040}
	[\mathbf{M}_\Omega  ]_{ij}   = \left \langle b_i b_j \right \rangle_\Omega  , \quad
	[\mathbf{M}_\Omega^x]_{ij}   = \left \langle b_i b_j \right \rangle_\Omega^x, \quad
	[\mathbf{M}_\Omega^y]_{ij}   = \left \langle b_i b_j \right \rangle_\Omega^y, \quad i,j = 1,\dots,D.
\end{equation}
For any polynomial $q \in \mathbb{V}^{ \left \lfloor \frac{k}2 \right \rfloor } \setminus  \{0\}$, it has a unique expansion:
\[
q(x,y) = \sum_{i = 1}^{D} q_i b_i(x,y) = {\bm q}^\top {\bm  b}(x,y).
\]
With the matrices defined in \eqref{eq:2040}, one can express $\langle q^2 \rangle_\Omega$, $\langle q^2 \rangle_\Omega^x$ and $\langle q^2 \rangle_\Omega^y$ in quadratic forms:
\begin{equation}\label{eq:398}
	\langle q^2 \rangle_{\Omega}   = {\bm{q}}^\top \mathbf{M}_{\Omega}   \; {\bm{q}}, \qquad
	\langle q^2 \rangle_{\Omega}^x = {\bm{q}}^\top \mathbf{M}_{\Omega}^x \; {\bm{q}}, \qquad
	\langle q^2 \rangle_{\Omega}^y = {\bm{q}}^\top \mathbf{M}_{\Omega}^y \; {\bm{q}}. 
\end{equation}
If ${\bm q} \neq {\bf 0}$ or equivalently $q \not\equiv 0$, then $\langle q^2 \rangle_{\Omega} >0$,  $\langle q^2 \rangle_{\Omega}^x \ge 0$, and $\langle q^2 \rangle_{\Omega}^y \ge 0$. 
Thus, the symmetric matrix $\mathbf{M}_{\Omega}$ is positive definite, 
and the matrices $\mathbf{M}_{\Omega}^x$ and $\mathbf{M}_{\Omega}^y$ are  
positive semi-definite so that 
the matrix 
$$\mathbf{M}_\theta := (1+\theta)\mathbf{M}_{\Omega}^x + (1-\theta)\mathbf{M}_{\Omega}^y$$ 
is positive semi-definite for any $\theta \in [-1,1]$.

\begin{theorem}\label{lem:2080}
	For any $\theta \in [-1,1]$, it holds that 
	\begin{equation}\label{eq:122}
		\phi^\star \left(\theta,(\mathbb{V}^{ \left \lfloor \frac{k}2 \right \rfloor })^2\right) 
		=\phi ( q_\star^2; \theta  )  = \frac{1}{  \left\| \mathbf{M}_{\Omega}^{-\frac12} \mathbf{M}_\theta \mathbf{M}_{\Omega}^{-\frac12}  \right\|_2   },
	\end{equation}
	namely, 
	where 
	$\| \cdot \|_2$ denotes the 2-norm of matrix, 
	and the critical positive polynomial 
	\begin{equation}\label{eq:2065}
	p_\star (x,y)= q_\star^2(x,y) \quad \mbox{with} \quad q_\star (x,y) :=  \widehat {\bm{q}}_\theta^\top \mathbf{M}_{\Omega}^{-\frac12} {\bm b}(x,y)  
	\end{equation}
	where $\widehat {\bm{q}}_\theta$ is 
	an eigenvector corresponding to the spectral radius of $\mathbf{M}_{\Omega}^{-\frac12} \mathbf{M}_\theta \mathbf{M}_{\Omega}^{-\frac12}$. 
	In other words, $\phi^\star\left(\theta,(\mathbb{V}^{ \left \lfloor \frac{k}2 \right \rfloor })^2\right)$ is the smallest real root of the following polynomial
		\[
		\mathcal{F}_\theta(\phi) := \det \left( \mathbf{M}_\Omega - \phi \mathbf{M}_\theta \right).
		\]
\end{theorem}

\begin{proof}
	For any $q \in \mathbb{V}^{ \left \lfloor \frac{k}2 \right \rfloor } \setminus  \{0\}$,  
	denote by $\bm q$ the associated expansion coefficients. 
	Noting that $\mathbf{M}_{\Omega}$ is positive definite, we define 
	$\widehat {\bm q} := \mathbf{M}_{\Omega}^\frac12  {\bm q} \neq {\bf 0}$. 
	It follows from \eqref{eq:398} that 
	\begin{align*}
		\phi(q^2;\theta) = \frac
		{{\bm{q}}^\top \mathbf{M}_{\Omega}{\bm{q}}}
		{{\bm{q}}^\top \big( \theta \mathbf{M}^x_{\Omega} + (1-\theta) \mathbf{M}^y_{\Omega} \big) {\bm{q}}} = \frac{{\bm{q}}^\top \mathbf{M}_{\Omega}{\bm{q}}} {{\bm{q}}^\top \mathbf{M}_{\theta}{\bm{q}}} = \frac{ \| \widehat {\bm q}\|^2 } 
		{ \widehat {\bm{q}}^\top  \mathbf{M}_{\Omega}^{-\frac12} \mathbf{M}_\theta \mathbf{M}_{\Omega}^{-\frac12} \widehat {\bm{q}}}.
	\end{align*}
	Therefore, 
	\begin{align*}
		\phi^\star \left(\theta,(\mathbb{V}^{ \left \lfloor \frac{k}2 \right \rfloor })^2 \right) & = \inf_{ 0 \neq  q \in \mathbb{V}^{  \lfloor \frac{k}2  \rfloor } } \phi(q^2;\theta) = \inf_{{\bm 0}\neq{\bm q}\in\mathbb{R}^D } 
		\frac
		{{\bm{q}}^\top \mathbf{M}_{\Omega}{\bm{q}}}
		{{\bm{q}}^\top  \mathbf{M}_\theta {\bm{q}}} 
		=  
		\inf_{{\bm 0}\neq \widehat {\bm q}\in\mathbb{R}^D } 
		\frac{ \| \widehat {\bm q}\|^2 } 
		{ \widehat {\bm{q}}^\top  \mathbf{M}_{\Omega}^{-\frac12} \mathbf{M}_\theta \mathbf{M}_{\Omega}^{-\frac12} \widehat {\bm{q}}}
		\\
		&	= 
		\left( \sup_{{\bm 0}\neq \widehat {\bm q}\in\mathbb{R}^D }  \frac
		{ \widehat {\bm{q}}^\top  \mathbf{M}_{\Omega}^{-\frac12} \mathbf{M}_\theta \mathbf{M}_{\Omega}^{-\frac12} \widehat {\bm{q}}}{ \| \widehat {\bm q}\|^2 } \right)^{-1}
		= \frac{1}{  \left\| \mathbf{M}_{\Omega}^{-\frac12} \mathbf{M}_\theta \mathbf{M}_{\Omega}^{-\frac12}  \right\|_2    } = \phi(q_\star^2;\theta).
	\end{align*}
	It implies $\phi^\star \left(\theta,(\mathbb{V}^{ \left \lfloor \frac{k}2 \right \rfloor })^2\right)^{-1}$ is the largest eigenvalue of $\mathbf{M}_{\Omega}^{-\frac12} \mathbf{M}_\theta \mathbf{M}_{\Omega}^{-\frac12}$, consequenctly, is the largest root of characteristic polynomial
	\[
		\det \left[ \lambda \mathbf{I} - \mathbf{M}_{\Omega}^{-\frac12} \mathbf{M}_\theta \mathbf{M}_{\Omega}^{-\frac12} \right] = 0
	\]
	which is equilvalent to
	\[
		\det\left[ \lambda \mathbf{M}_{\Omega} - \mathbf{M}_{\theta} \right] = 0.
	\]
	Thus, $\phi^\star \left(\theta,(\mathbb{V}^{ \left \lfloor \frac{k}2 \right \rfloor })^2\right) $ is the smallest real root of 
	\[
		\det\left[ \mathbf{M}_{\Omega} - \phi \mathbf{M}_{\theta} \right] = 0.
	\]
	The proof is completed. 
\end{proof}

\subsection{More critical findings and geometric insights}

Based on extensive numerical experiments, we achieve two interesting findings, which are critical for further understanding, constructing, and verifying the 2D OCADs.

First, we discover that, for $\mathbb V^k$ being either $\mathbb Q^k$ or $\mathbb P^k$, 
the inequality \eqref{eq:1579} should be an equality. 

\begin{conjecture}\label{con:2070}
	For $\mathbb V^k$ being either $\mathbb Q^k$ or $\mathbb P^k$ with 
	$k \in \mathbb{N}_+$, it holds that 
	\begin{equation}\label{eq:conject1}
		\Wmu_\star(\theta,\mathbb{V}^k) = \phi^\star (\theta,\mathbb{V}^k_{+}) \qquad \forall \theta \in [-1,1],
	\end{equation}
	and there always exists a critical positive polynomial for $\phi^\star (\theta,\mathbb{V}^k_{+})$. 
\end{conjecture}

Furthermore, we observe that there exists a critical positive polynomial in the squared form. 
 
\begin{conjecture}\label{con:2071}
	For $\mathbb V^k$ being either $\mathbb Q^k$ or $\mathbb P^k$ with 
	$k \in \mathbb{N}_+$, it holds that 
	\[
	\phi^\star(\theta,\mathbb{V}^{k}_+) = \phi^\star(\theta,(\mathbb{V}^{ \left \lfloor \frac{k}2 \right \rfloor })^2) := \inf_{q \in \mathbb{V}^{  \lfloor \frac{k}2  \rfloor } } \phi(q^2;\theta) \qquad \forall \theta \in [-1,1],
	\]
	and there exists a critical positive polynomial 
	$p^\star = q_\star^2$ with $q_\star \in \mathbb{V}^{  \lfloor \frac{k}2  \rfloor } \setminus \{0\}$. 
\end{conjecture}

We can show that for all the 2D OCADs found in this paper, the above two conjectures are true. 
More specifically, we will prove these two conjectures for the spaces $\mathbb Q^k$ with any $k \in \mathbb N_+$, and for the spaces $\mathbb P^k$ with $1	 \le k \le 7$. 
Numerical evidence has demonstrated the validity of these conjectures for $\mathbb P^k$ spaces with higher $k \in \{8,9,\dots,15\}$, although a rigorous analytical proof is not available yet. 
It is worth noting that if the optimality criteria in \Cref{lem:511,lem:512,lem:513} work, 
then Conjecture \ref{con:2070} and \ref{con:2071} will get proved immediately.

\begin{figure}[h!]
	\centerline{
		\begin{subfigure}[t][][t]{0.5\textwidth}
			\centering
			\includegraphics[width=0.9\textwidth]{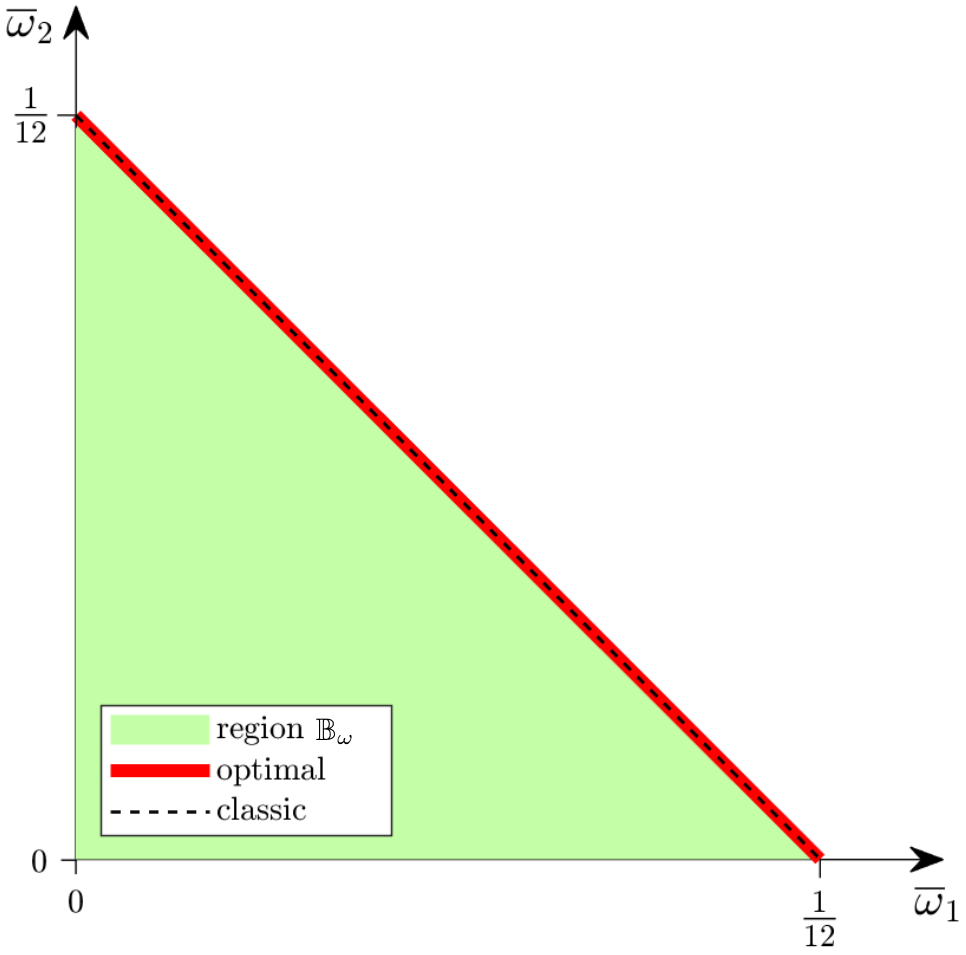}
			\caption{$\mathbb{Q}^4$ and $\mathbb{Q}^5$.}
		\end{subfigure}
		\begin{subfigure}[t][][t]{0.5\textwidth}
			\centering
			\includegraphics[width=0.9\textwidth]{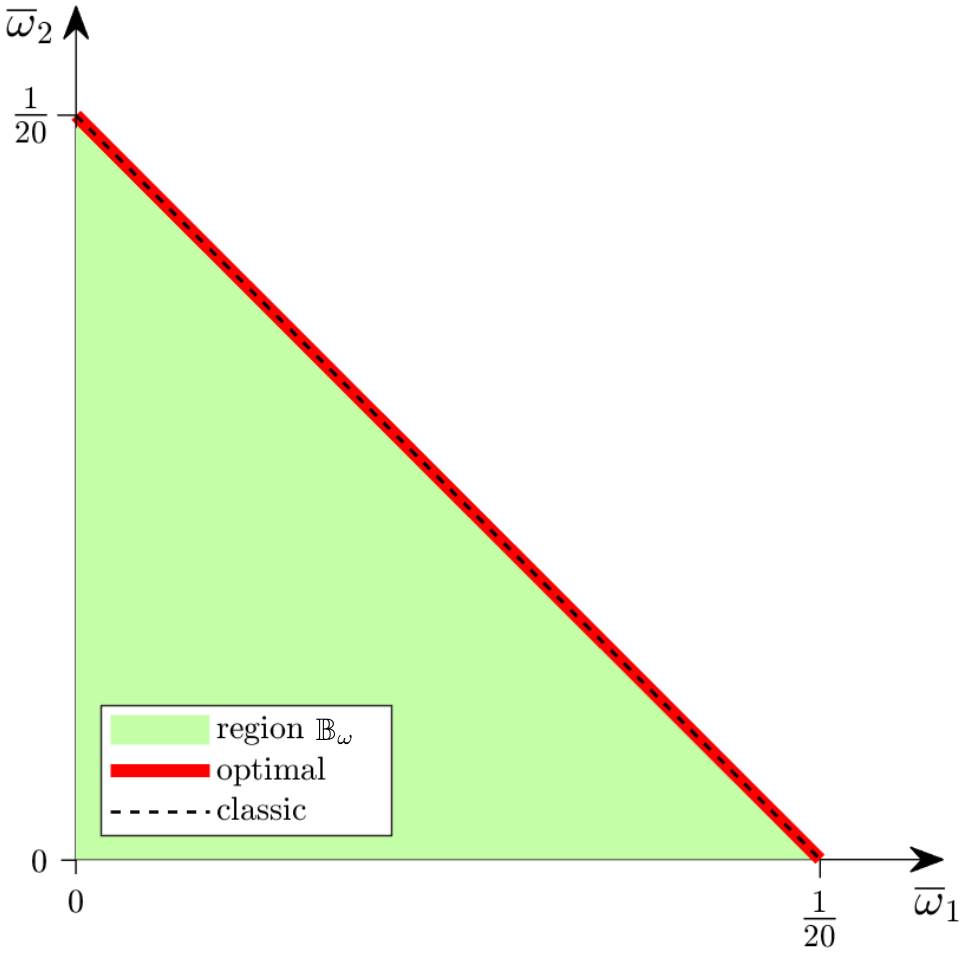}
			\caption{$\mathbb{Q}^6$ and $\mathbb{Q}^7$.}
		\end{subfigure}
	}
	\caption{Geometric illustration of the region $\mathbb{B}_\omega$, the boundary weights of the 
		classic CAD, the boundary weights of the OCAD     
		for $\mathbb{Q}^k$ spaces: the boundary weights of the classic CAD lies on $\partial \mathbb{B}$ and coincide with the boundary weights of the OCAD. }\label{fig:1939}
\end{figure}

For a given space $\mathbb V^k$, we define the following region 
\begin{equation}\label{eq:Bw}
	\mathbb B^\star_\omega := \left\{ 
	(\overline \omega_1, \overline \omega_2)\in [0,1]^2:~ 
	\langle p \rangle_{\Omega} - 
	2\overline \omega_1 \langle p \rangle_{\Omega}^x -
	2\overline \omega_2 \langle p \rangle_{\Omega}^y \ge 0~~~\forall p \in \mathbb V^k_+
	 \right \}. 
\end{equation}

\begin{theorem}\label{thm:Bw}
For any given space $\mathbb V^k$, we have $\mathbb B_\omega \subseteq \mathbb B^\star_\omega$. If \Cref{con:2070} holds true, then we further have $\mathbb B_\omega = \mathbb B^\star_\omega$.  
\end{theorem}

\begin{proof}
	For any $( \overline \omega_1, \overline \omega_2 ) \in \mathbb B_\omega$, there is a symmetric CAD with $( \overline \omega_1, \overline \omega_2 )$ as the boundary weights: 
	$$
	\langle p \rangle_\Omega = 2\overline \omega_1 \langle p \rangle_{\Omega}^x+
	2\overline \omega_2 \langle p \rangle_{\Omega}^y+
	\sum_{s} \omega_s  \overline{ p( x^{(s)}, y^{(s)}) }.
	$$
	It follows that 
		$$
	\langle p \rangle_\Omega - 2\overline \omega_1 \langle p \rangle_{\Omega}^x-
	2\overline \omega_2 \langle p \rangle_{\Omega}^y = 
	\sum_{s} \omega_s  \overline{ p( x^{(s)}, y^{(s)}) } \ge 0 \qquad \forall  p \in \mathbb V_+^k,
	$$
	so that $( \overline \omega_1, \overline \omega_2 ) \in \mathbb B^\star_\omega$. Hence $\mathbb B_\omega \subseteq \mathbb B^\star_\omega$. 
	On 
	the other hand, if $( \overline \omega_1, \overline \omega_2 ) \in \mathbb B^\star_\omega$, 
	we define $\overline \omega_1 = \frac12 \overline \omega (1+\theta)$ 
	and  $\overline \omega_2 = \frac12 \overline \omega (1-\theta)$ with $\theta \in [-1,1]$. 
	Then 
	$ \langle p \rangle_{\Omega} - 
	\overline \omega \left[ (1+\theta) \langle p \rangle_{\Omega}^x + (1-\theta)
	 \langle p \rangle_{\Omega}^y \right]
	 \ge 0 $ for all $ p \in \mathbb V^k_+$, which yields $\overline \omega \le   \inf_{ p \in \mathbb{V}_+^k  }\phi (p;\theta) = \phi^\star (\theta,\mathbb{V}^k_{+}) =  \Wmu_\star(\theta,\mathbb{V}^k)$ if \Cref{con:2070} holds true. Thanks to the convexity of 
	 $\mathbb B_\omega$, we know that 
	$( \overline \omega_1, \overline \omega_2 )$ is the boundary weights of a feasible symmetric CAD. That is, $( \overline \omega_1, \overline \omega_2 ) \in \mathbb B_\omega$, and 
	$\mathbb B_\omega = \mathbb B^\star_\omega$. 
\end{proof}

\begin{figure}[h!]
	\centerline{
		\begin{subfigure}[t][][t]{0.5\textwidth}
			\centering
			\includegraphics[width=0.9\textwidth]{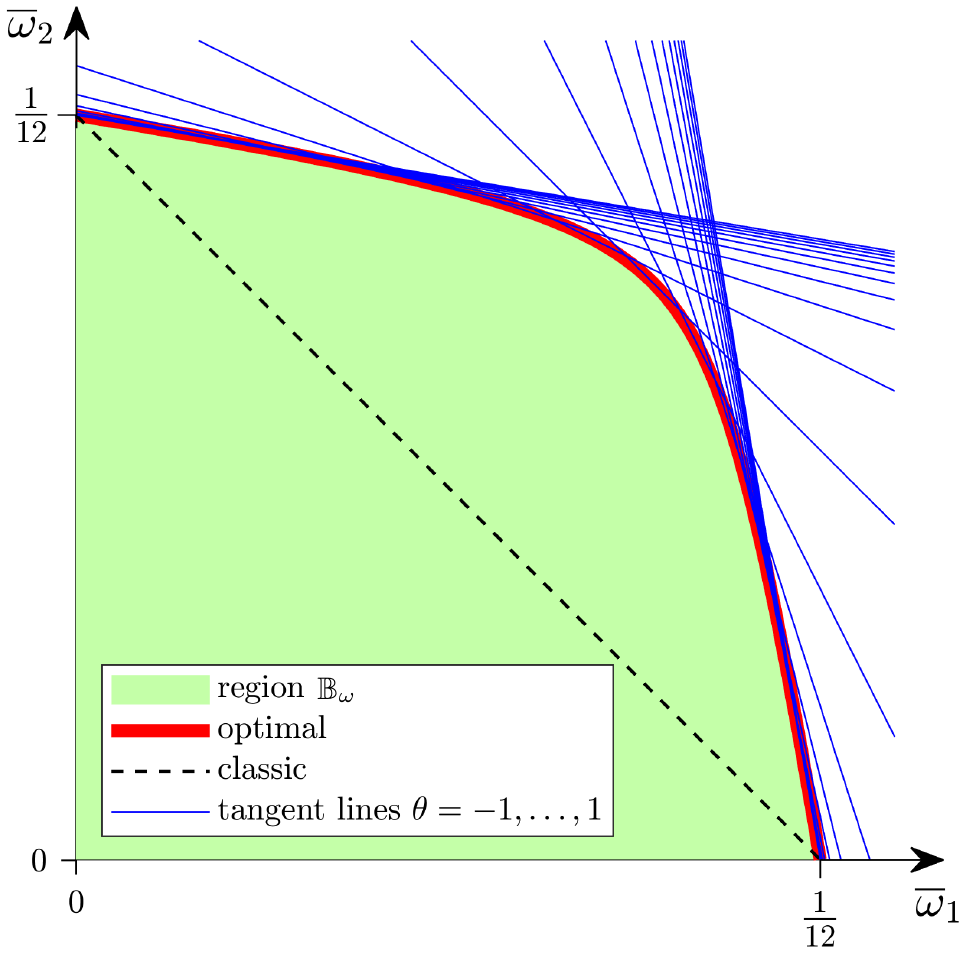}
			\caption{$\mathbb{P}^4$ and $\mathbb{P}^5$.}
		\end{subfigure}
		\begin{subfigure}[t][][t]{0.5\textwidth}
			\centering
			\includegraphics[width=0.9\textwidth]{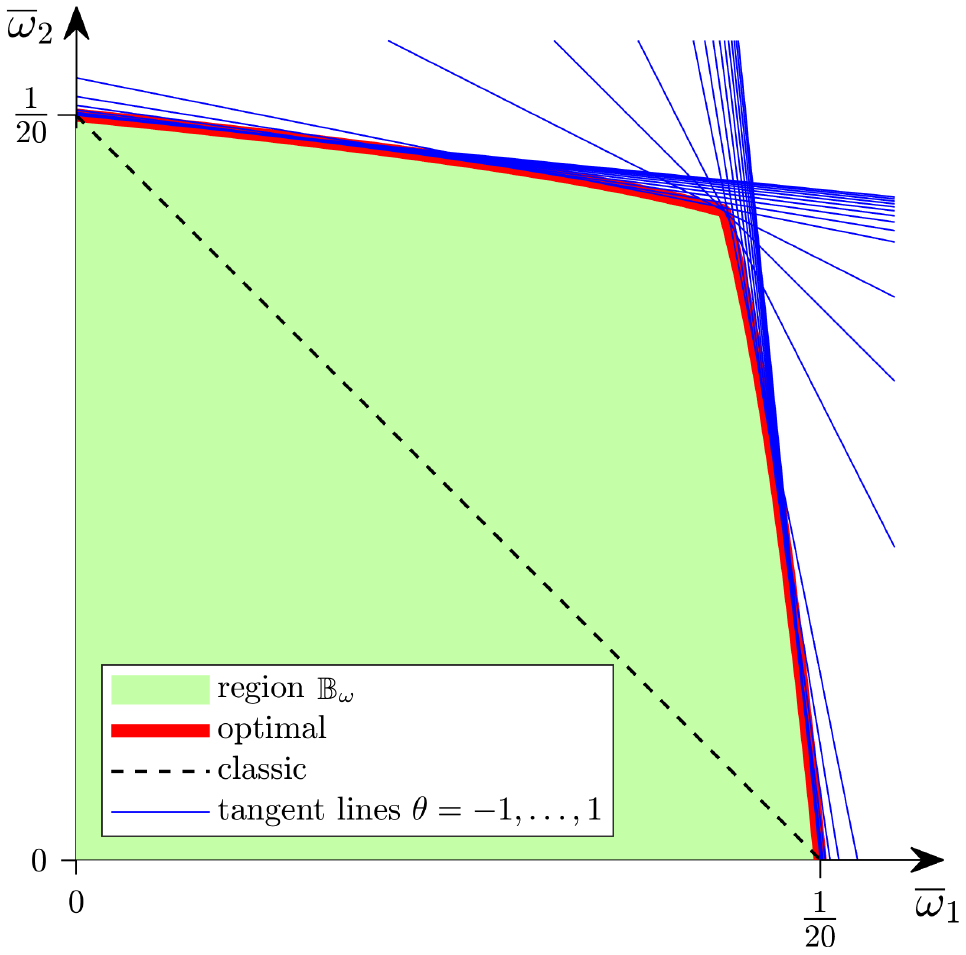}
			\caption{$\mathbb{P}^6$ and $\mathbb{P}^7$.}
			\label{fig:1898b}
		\end{subfigure}
	}
	\caption{Geometric illustration of the region $\mathbb{B}_\omega$, the boundary weights of the 
		classic CAD, the boundary weights of the OCAD   
		for $\mathbb{P}^k$ spaces. 
		We also plot the straight 
		 lines $\langle p^\star \rangle_{\Omega}=2\overline{\omega}_1\langle p^\star \rangle_{\Omega}^x+2\overline{\omega}_2\langle p^\star \rangle_{\Omega}^y$ 
		with $p^\star$ denoting the critical positive polynomial for different $\theta = -1, -0.9, \dots, 1$, which are exactly tangent to $\partial {\mathbb B}_\omega$. 
		  }\label{fig:1898} 
\end{figure}

Based on \Cref{thm:Bw} and the OCADs found in \Cref{sec:2DQk,sec:2DPk} later, we can  visualize the region $\mathbb B_\omega$ for several $\mathbb Q^k$ and $\mathbb P^k$ spaces, as shown in  \Cref{fig:1363,fig:1939,fig:1898}. Based on the visualization results, we have the following important observations and insights from the geometric point of view:
\begin{itemize} 
	\item The region $\mathbb B_\omega$ is compact, confirming the theoretical result in \Cref{lem:compact}. 
	\item The region $\mathbb B_\omega$ is 
	symmetric with respect to the line $\overline \omega_1 = \overline \omega_2$, validating the analysis in \Cref{lem:1525}. 
	\item The region $\mathbb B_\omega$ is convex, confirming the theoretical result in \Cref{lem:convex}. Any line segment connected two points in $\mathbb B_\omega$ represents 
	a convex combination of two feasible CADs. Consequently, 
	the classic CAD is actually a convex combination of the two OCADs at $\theta =\pm 1$, 
	as it will be proved in \Cref{rem:CCADcombine}. 
	\item For the $\mathbb Q^k$ spaces, we see from \Cref{fig:1939} that 
	the region $\mathbb B_\omega$ is triangular in shape. Moreover, 
	the boundary weights of the classic CAD lie on $\partial \mathbb B_\omega$  
	and coincide with the boundary weights of OCAD, inferring the 
	classic CAD is optimal for the $\mathbb Q^k$ spaces as it will be proved in \Cref{sec:2DQk}. 
	\item For the $\mathbb P^k$ spaces, we see from \Cref{fig:1898} that 
	the boundary weights of the classic CAD lie inside $\mathbb{B}_\omega$ rather than on $\partial \mathbb B_\omega$, which implies that the classic CAD is feasible but indeed not optimal in general (except in the special cases of $\theta = \pm 1$). 
	\item  \Cref{fig:1898} displays 
	the straight lines $\langle p^\star \rangle_{\Omega}=2\overline{\omega}_1\langle p^\star \rangle_{\Omega}^x+2\overline{\omega}_2\langle p^\star \rangle_{\Omega}^y$ 
	with $p^\star$ denoting the critical positive polynomial for different $\theta = -1, -0.9, \dots, 1$. 
	We find that these lines are exactly tangent to $\partial {\mathbb B}_\omega$. 
	This finding is consistent with \Cref{thm:Bw} and is closely related to 
	the geometric quasilinearization proposed in \cite{WuShu2021GQL} for convex regions.   
	\item As shown in \Cref{fig:P2P3,fig:1898}, the region boundary $\partial {\mathbb B}_\omega$ corresponding 
	to OCADs is generally smooth, except at $\theta=0$ (namely, $\overline{\omega}_1=\overline{\omega}_2$) for $\mathbb{P}^2$, $\mathbb{P}^3$, $\mathbb{P}^6$ and $\mathbb{P}^7$ (more generally for $\mathbb{P}^{4k+2}$ and $\mathbb{P}^{4k+3}$ with $k \in \mathbb{N}$). At the nonsmooth point on $\partial {\mathbb B}_\omega$, the tangent lines 
	are clearly not unique. 
	This indicates in these special cases, the OCADs and the critical positive polynomials  (normalized with $\langle p^\star \rangle_{\Omega} = 1$) are both not unique (see \Cref{rem:pstar_P2P3,rem:pstar_P6P7}). 
\end{itemize} 

\section{2D OCAD for $\mathbb Q^k$ spaces}\label{sec:2DQk}
With the help of \Cref{lem:512}, we now prove that the classic 2D CAD \eqref{eq:U2Dsplit} is optimal for $\mathbb Q^k$ with $k \in \mathbb N_+$. The 2D OCAD for $\mathbb P^k$ is much more difficult and will be explored carefully in \Cref{sec:2DPk}.

Recall that the classic CAD (\ref{eq:U2Dsplit}) is symmetric. 
It can be transformed onto the reference cell $\Omega=[-1,1]^2$ as 
\begin{equation}\label{eq:1471}
\begin{aligned}
		\langle p \rangle_{\Omega} = & 
	 \omega_1^{{\tt GL}} \left[  (1+\theta)    \langle p \rangle^x_{\Omega}+	
	 (1-\theta) \langle p \rangle^y_{\Omega} \right] 
	 \\
	 &+
	\frac{1+\theta}{2}    \sum_{\ell=2}^{L-1} \sum_{q=1}^Q \omega_\ell^{{\tt GL}} \omega_q^{{\tt G}} p(x_{\ell}^{{\tt GL}},y_{q}^{{\tt G}})+
	\frac{1-\theta}2 \sum_{\ell=2}^{L-1} \sum_{q=1}^Q \omega_\ell^{{\tt GL}} \omega_q^{{\tt G}} p(x_{q}^{{\tt G}},y_{\ell}^{{\tt GL}}),
\end{aligned}
\end{equation}
where 
$\{x_{\ell}^{{\tt GL}}\}$ and $\{y_{\ell}^{{\tt GL}}\}$ denote the $L$-point Gauss--Lobatto quadrature nodes in the internal $[-1,1]$, 
$\{x_{q}^{{\tt G}}\}$ and $\{y_{q}^{{\tt G}}\}$ are the $Q$-point Gauss quadrature nodes in the internal $[-1,1]$.

\begin{theorem}\label{thm:968}
	For any $\theta \in [-1,1]$ and the spaces $\mathbb Q^k$ with any $k \in \mathbb N_+$, 
	the classic 
	 2D CAD (\ref{eq:1471}) with $L = \left \lceil \frac{k+3}2 \right \rceil$ and $Q\ge \frac{k+1}2$ 
	is an OCAD, and Conjectures \ref{con:2070} and \ref{con:2071} hold true with 
\begin{equation}\label{eq:Qk-conj}
		\Wmu_\star(\theta,\mathbb{Q}^k) = \phi^\star (\theta,\mathbb{Q}^k_{+})= \phi^\star(\theta,(\mathbb{Q}^{ \left \lfloor \frac{k}2 \right \rfloor })^2) = \omega_1^{{\tt GL}} = \frac{1}{ L(L-1) }. 
\end{equation}
\end{theorem}
\begin{proof}
	Consider the polynomial
\begin{equation}\label{eq:dualpforQk}
		p^\star (x,y)=  q_\star^2(x,y)   \quad \mbox{with} \quad  q_\star(x,y) :=
		\prod_{\ell = 2}^{L-1} \left(x- x_{\ell }^{{\tt GL}} \right) 
		\cdot
		\prod_{\ell = 2}^{L-1} \left(y- y_{\ell}^{{\tt GL}} \right),  
\end{equation}
which vanishes at all the internal nodes of CAD \eqref{eq:1471}. Besides, 
the degrees of $q_\star$ in $x$ and $y$ are both  
$$
L-2 =    \left \lceil \frac{k+3}2 \right \rceil -2 =     \left \lceil \frac{k-1}2 \right \rceil   \le \left \lfloor \frac{k}2 \right \rfloor, 
$$
so that $q_\star \in \mathbb{Q}^{ \lfloor \frac{k}2  \rfloor }$ and $p^\star \in {\mathbb Q}^k$. 
	According to \Cref{lem:512}, 
	we have \eqref{eq:Qk-conj}, 
	$p^\star$ is a critical positive polynomial for $\phi^\star (\theta,\mathbb{Q}^k_{+})$, and the classic CAD (\ref{eq:1471}) is an OCAD for the space $\mathbb Q^k$.
\end{proof}

\begin{figure}[h]
	\centerline{
		\begin{subfigure}[t][][t]{0.23\textwidth}
			\includegraphics[scale=0.32,trim=5 0 70 0,clip]{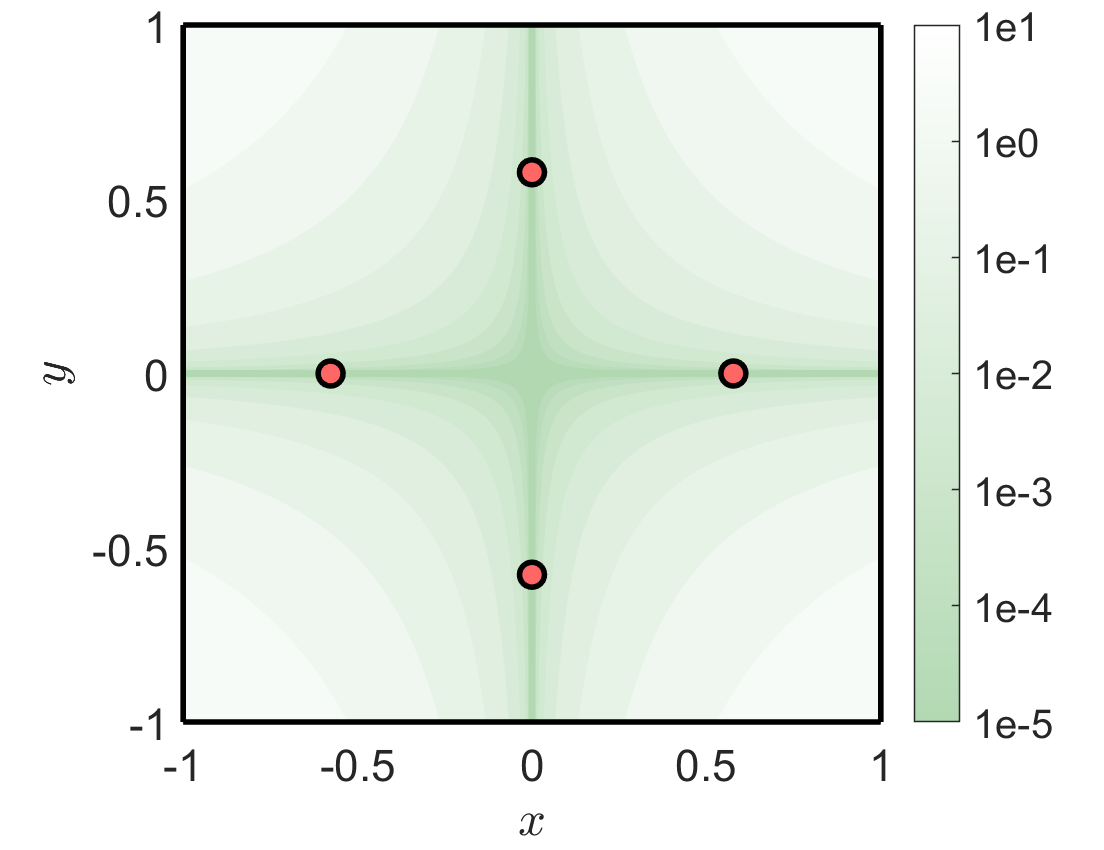}
			\caption{$\mathbb{Q}^2$}
		\end{subfigure}
		\begin{subfigure}[t][][t]{0.23\textwidth}
			\includegraphics[scale=0.32,trim=5 0 70 0,clip]{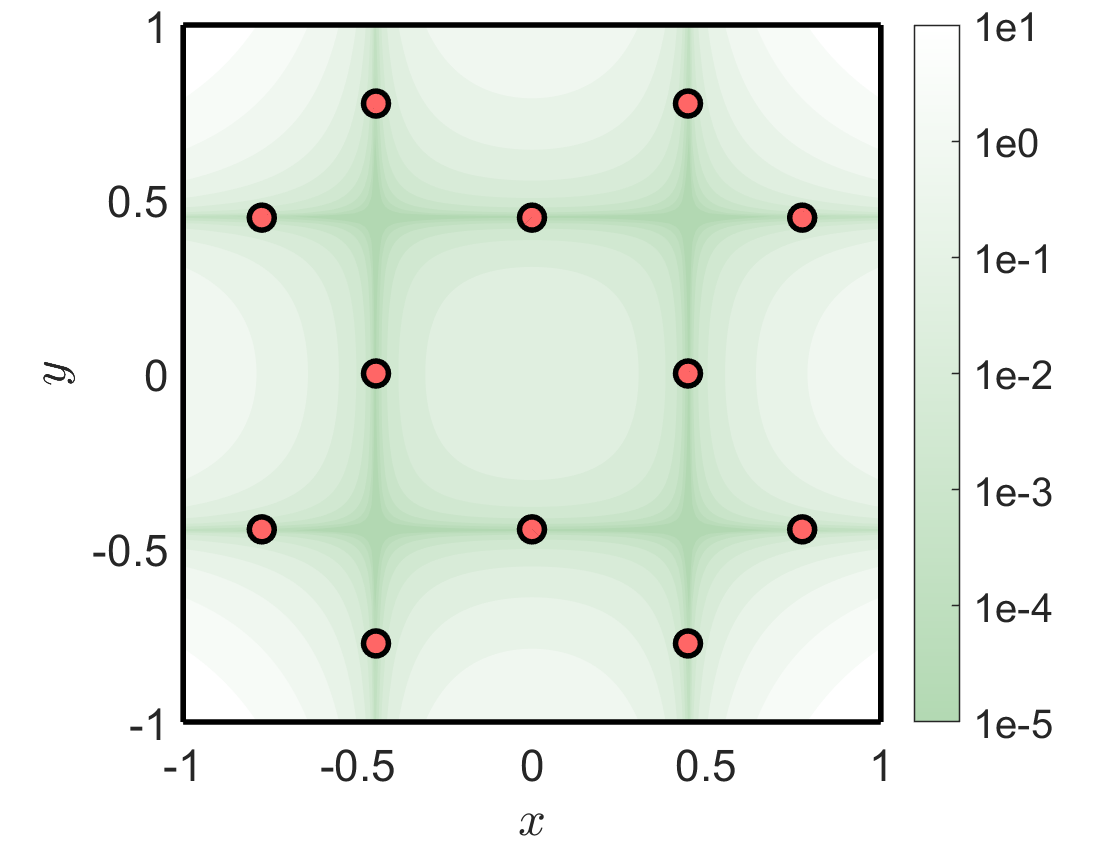}
			\caption{$\mathbb{Q}^4$}
		\end{subfigure}
		\begin{subfigure}[t][][t]{0.23\textwidth}
			\includegraphics[scale=0.32,trim=5 0 70 0,clip]{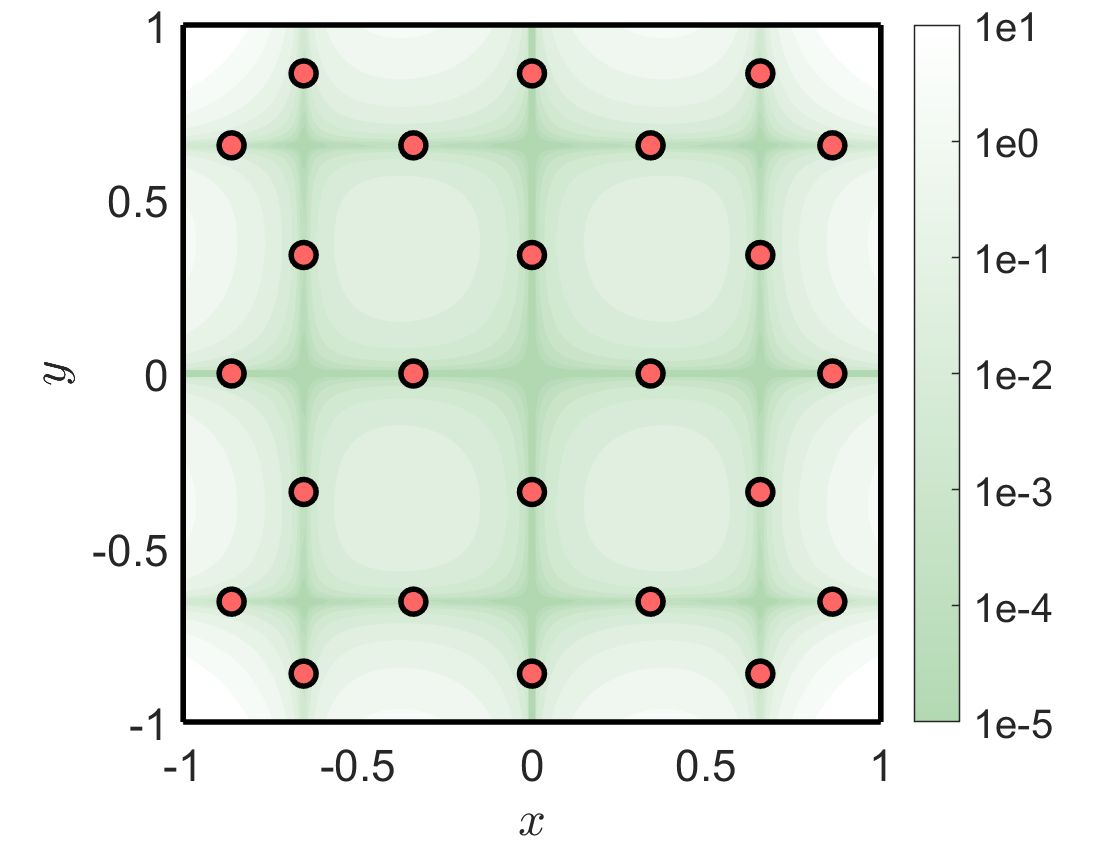}
			\caption{$\mathbb{Q}^6$}
		\end{subfigure}
		\begin{subfigure}[t][][t]{0.23\textwidth}
			\includegraphics[scale=0.32,trim=5 0 70 0,clip]{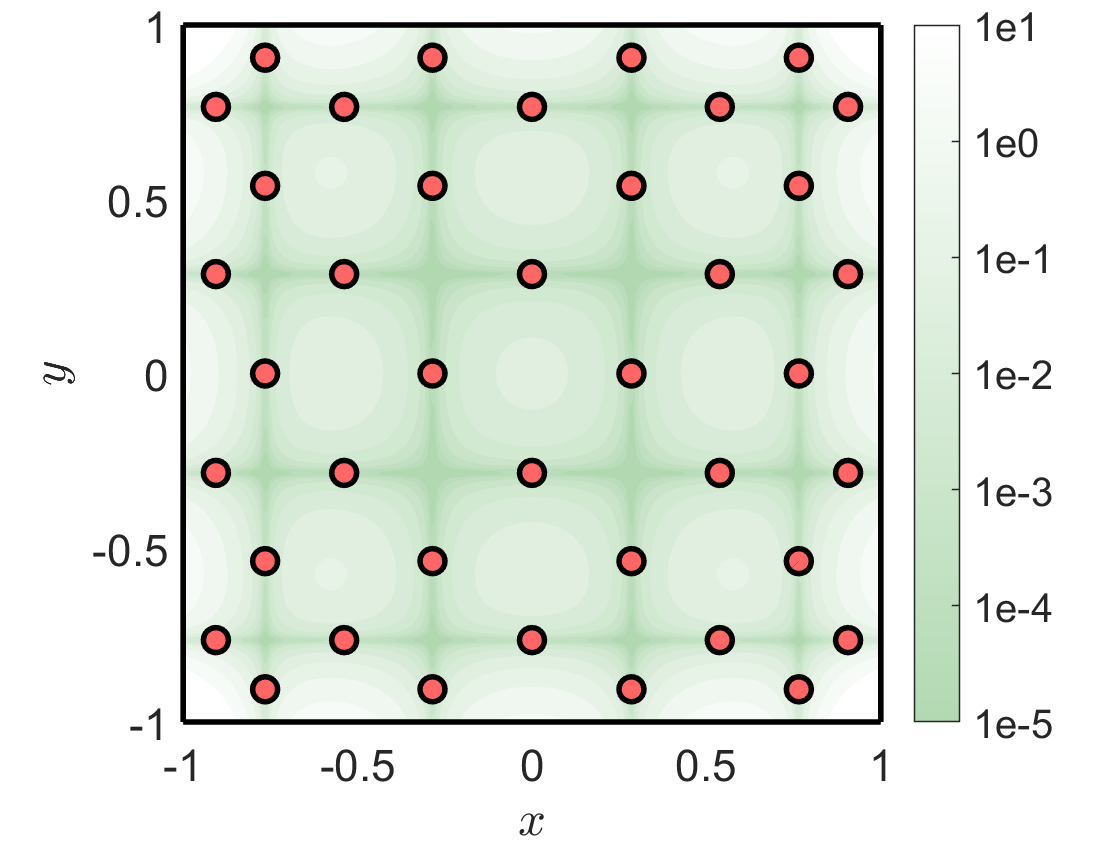}
			\caption{$\mathbb{Q}^8$}
		\end{subfigure}
		\begin{subfigure}[t][][t]{0.05\textwidth}
			\includegraphics[scale=0.32,trim=323 0 0 0,clip]{fig/C2_Qk/OCAD_C2_Q8_FV.png}
		\end{subfigure}
	}
	\caption{Internal nodes of the CAD (\ref{eq:1471}) for polynomial spaces $\mathbb{Q}^k$, $k = 2, 4, 6 ,8$, with $L =  \lceil \frac{k+3}2  \rceil$ and $Q = \lceil \frac{k+1}2  \rceil$. In the background, each corresponding critical positive polynomial (\ref{eq:dualpforQk}) is shown in the logarithmic scale.}\label{fig:1661}
\end{figure}

The internal nodes of the classic CAD (\ref{eq:1471}) and the critical positive polynomial $p^\star$ are illustrated in Figure \ref{fig:1661}. 
	Note that $p^\star$ lies in the space $\mathbb{Q}^k$, but does not belong to the space $\mathbb{P}^k$. Thus, the above proof does not imply that the CAD \eqref{eq:1471} is an OCAD for $\mathbb{P}^k$. As it will be shown in \Cref{sec:2DPk}, solving the 2D OCAD problem for $\mathbb{P}^k$ is much more complicated and challenging.

\section{2D Symmetric OCAD for $\mathbb P^k$ spaces}\label{sec:2DPk}

In this section, we discuss the construction of the symmetric OCAD for $\mathbb P^{k}$ with $k \in \mathbb N_+$. 
When $k=1$, the following CAD on the reference cell $\Omega=[-1,1]^2$ is obviously the symmetric OCAD for $\mathbb P^{1}$: 
$$
\langle p \rangle_{\Omega} =  
\frac12 \left[  (1+\theta)    \langle p \rangle^x_{\Omega}+	
(1-\theta) \langle p \rangle^y_{\Omega} \right],
$$
which is also the OCAD for $\mathbb Q^{1}$. 
However, when $k\ge 2$, the classic CAD \eqref{eq:1471} is generally not optimal for $\mathbb P^{k}$. 
Recently, the OCADs were found in \cite{cui2022classic} 
for $\mathbb P^{2}$ and $\mathbb P^{3}$. 
For $\mathbb P^{k}$ with higher degrees  $k\ge 4$, the OCADs remain unknown yet, 
and their exploration is highly nontrivial and becomes our goal.


We will start our exploration of symmetric OCADs for three special $\theta \in \{-1,1,0\}$ in \Cref{sec:5.1,sec:theta0}, and then study the symmetric OCADs for general $\theta \in [-1,1]$ in \Cref{sec:general-theta}. 
The OCADs for the special $\theta \in \{-1,1,0\}$ will also be helpful for establishing 
the near-optimal or quasi-optimal CADs for general  $\theta \in [-1,1]$; see \Cref{sec:quasi-optimal}.


\subsection{Symmetric OCAD for $\mathbb{P}^{k}$ and $\theta = \pm1$} \label{sec:5.1}

\begin{theorem}\label{thm:1686}
In case of $\theta=-1$, the following symmetric 2D CAD with $L = \left \lceil \frac{k+3}2 \right \rceil$ and $Q\ge \frac{k+1}2$ is an OCAD on $\Omega$ for $\mathbb{P}^k$ with any $k \in \mathbb N_+$.
\begin{equation}\label{eq:1688}
\langle p \rangle_{\Omega} = 
2\omega_{1}^{{\tt GL}} \langle p \rangle_{\Omega}^y+  
\sum_{\ell=2}^{L-1} \sum_{q=1}^Q \omega_\ell^{{\tt GL}} \omega_q^{{\tt G}} p( x_q^{{\tt G}}, y_\ell^{{\tt GL}}).
\end{equation}
In this case, 
Conjectures \ref{con:2070} and \ref{con:2071} hold true with 
\begin{equation}\label{eq:Pk-conj:th-1}
	\Wmu_\star(-1,\mathbb{P}^k) = \phi^\star (-1,\mathbb{P}^k_{+})= \phi^\star(-1,(\mathbb{P}^{ \left \lfloor \frac{k}2 \right \rfloor })^2) = \omega_1^{{\tt GL}} = \frac{1}{ L(L-1) }. 
\end{equation}
\end{theorem}
\begin{proof}
Consider the polynomial
\begin{equation}\label{eq:1696}
		p^\star (x,y)= q_\star^2(x,y) \quad \mbox{with} \quad  q_\star(x,y) := \prod_{\ell = 2}^{L-1} \left(y- y_{\ell}^{{\tt GL}} \right),
\end{equation}
which vanishes at all the internal nodes of CAD \eqref{eq:1688}. Besides, 
the degree of $q_\star$ in $x$ is 0, and the degree of $q_\star$ in $y$ is 
$ (L-2) =       \left \lceil \frac{k-1}2 \right \rceil   \le \left \lfloor \frac{k}2 \right \rfloor$
so that $q_\star \in \mathbb{P}^{ \left \lfloor \frac{k}2 \right \rfloor }$. 
According to \Cref{lem:512}, 
we have \eqref{eq:Pk-conj:th-1}, 
$p^\star$ is a critical positive polynomial $\phi^\star (-1,\mathbb{P}^k_{+})$, and the CAD (\ref{eq:1471}) is an OCAD for $\mathbb P^k$ and $\theta=-1$.
\end{proof}

Based on \Cref{thm:sym_theta} and \Cref{cor:1575}, we immediately obtain 
an OCAD for $\theta = 1$.

\begin{theorem}\label{thm:1699}
In case of $\theta=1$, the following symmetric 2D CAD with $L = \left \lceil \frac{k+3}2 \right \rceil$ and $Q\ge \frac{k+1}2$ is an OCAD on $\Omega$ for $\mathbb{P}^k$ with any $k \in \mathbb N_+$. 
\begin{equation}\label{eq:1712}
\langle p \rangle_{\Omega} = 
2\omega_1^{{\tt GL}} \langle p \rangle_{\Omega}^x+  
\sum_{\ell=2}^{L-1} \sum_{q=1}^Q \omega_\ell^{{\tt GL}} \omega_q^{{\tt G}} p(x_\ell^{{\tt GL}},y_q^{{\tt G}}).
\end{equation}
In this case, 
Conjectures \ref{con:2070} and \ref{con:2071} hold true with 
\begin{equation}\label{eq:Pk-conj:th1}
	\Wmu_\star(1,\mathbb{P}^k) = \phi^\star (1,\mathbb{P}^k_{+})= \phi^\star(1,(\mathbb{P}^{ \left \lfloor \frac{k}2 \right \rfloor })^2) = \omega_1^{{\tt GL}} = \frac{1}{ L(L-1) }. 
\end{equation}
\end{theorem}

\begin{remark}\label{rem:CCADcombine}
	Note that 
	\begin{equation}\label{eq:classicCADconvex}
		\mbox{CAD~\eqref{eq:1471}} = \frac{1+\theta}2 \times \mbox{CAD~\eqref{eq:1688}} 
		+ \frac{1-\theta}2 \times \mbox{CAD~\eqref{eq:1712}},
	\end{equation}
	which implies the classic CAD \eqref{eq:1471} is actually a convex combination of 
	the above two OCADs (\ref{eq:1688}) and (\ref{eq:1712}). 
In the special cases of $\theta = \pm 1$, the above OCADs (\ref{eq:1688}) and (\ref{eq:1712}) for $\mathbb P^k$	
	coincide with the classic CAD (\ref{eq:1471}), which is also optimal for $\mathbb{Q}^k$. 
However, unfortunately, in the cases of $\theta \in (-1,1)$, the classic CAD (\ref{eq:1471}) is no longer optimal for $\mathbb P^k$. 	
\end{remark}

\subsection{Fully symmetric OCAD for  $\mathbb{P}^{k}$ and $\theta=0$}\label{sec:theta0}

In this subsection, we focus on the case of $\theta = 0$, namely, $a_1/\Delta x = a_2 / \Delta y$. In this case,  there is a symmetric OCAD for $\mathbb{P}^k$ 
in the form of 
\[
\langle p \rangle_{\Omega} = \Wmu_\star(0,\mathbb{P}^k) \left[ \langle p \rangle_{\Omega}^x + \langle p \rangle_{\Omega}^y \right] + \sum_s \omega_s \overline{p(x^{(s)},y^{(s)})}. 
\]
Then, according to \Cref{thm:sym_theta} and \Cref{cor:1575}, 
the following symmetric CAD is also optimal for $\mathbb{P}^k$ and $\theta = 0$: 
\[
\langle p \rangle_{\Omega} = \Wmu_\star(0,\mathbb{P}^k) \left[ \langle p \rangle_{\Omega}^x + \langle p \rangle_{\Omega}^y \right] + \sum_s \omega_s \overline{p(y^{(s)},x^{(s)})}.
\]
Taking an average of the above two OCADs leads to a new fully symmetric OCAD: 
\begin{equation}\label{2014}
	\langle p \rangle_{\Omega} = \Wmu_\star(0,\mathbb{P}^k) \left[ \langle p \rangle_{\Omega}^x + \langle p \rangle_{\Omega}^y \right] + \sum_s \frac{\omega_s}{2} \left[ \, \overline{p(x^{(s)},y^{(s)})}+\overline{p(y^{(s)},x^{(s)})} \, \right],
\end{equation}
whose internal nodes are symmetric not only with respect to the $x$- and $y$-axes, but also with respect to $x=y$ and $x=-y$. Such full symmetry 
is very helpful for finding the OCAD \eqref{2014}. 

In order to precisely describe such fully symmetric structures, we can 
invoke the invariance with respect to the following full symmetric group of transformations: 
\begin{equation}\label{eq:Gfs}
{\mathscr G}_{f} := 
\left\{
\begin{array}{cccc}
	(x,y)\mapsto( x, y), &
	(x,y)\mapsto(-x, y), &
	(x,y)\mapsto( x,-y), &
	(x,y)\mapsto(-x,-y) \\
	(x,y)\mapsto( y, x), &
	(x,y)\mapsto(-y, x), &
	(x,y)\mapsto( y,-x), &
	(x,y)\mapsto(-y,-x)
\end{array}
\right\}.
\end{equation}
Clearly, the reference cell $\Omega=[-1,1]^2$ is ${\mathscr G}_{f}$-invariant, namely, 
$g(\Omega) = \Omega$ for all $g \in {\mathscr G}_{f}$. 
In addition, the space $\mathbb{P}^k$  is 
${\mathscr G}_{f}$-invariant, namely, 
\begin{equation}\label{eq:Vk-invar2}
	g(p):=p(g(x,y)) \in \mathbb{P}^k \qquad \forall g \in {\mathscr G}_{f}.	
\end{equation}

\begin{definition}(${\mathscr G}_{f}$-invariant subspace)
	The polynomial space 
	$$\mathbb{P}^k({\mathscr G}_{f}) := \{p \in \mathbb{P}^k : 
	g(p) = p ~~~~~ \forall g \in {\mathscr G}_{f} \}
	$$ 
	is called the ${\mathscr G}_{f}$-invariant subspace of $\mathbb{P}^k$.  
\end{definition}

 The OCAD in the form of \eqref{2014} is called fully symmetric as its internal nodes form a ${\mathscr G}_{f}$-invariant set. 
The fully symmetric structure is helpful for reducing the feasibility requirement in condition (i) of \Cref{def:2D_FCAD}, as shown in the following lemma.

\begin{lemma}\label{lem:Gfs-invarint}
	If a fully symmetric CAD \eqref{2014} is feasible for the ${\mathscr G}_{f}$-invariant subspace $\mathbb{P}^k({\mathscr G}_{f})$, then it is feasible for $\mathbb{P}^k$. 
\end{lemma}

\begin{proof}
	The proof is similar to that of \Cref{lem:Gs-invarint} and is thus omitted. 
\end{proof}

\begin{remark}
		Note that 
		the dimension of the subspace $\mathbb{P}^k({\mathscr G}_{f})$ is typically much smaller than the dimension of $\mathbb{P}^k$. 
		For example, for $\mathbb{P}^3 = \operatorname{span}\{1,x,y,x^2,xy,y^2,x^3,x^2y,xy^2,y^3\}$, we have  $\mathbb{P}^3({\mathscr G}_{f}) = \operatorname{span}\{1,x^2+y^2\}$, 
		and $\dim \mathbb{P}^3({\mathscr G}_{f}) =2 \ll  \dim \mathbb{P}^3 = 10$. 
		In general, 
	by the invariant theory \cite{flatto1968basic}, the basis $\mathbb{P}^k(\mathscr G_{f})$ for $k \in \mathbb{N}_{+}$ can be represented using $x^2+y^2$ and $x^2 y^2$: 
	\[
	\mathbb{P}^k(\mathscr G_{f}) = \operatorname{span}\left\{ (x^2 y^2)^{\alpha} (x^2+y^2)^{\beta} ~|~ 4\alpha+2\beta \le k, \alpha, \beta \in \mathbb{N} \right\}.
	\]
	For example, for $k = 2,4,6,8$, we have
	the following basis of $\mathbb{P}^k(\mathscr G_{f})$
	\[
	\begin{aligned}
		\mathbb{P}^2(\mathscr G_{f}) =& \operatorname{span}\{ 1, x^2+y^2 \}, \\
		\mathbb{P}^4(\mathscr G_{f}) =& \operatorname{span}\{ 1, x^2+y^2, (x^2+y^2)^2, x^2y^2 \}, \\
		\mathbb{P}^6(\mathscr G_{f}) =& \operatorname{span}\{ 1, x^2+y^2, (x^2+y^2)^2, x^2y^2 , (x^2+y^2)^3, (x^2+y^2) x^2 y^2\}, \\
		\mathbb{P}^8(\mathscr G_{f}) =& \operatorname{span}\{ 1, x^2+y^2, (x^2+y^2)^2, x^2y^2 , (x^2+y^2)^3, (x^2+y^2) x^2 y^2, (x^2+y^2)^4, (x^2+y^2)^2 x^2 y^2, x^4 y^4\}. \\
	\end{aligned}
	\]
	It is clear that $\dim \mathbb{P}^k(\mathscr G_{f}) \ll \dim \mathbb{P}^k = \frac{(k+2)(k+1)}{2}$.
Therefore, \Cref{lem:Gfs-invarint} greatly simplifies the feasibility requirement and  
will be very useful for seeking fully symmetric OCADs.
\end{remark}

As shown in the following theorems, we have 
successfully found out 
the fully symmetric OCADs in the form of \eqref{2014} for $\theta =0$ and $\mathbb{P}^k$ spaces with $2\le k \le 7$. With the help of \Cref{lem:Gfs-invarint}, we will provide a general algorithm to numerically 
compute the fully symmetric OCADs \eqref{2014} for $\theta =0$ and $\mathbb{P}^k$ with higher $k \ge 8$. 

\begin{theorem}[OCAD for $\theta = 0$, $\mathbb{P}^2$ and $\mathbb{P}^3$]
	In case of $\theta = 0$, the following fully symmetric CAD is an OCAD on $\Omega$ for the spaces $\mathbb{P}^2$ and $\mathbb{P}^3$:
	\begin{equation}\label{eq:2276}
		\begin{aligned}
			\langle p \rangle_{\Omega}
			=& \frac{1}{4} \, [ \langle p \rangle_{\Omega}^x+ \langle p \rangle_{\Omega}^y ] + \frac12 \, p(0,0).
		\end{aligned}
	\end{equation}
	In this case, 
	Conjectures \ref{con:2070} and \ref{con:2071} hold true for $2\le k \le 3$ with 
	\begin{equation}\label{eq:P2P3-conj:th0}
		\Wmu_\star(0,\mathbb{P}^k) = \phi^\star (0,\mathbb{P}^k_{+})= \phi^\star(0,(\mathbb{P}^1)^2) = \frac14, \qquad k=2,3. 
	\end{equation}
\end{theorem}

\begin{proof}
	It can be verified that the fully symmetric CAD \eqref{eq:2276} is feasible 
	for the ${\mathscr G}_{f}$-invariant subspace $\mathbb{P}^2({\mathscr G}_{f}) = \mathbb{P}^3({\mathscr G}_{f}) ={\rm span}\{1,x^2+y^2\}$. Thus, by \Cref{lem:Gfs-invarint}, it is feasible 
	for $\mathbb{P}^2$ and $\mathbb{P}^3$. 
	Both polynomials 
	\begin{equation}\label{eq:2334}
		p_1^\star(x,y) = x^2 \quad \text{and} \quad p_2^\star(x,y) = y^2
	\end{equation}
	vanish at the internal node $(0,0)$. Additionally, both $p_1^\star$ and $p_2^\star$ belong to  $(\mathbb{P}^1)^2$. By \Cref{lem:512}, the fully symmetric CAD \eqref{eq:2276} is optimal for $\mathbb{P}^2$ and $\mathbb{P}^3$ spaces, and both $p_1^\star$ and $p_2^\star$ are the corresponding critical positive polynomials.
\end{proof}

\begin{remark}\label{rem:pstar_P2P3}
	As shown in \eqref{eq:2334}, the critical positive polynomials for either  
	$\phi^\star (0,\mathbb{P}^2_{+})$ or $\phi^\star (0,\mathbb{P}^3_{+})$ 
	are not unique. In fact, there are infinitely many critical positive polynomials in this case, 
	for example, 
	$(x + t y)^2$ for any $t \in \mathbb R$, $\alpha_1 x^2 + \alpha_2 y^2$ for any $\alpha_1,\alpha_2 \in \mathbb R_+$, as discussed in \Cref{lem:notUnique}.   
\end{remark}

\begin{theorem}[OCAD for $\theta = 0$, $\mathbb{P}^4$ and $\mathbb{P}^5$]
	In case of $\theta = 0$, the following fully symmetric CAD is an OCAD on $\Omega$ for the spaces $\mathbb{P}^4$ and $\mathbb{P}^5$:
	\begin{equation}\label{eq:2286a}
		\begin{aligned}
			\langle p \rangle_{\Omega}
			=& \left( 2-\frac{\sqrt{14}}{2} \right) [ \langle p \rangle_{\Omega}^x+ \langle p \rangle_{\Omega}^y ]
			+ \frac{5\sqrt{14}-15}{7} \, \overline{p\left(\sqrt{\frac{7-\sqrt{14}}{15}},\sqrt{\frac{7-\sqrt{14}}{15}}\right)} \\
			+& \frac{\sqrt{14}-3}{7} \left[ \, \overline{p\left(\sqrt{\frac{14-2\sqrt{14}}{15}},0\right)} + \overline{p\left(0,\sqrt{\frac{14-2\sqrt{14}}{15}} \right)} \right].
		\end{aligned}
	\end{equation}
	In this case, 
	Conjectures \ref{con:2070} and \ref{con:2071} hold true for $4\le k \le 5$ with 
	\begin{equation}\label{eq:P4P5-conj:th0}
		\Wmu_\star(0,\mathbb{P}^k) = \phi^\star (0,\mathbb{P}^k_{+})= \phi^\star(0,(\mathbb{P}^2)^2) = 2-\frac{\sqrt{14}}{2}, \qquad k=4,5. 
	\end{equation}
\end{theorem}

\begin{proof}
	It can be verified that the fully symmetric CAD \eqref{eq:2286a} is feasible 
	for the ${\mathscr G}_{f}$-invariant subspace $\mathbb{P}^4({\mathscr G}_{f}) = \mathbb{P}^5({\mathscr G}_{f}) ={\rm span}\{1,x^2+y^2,(x^2+y^2)^2, x^2y^2 \}$. Thus, by \Cref{lem:Gfs-invarint}, it is feasible 
	for $\mathbb{P}^4$ and $\mathbb{P}^5$. The polynomial 
	\begin{equation}\label{eq:2354}
		p^\star(x,y) = q_\star^2 (x,y) \quad \mbox{with} \quad 
		q_\star(x,y) := x^2 + y^2 - \frac{14-2\sqrt{14}}{15}  
	\end{equation}
	vanishes at all the internal nodes. By \Cref{lem:512}, the fully symmetric CAD \eqref{eq:2286a} is optimal for $\mathbb{P}^4$ and $\mathbb{P}^5$, and $p^\star$ is the corresponding critical positive polynomial.
\end{proof}

\begin{theorem}[OCAD for $\theta = 0$, $\mathbb{P}^6$ and $\mathbb{P}^7$]
	In case of $\theta = 0$, the following fully symmetric CAD is an OCAD on $\Omega$ for the spaces $\mathbb{P}^6$ and $\mathbb{P}^7$:
	\begin{equation}\label{eq:2004}
		\begin{aligned}
			\langle p \rangle_{\Omega}
			=& \left(1-\frac{\sqrt{30}}{6}\right) [ \langle p \rangle_{\Omega}^x+ \langle p \rangle_{\Omega}^y ] 
			+ \frac{875\sqrt{30}-3125}{4563} \, \overline{p\left(\sqrt{\frac{3}{5}-\frac{\sqrt{30}}{25}},\sqrt{\frac{3}{5}-\frac{\sqrt{30}}{25}}\right)} \\
			+& \frac{343\sqrt{30}-1225}{4563} \left[ \, \overline{p\left(\sqrt{\frac{6}{7}-\frac{2\sqrt{30}}{35}},0\right)} + \overline{p\left(0,\sqrt{\frac{6}{7}-\frac{2\sqrt{30}}{35}}\right)} \, \right]
			+ \frac{1012-40\sqrt{30}}{4563} \, p(0,0).
		\end{aligned}
	\end{equation}
	In this case, 
Conjectures \ref{con:2070} and \ref{con:2071} hold true for $6\le k \le 7$ with 
\begin{equation}\label{eq:P6P7-conj:th0}
	\Wmu_\star(0,\mathbb{P}^k) = \phi^\star (0,\mathbb{P}^k_{+})= \phi^\star(0,(\mathbb{P}^3)^2) = 1-\frac{\sqrt{30}}{6}, \qquad k=6,7. 
\end{equation}
\end{theorem}

\begin{proof}
	It can be verified that the fully symmetric CAD \eqref{eq:2004} is feasible 
		for the ${\mathscr G}_{f}$-invariant subspace $\mathbb{P}^6({\mathscr G}_{f}) = \mathbb{P}^7({\mathscr G}_{f}) ={\rm span}\{1,x^2+y^2,(x^2+y^2)^2, x^2y^2, (x^2+y^2)^3, (x^2+y^2)x^2y^2 \}$. Thus, by \Cref{lem:Gfs-invarint}, it is feasible 
	for $\mathbb{P}^6$ and $\mathbb{P}^7$. Both polynomials
	\begin{equation}\label{eq:2017}
		p_1^\star(x,y) 
		= y^2\,\left(15\,x^2+35\,y^2+2\,\sqrt{30}-30\right)^2
		\quad \textrm{and} \quad 
		p_2^\star(x,y) 
		= x^2\,\left(35\,x^2+15\,y^2+2\,\sqrt{30}-30\right)^2
	\end{equation}
	vanish at all the internal nodes of CAD \eqref{eq:2004}. Additionally, both $p_1^\star$ and $p_2^\star$ belong to $(\mathbb{P}^3)^2$, thus belong to $\mathbb{P}^6_{+}$ and $\mathbb{P}^7_{+}$. By \Cref{lem:512}, the fully symmetric CAD \eqref{eq:2004} is optimal for $\mathbb{P}^6$ and $\mathbb{P}^7$, and both $p_1^\star$ and $p_2^\star$ are the corresponding critical positive polynomials.
\end{proof}

\begin{remark}\label{rem:pstar_P6P7}
	One can add $p_1^\star$ and $p_2^\star$ in \eqref{eq:2334} (or \eqref{eq:2017}), resulting in a new critical positive polynomial, which is ${\mathscr G}_{f}$-invariant. In \Cref{fig:2036}, we plot the internal nodes and the critical polynomial with such a symmetry for $\mathbb{P}^2$ to $\mathbb{P}^7$ spaces. We observe that the internal nodes of our OCADs are much fewer than those of the classic 
	CAD (\ref{eq:1471}). 
\end{remark}

\begin{figure}[h!]
	\begin{subfigure}[t][][t]{0.3\textwidth}
		\includegraphics[scale=0.32,trim= 0 5 0 0,clip]{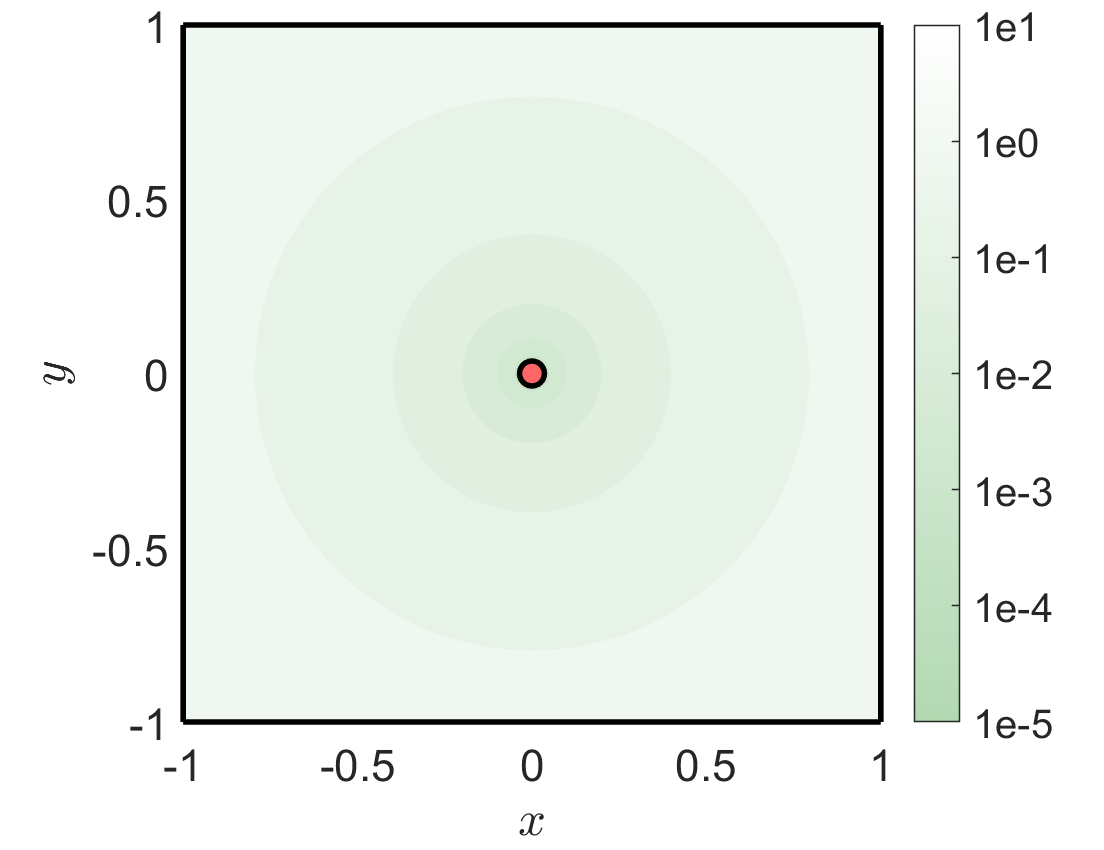}
		\caption{$\mathbb{P}^2$ and $\mathbb{P}^3$. The critical positive polynomial is the sum of $p_1^\star$ and $p_2^\star$ given in \eqref{eq:2334}.}
	\end{subfigure}
	\hfill
	\begin{subfigure}[t][][t]{0.3\textwidth}
		\includegraphics[scale=0.32,trim= 0 5 0 0,clip]{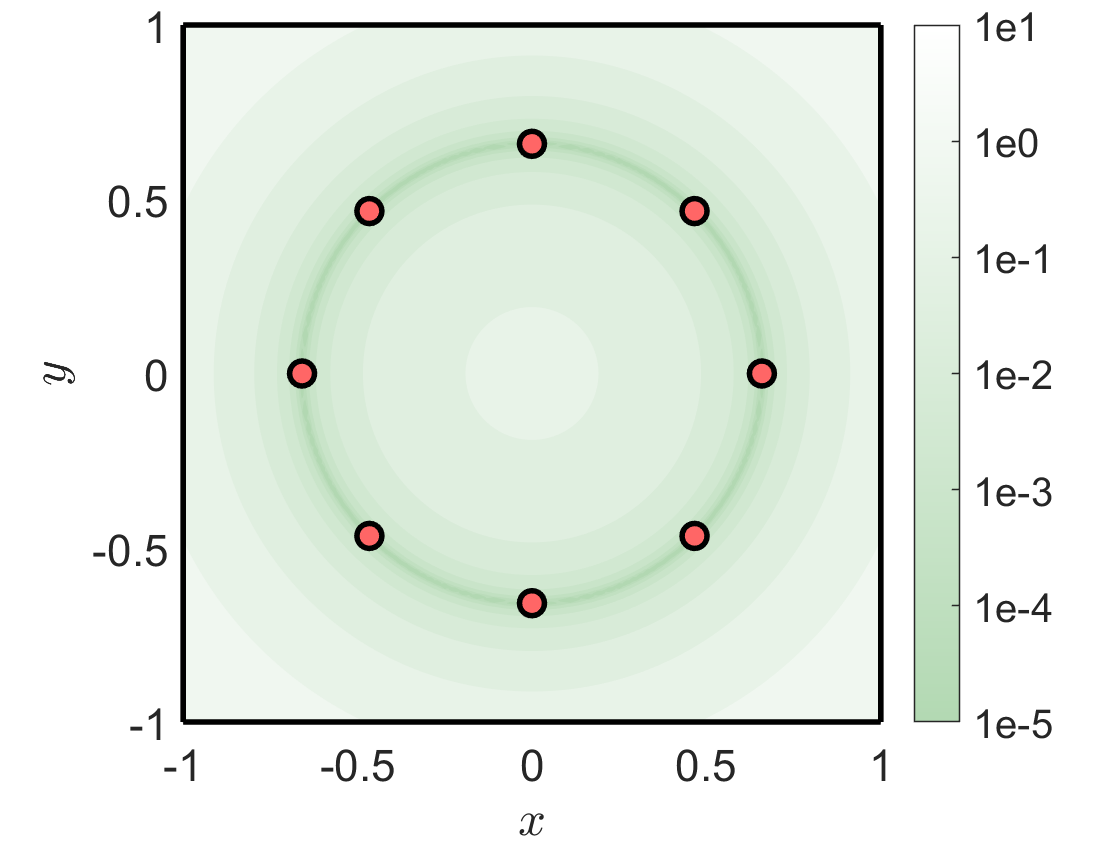}
		\caption{$\mathbb{P}^4$ and $\mathbb{P}^5$. The critical positive polynomial is $p^\star$ in \eqref{eq:2017}.}
	\end{subfigure}
	\hfill
	\begin{subfigure}[t][][t]{0.3\textwidth}
		\includegraphics[scale=0.32,trim= 0 5 0 0,clip]{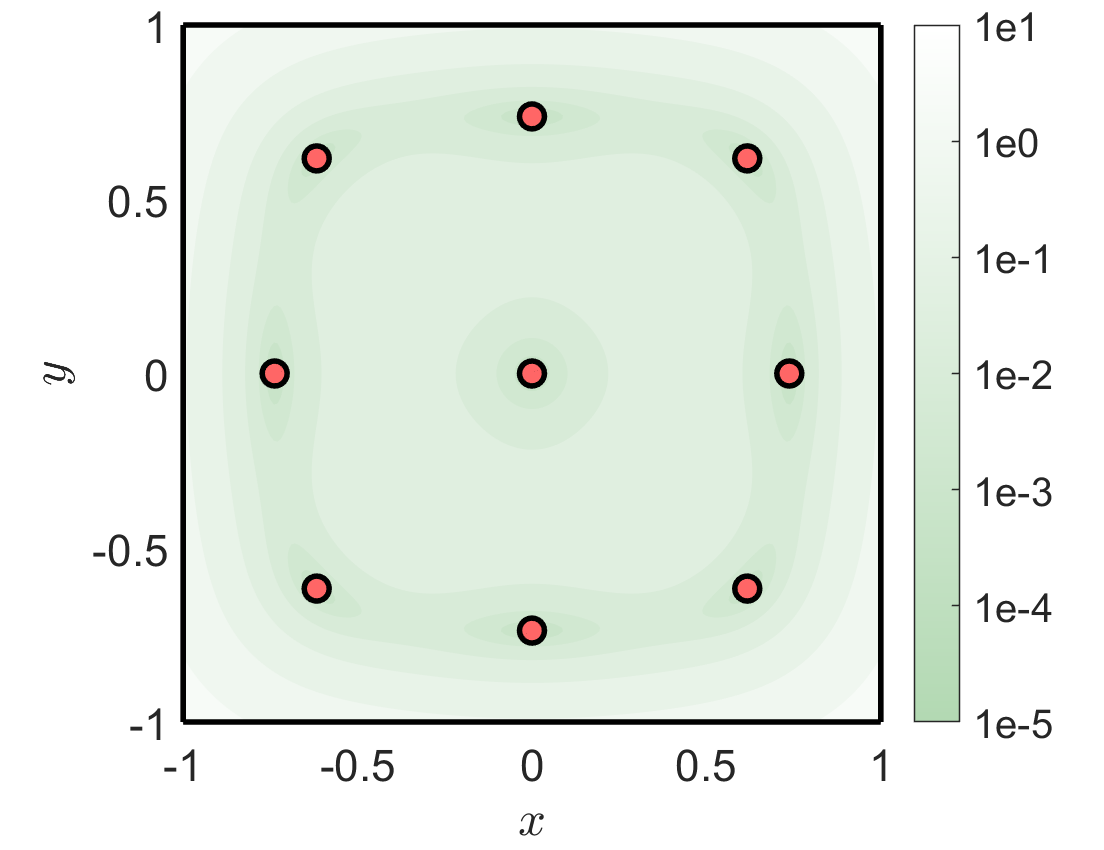}
		\caption{$\mathbb{P}^6$ and $\mathbb{P}^7$. The critical positive polynomial is the sum of $p_1^\star$ and $p_2^\star$ given in \eqref{eq:2354}.}
	\end{subfigure}
	\caption{Internal nodes of the fully symmetric OCADs for $\theta =0$ and $\mathbb{P}^k$ spaces with $2 \le k \le 7$. The background also displays the corresponding critical positive polynomials. 
	}
		\label{fig:2036}
\end{figure}

As the degree $k$ increases, 
it becomes more and more difficult to find the analytical form of the fully symmetric OCADs, even in the special case of $\theta =0$. 
We propose the following general theorem, based on which we can numerically seek the fully symmetric OCADs for $\mathbb P^k$ spaces with $k\ge 8$ (\Cref{sec:A2}).

	\begin{theorem}\label{thm:3125}
		For any given $k \in \mathbb{N}_+$, let 
		 $q_\star(x,y) \in \mathbb{P}^{\left \lfloor \frac{k}2 \right \rfloor}$ be the polynomial defined in  \eqref{eq:2065}. 
		 If there exist $\{(\omega_s,x^{(s)},y^{(s)}): \omega_s > 0,0 \le y^{(s)} \le x^{(s)} \le 1\}_{s=1}^S$ satisfying the following system 
		\begin{equation}\label{eq:3137}
			\begin{cases}
				 q_\star(x^{(s)},y^{(s)})  = 0, 
				 \\ \displaystyle 
				 \langle g_i \rangle_\Omega = 2 \phi^\star(0,(\mathbb{P}^{ \left \lfloor \frac{k}2 \right \rfloor })^2)
				 \langle g_i \rangle_\Omega^{x}
				  +
				\sum_{s=1}^S \omega_s g_i(x^{(s)},y^{(s)}), \qquad 1 \le i \le \dim( \mathbb{P}^k(\mathscr G_{f}) )
			\end{cases}
		\end{equation}
		with $\{g_i\}$ being a basis of $\mathbb{P}^k(\mathscr G_{f})$, then the following CAD 
		\begin{equation}\label{eq:3148}
			\langle p \rangle_\Omega
			=  \phi^\star(0,(\mathbb{P}^{ \left \lfloor \frac{k}2 \right \rfloor })^2)
			\Big[
			\langle p \rangle_\Omega^{x}
			+ \langle p \rangle_\Omega^{y}
			\Big]
			+ \sum_{s=1}^S \frac{\omega_s}{2} [\overline{p(x^{(s)},y^{(s)})}+\overline{p(y^{(s)},x^{(s)})}] 
		\end{equation}
		is a fully symmetric OCAD for $\theta = 0$ and $\mathbb{P}^k$. Furthermore, 
		in this case, 
		Conjectures \ref{con:2070} and \ref{con:2071} hold true with 
		\begin{equation}\label{eq:Pk-conj:th0}
			\Wmu_\star(0,\mathbb{P}^k) = \phi^\star (0,\mathbb{P}^k_{+})= \phi^\star(0,(\mathbb{P}^{ \left \lfloor \frac{k}2 \right \rfloor })^2) = \phi(q^2_\star;0), 
		\end{equation}
		and $q_\star^2(x,y)$ is 
		the critical positive polynomial for $\phi^\star (0,\mathbb P_+^k)$. 
	\end{theorem}
	
	\begin{proof}
	Due to the symmetry of $g_i \in \mathbb{P}^k(\mathscr G_{f})$, it satisfies
	\[
		g_i(x^{(s)},y^{(s)}) = \overline{g_i(x^{(s)},y^{(s)})} = \overline{g_i(y^{(s)},x^{(s)})} \qquad 1 \le s \le S.
	\]
	and $\langle g_i \rangle_\Omega^{x} = \langle g_i \rangle_\Omega^{y}$. Thus, \eqref{eq:3137} implies 
		\begin{equation}\label{eq:3183}
				 \langle g_i \rangle_\Omega = \phi^\star(0,(\mathbb{P}^{ \left \lfloor \frac{k}2 \right \rfloor })^2)
				\left[\langle g_i \rangle_\Omega^{x}+\langle g_i \rangle_\Omega^{y}\right]+
				\sum_{s=1}^S \frac{\omega_s}{2} [\overline{g_i(x^{(s)},y^{(s)})} + \overline{g_i(y^{(s)},x^{(s)})}] 
		\end{equation}
		for all $1 \le i \le \dim( \mathbb{P}^k(\mathscr G_{f}) )$. Since $\{g_i\}$ are a basis of $\mathbb{P}^k(\mathscr G_{f})$, 
		the decomposition \eqref{eq:3148} holds for any $g \in \mathbb{P}^k(\mathscr G_{f})$. 
		Thanks to \Cref{lem:Gfs-invarint}, the fully symmetric CAD \eqref{eq:3148} is feasible for $\mathbb{P}^k$. 
		Next, we verify the optimality of the CAD \eqref{eq:3148}. 
		Because $\{(\omega_s,x^{(s)},y^{(s)})\}_{s=1}^S$ satisfy \eqref{eq:3137}, 
		the polynomial $q^2_\star \in \mathbb{P}^k$ vanishes at all the internal nodes of the CAD \eqref{eq:3148}. By \Cref{lem:512}, the CAD \eqref{eq:3148} is optimal for $\theta = 0$ and $\mathbb{P}^k$, and  $q_\star^2(x,y)$ is the critical positive polynomial for $\phi^\star (0,\mathbb P_+^k)$.
	\end{proof}
	


When $k\ge 8$, it become very difficult to find the analytical solution to the system \eqref{eq:3137} or to 
rigorously prove the existence of its solution. Nevertheless, we can always 
construct the fully symmetric OCAD \eqref{eq:3183} by numerically solving the equations \eqref{eq:3137} using an iterative method. 
The readers are referred to \Cref{sec:A2} for the fully symmetric OCADs we found for $\mathbb P^k$ spaces with $8 \le k \le 15$.

\subsection{Symmetric OCADs for $\mathbb{P}^{k}$ and general $\theta \in [-1,1]$}\label{sec:general-theta}

In this subsection, we discuss the symmetric OCADs for $\mathbb{P}^{k}$ and general $\theta \in [-1,1]$. Seeking OCADs for general $\theta$ is much more difficult than it for 
the three special cases in the previous subsections. 
As shown in the following theorems, we have 
successfully found out 
the analytical formulas of the symmetric OCADs for $\mathbb{P}^k$ spaces with $k \le 7$, 
which cover the widely used polynomial spaces. We will also provide a general algorithm to numerically 
compute the symmetric OCADs for general $\theta \in [-1,1]$ and $\mathbb{P}^k$ with higher degree $k \ge 8$.

\begin{theorem}[OCAD for $\mathbb{P}^2$ and $\mathbb{P}^3$]\label{thm:3129}
	For any $\theta \in [-1,1]$, the 2D symmetric OCAD on $\Omega$ for $\mathbb P^2$ and $\mathbb P^3$ spaces is given by
	\begin{equation}\label{eq:2849}
		\langle p \rangle_{\Omega} = \overline{\omega}_\star
		\Big[  
		(1+\theta) \langle p \rangle^x_{\Omega}+	
		(1-\theta) \langle p \rangle^y_{\Omega} \Big] 
		+
		\omega_1 \, \overline{p\left(x^{(1)},y^{(1)}\right)}
	\end{equation}
	with
	\begin{equation}\label{eq:2858}
		\overline{\omega}_\star = \frac{1}{4+2|\theta|}, \quad
		\omega_1 = \frac{1+|\theta|}{2+|\theta|}, \quad
		\left(\,x^{(1)},y^{(1)}\,\right) = 
		\begin{dcases}
			\left(\, \sqrt{\frac{2|\theta|}{3+3|\theta|}},0 \,\right) & \text{ if } \theta \in [-1,0], \\
			\left(\, 0,\sqrt{\frac{2|\theta|}{3+3|\theta|}} \,\right) & \text{ if } \theta \in [0,1].
		\end{dcases}
	\end{equation}
Moreover, Conjectures \ref{con:2070} and \ref{con:2071} hold true for all $\theta\in [-1,1]$ and $2\le k \le 3$ with 
\begin{equation}\label{eq:P2P3-conj}
	\Wmu_\star(\theta,\mathbb{P}^k) = \phi^\star (\theta,\mathbb{P}^k_{+})= \phi^\star(\theta,(\mathbb{P}^1)^2) =  \frac{1}{4+2|\theta|}, \qquad k=2,3. 
\end{equation}
\end{theorem}
\begin{proof}
	It is easy to verify that the symmetric CAD \eqref{eq:2849} is feasible 
	for the ${\mathscr G}_s$-invariant subspace $\mathbb{P}^2({\mathscr G}_{s}) = \mathbb{P}^3({\mathscr G}_{s})={\rm span} 
	\{1,x^2,y^2\}$. Thus by \Cref{lem:Gs-invarint}, it is feasible for $\mathbb{P}^2$ and $\mathbb{P}^3$ spaces. 
	In case of $\theta \in [-1,0]$, the nonnegative polynomial $p^\star(x,y) = y^2$ belong to $(\mathbb{P}^1)^2$, thus is in both $\mathbb{P}^2_{+}$ and $\mathbb{P}^3_{+}$. Furthermore, $p^\star$ vanishes at all the internal nodes given in \eqref{eq:2858}. 
	By \Cref{lem:512}, the symmetric CAD \eqref{eq:2849} is optimal for both $\mathbb{P}^2$ and $\mathbb{P}^3$ spaces when $\theta \in [-1,0]$, and \eqref{eq:P2P3-conj} holds for all $\theta\in [-1,0]$ and $2\le k \le 3$. 
	By \Cref{thm:sym_theta}, the CAD \eqref{eq:2849} is also optimal when $\theta \in [0,1]$. The proof is completed.
\end{proof}

\begin{remark}
For comparison,	\Cref{fig:1882} illustrates our symmetric OCAD \eqref{eq:2849} and the classic CAD \eqref{eq:1471}. We clearly see 
that the OCAD \eqref{eq:2849} involves much fewer internal nodes than the classic CAD. 
It is worth mentioning that the symmetric OCAD \eqref{eq:2849} is consistent with the OCAD derived in our previous paper \cite{cui2022classic}, where only the OCAD for $\mathbb{P}^2$ and $\mathbb{P}^3$ spaces was found. 
\end{remark}


\begin{theorem}[OCAD for $\mathbb{P}^4$ and $\mathbb{P}^5$]\label{thm:3161}
	For any $\theta \in [-1,1]$, the 2D symmetric OCAD on $\Omega$ for $\mathbb P^4$ and $\mathbb P^5$ spaces is given by
	\begin{equation}\label{eq:2879}
		\langle p \rangle_{\Omega} = \overline \omega_\star 
		\left[ 
		(1+\theta) \langle p \rangle^x_{\Omega}+ 
		(1-\theta) \langle p \rangle^y_{\Omega} \right] 
		+
		\omega_1 \, \overline{p\left(x^{(1)},y^{(1)}\right)}
		+
		\omega_2 \, \overline{p\left(x^{(2)},y^{(2)}\right)}
	\end{equation}
	with
	\begin{subequations}\label{eq:2890}
		\begin{equation}\label{eq:2890a}
			\overline{\omega}_\star = \left[\frac{14}{3}+\frac{2}{3}\sqrt{78\,\theta^2+46} \cos \left( \frac{1}{3} \arccos\frac{1476\,\theta^2-244}{(78\,\theta^2+46)^{\frac{3}{2}}} \right) \right]^{-1},
		\end{equation}
		\begin{equation}\label{eq:2890b}
			\omega_1 = \frac{5(1-4\overline{\omega}_\star+2|\theta|\overline{\omega}_\star)^2}{9\,(1-6\overline{\omega}_\star+4|\theta|\overline{\omega}_\star)}, \quad
			\omega_2 = 1-2\overline{\omega}_\star-\omega_1,
		\end{equation}
		\begin{equation}\label{eq:2890c}
			\left(\, x^{(1)}, y^{(1)} \,\right) = 
			\begin{dcases}
				\left(\, 
				\sqrt{\frac{3\,(1-6\overline{\omega}_\star+4|\theta|\overline{\omega}_\star)}{5\,(1-4\overline{\omega}_\star+2|\theta|\overline{\omega}_\star)}}, 
				\sqrt{\frac{1-6\overline{\omega}_\star}{3\,(1-4\overline{\omega}_\star+2|\theta|\overline{\omega}_\star)}}
				\,\right) 
				& \text{ if } \theta \in [-1,0], \\
				\left(\, 
				\sqrt{\frac{1-6\overline{\omega}_\star}{3\,(1-4\overline{\omega}_\star+2|\theta|\overline{\omega}_\star)}},
				\sqrt{\frac{3\,(1-6\overline{\omega}_\star+4|\theta|\overline{\omega}_\star)}{5\,(1-4\overline{\omega}_\star+2|\theta|\overline{\omega}_\star)}}
				\,\right) & \text{ if } \theta \in [0,1],
			\end{dcases}
		\end{equation}
		\begin{equation}\label{eq:2890d}
			\left(\, x^{(2)}, y^{(2)} \,\right) = 
			\begin{dcases}
				\left(\, 
				0, 
				\sqrt{\frac{1-4\overline{\omega}_\star-2|\theta| \overline{\omega}_\star-3 \omega_1 (y^{(1)})^2}{3\,\omega_2}}
				\,\right) 
				& \text{ if } \theta \in [-1,0], \\
				\left(\,
				\sqrt{\frac{1-4\overline{\omega}_\star-2|\theta| \overline{\omega}_\star-3 \omega_1 (x^{(1)})^2}{3\,\omega_2}},
				0
				\,\right) & \text{ if } \theta \in [0,1].
			\end{dcases}
		\end{equation}
	\end{subequations}
	Moreover, Conjectures \ref{con:2070} and \ref{con:2071} hold true for all $\theta\in [-1,1]$ and $4\le k \le 5$.
\end{theorem}

	For better readability, we put the proof of \Cref{thm:3161} in \Cref{sec:A0}. The weights and nodes \eqref{eq:2890a}--\eqref{eq:2890d} are illustrated in \Cref{fig:2417}, which clearly verifies the feasibility conditions (ii) and (iii) in \Cref{def:2D_FCAD}.

\begin{figure}[h!]
	\centering
	\includegraphics[width=0.7\textwidth]{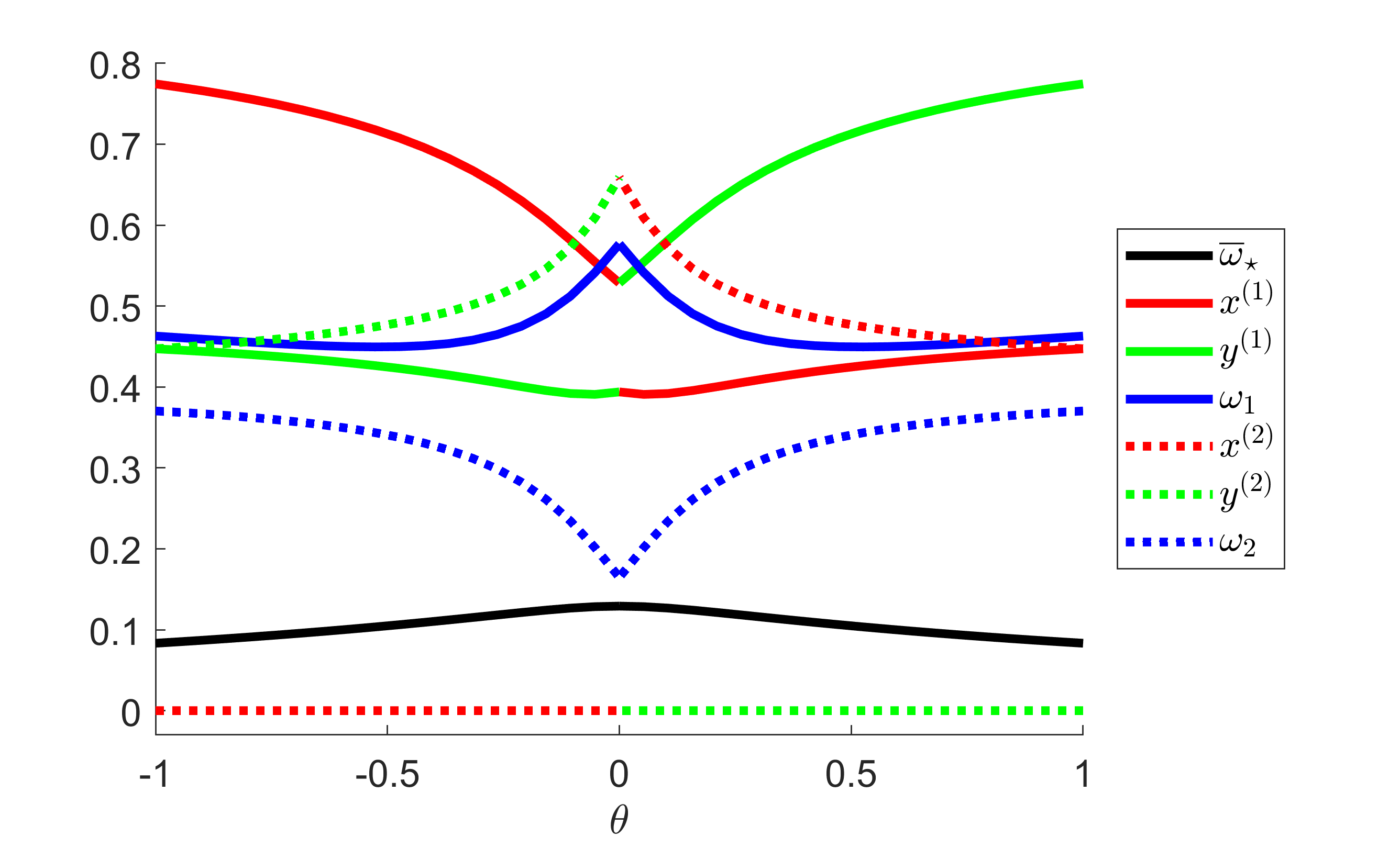}
	\caption{The boundary weight $\Wmu_\star$,  internal weights $\{\omega_s\}$ and internal nodes $(x^{(s)},y^{(s)})$ of the symmetric OCAD \eqref{eq:2879} for $\theta \in [-1,1]$. }\label{fig:2417}
\end{figure}

\begin{remark}
	For comparison,	\Cref{fig:1882} illustrates our symmetric OCAD \eqref{eq:2879} and the classic CAD \eqref{eq:1471}. We clearly see 
	that the OCAD \eqref{eq:2879} involves much fewer internal nodes than the classic CAD. 
\end{remark}

We have also found out the 2D symmetric OCAD \eqref{eq:3724} for $\mathbb P^6$ and $\mathbb P^7$ spaces, whose discovery is highly nontrivial; see \Cref{thm:3650}. 
	
	\begin{theorem}\label{thm:3650}
		For any $\theta \in [-1,1]$, the 2D symmetric OCAD on $\Omega$ for $\mathbb P^6$ and $\mathbb P^7$ spaces is given by
		\begin{equation}\label{eq:3724}
		\langle p \rangle_{\Omega} = \overline \omega_\star 
		\left[  
		(1+\theta) \langle p \rangle^x_{\Omega}+	
		(1-\theta) \langle p \rangle^y_{\Omega} \right] 
		+
		\sum_{s = 1}^4\omega_s \, \overline{p\left(x^{(s)},y^{(s)}\right)}.
		\end{equation}
		The boundary weight is given by
		\begin{equation}\label{eq:3733}
		\overline{\omega}_\star = \left[2|\theta|+\frac{20}{3}+\frac{2}{3}\sqrt{126\,\theta^2+96|\theta|+94}\cos\left(\frac{1}{3}\arccos\frac{864|\theta|^3+2916\,\theta^2+288|\theta|-532}{(126\,\theta^2+96|\theta|+94)^{\frac{3}{2}}}\right)\right]^{-1}.
		\end{equation}
		The coordinates and the weights of internal nodes \# 1 and \# 2 are given by
		\begin{equation}\label{eq:3737}
		\left(x^{(1)}, y^{(1)}, \omega_1 \right) = 
		\begin{dcases}
		\left(
		\sqrt{\frac{m_{22}-\sqrt{\frac{(1-\beta_1)\beta_2}{\beta_1}} }{m_{02}}},
		\sqrt{\frac{m_{04}+\sqrt{\frac{(1-\beta_1)\beta_3}{\beta_1}} }{m_{02}}},
		\frac{\beta_1 m_{02}}{ ( y^{(1)} )^2}
		\right), & \textrm{if } \theta \in [-1,0], \\
		\left(
		\sqrt{\frac{m_{40}+\sqrt{\frac{(1-\beta_1)\beta_2}{\beta_1}} }{m_{20}}},
		\sqrt{\frac{m_{22}-\sqrt{\frac{(1-\beta_1)\beta_3}{\beta_1}} }{m_{20}}},
		\frac{\beta_1 m_{20}}{ ( x^{(1)} )^2}
		\right), & \textrm{if } \theta \in [ 0,1],
		\end{dcases}
		\end{equation}
		\begin{equation}\label{eq:3752}
		( x^{(2)}, y^{(2)}, \omega_2) = 
		\begin{dcases}
		\left(
		\sqrt{\frac{m_{22}+\sqrt{\frac{\beta_1\beta_2}{1-\beta_1}} }{m_{02}}},
		\sqrt{\frac{m_{04}-\sqrt{\frac{\beta_1\beta_3}{1-\beta_1}} }{m_{02}}},
		\frac{(1-\beta_1) m_{02}}{ ( y^{(2)} )^2}
		\right), & \textrm{if } \theta \in [-1,0], \\
		\left(
		\sqrt{\frac{m_{40}-\sqrt{\frac{\beta_1\beta_2}{1-\beta_1}} }{m_{20}}},
		\sqrt{\frac{m_{22}+\sqrt{\frac{\beta_1\beta_3}{1-\beta_1}} }{m_{20}}},
		\frac{(1-\beta_1) m_{20}}{ ( x^{(2)} ) ^2}
		\right), & \textrm{if } \theta \in [ 0,1]
		\end{dcases}
		\end{equation}
		with
		\begin{equation}\label{eq:3768}
			(\beta_1,\beta_2,\beta_3) = 
			\begin{dcases}
			\left(
			1-\frac{m_{22}^2}{m_{42} m_{02}} + \frac{\sqrt{30}+2}{36} \theta^2, 
			m_{42}m_{02}-m_{22}^2,
			m_{06}m_{02}-m_{04}^2
			\right)
			& \textrm{ if } \; \theta \in [-1,0], \\
			\left(
			1-\frac{m_{22}^2}{m_{24} m_{02}} + \frac{\sqrt{30}+2}{36} \theta^2, 
			m_{60}m_{20}-m_{40}^2,
			m_{24}m_{20}-m_{22}^2
			\right)
			& \textrm{ if } \; \theta \in [ 0,1],
			\end{dcases}
		\end{equation}
		and
		\[
			m_{ij} = \frac{1}{(i+1)(j+1)} - \overline{\omega}_\star 
			\left[\frac{1+\theta}{1+j} + \frac{1-\theta}{1+i} \right], \quad i,j = 0,2,4,6, \; i+j \le 6.
		\]
		The weights and the coordinates of internal nodes \# 3 and \# 4 are given by
		\begin{equation}\label{eq:3791}
			\omega_3 = \frac{m_0}{2}\left(1+\frac{\beta_4}{\sqrt{\beta_4^2+4}}\right),
			\quad
			\omega_4 = \frac{m_0}{2}\left(1-\frac{\beta_4}{\sqrt{\beta_4^2+4}}\right),
		\end{equation}
		\begin{equation}\label{eq:3796}
			(x_3,y_3,x_4,y_4) = 
			\begin{dcases}
				\left( 
				\sqrt{\frac{m_2}{m_0}-\sqrt{\frac{\omega_4}{\omega_3}} \frac{\sqrt{m_4 m_0-m_2^2}}{m_0}},
				0,
				\sqrt{\frac{m_2}{m_0}+\sqrt{\frac{\omega_3}{\omega_4}} \frac{\sqrt{m_4 m_0-m_2^2}}{m_0}},
				0
				\right)
				& \textrm{ if } \; \theta \in [-1,0],
				\\
				\left( 
				0,
				\sqrt{\frac{m_2}{m_0}-\sqrt{\frac{\omega_4}{\omega_3}} \frac{\sqrt{m_4 m_0-m_2^2}}{m_0}},
				0,
				\sqrt{\frac{m_2}{m_0}+\sqrt{\frac{\omega_3}{\omega_4}} \frac{\sqrt{m_4 m_0-m_2^2}}{m_0}}
				\right)
				& \textrm{ if } \; \theta \in [0,1]
			\end{dcases}
		\end{equation}
		with	
\begin{align*}
				\beta_4 &= \frac{m_6 m_0^2 - 3 m_4 m_2 m_0 + 2 m_2^3}{(m_4 m_0 - m_2^2)^{3/2}}, 
			\\
			m_k &= 
			\begin{dcases}
			m_{k0}-\omega_1 \left( x^{(1)} \right)^k-\omega_2  \left( x^{(2)} \right)^k, & \textrm{ if } \; \theta \in [-1,0], \\
			m_{0k}-\omega_1 \left( y^{(1)} \right)^k-\omega_2 \left( y^{(2)} \right)^k, & \textrm{ if } \; \theta \in [0 ,1],
			\end{dcases}  
			\quad k = 0,2,4,6.
\end{align*}
		Moreover, Conjectures \ref{con:2070} and \ref{con:2071} hold true for all $\theta\in [-1,1]$ and $6\le k \le 7$.
	\end{theorem}
	
	For better readability, we put the proof of \Cref{thm:3650} in \Cref{sec:A1}. The weights and nodes \eqref{eq:3733}--\eqref{eq:3796} are illustrated in \Cref{fig:3862}, which clearly verifies the feasibility conditions (ii) and (iii) in \Cref{def:2D_FCAD}.
	
	\begin{figure}
	\centerline{
		\includegraphics[width=0.24\textwidth]{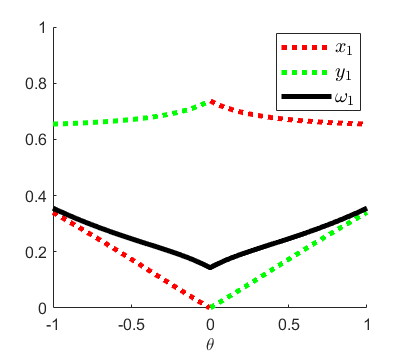}
		\includegraphics[width=0.24\textwidth]{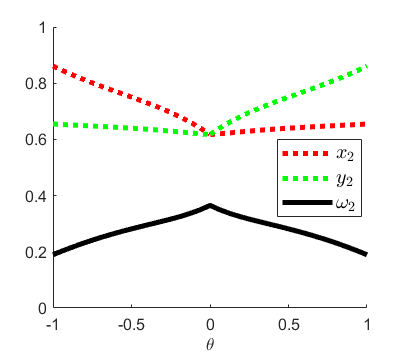}
		\includegraphics[width=0.24\textwidth]{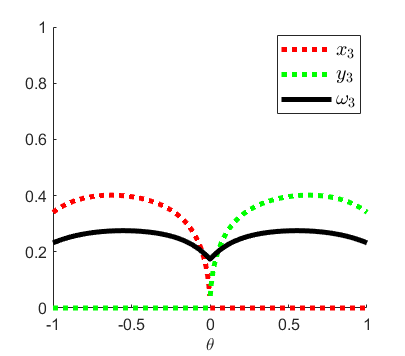}
		\includegraphics[width=0.24\textwidth]{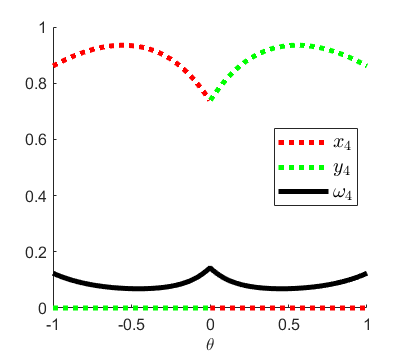}
	}
	\caption{The weights and the internal nodes of the OCAD \eqref{eq:3724} for $\theta \in [-1,1]$.}\label{fig:3862}
\end{figure}
	
	\begin{remark}
		For comparison,	\Cref{fig:1882} illustrates our symmetric OCAD \eqref{eq:3724} and the classic CAD \eqref{eq:1471}. We clearly see 
		that the OCAD \eqref{eq:3724} involves much fewer internal nodes than the classic CAD. 
	As $\theta \to 0$, the symmetric OCAD \eqref{eq:3724} converges to the fully symmetric OCAD in \eqref{eq:2004}, for which the critical positive polynomial is not unique (as it was explained in \Cref{fig:1898b} due to the non-smoothness of $\partial \mathbb B_\omega$).
	\end{remark}
	

\begin{table}
	\centering
	\caption{ Internal nodes of the OCADs (for $\theta = -1, -0.2, 0$) and the classic CAD for $\mathbb{P}^k$, $k = 2,\dots,11$.  The critical polynomials of the OCADs are also displayed.}\label{fig:1882}
	\begin{tabular}{ c @{\hspace{-1mm}} m{34mm} m{30mm} m{30mm} m{30mm} m{5mm} }
		
		\toprule[1.5pt]
		
		
		polynomial & \hspace{ 7mm} {\tt classic} CAD & \multicolumn{4}{c}{{\tt optimal} CAD } \\[0.1mm]
		\cmidrule(lr){2-2} \cmidrule(lr){3-6}
		space & \hspace{11mm} $\theta \in [-1,1] $ & \hspace{13mm} $\theta = -1$ & \hspace{13mm} $\theta = -0.2$ & \hspace{13mm} $\theta = 0$ & \\
		
		\midrule[1.5pt]
		
		$\mathbb{P}^2$ or $\mathbb{P}^3$ &
		\includegraphics[scale=0.28,trim=7 5 70 0,clip]{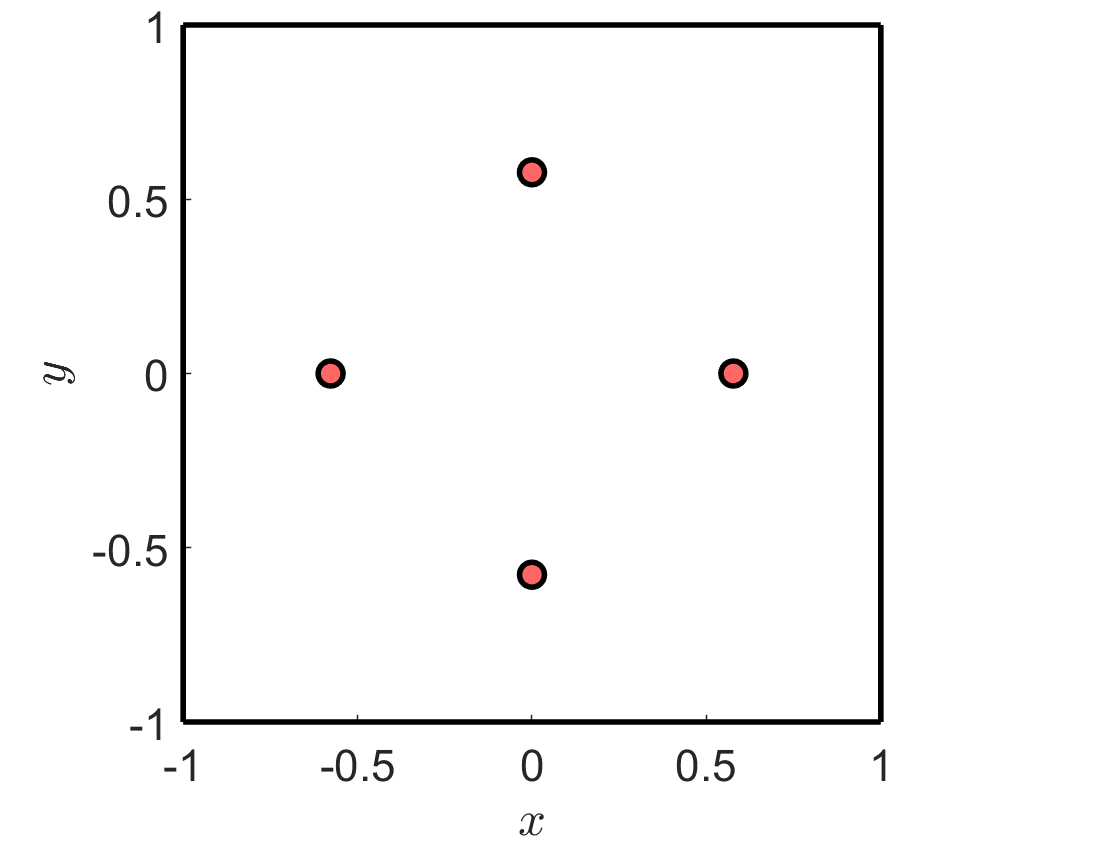} &
		\includegraphics[scale=0.28,trim=7 5 70 0,clip]{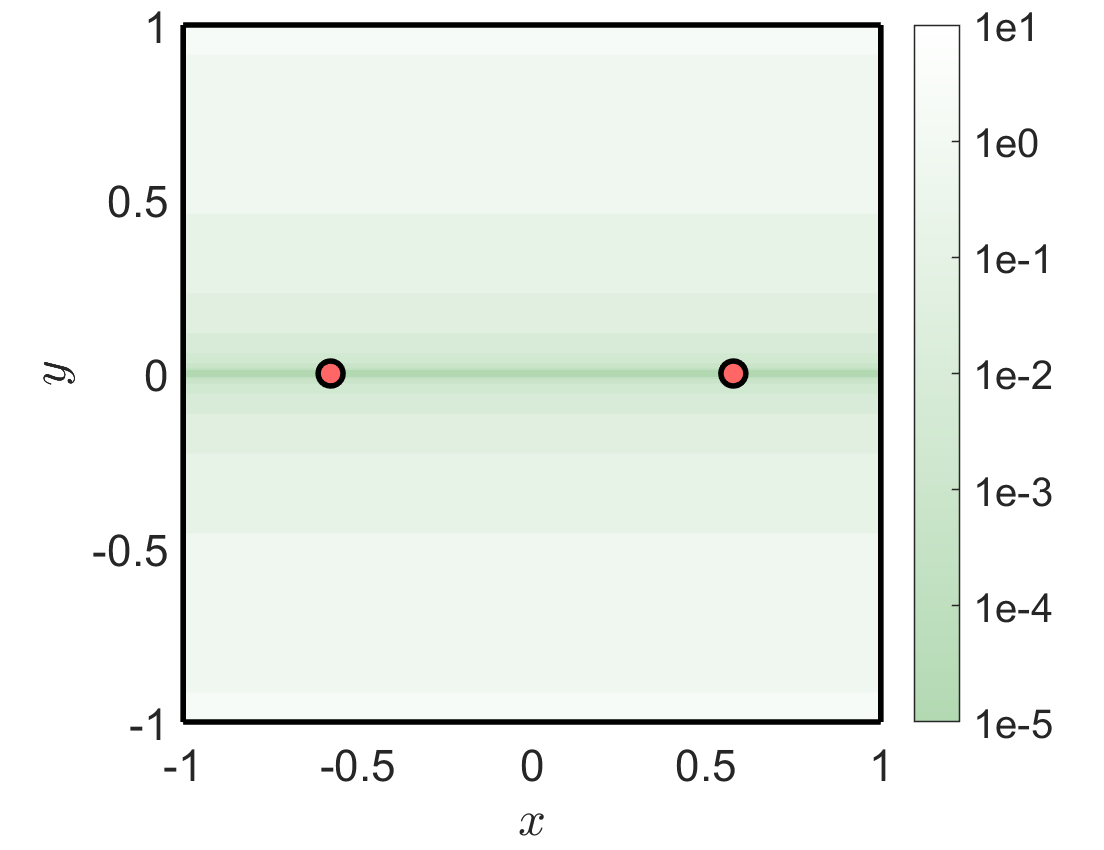} &
		\includegraphics[scale=0.28,trim=7 5 70 0,clip]{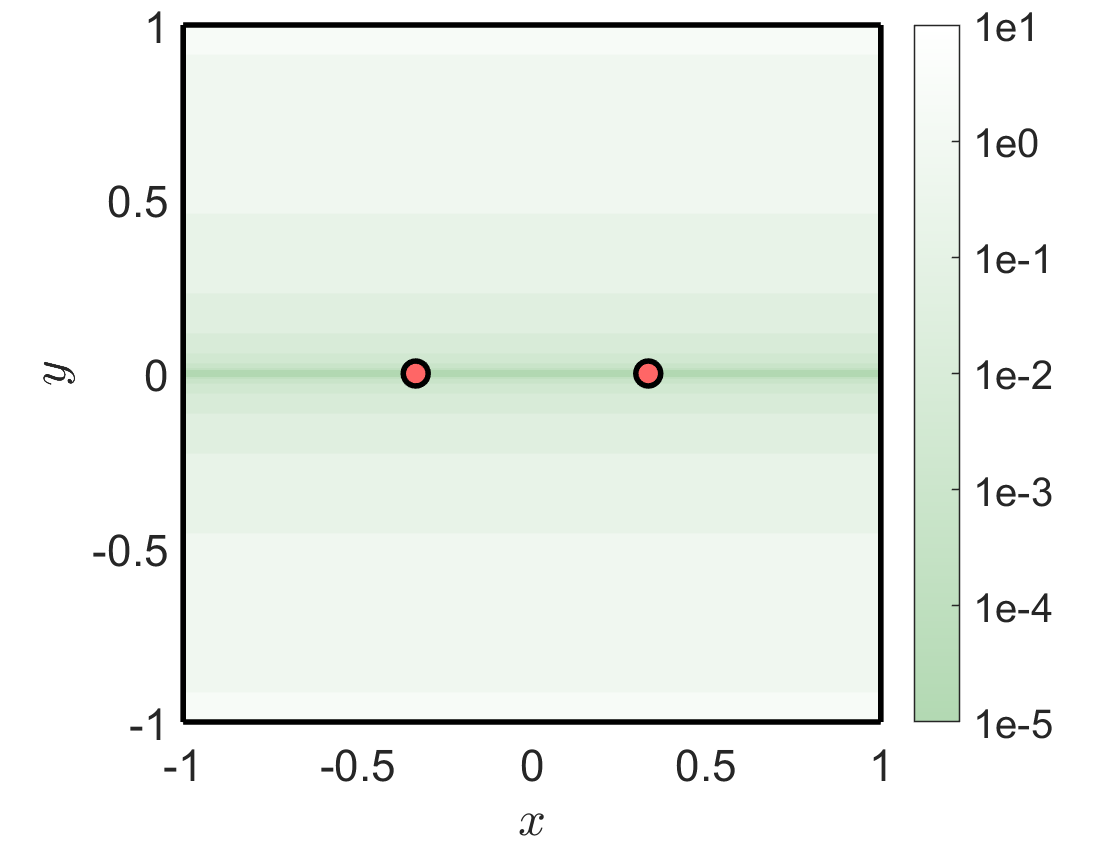} &
		\includegraphics[scale=0.28,trim=7 5 70 0,clip]{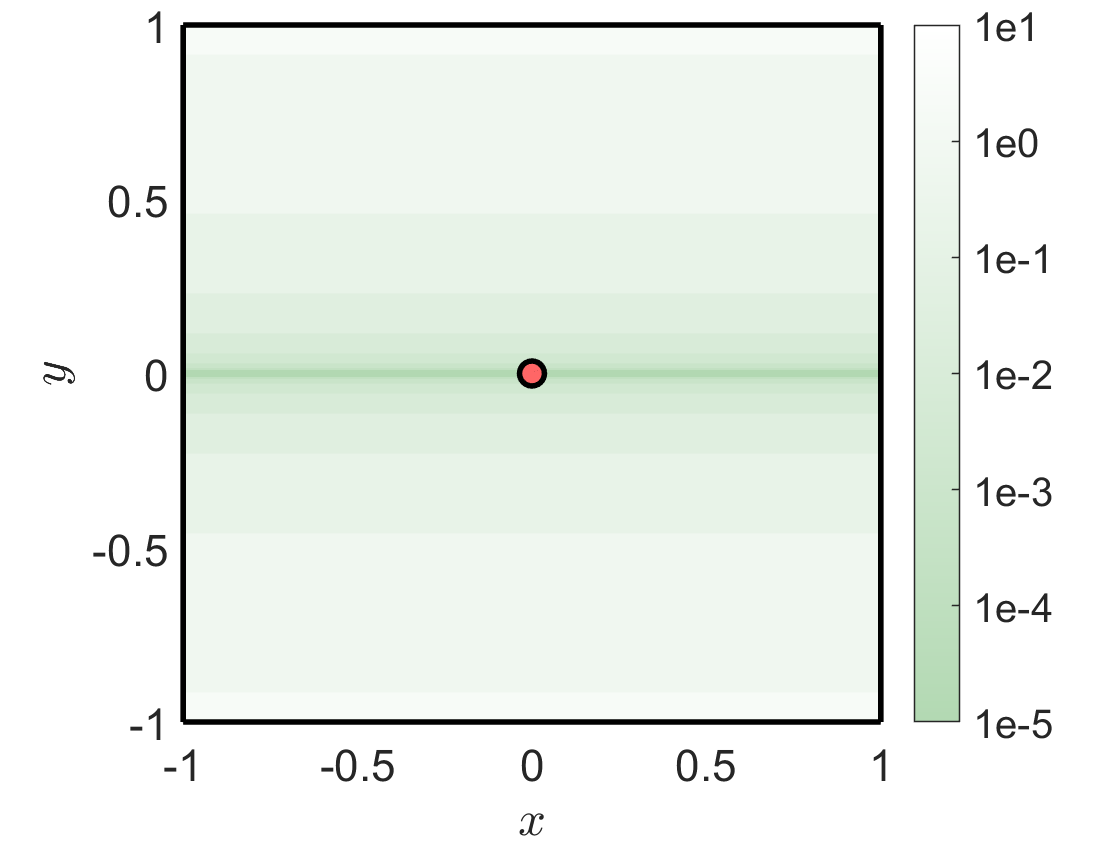} &
		\includegraphics[scale=0.28,trim=323 5 0 0,clip]{fig/C2_Pk/P2/movie_node_2/00000000.png} \\
		
		$\mathbb{P}^4$ or $\mathbb{P}^5$ &
		\includegraphics[scale=0.28,trim=7 5 70 0,clip]{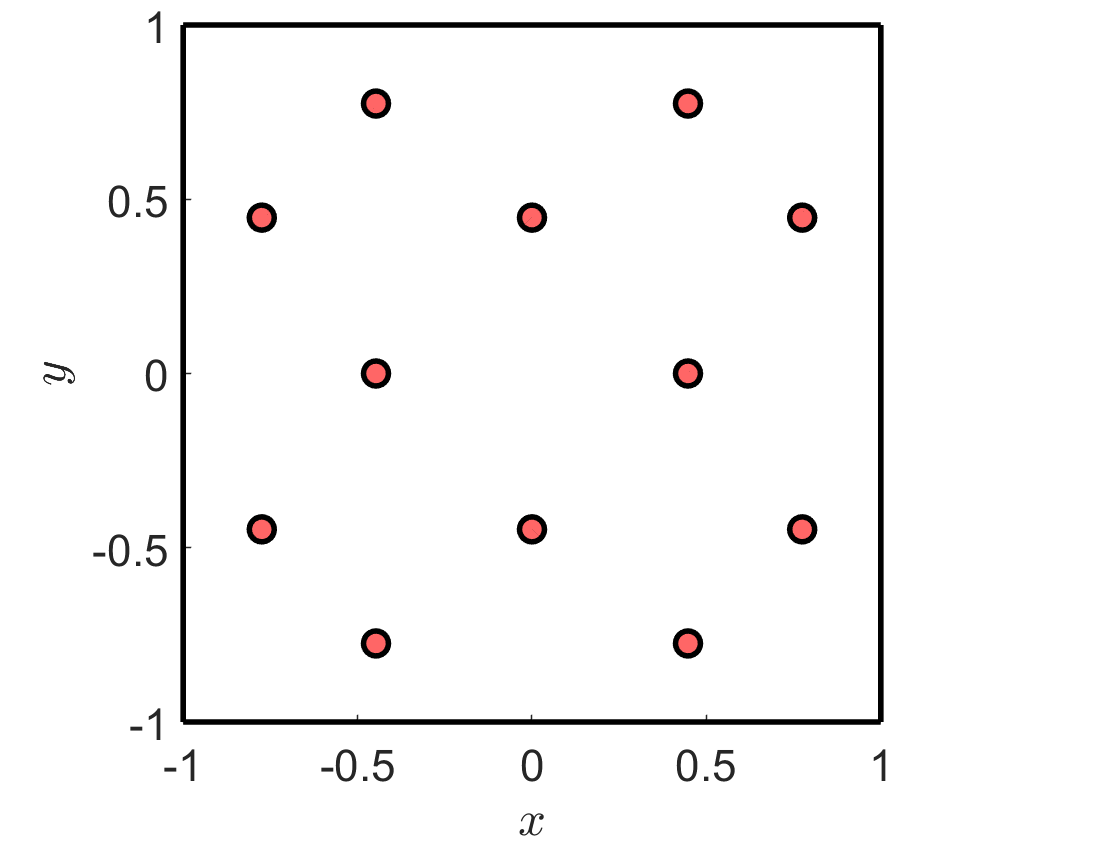} &
		\includegraphics[scale=0.28,trim=7 5 70 0,clip]{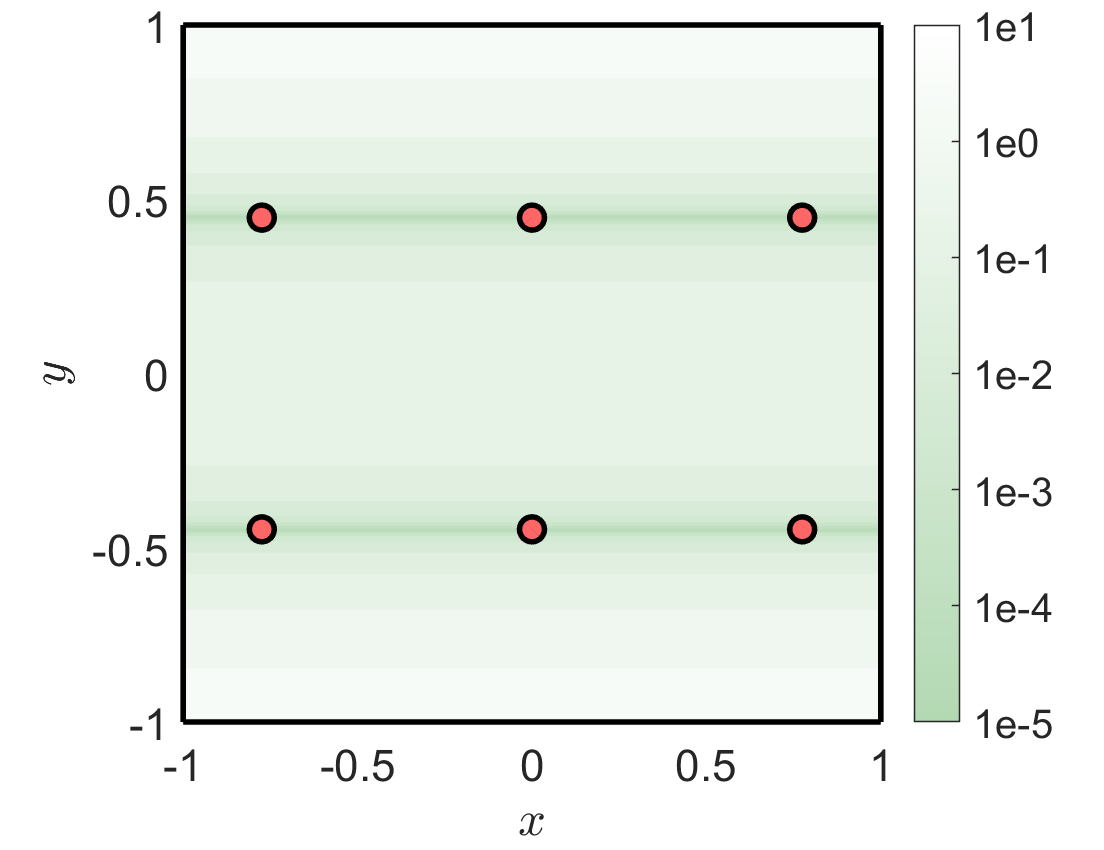} &
		\includegraphics[scale=0.28,trim=7 5 70 0,clip]{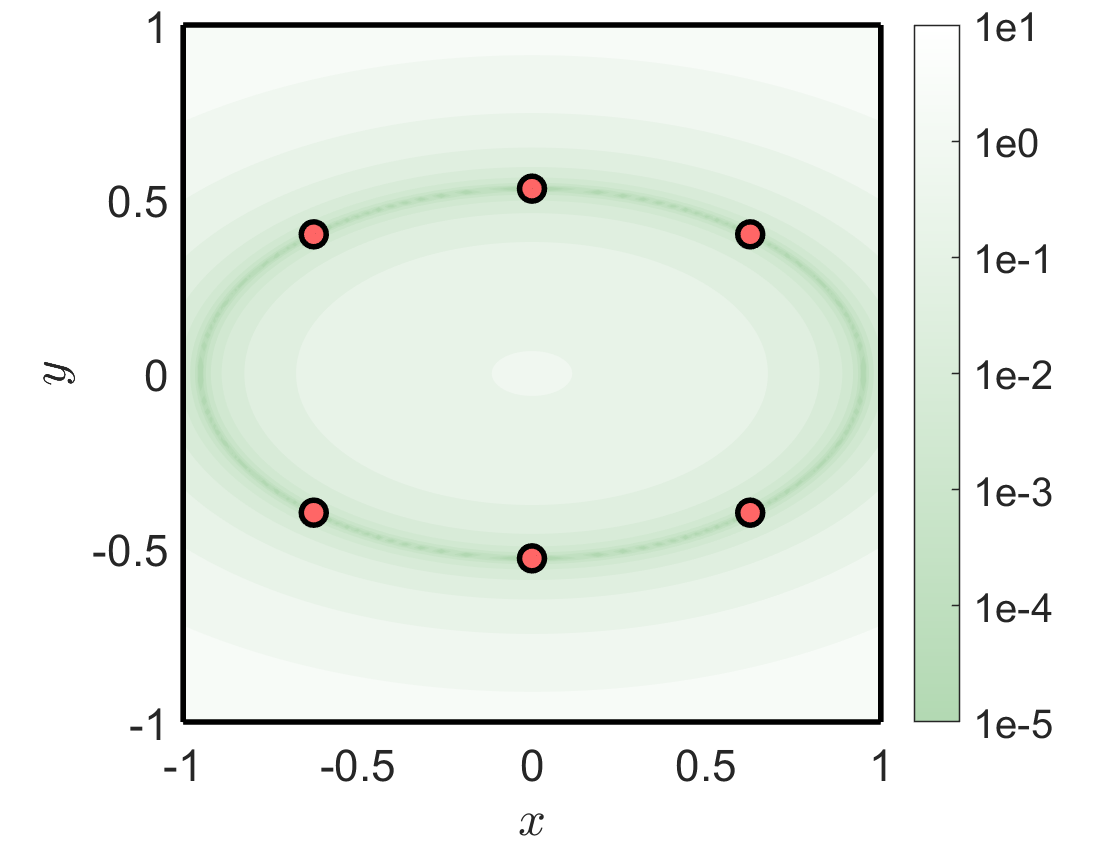} &
		\includegraphics[scale=0.28,trim=7 5 70 0,clip]{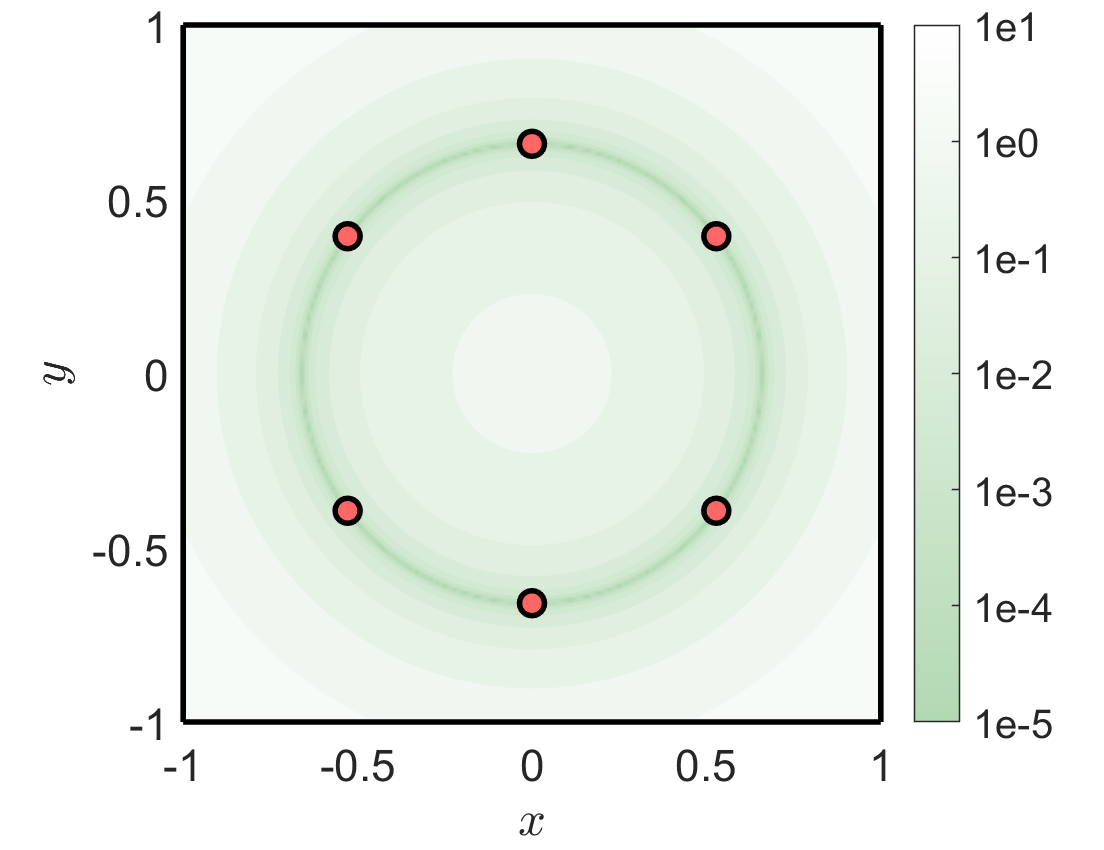} &
		\includegraphics[scale=0.28,trim=323 5 0 0,clip]{fig/C2_Pk/P4/movie_node_3/000000.png} \\
		
		$\mathbb{P}^6$ or $\mathbb{P}^7$ &
		\includegraphics[scale=0.28,trim=7 5 70 0,clip]{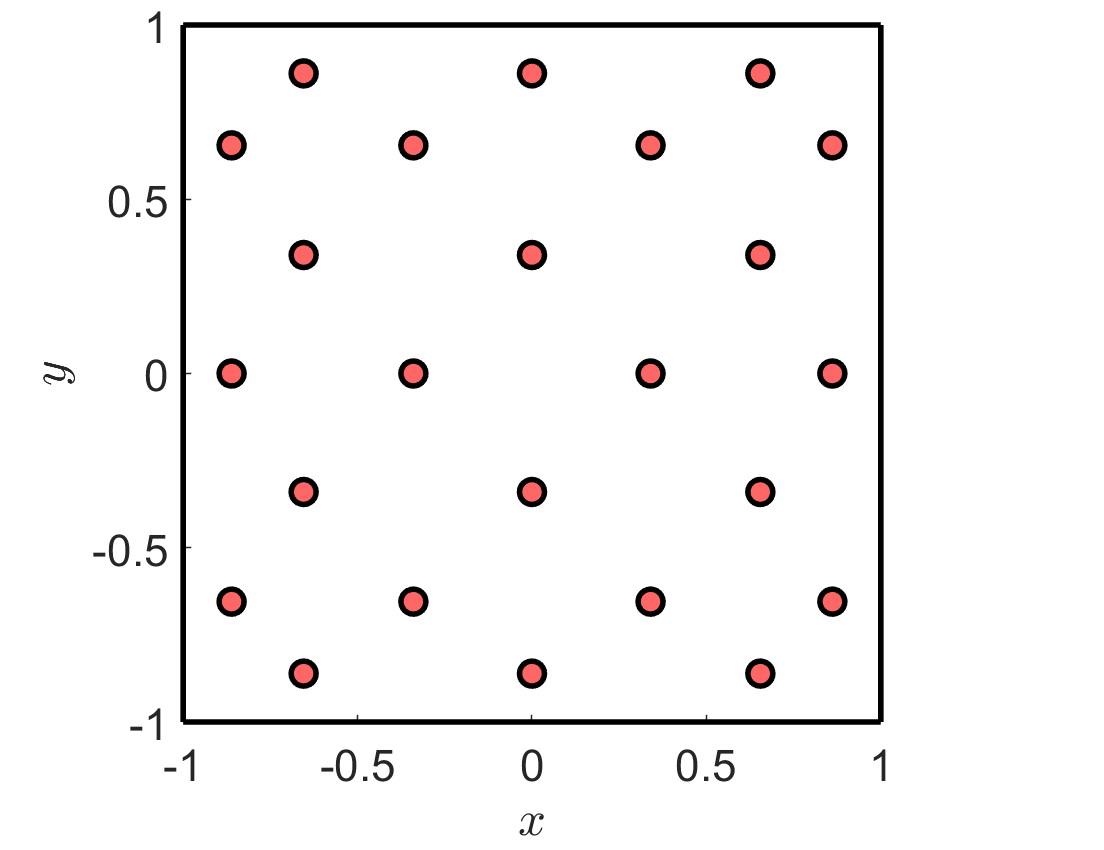} &
		\includegraphics[scale=0.28,trim=7 5 70 0,clip]{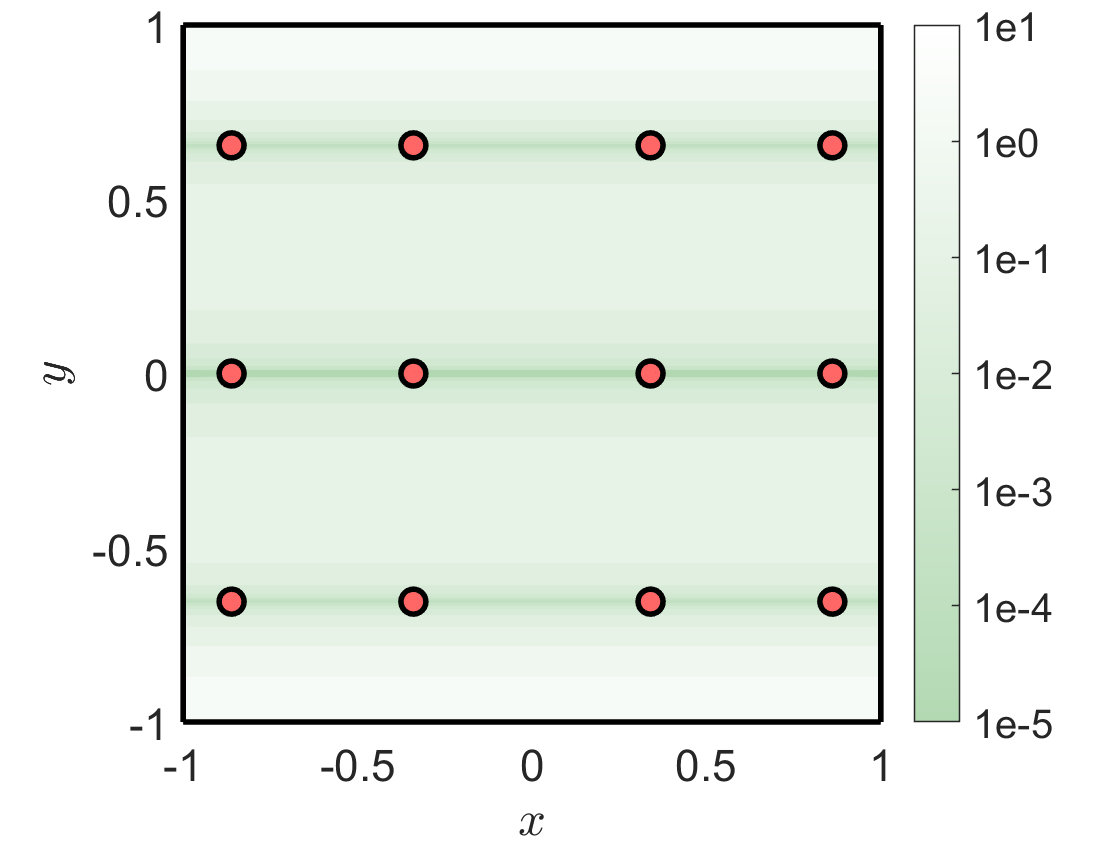} &
		\includegraphics[scale=0.28,trim=7 5 70 0,clip]{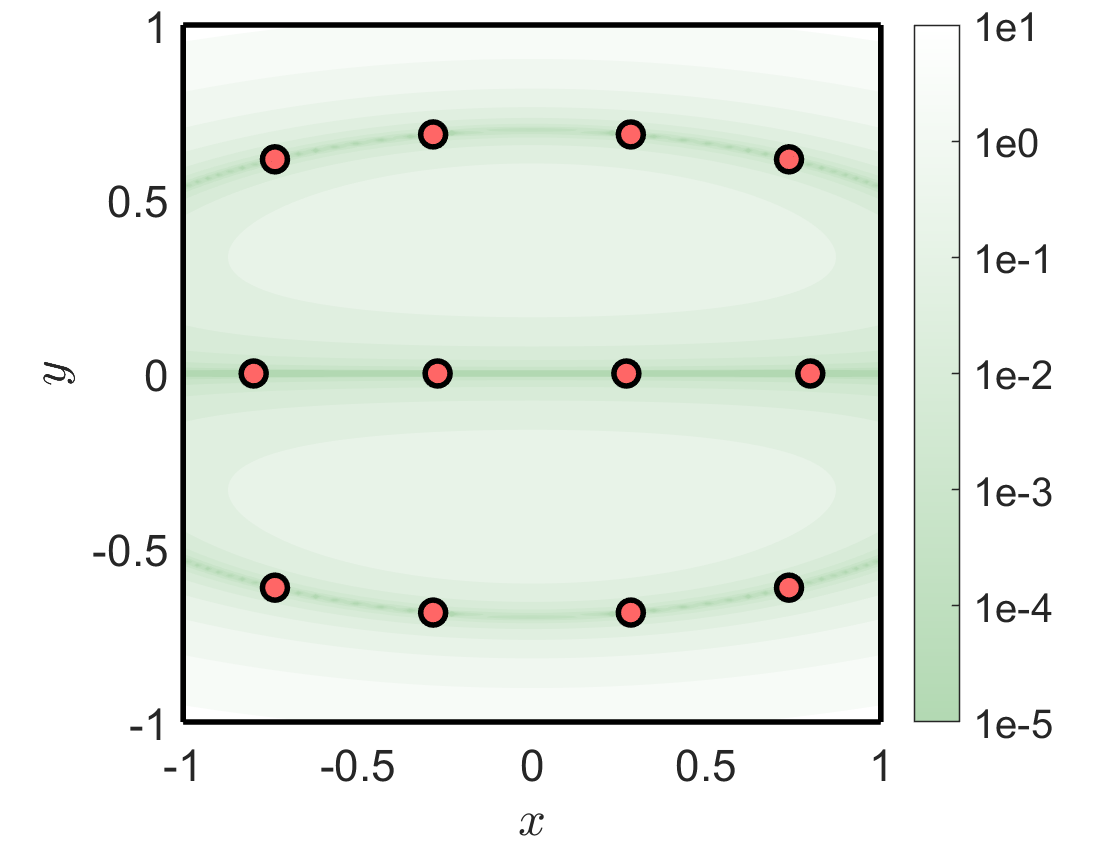} &
		\includegraphics[scale=0.28,trim=7 5 70 0,clip]{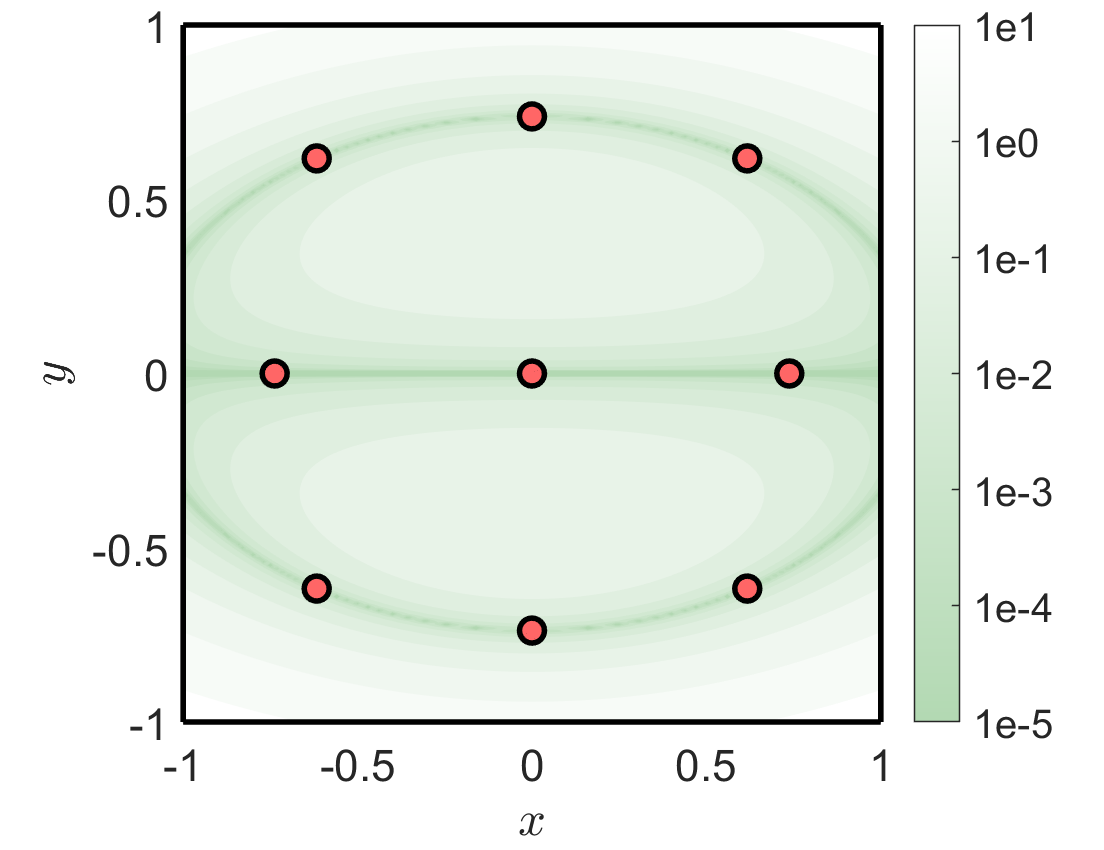} &
		\includegraphics[scale=0.28,trim=323 5 0 0,clip]{fig/C2_Pk/P6/movie_node_44/000000.png} \\
		
		$\mathbb{P}^8$ or $\mathbb{P}^9$ &
		\includegraphics[scale=0.28,trim=7 5 70 0,clip]{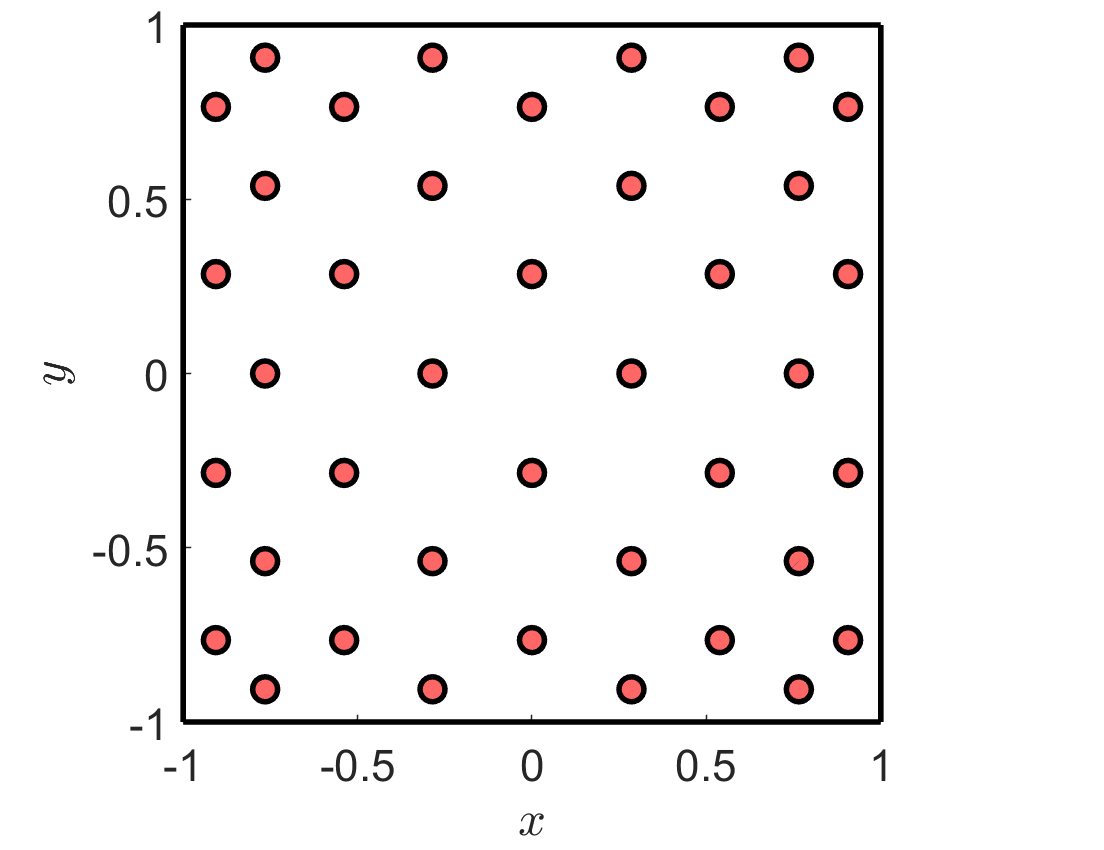} &
		\includegraphics[scale=0.28,trim=7 5 70 0,clip]{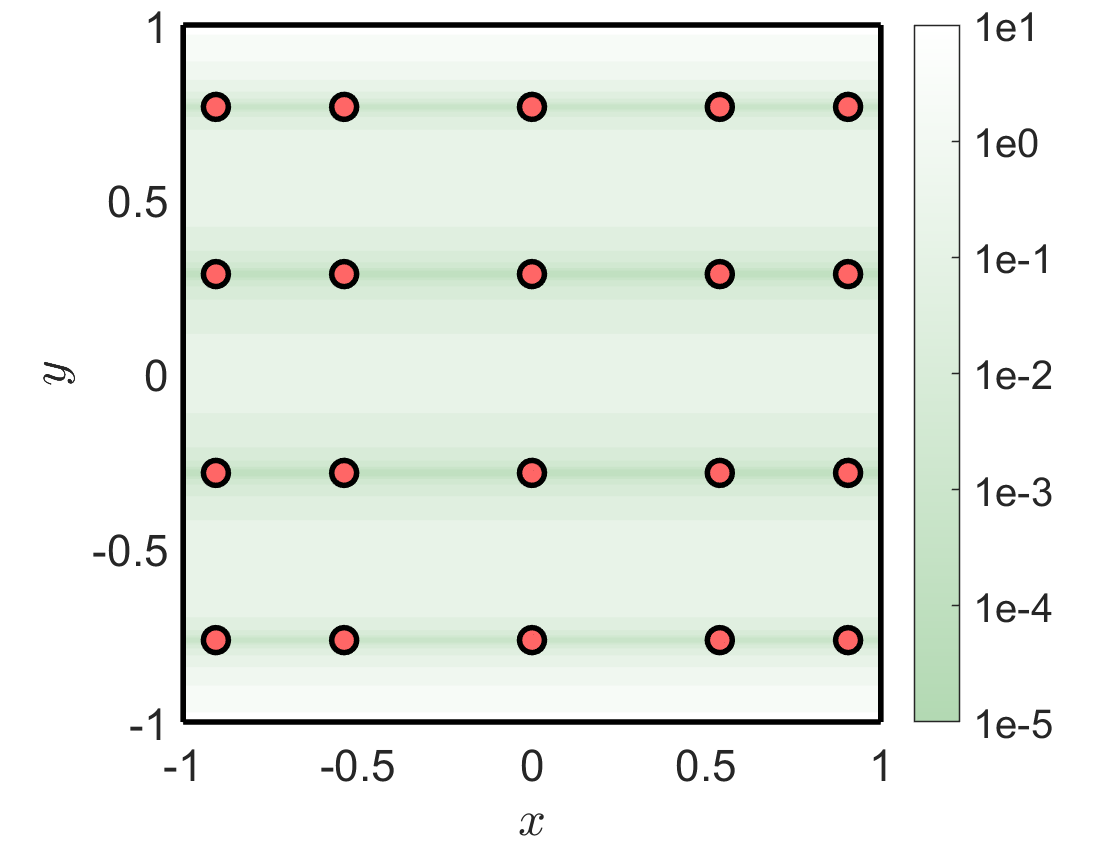} &
		\includegraphics[scale=0.28,trim=7 5 70 0,clip]{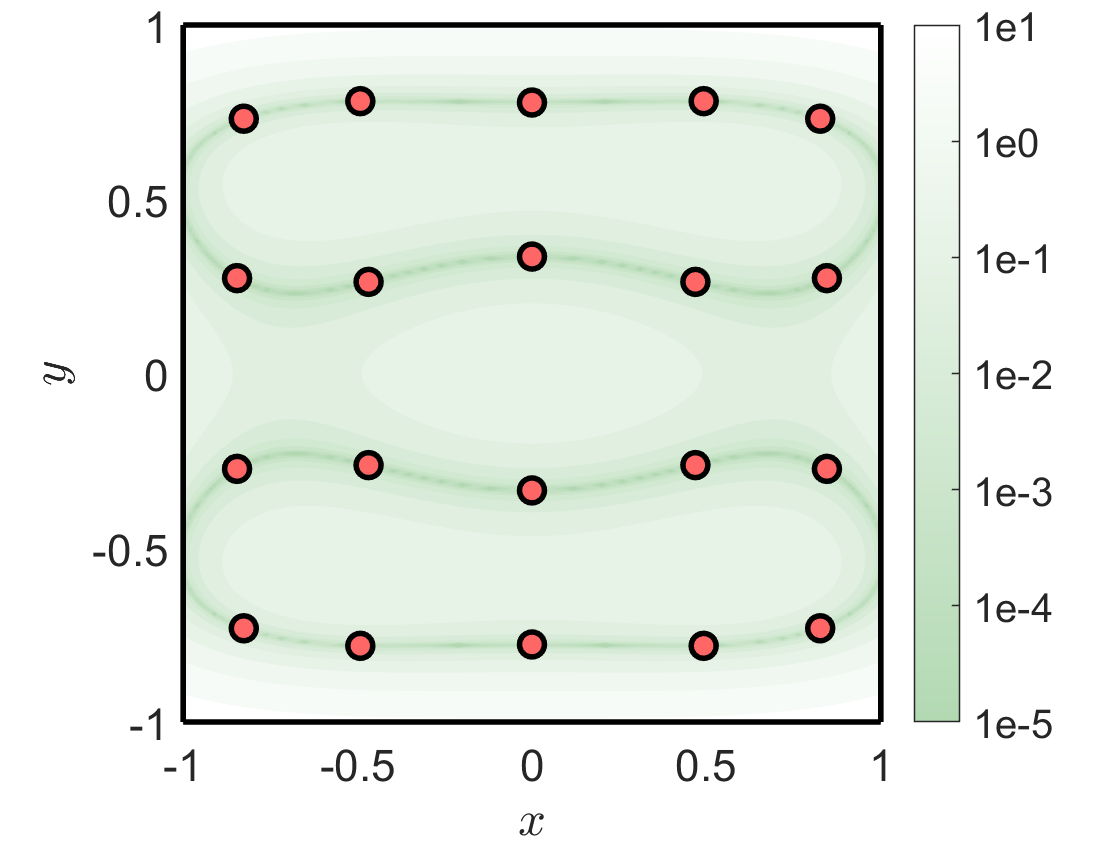} &
		\includegraphics[scale=0.28,trim=7 5 70 0,clip]{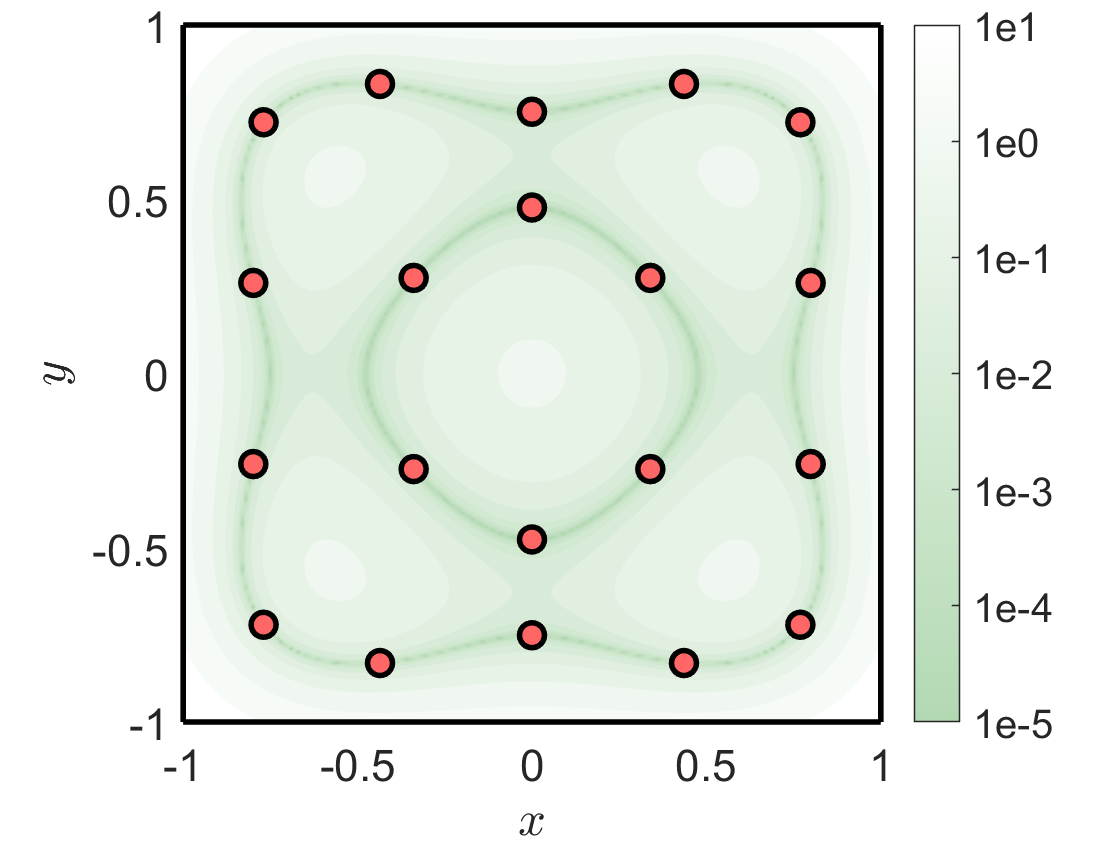} &
		\includegraphics[scale=0.28,trim=323 5 0 0,clip]{fig/C2_Pk/P8/movie_node_55/000000.png} \\
		
		$\mathbb{P}^{10}$ or $\mathbb{P}^{11}$ &
		\includegraphics[scale=0.28,trim=7 5 70 0,clip]{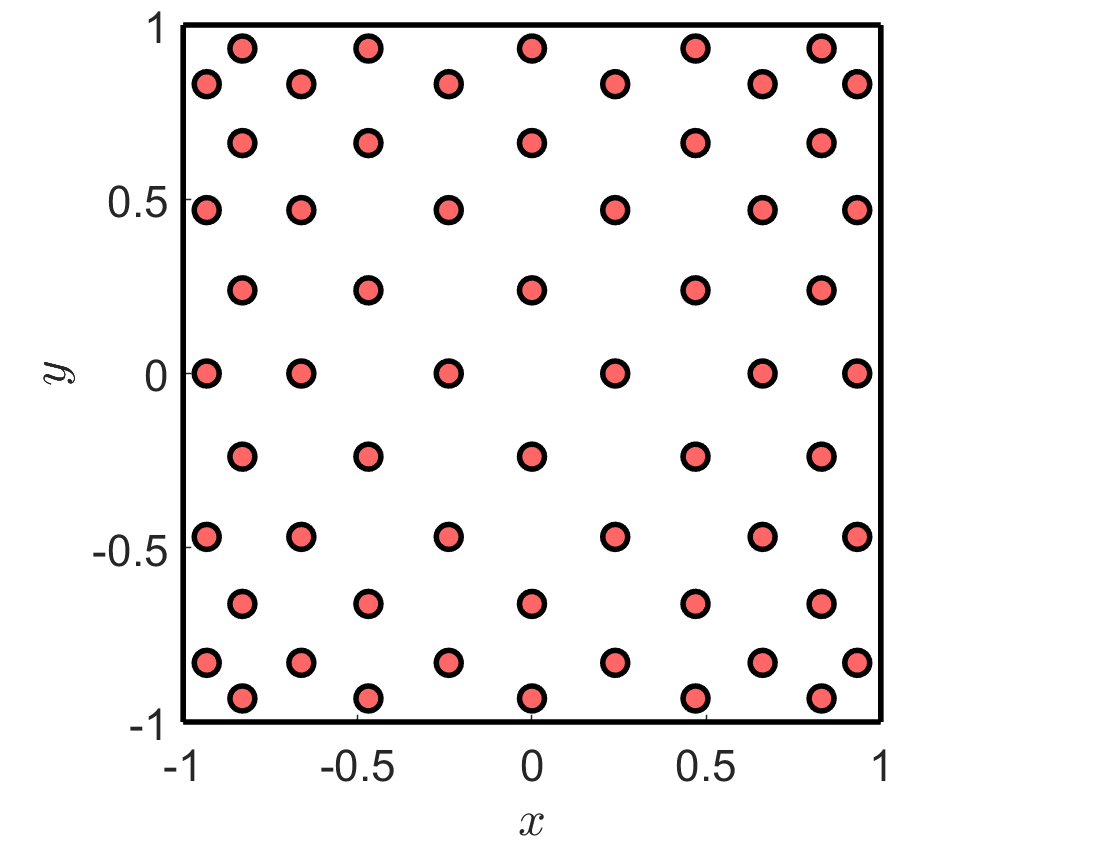} &
		\includegraphics[scale=0.28,trim=7 5 70 0,clip]{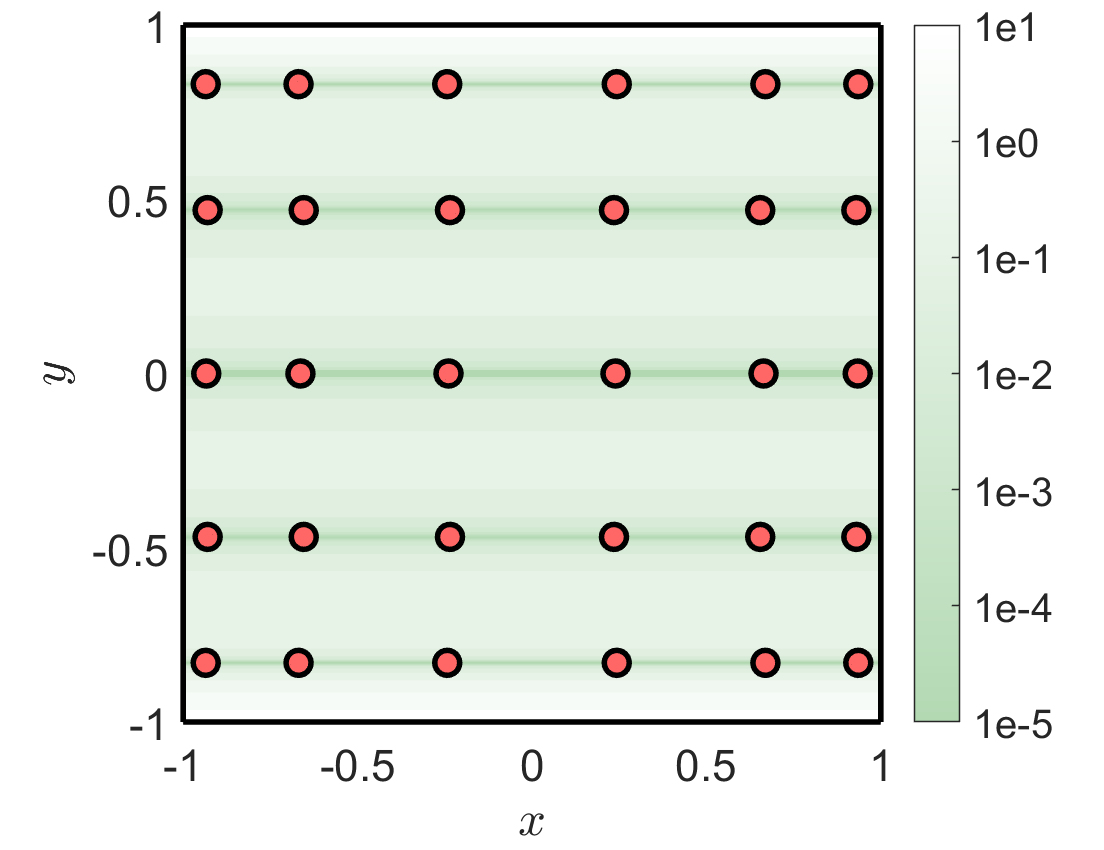} &
		\includegraphics[scale=0.28,trim=7 5 70 0,clip]{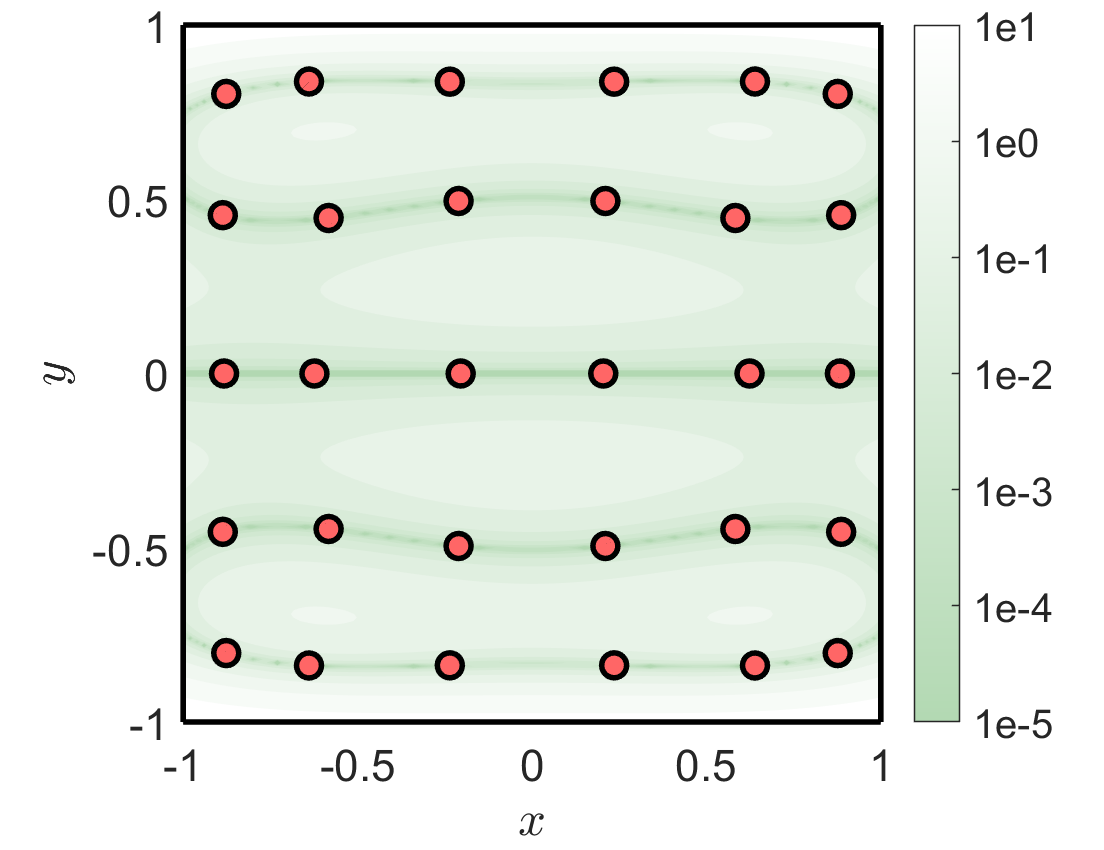} &
		\includegraphics[scale=0.28,trim=7 5 70 0,clip]{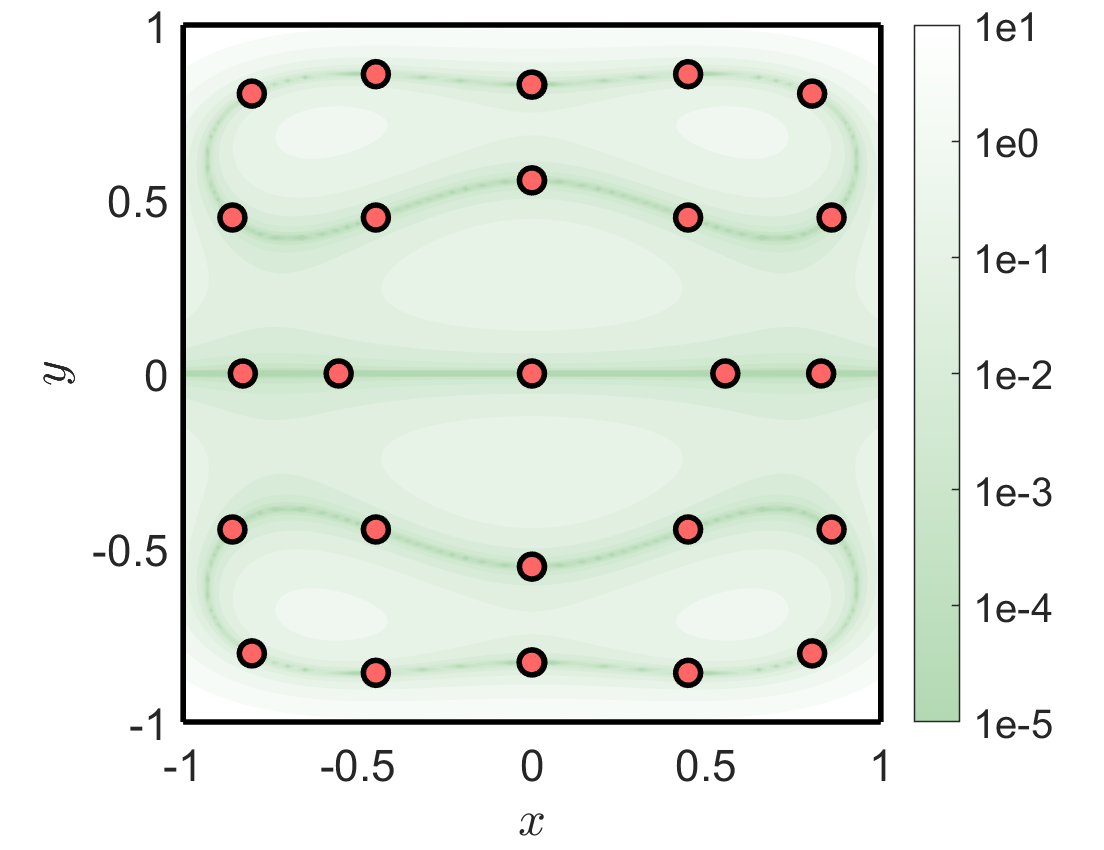} &
		\includegraphics[scale=0.28,trim=323 5 0 0,clip]{fig/C2_Pk/P10/movie_node_666/00000000.png} \\
		
		\bottomrule[1.5pt]
	\end{tabular}
\end{table}

	We have derived the explicit formulas of OCADs for $\mathbb{P}^2$ to $\mathbb{P}^7$ spaces. 
	However, for $\mathbb P^k$ spaces with $k\ge 8$, seeking the analytical form of the OCADs for general $\theta \in [-1,1]$ is very challenging (if not impossible). In the following, we propose a systematic approach to numerically construct OCADs for higher $k \ge 8$.  
	

	\begin{theorem}\label{thm:generalthetaPK}
		For any given $k \in \mathbb{N}_+$ and $\theta \in [-1,1]$, 
		let   
		$q_\star(x,y) \in \mathbb{P}^{\left \lfloor \frac{k}2 \right \rfloor}$ be the polynomial defined in  \eqref{eq:2065}. 
		If there exist $\{(\omega_s,x^{(s)},y^{(s)}): \omega_s > 0,  x^{(s)}, y^{(s)} \in [0,1]   \}_{s=1}^S$ satisfying the following system 
				\begin{equation}\label{eq:3407}
			\left\{
			\begin{aligned}
				& q_\star(x^{(s)},y^{(s)})  = 0, \\
				& \langle g_i \rangle_\Omega = \phi^\star(\theta,(\mathbb{P}^{ \left \lfloor \frac{k}2 \right \rfloor })^2) \Big[
				(1+\theta) \langle g_i \rangle_\Omega^{x}
				+ (1-\theta) \langle g_i \rangle_\Omega^{y}
				\Big]+\sum^{S}_{s=1} \omega_s \overline{g_i(x^{(s)},y^{(s)})}, \quad 1\le i \le \dim( \mathbb{P}^k(\mathscr G_{s})  )
			\end{aligned}
			\right.
		\end{equation}
		with $\{g_i\}$ being a basis of $\mathbb{P}^k(\mathscr G_{s})$, then the following symmetric CAD 
		\begin{equation}\label{eq:3397}
	\langle p \rangle_\Omega
	= \phi^\star(\theta,(\mathbb{P}^{ \left \lfloor \frac{k}2 \right \rfloor })^2)
	\Big[
	(1+\theta) \langle p \rangle_\Omega^{x}
	+ (1-\theta) \langle p \rangle_\Omega^{y}
	\Big]
	+
	\sum^{S}_{s=1} \omega_s \overline{p(x^{(s)},y^{(s)})} 
\end{equation}
is an OCAD for $\mathbb P^k$ and $\theta \in [-1,1]$. 		
	Furthermore, 
	in this case, 
	Conjectures \ref{con:2070} and \ref{con:2071} hold true with 
	\begin{equation}\label{eq:Pk-conj}
		\Wmu_\star(\theta,\mathbb{P}^k) = \phi^\star (\theta,\mathbb{P}^k_{+})= \phi^\star(\theta,(\mathbb{P}^{ \left \lfloor \frac{k}2 \right \rfloor })^2) = \phi(q^2_\star;\theta), 
	\end{equation}
	and $q_\star^2(x,y)$ is 
	the critical positive polynomial for $\phi^\star (\theta,\mathbb P_+^k)$. 	
	\end{theorem}
	
	\begin{proof}
		Since $\{g_i\}$ are a basis of $\mathbb{P}^k(\mathscr G_{s})$, 
		the decomposition \eqref{eq:3397} holds for any $g \in \mathbb{P}^k(\mathscr G_{s})$. 
		Thanks to \Cref{lem:Gs-invarint}, the symmetric CAD \eqref{eq:3397} is feasible for $\mathbb{P}^k$.
		Because the non-negative polynomial $q^2_\star \in \mathbb{P}^{\left \lfloor \frac{k}2 \right \rfloor}$ vanishes at all the internal nodes of \eqref{eq:3397}, the CAD \eqref{eq:3397} is optimal for $\mathbb{P}^k$ according to \Cref{lem:512}. Moreover, $\overline{\omega}_\star (\theta,\mathbb{P}^k) = \phi^\star(\theta,(\mathbb{P}^{ \left \lfloor \frac{k}2 \right \rfloor })^2)$, which implies the validity of Conjectures \ref{con:2070} and \ref{con:2071}.
	\end{proof}
	

Yet, when $k\ge 8$, we cannot find the analytical solution to the system \eqref{eq:3407} nor 
rigorously prove the existence of its solution. Based on \Cref{thm:generalthetaPK}, we 
propose the following algorithm for  
constructing the OCAD \eqref{eq:3183} by numerically solving the equations \eqref{eq:3407}.

	\begin{algorithm}\label{alg:3427}
		For $\theta \in (-1,1)$ and $k \ge 8$, 
		the symmetric OCAD for $\mathbb{P}^k$ space can be obtained via the following steps:
		\begin{enumerate}
			\item Calculate the value of $\phi^\star(\theta,(\mathbb{P}^{ \left \lfloor \frac{k}2 \right \rfloor })^2)$ according \eqref{eq:122}.
			\item Find the polynomial $q_\star(x,y)$ according to \eqref{eq:2065}.
			\item Solve the system \eqref{eq:3407} to obtain $\{\omega_s,x^{(s)},y^{(s)}\}$.
			\item Obtain the OCAD for $\mathbb{P}^k$ space in the form of \eqref{eq:3397}. 
		\end{enumerate}
	For $\theta = \pm 1$ and any $k \in \mathbb{N}_+$, the symmetric OCADs for $\mathbb{P}^k$ were given in \Cref{thm:1686,thm:1699}.
	\end{algorithm}



\begin{remark}
According to \Cref{lem:2k2kp1}, the symmetric OCAD for $\mathbb{P}^{2k}$ space is also a symmetric OCAD for $\mathbb{P}^{2k+1}$. According to 
\Cref{thm:sym_theta} and \Cref{cor:1575}, given the OCAD for $\mathbb{P}^{2k}$ space and $\theta \in [-1,0]$, 
we immediately obtain the corresponding symmetric OCAD for $\theta \in [0,1]$ via \eqref{eq:nTheta}. 
As such, we only need to seek the symmetric OCAD for $\mathbb{P}^{2k}$ space ($k\ge 4$) and $\theta \in [-1,0]$ via \Cref{alg:3427}. 
\end{remark}

In Step (3) of \Cref{alg:3427}, numerically solving the nonlinear algebraic equations \eqref{eq:3407} is a  nontrivial task. 
This is typically based on an iterative algorithm (we use the MATLAB built-in function {\tt fsolve}), which requires us to provide a good initial guess sufficiently close to the true solution so as to ensure the convergence. 
We choose the initial guess as follows. 
Given $k \in \mathbb{N}_{+}$ and $\theta \in [-1,0]$, we first partition the interval $[-1,\theta]$ into  $-1 = \theta_0 < \theta_1 < \dots < \theta_j < \dots <\theta_J = \theta$, and then perform \Cref{alg:3427} from $\theta_0$ to $\theta_J$ sequentially. 
For each $j\in \{1,\dots,J\}$, the initial guess of solving the system \eqref{eq:3407} corresponding $\theta_j$ is given by the final numerical solution to the system \eqref{eq:3407} coresponding to $\theta_{j-1}$. 
It is worth noting that, at the beginning of this process, 
the exact solution to the system \eqref{eq:3407} corresponding to $\theta_0 = -1$ was already analytically given in \Cref{thm:1686}.


For example, we apply the above-mentioned strategy  to find the OCADs for $\mathbb{P}^k$ spaces ($8\le k \le 11$) with $\theta = -1, -0.8, \dots, -0.2, 0$. The residual of each iteration is plotted in \Cref{fig:3491}, indicating the fast convergence of the iterations to machine accuracy. The internal nodes of the found OCADs (for $\theta = -1, -0.2, 0$) are illustrated in \Cref{fig:1882}. 

\begin{figure}[hbtp]
\centerline{
\begin{subfigure}[t][][t]{0.5\textwidth}
\includegraphics[width=\textwidth]{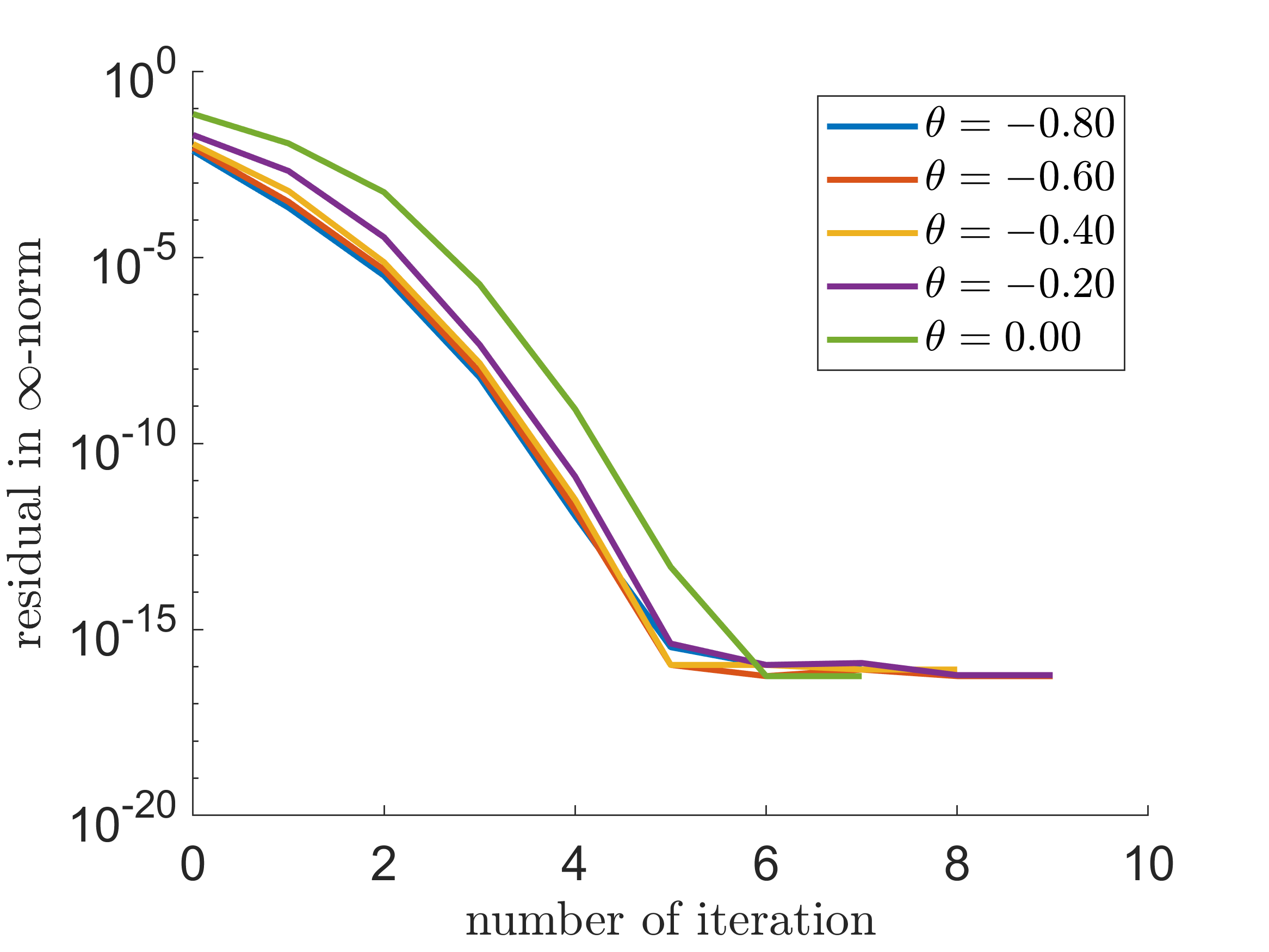}
\caption{$\mathbb{P}^8$ and $\mathbb{P}^9$}
\end{subfigure}
\begin{subfigure}[t][][t]{0.5\textwidth}
\includegraphics[width=\textwidth]{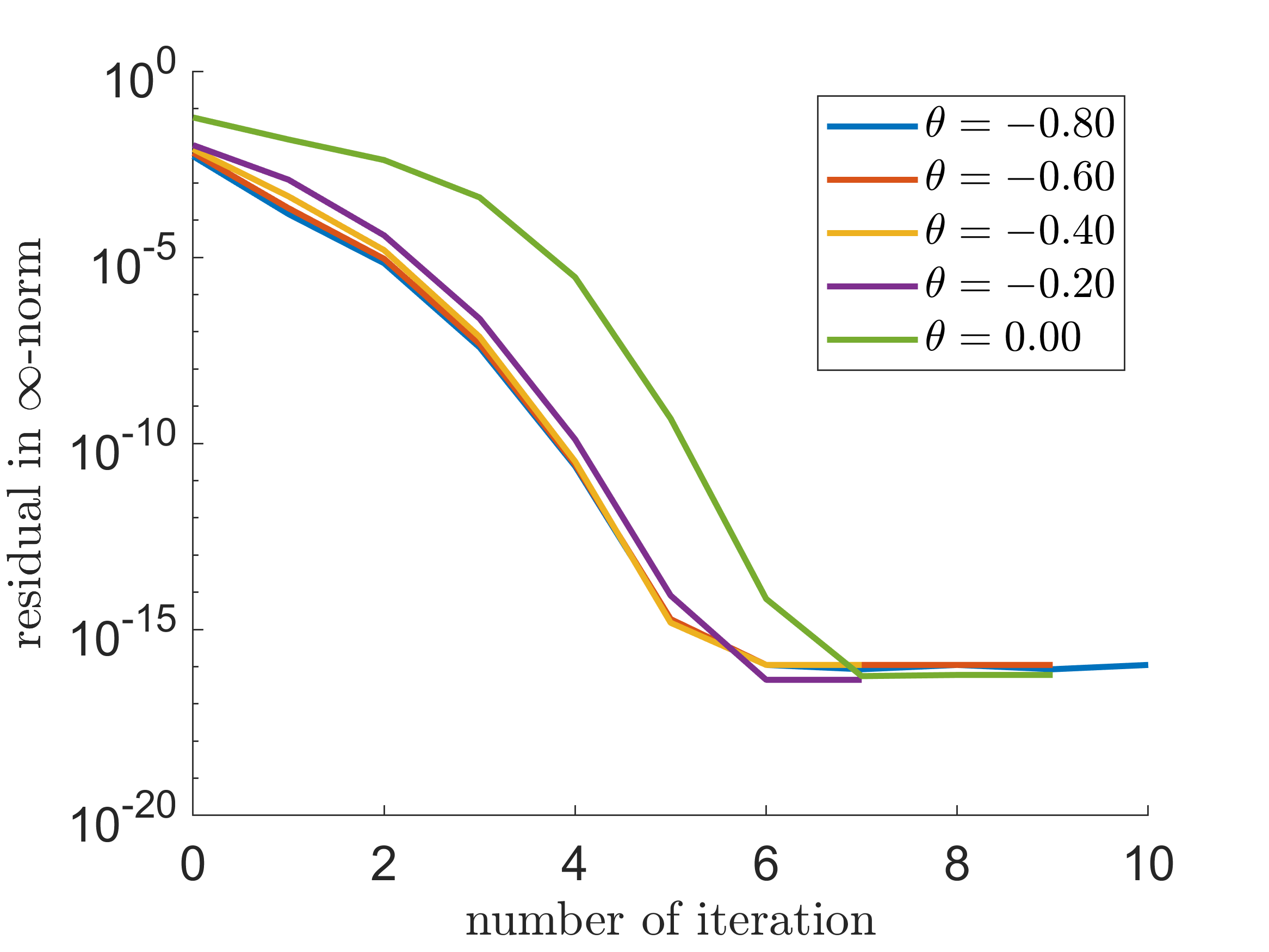}
\caption{$\mathbb{P}^{10}$ and $\mathbb{P}^{11}$}
\end{subfigure}
}
\caption{Fast convergence of the iterations for solving system \eqref{eq:3407} using the MATLAB built-in function {\tt fsolve} with the proposed initial-guess strategy. 
 }\label{fig:3491}
\end{figure}


\section{2D quasi-optimal CAD for $\mathbb{P}^{k}$ spaces}\label{sec:quasi-optimal}

As discussed in \Cref{rem:CCADcombine}, the classic CAD  
\eqref{eq:1471} is optimal for only $\theta = \pm 1$ but not optimal for general $\theta\in (-1,1)$. 
Recall that the classic CAD  
\eqref{eq:1471} 
is actually convex combinations of 
the two special OCADs (\ref{eq:1688}) and (\ref{eq:1712}) for $\theta = \pm 1$. 
For the OCADs (\ref{eq:1688}) and (\ref{eq:1712}), the corresponding 
boundary weights are repetitively $(0,  \overline \omega_\star (-1, \mathbb P^k)  )$ and 
$(\overline \omega_\star (1, \mathbb P^k), 0  )$, which are two 
vertexes on $\partial {\mathbb B}_\omega$. See \Cref{fig:1666} for an illustration. 
The boundary weights of 
the classic CAD  
\eqref{eq:1471} form a straight line segment between  
the two vertexes $(0,  \overline \omega_\star (-1, \mathbb P^k)  )$ and 
$(\overline \omega_\star (1, \mathbb P^k), 0  )$. 
Any straight line segment in 
the region ${\mathbb B}_\omega$ represents convex combinations of two feasible CADs. 
As observed from \Cref{fig1666:P2P3} for $\mathbb P^2$ and $\mathbb P^3$, the boundary weights of  
OCAD for all $\theta \in [-1,1]$ lie on $\partial {\mathbb B}_\omega$, forming two line segments that connect three vertexes $(0,  \overline \omega_\star (-1, \mathbb P^k)  )$,  
$(\overline \omega_\star (1, \mathbb P^k), 0  )$, and $(\frac12 \overline \omega_\star (0, \mathbb P^k), \frac12 \overline \omega_\star (0, \mathbb P^k)  )$. 
This intuitively reveals that the general OCAD \eqref{eq:2849} 
for $\mathbb P^2$ and $\mathbb P^3$ with $\theta \in [-1,1]$ is actually convex combinations of three 
special OCADs for $\theta \in \{-1,0,1\}$. 
However, as shown in \Cref{fig1666:P4,fig1666:P6,fig1666:P8}, for $\mathbb P^k$ with higher $k\ge 4$, the part of $\partial {\mathbb B}_\omega$ related to OCAD weights becomes curved, namely, it is no longer formed by straight edges. 
In spite of this, we discover that some convex combinations of three 
special OCADs for $\theta \in \{-1,0,1\}$ will still provide a feasible CAD, which  is very close to the OCAD and thus named as ``{\em quasi-optimal} CAD'' in the following.


In order to define the quasi-optimal CAD, let us denote the symmetric OCAD for $\mathbb{P}^{k}$  and $\theta = 0$ by
\begin{equation}\label{eq:2177}
	\langle p \rangle_{\Omega} = 
	\Wmu_{\star,0} \left [\langle p \rangle_{\Omega}^x+\langle p \rangle_{\Omega}^y \right]+  \sum_{s=1}^{S_0} \omega^{(s)}_{0} p( x_0^{(s)}, y_0^{(s)}) \qquad \forall p \in \mathbb{P}^{k},
\end{equation}
where $\Wmu_{\star,0}:= \Wmu_{\star}(0, \mathbb P^k)$. 
The values of $\Wmu_{\star,0}$, internal weights $\{\omega^{(s)}_{0}\}$, and nodes $\{(x_0^{(s)}, y_0^{(s)})\}$ can be obtained via \Cref{sec:theta0,sec:general-theta}. 

\begin{figure}[htbp]
	\centering
	\begin{subfigure}[t][][t]{0.49\textwidth}
		\centering
		\includegraphics[width=0.88\textwidth]{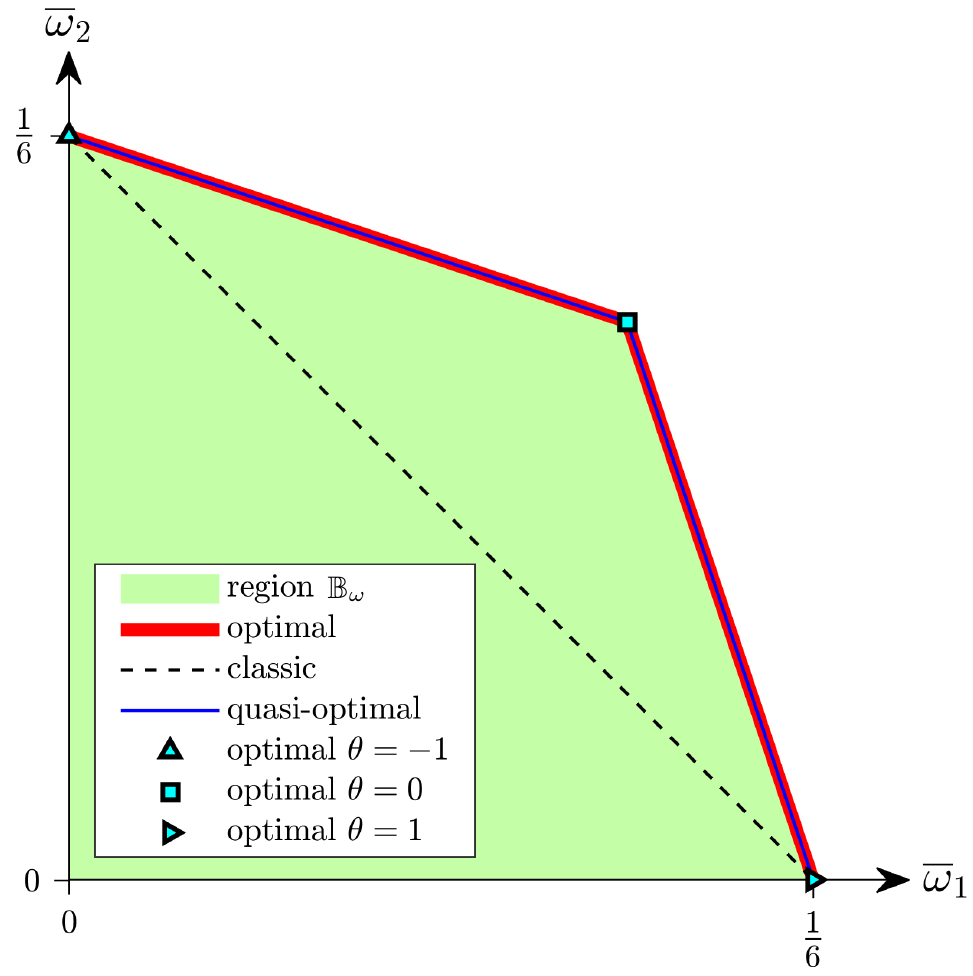}
		\caption{$\mathbb{P}^2$ and $\mathbb{P}^3$.}
		\label{fig1666:P2P3}
	\end{subfigure}
	\begin{subfigure}[t][][t]{0.49\textwidth}
		\centering
		\includegraphics[width=0.88\textwidth]{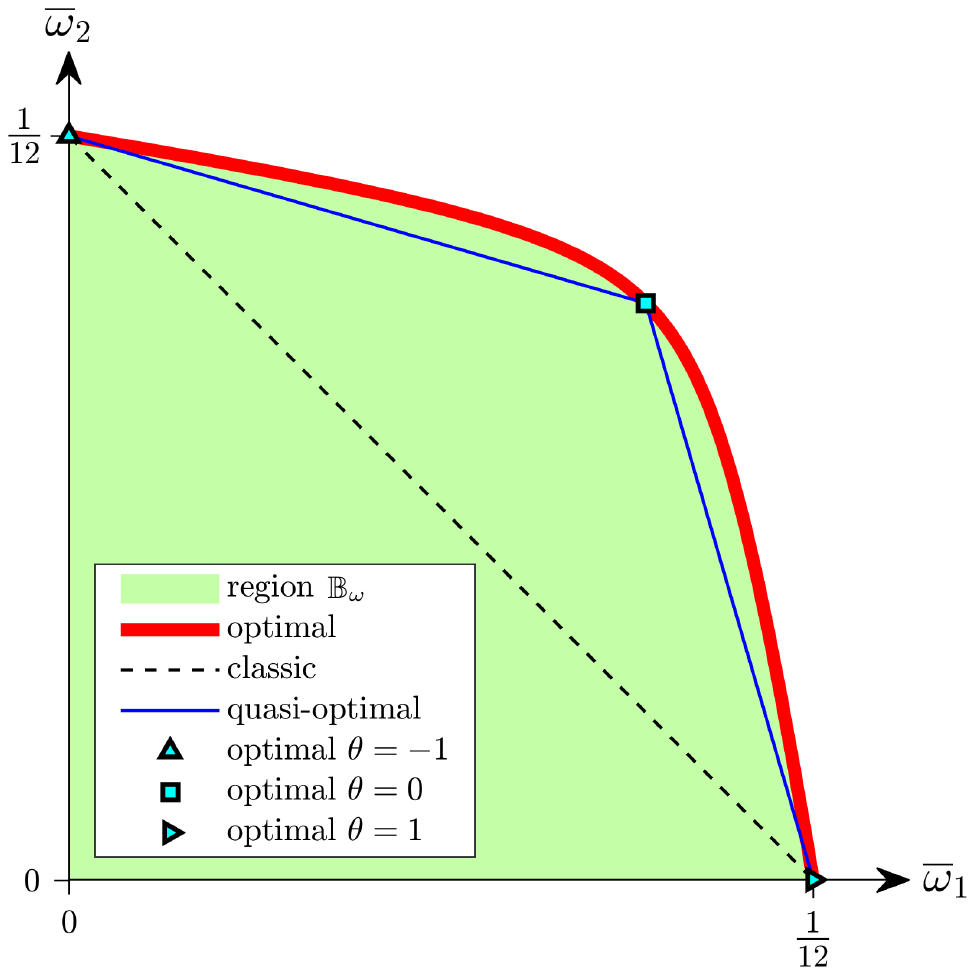}
		\caption{$\mathbb{P}^4$ and $\mathbb{P}^5$.}
		\label{fig1666:P4}
	\end{subfigure}
	\\
	\begin{subfigure}[t][][t]{0.49\textwidth}
		\centering
		\includegraphics[width=0.88\textwidth]{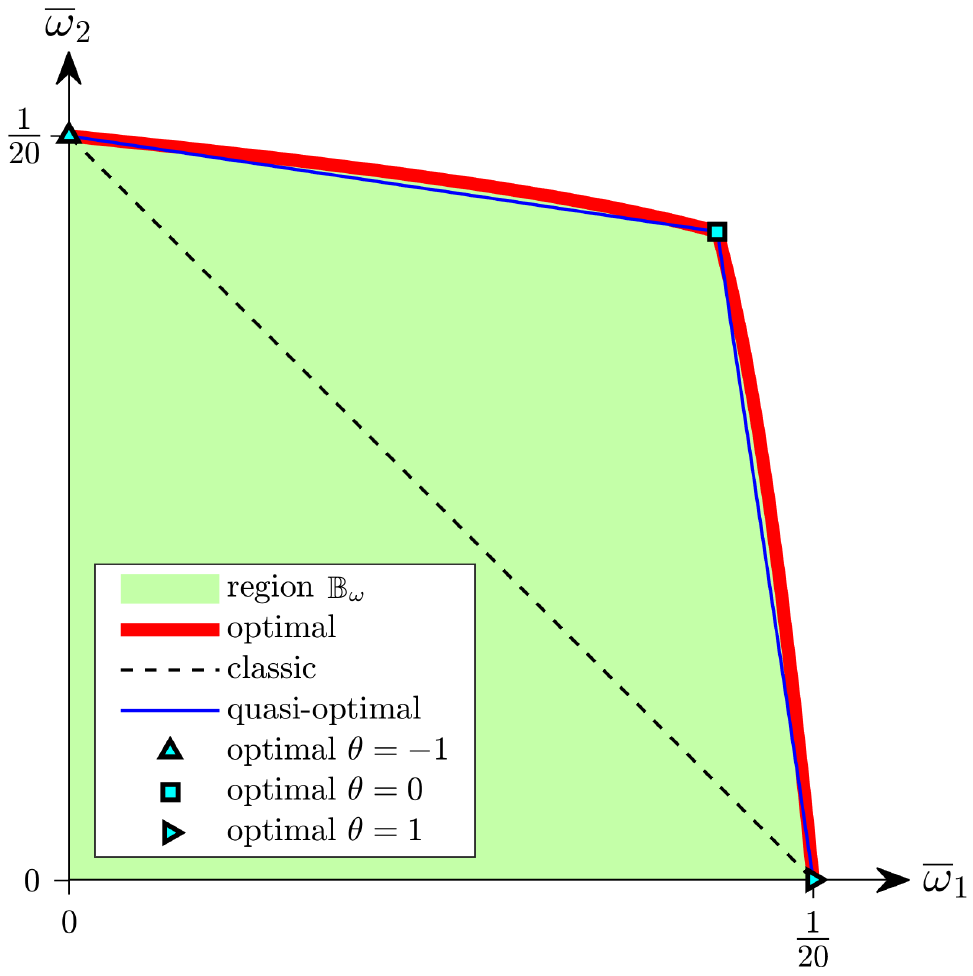}
		\caption{$\mathbb{P}^6$ and $\mathbb{P}^7$.}
		\label{fig1666:P6}
	\end{subfigure}
	\begin{subfigure}[t][][t]{0.49\textwidth}
		\centering
		\includegraphics[width=0.88\textwidth]{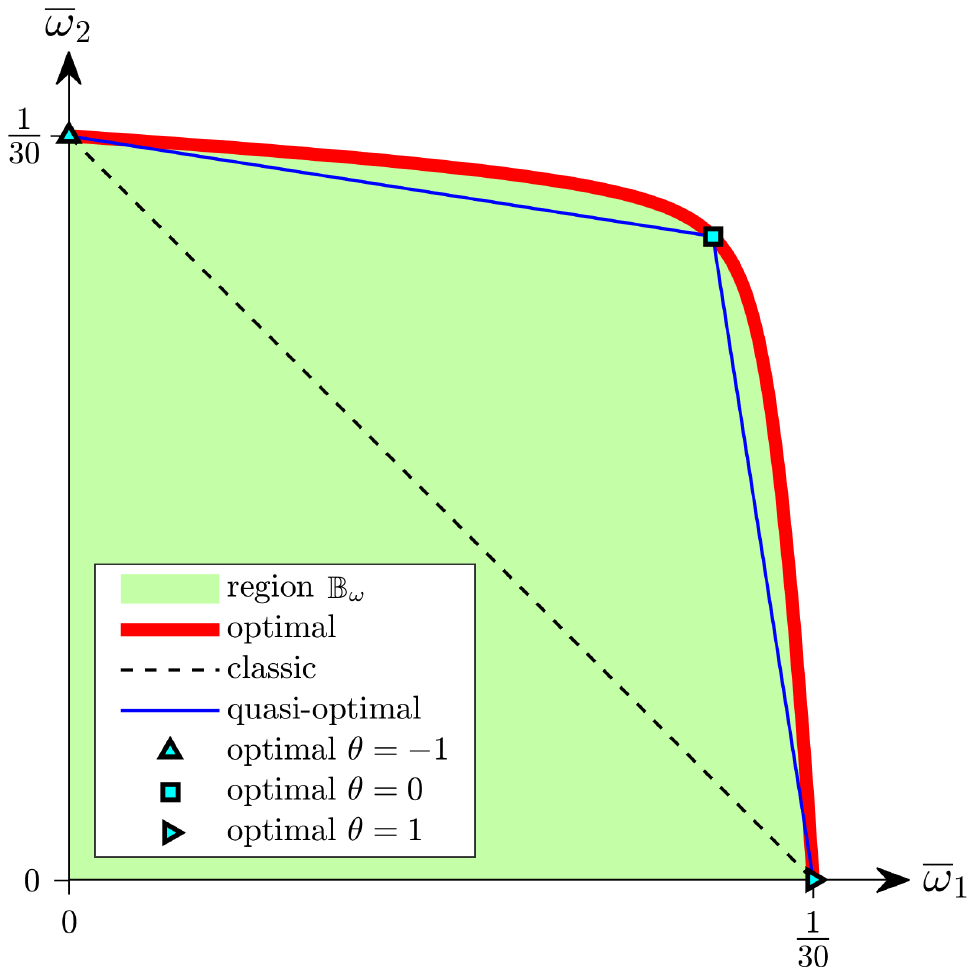}
		\caption{$\mathbb{P}^8$ and $\mathbb{P}^9$.}
		\label{fig1666:P8}
	\end{subfigure}
	\caption{Geometric illustration and comparison of the optimal, classic, quasi-optimal CADs for their boundary weights in the region $\mathbb{B}_\omega$. 
	}\label{fig:1666}
\end{figure}

\begin{theorem}[Quasi-optimal CAD]
	The following symmetric CAD is feasible for $\mathbb{P}^{k}$:
	\begin{equation*}
		\langle p \rangle_{\Omega} = 
		\begin{dcases}
			\tau \cdot \Big[\, \text{OCAD in } (\ref{eq:1688}) \text{ for $\theta = -1$} \,\Big] + 
			(1-\tau) \cdot \Big[\, \text{OCAD in } (\ref{eq:2177}) \text{ for $\theta = 0$} \,\Big] & \text{if } \theta \in [-1,0], \\
			\tau \cdot \Big[\, \text{OCAD in } (\ref{eq:1696}) \text{ for $\theta = 1$} \,\Big] +
			(1-\tau) \cdot \Big[\, \text{OCAD in } (\ref{eq:2177}) \text{ for $\theta = 0$} \,\Big] & \text{if } \theta \in [0,1],
		\end{dcases}
	\end{equation*}
	which can be equivalently written as 
	\begin{equation}\label{eq:2194}
		\langle p \rangle_{\Omega} = 
		\Wmu^{\tt Q}_\star \,
		\left[ (1+\theta) \langle p \rangle_{\Omega}^x + (1-\theta) \langle p \rangle_{\Omega}^y \right]+
		\tau \sum_{\ell=2}^{L-1} \sum_{q=1}^Q \omega_\ell^{{\tt GL}} \omega_q^{{\tt G}} p( x_{\ell q}^{{\tt Q}}, y_{\ell q}^{{\tt Q}})+
		(1-\tau) \sum_{s=1}^{S_0} \omega^{(s)}_{0} \overline{p( x_0^{(s)}, y_0^{(s)})}
	\end{equation}
	with
	\[
	\tau = \frac{\Wmu_{\star,0}|\theta|}{\Wmu_{\star,0} |\theta|+ \omega_1^{{\tt GL}}(1-|\theta|)}, \qquad 
	\Wmu^{\tt Q}_\star = \frac{\Wmu_{\star,0} \omega_1^{{\tt GL}} }{\Wmu_{\star,0} |\theta|+ \omega_1^{{\tt GL}}(1-|\theta|)},
	\]
	and
	\[
	( x_{\ell q}^{{\tt Q}}, y_{\ell q}^{{\tt Q}}) = 
	\begin{dcases}
		( x_q^{{\tt G}}, y_\ell^{{\tt GL}}), & \text{if } \theta \in [-1,0], \\
		( x_\ell^{{\tt GL}}, y_q^{{\tt G}}), & \text{if } \theta \in [0,1].
	\end{dcases}
	\]
\end{theorem}

\begin{proof}
The quasi-optimal CAD \eqref{eq:2194} is 
either the convex combination of the two special OCADs for $\theta = -1$ and $\theta =0$ (if $\theta \in [-1, 0]$) or the convex combination of the two special OCADs for $\theta = 1$ and $\theta =0$ (if $\theta \in [0,1]$). Consequently, the quasi-optimal CAD \eqref{eq:2194} is feasible, according to \Cref{lem:convex}. 
\end{proof}

For comparison, \Cref{fig:1666} illustrates 
the boundary weights of the optimal, classic, quasi-optimal CADs in the region $\mathbb{B}_\omega$. 
One can see that the classic CAD   
is notably far from the OCAD, especially when $\theta$ is near $0$, while  
the quasi-optimal CAD is much closer to the OCAD. 
In particular, the quasi-optimal CAD 
exactly coincides with the OCAD for $\mathbb P^2$ and $\mathbb P^3$.

To quantitatively analyze how close the quasi-optimal CAD is to the OCAD for $\mathbb P^k$ with  $k\ge 4$, we define the $\overline \omega$-ratio as follow. 

\begin{definition}[$\overline \omega$-ratio] 
	For a feasible symmetric CAD 
	$$
	\langle p \rangle_{\Omega} = 
	\Wmu \left[ (1+\theta) \langle p \rangle_{\Omega}^x + (1-\theta) \langle p \rangle_{\Omega}^y \right] +  \sum_{s} \omega_s   \overline{ p( x^{(s)}, y^{(s)}) } \qquad \forall p \in \mathbb P^k, 
	$$
	the ratio of its weight $\Wmu$ to the OCAD weight $\overline{\omega}_\star (\theta,\mathbb{P}^k)$ is called the $\overline \omega$-ratio. 
\end{definition}

In particular, the $\overline \omega$-ratio of the classic CAD \eqref{eq:1471} for $\mathbb P^k$ is given by 
\begin{equation}\label{eq:omega-ratio-classic}
	\frac{ \omega_1^{\tt GL} }{ \overline{\omega}_\star } = 
	\frac{ 1 }{ \left \lceil \frac{k+3}2 \right \rceil \left \lceil \frac{k+1}2 \right \rceil  \overline{\omega}_\star (\theta,\mathbb{P}^k)  }, 
\end{equation}
and $\overline \omega$-ratio of the quasi-optimal CAD \eqref{eq:2194} for $\mathbb P^k$ is given by 
\begin{equation}\label{eq:omega-ratio-qOCAD}
	\frac{ \Wmu^{\tt Q}_\star }{ \overline{\omega}_\star } = 
	\frac{\Wmu_{\star,0}  }{\Wmu_{\star,0} |\theta|+ \omega_1^{{\tt GL}}(1-|\theta|)} 
	\times \frac{ \omega_1^{\tt GL} }{ \overline{\omega}_\star }. 
\end{equation}
For $\mathbb P^2$ and $\mathbb P^3$, the quasi-optimal CAD \eqref{eq:2194} is OCAD, so that 
$\frac{ \Wmu^{\tt Q}_\star }{ \overline{\omega}_\star } = 1 $. 
For $\mathbb P^k$ with higher $k\ge 4$, \Cref{fig:188} gives a comparison of  
the optimal, classic, quasi-optimal CADs in terms of their 
boundary weights and $\overline \omega$-ratios. 
One can see that the $\overline \omega$-ratio of the classic CAD \eqref{eq:1471}
can reach as low as 
57\% $\sim$ 65\%, and  
$$
\omega_1^{\tt GL}  \ge 57\% ~~ \overline{\omega}_\star.
$$
Consequently, 
using OCAD to replace the classic CAD would help to notably improve the BP CFL condition. 
From \Cref{fig:188}, we also clearly observe  that $\Wmu^{\tt Q}_\star$ is very close to 
$\overline{\omega}_\star$, and the $\overline \omega$-ratio of the quasi-optimal CAD  
\eqref{eq:2194} is always above 95\% (overall much higher than that of the classic CAD \eqref{eq:1471}), namely,  
$$
\Wmu^{\tt Q}_\star \ge 95\% ~~ \overline{\omega}_\star.
$$
In other words, $\Wmu^{\tt Q}$ is only less than 5\% lower than $\Wmu_\star$. 
Since the construction of 
quasi-optimal CAD \eqref{eq:2194} only requires the special OCAD for $\theta \in \{-1,0,1\}$, 
and thus is  much easier than the construction of OCAD for all $\theta \in [-1,1]$. 
Therefore, the quasi-optimal CAD \eqref{eq:2194} is a good alternative for OCAD.


%
%
%
%

\begin{figure}[htbp]
	\centering
	\begin{subfigure}[t][][t]{0.32\textwidth}
		\centering
		\includegraphics[width=0.99\textwidth,trim=15 0 0 0,clip]{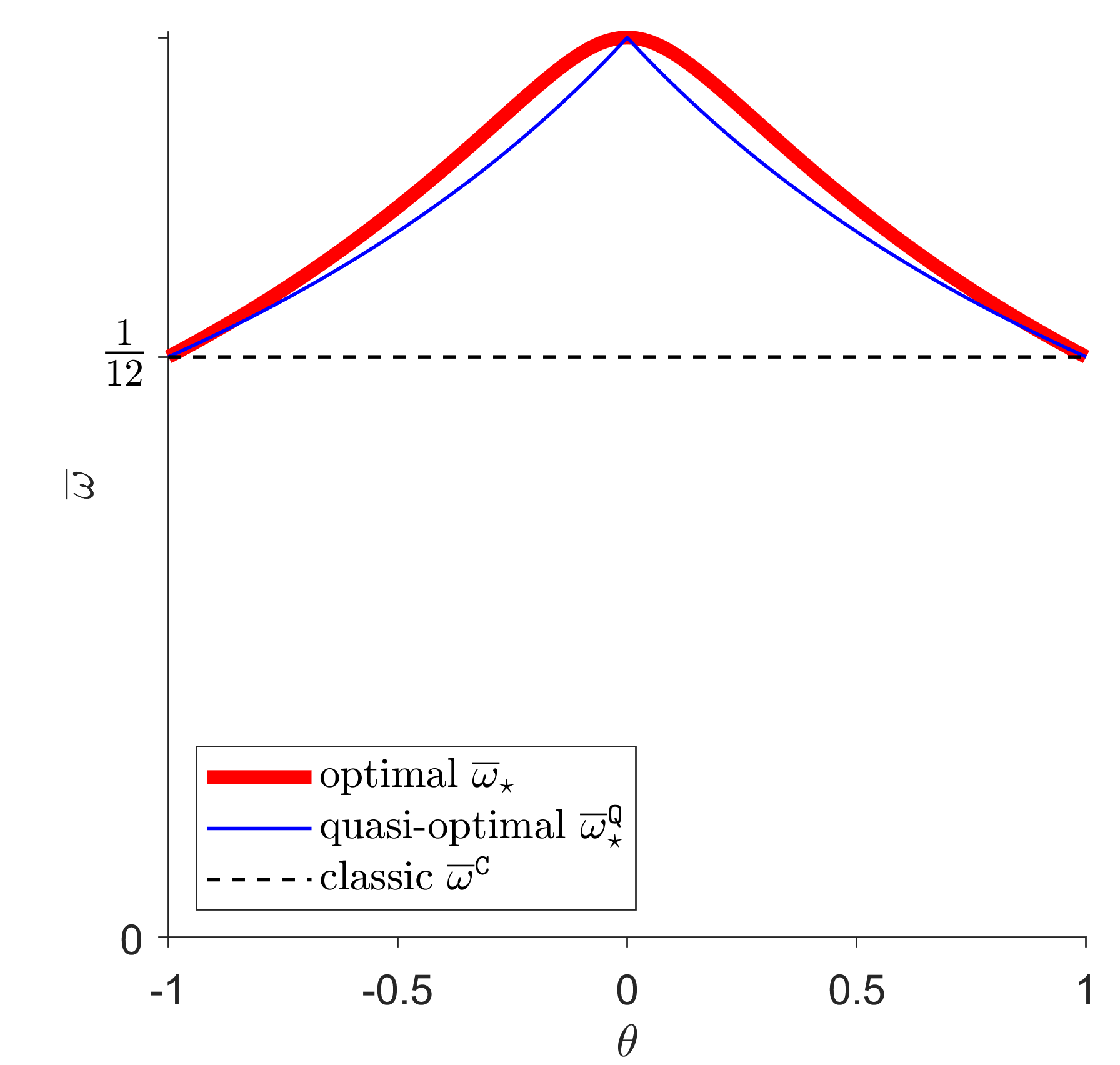}
	\end{subfigure}
	\begin{subfigure}[t][][t]{0.32\textwidth}
		\centering
		\includegraphics[width=0.99\textwidth,trim=15 0 0 0,clip]{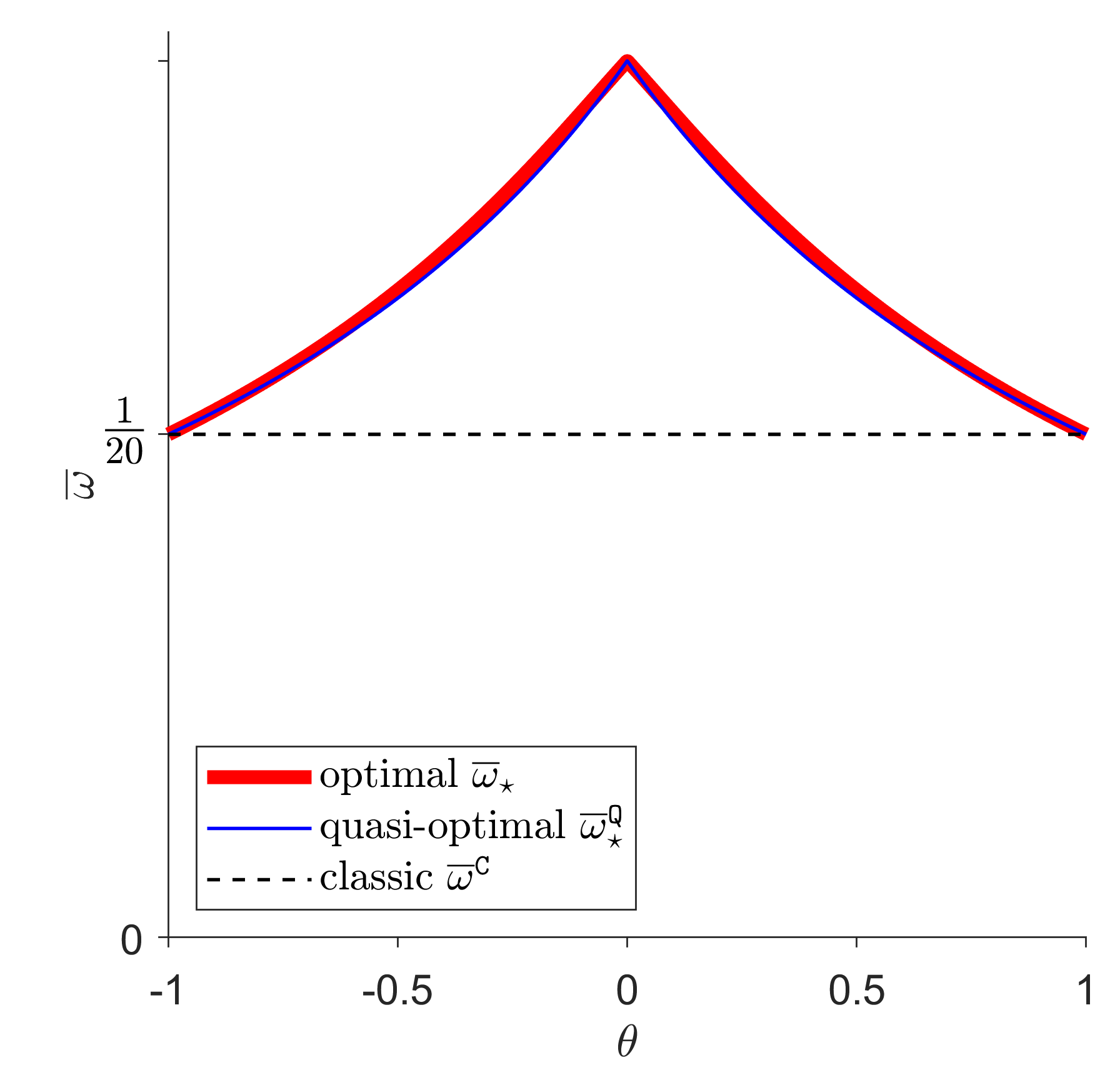}
	\end{subfigure}
	\begin{subfigure}[t][][t]{0.32\textwidth}
		\centering
		\includegraphics[width=0.99\textwidth,trim=15 0 0 0,clip]{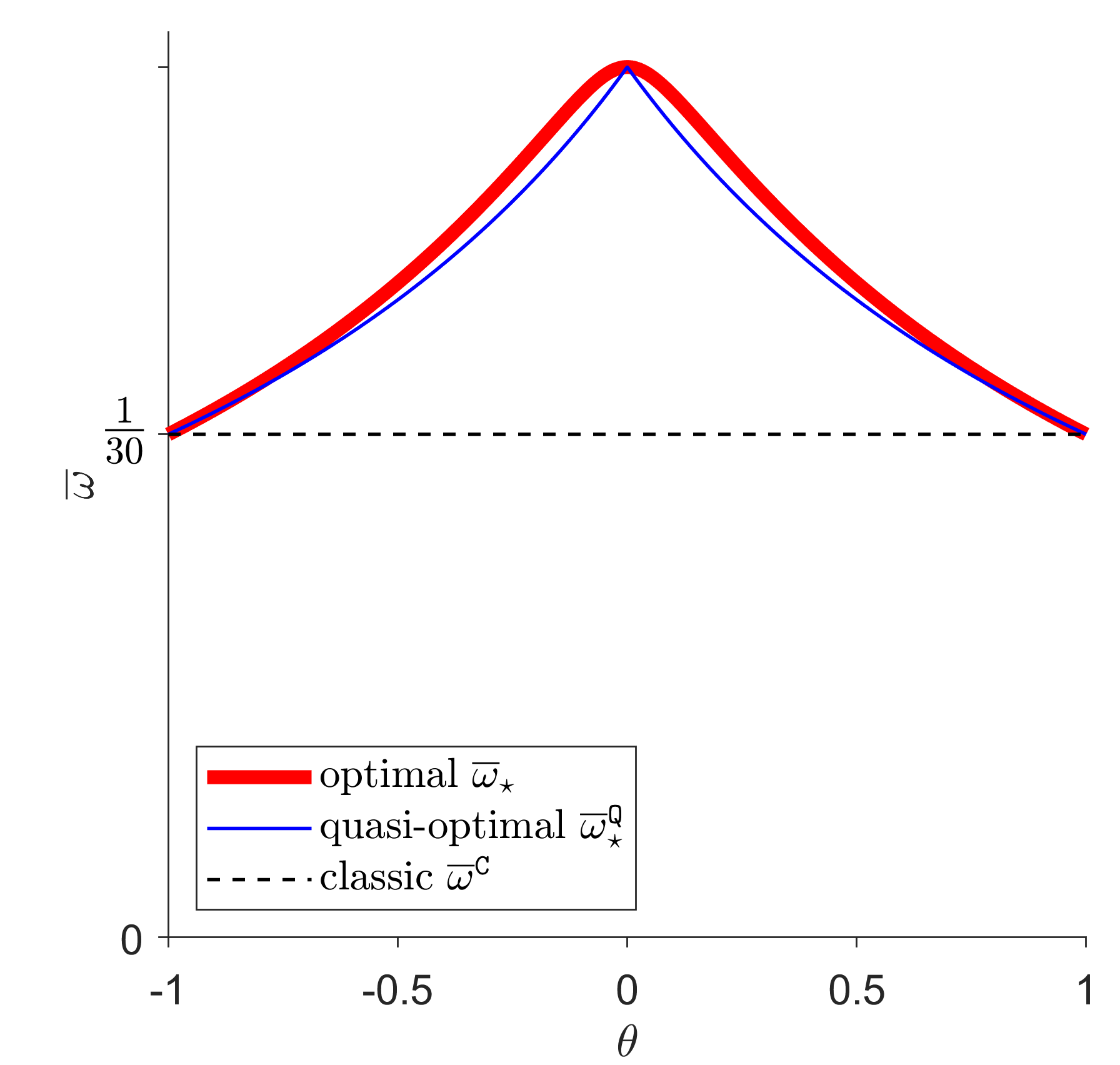}
	\end{subfigure}
	\\
	\begin{subfigure}[t][][t]{0.32\textwidth}
		\centering
		\includegraphics[width=0.99\textwidth,trim=17 0 0 0,clip]{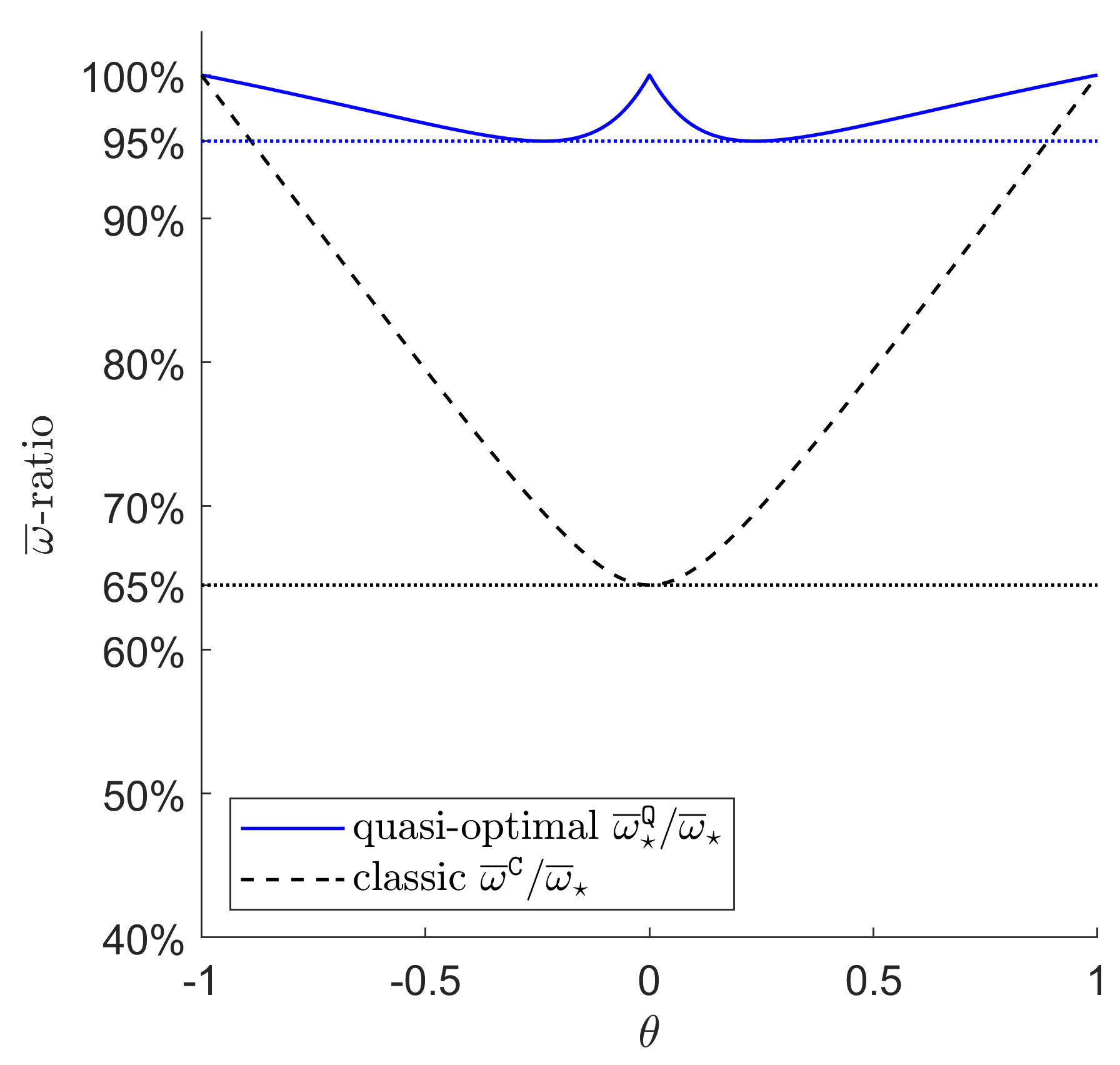}
		\caption{$\mathbb{P}^4$ and $\mathbb{P}^5$.}
	\end{subfigure}
	\begin{subfigure}[t][][t]{0.32\textwidth}
		\centering
		\includegraphics[width=0.99\textwidth,trim=17 0 0 0,clip]{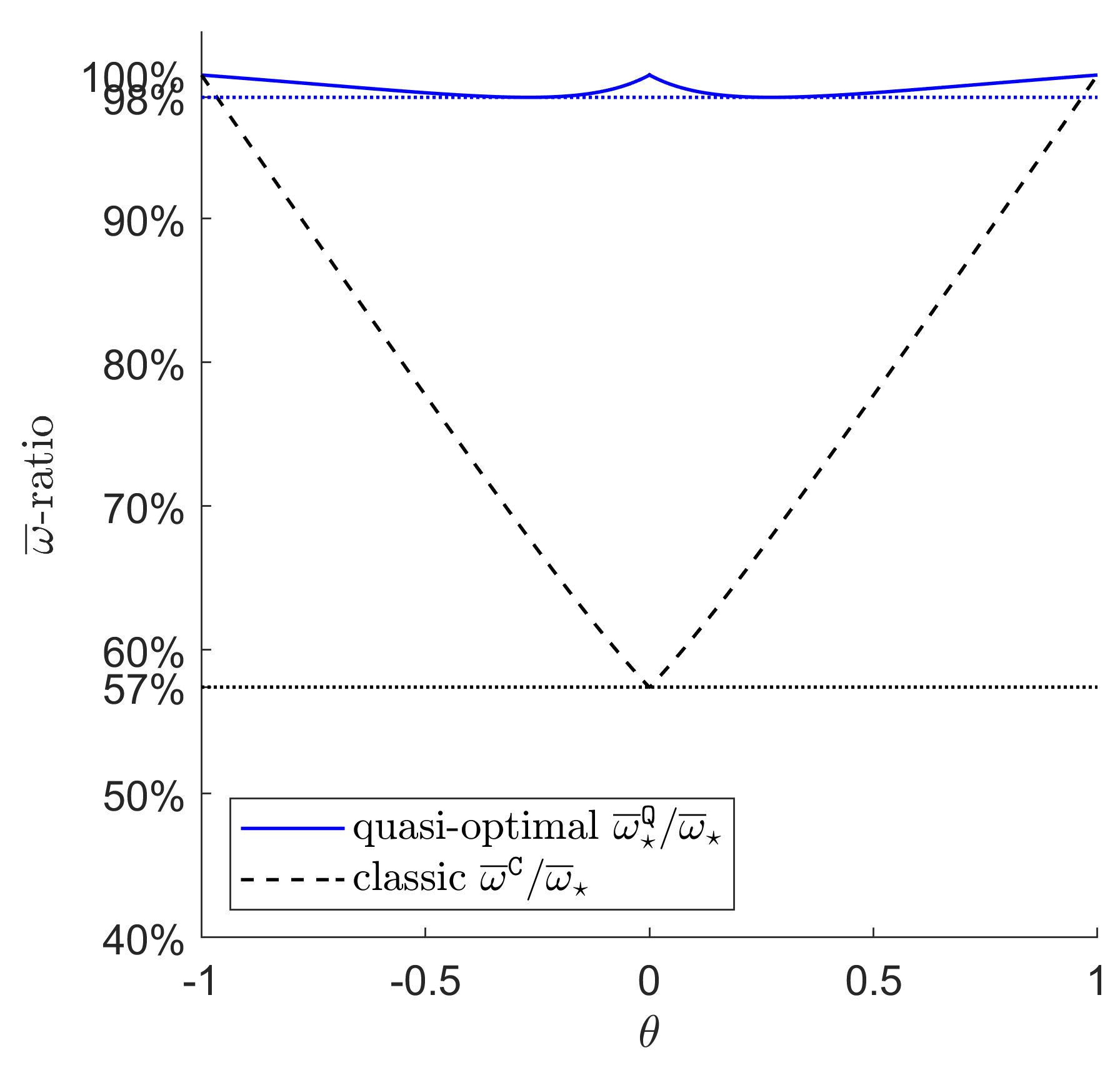}
		\caption{$\mathbb{P}^6$ and $\mathbb{P}^7$.}
	\end{subfigure}
	\begin{subfigure}[t][][t]{0.32\textwidth}
		\centering
		\includegraphics[width=0.99\textwidth,trim=17 0 0 0,clip]{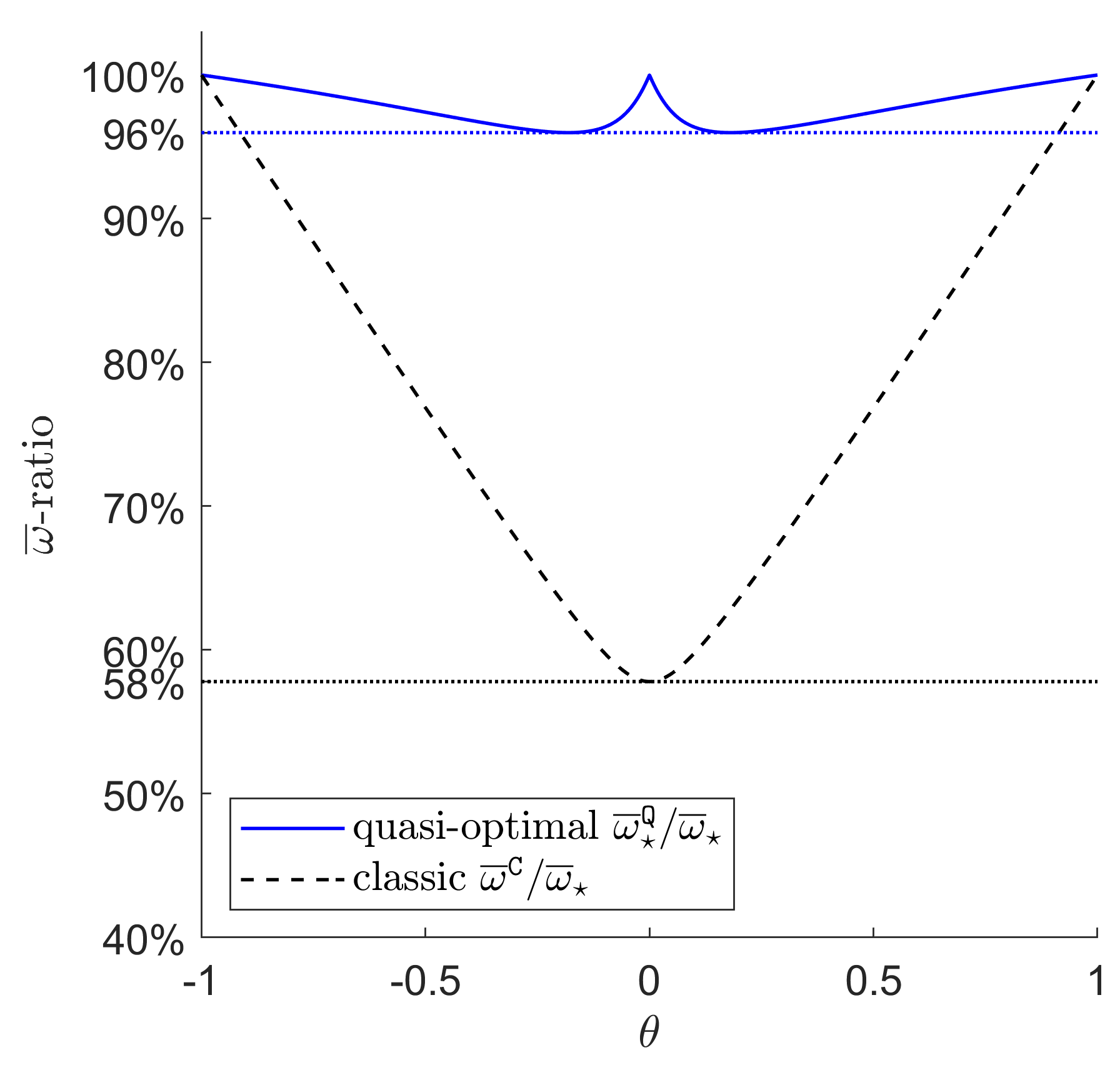}
		\caption{$\mathbb{P}^8$ and $\mathbb{P}^9$.}
	\end{subfigure}
	\caption{Comparison of the optimal, classic, quasi-optimal CADs in terms of their 
		boundary weights (top row) and $\overline \omega$-ratios (bottom row). 
	}\label{fig:188}
\end{figure}

\section{BP schemes based on OCAD and quasi-optimal CAD for hyperbolic systems}\label{sec:BPscheme} 

In this section, we apply the proposed OCAD and quasi-optimal CAD to designing 
efficient BP schemes for hyperbolic systems of conservation laws. 
Since we have proved that the existing classic CAD \eqref{eq:U2Dsplit} by Zhang and Shu \cite{zhang2010,zhang2010b} is already optimal in 1D and for $\mathbb Q^k$ spaces in 2D, 
this section is only focused on the  $\mathbb P^k$-based BP  DG schemes in 2D.


Consider the $(k+1)$th-order $\mathbb P^k$-based DG scheme with the forward Euler time discretization\footnote{All our discussions are also valid for high-order strong-stability-preserving (SSP) time discretizations \cite{GottliebKetchesonShu2011}, which are convex combinations of forward Euler step.} for  
the 2D hyperbolic conservation laws
\begin{equation}\label{2DCL}
	\frac{\partial }{\partial t} u+ \frac{\partial }{\partial x} f_1(u) + \frac{\partial }{\partial y} f_2(u)  = 0  \qquad (x,y,t) \in \mathbb{R} \times \mathbb{R} \times \mathbb R^{+}.
\end{equation}
Following the Zhang--Shu framework \cite{zhang2010,zhang2010b}, in order to design a BP DG scheme, we only need to ensure the cell averages within the region $G$. As long as the BP property of the updated cell averages is guaranteed, 
one may employ a simple BP limiter to enforce the pointwise bounds 
of the piecewise DG polynomial solutions without affecting the high-order accuracy \cite{zhang2010,zhang2010b}. 
On a rectangular cell $\Omega_{ij}:=[x_{i-\frac12},x_{i+\frac12}]\times[y_{j-\frac12},y_{j+\frac12}]$, the evolution equation of cell averages for the $(k+1)$th-order DG scheme reads
\begin{equation}\label{eq:171}
	\begin{aligned}
		\bar {u}_{ij}^{n+1} = 
		\bar u_{ij}^n & - 
		\frac{\Delta t}{\Delta x} \sum_{q=1}^Q \omega_q^{\tt G}
		\left[  \hat f_1 \big( u_{i+\frac12,q}^{-}, u_{i+\frac12,q}^{+}  \big) - \hat f_1 \big( u_{i-\frac12,q}^{-}, u_{i-\frac12,q}^{+}  \big)  \right]
		\\
		& - 
		\frac{\Delta t}{\Delta y} \sum_{q=1}^Q \omega_q^{\tt G} 
		\left[  \hat f_2 \big( u_{q,j+\frac12}^{-}, u_{q,j+\frac12}^{+}  \big) - \hat f_2 \big( u_{q,j-\frac12}^{-}, u_{q,j-\frac12}^{+}  \big)  \right], 
	\end{aligned}
\end{equation} 
where 
$Q$ is typically taken as $k+1$ such that the Gauss quadrature has sufficient high-order accuracy, and 
the limiting values at the cell interfaces are computed by  
\begin{align} \label{eq:u1}
	&u_{i-\frac12,q}^{+} = p_{ij} \big( x_{i-\frac12}, y_{j,q}^{\tt G} \big), \quad 
	u_{i+\frac12,q}^{-} = p_{ij} \big( x_{i+\frac12}, y_{j,q}^{\tt G} \big), \\
	& \label{eq:u2}
	u_{q,j-\frac12}^{+} = p_{ij} \big( x_{i,q}^{\tt G}, y_{j-\frac12} \big), \quad 
	u_{q,j+\frac12}^{-} = p_{ij} \big( x_{i,q}^{\tt G}, y_{j+\frac12} \big)
\end{align}
with $p_{ij}(x,y) \in \mathbb{P}^k$ denoting the DG solution polynomial on $\Omega_{ij}$ at time level $n$ satisfying
\begin{equation*}
	\bar {u}_{ij}^{n} = \frac{1}{\dx \dy} \int_{ x_{i-\frac12} }^{ x_{i+\frac12} } \int_{ y_{j-\frac12} }^{ y_{j+\frac12} } p_{ij}(x,y) \, {\rm d} y {\rm d} x. 
\end{equation*}
In \eqref{eq:171}, we take the numerical fluxes $\hat{f}_1$ and $\hat{f}_2$ as the BP numerical fluxes with which the corresponding 1D three-point first-order schemes are BP, {\rm i.e.}, for any $u_1,u_2,u_3 \in G$ it holds that 
\begin{equation}\label{1DBP}
	u_2 - \frac{ \dt }{ \dx} \left( \hat{f}_1(u_2,u_3) - \hat{f}_1(u_1,u_2) \right) \in G, \qquad 
	u_2 - \frac{\dt }{ \dy} \left( \hat{f}_2(u_2,u_3) - \hat{f}_2(u_1,u_2) \right) \in G \qquad 
\end{equation}
under a suitable CFL condition $\max\{a_1 \dt / \dx,a_2 \dt/\dy\} \le c_0$, where $a_1$ and $a_2$ denote the maximum characteristic speeds in $x$- and $y$-directions, and $c_0$ is the maximum allowable CFL number for the 1D first-order schemes. For example, typically $c_0=1$ for the Lax--Friedrichs flux \cite{zhang2010b,jiang2018invariant},  and $c_0=\frac12$ for the HLL and HLLC fluxes \cite{jiang2018invariant}.

We now discuss the BP conditions for the scheme (\ref{eq:171}) based on a 2D feasible  symmetric CAD for $\mathbb P^k$ in the form of 
\begin{equation}\label{eq:15275}
	\langle p \rangle_{\Omega} = \Wmu \left[ (1+\theta) \langle p \rangle_{\Omega}^x + (1-\theta) \langle p \rangle_{\Omega}^y \right] + \sum_{s=1}^S \omega_s \overline{p(x^{(s)},y^{(s)})} \qquad \forall p \in \mathbb P^k
\end{equation}
with $\theta$ defined in \eqref{eq:eta}. Given a feasible 
CAD \eqref{eq:15275} on the reference cell $\Omega = [-1,1]^2$, we should first transfer it onto 
$\Omega_{ij}$ as 
\begin{equation}\label{eq:15275b}
	\langle p \rangle_{\Omega_{ij}} = \Wmu \left[ (1+\theta) \langle p \rangle_{\Omega_{ij}}^x + (1-\theta) \langle p \rangle_{\Omega_{ij}}^y \right] + \sum_{s=1}^S \omega_s \overline{p(x_{ij}^{(s)},y_{ij}^{(s)})} \qquad \forall p \in \mathbb P^k 
\end{equation}
with the weights unchanged and 
the nodes $\{( x_{ij}^{(s)}, y_{ij}^{(s)} )\}$ given by the inverse transformation of (\ref{eq:trans}).

\begin{theorem}[BP via general symmetric CAD]\label{thm:CFL}
	Suppose that there is a 2D feasible symmetric CAD in the form of \eqref{eq:15275}. 
	If the DG solution polynomial $p_{ij}(x,y)$ satisfies for all $i$ and $j$ that 
	\begin{equation}\label{eq:478a}
		u_{i-\frac12,q}^{+}  \in G, \quad
		u_{i+\frac12,q}^{-}  \in G, \quad
		u_{q,j-\frac12}^{+}  \in G, \quad
		u_{q,j+\frac12}^{-}  \in G, \quad
		q = 1,\dots,Q, 
	\end{equation}
	and 
	\begin{equation}\label{eq:478a2}
		\Pi_{ij} := \frac{1}{1-2 \Wmu } \sum_{s=1}^S \omega_s \overline{p_{ij}(x_{ij}^{(s)},y_{ij}^{(s)})} \in G, 
	\end{equation}    
	then the high-order scheme \eqref{eq:171} preserves 
	$\bar u_{ij}^{n+1}\in G$ under the BP CFL condition
	\begin{equation}\label{key3dd}
		\Delta t \left( \frac{a_1 } {\Delta x} + \frac{ a_2 }{ \Delta y} \right) \le    \Wmu c_0. 
	\end{equation}
\end{theorem}

\begin{proof}
	Applying the symmetric CAD \eqref{eq:15275b} to 
	$\bar u_{ij}^{n} = \langle p_{ij} \rangle_{ \Omega_{ij} }$ gives 
	\begin{equation}\label{eq:4540}
		\bar u_{ij}^{n} =   \widehat \omega_1  \sum_{q=1}^Q   \omega_q^{{\tt G}}  
		\left(  u_{i-\frac12,q}^+ +  u_{i+\frac12,q}^- \right) 
		+  \widehat \omega_2  \sum_{q=1}^Q   \omega_q^{{\tt G}} 
		\left( u_{q,j-\frac12}^+ +  u_{q,j+\frac12}^- \right) 
		+ ( 1 - 2 \Wmu ) \Pi_{ij}
	\end{equation}
	where the Gauss quadrature and \eqref{eq:478a2} have been used, 
	and 
	$$\widehat \omega_1:= \frac{ \Wmu  (1+\theta) } 2 =  \frac{\Wmu a_1/\Delta x  }{a_1 / \Delta x + a_2 / \Delta y}, \qquad 
	\widehat \omega_2:= \frac{ \Wmu  (1-\theta) } 2 =  \frac{\Wmu a_2/\Delta x  }{a_1 / \Delta x + a_2 / \Delta y}.
	$$
	Under the assumption of \eqref{eq:478a} and \eqref{eq:478a2}, it holds that $\bar u_{ij}^{n} \in G$.  
	Substituting the decomposition \eqref{eq:4540} into \eqref{eq:171}, one can rewrite the scheme \eqref{eq:171} as 
	\begin{equation}\label{eq:476}
		\bar {u}_{ij}^{n+1}  = \sum_{q=1}^Q    \omega_q^{\tt G}  \,
		\widehat \omega_1 (  H_{i+\frac12,q}^- + H_{i-\frac12,q}^+  ) + 
		\sum_{q=1}^Q    \omega_q^{\tt G}  \,
		\widehat \omega_2 ( H_{q,j+\frac12}^- + H_{q,j-\frac12}^+ ) + ( 1 - 2 \Wmu ) \Pi_{ij}
	\end{equation} 
	with
	\begin{align*}
		H_{i+\frac12,q}^- &= u_{i+\frac12,q}^- - \frac{\Delta t}{\widehat \omega_1 \Delta x} 
		\left( \hat f_1 \big( u_{i+\frac12,q}^-, u_{i+\frac12,q}^+ \big)  
		- \hat f_1 \big( u_{i-\frac12,q}^+, u_{i+\frac12,q}^- \big)  \right),
		\\
		H_{i-\frac12,q}^+ & = u_{i-\frac12,q}^+ - \frac{\Delta t}{\widehat \omega_1 \Delta x} 
		\left(   \hat f_1 \big( u_{i-\frac12,q}^+, u_{i+\frac12,q}^- \big) 
		- \hat f_1 \big( u_{i-\frac12,q}^-, u_{i-\frac12,q}^+ \big)   \right),
		\\
		H_{q,j+\frac12}^- &= u_{q,j+\frac12}^- - \frac{\Delta t}{\widehat \omega_2 \Delta y} 
		\left( \hat f_2 \big( u_{q,j+\frac12}^-, u_{q,j+\frac12}^+ \big)  
		- \hat f_2 \big( u_{q,j-\frac12}^+, u_{q,j+\frac12}^- \big)  \right),
		\\
		H_{q,j-\frac12}^+ &= u_{q,j-\frac12}^+ - \frac{\Delta t}{\widehat \omega_2 \Delta y} 
		\left(  \hat f_2 \big( u_{q,j-\frac12}^+, u_{q,j+\frac12}^- \big) 
		- \hat f_2 \big( u_{q,j-\frac12}^-, u_{q,j-\frac12}^+ \big) \right),
	\end{align*}
	which have the same form as the 1D three-point first-order schemes \eqref{1DBP} and thus satisfy  
	\begin{equation*}
		H_{i+\frac12,q}^-\in G, \quad  H_{i-\frac12,q}^+ \in G, \quad H_{q,j+\frac12}^- \in G, \quad H_{q,j-\frac12}^+ \in G,
	\end{equation*}
	under the CFL type conditions 
	\begin{equation}\label{eq:tempCFL00}
		a_1 {\Delta t}\le c_0 {\widehat \omega_1  \Delta x}, \quad  a_2 {\Delta t}\le c_0 { \widehat \omega_2  \Delta y}.  
	\end{equation}
	Because (\ref{eq:476}) is a convex combination form, by the convexity of $G$ we conclude that $\bar {u}_{ij}^{n+1} \in G$ under the CFL conditions \eqref{eq:tempCFL00}, which are exactly equivalent to \eqref{key3dd}. The proof is completed. 
\end{proof}

\begin{remark}[Classic CAD]\label{rem:CCAD}
	If the Zhang--Shu classic CAD is considered, then 
	the BP CFL condition \eqref{key3dd} becomes 
	\begin{equation}\label{key3dd-c}
		\Delta t \left( \frac{a_1 } {\Delta x} + \frac{ a_2 }{ \Delta y} \right) \le    \omega_1^{{\tt GL}} c_0. 
	\end{equation}	
\end{remark}

\begin{remark}[BP Limiter] 
	The condition \eqref{eq:478a2} is satisfied if ${p_{ij}(x,y)} \in G$ for all $(x,y)\in \mathbb S_{ij}$. 
	In general, the DG solution polynomial may not automatically meet  
	the conditions \eqref{eq:478a} and \eqref{eq:478a2}, which should be enforced by a BP limiter. 
	For the scalar conservation law with the maximum principle \eqref{scalar_bounds} and $G=[U_{\min}, U_{\max} ]$, 
	the BP limiter \cite{zhang2010} is given by 
	\begin{equation}\label{eq:BPlimiter}
		\tilde p_{ij}(x,y) =  \delta ( p_{ij}(x,y) - \bar u_{ij}^{n} ) + \bar u_{ij}^{n}, 
		\qquad \delta = \min \left\{  
		\left| \frac{U_{\max} - \bar u_{ij}^{n} }{ p_{ij}^{\max} -  \bar u_{ij}^{n} } \right|,  
		\left| \frac{U_{\min} - \bar u_{ij}^{n} }{ p_{ij}^{\min} -  \bar u_{ij}^{n} } \right|, 
		1
		\right\},
	\end{equation}
	where 
	$$p_{ij}^{\max} = \max_{ (x,y)\in \Theta_{ij} } p_{ij}(x,y),~~~ 
	p_{ij}^{\min} = \min_{ (x,y)\in \Theta_{ij} } p_{ij}(x,y),~~~
	\Theta_{ij} = \{ (x_{i \pm \frac12}, y_{j,q}^{\tt G}), 
	(x_{i,q}^{\tt G}, y_{j\pm \frac12}) \}_{q=1}^Q \cup \mathbb S_{ij}.$$ 
	One can verify that the limited DG solution polynomial $\tilde p_{ij}$ satisfies the desired conditions \eqref{eq:478a} and \eqref{eq:478a2}.    
	Similar local scaling BP limiters have also been designed for the Euler equations \cite{zhang2010b} and many other hyperbolic systems \cite{wang2012robust,QinShu2016,Wu2017,WuShu2019}. It is worth noting that the internal nodes of the OCAD are much fewer than those of the classic CAD (see \Cref{fig:1882}). 
	When the local scaling BP limiter is performed at the internal nodes in all computational cells, using our OCAD also reduces the computational cost in the BP limiting procedure.  
\end{remark}

\begin{remark}[Simplified BP Limiter]\label{rem:SBP} 
	One can also use a simplified BP limiter \cite{zhang2011b} to enforce the conditions \eqref{eq:478a} and \eqref{eq:478a2}, 
	without using the internal nodes $\mathbb S_{ij}$ of the CAD. 
	In fact, according to \eqref{eq:4540}, $\Pi_{ij}$ can also be represented as 
	\begin{equation}\label{eq:478a3}
		\Pi_{ij} = \frac{ \bar u_{ij}^{n}  -  \Wmu \left[  \frac{ 1 + \theta }2    \sum_{q=1}^Q   \omega_q^{{\tt G}}  
			\left(  u_{i-\frac12,q}^+ +  u_{i+\frac12,q}^- \right) 
			+  \frac{ 1 - \theta }2  \sum_{q=1}^Q   \omega_q^{{\tt G}} 
			\left( u_{q,j-\frac12}^+ +  u_{q,j+\frac12}^- \right)  \right]  }{1-2 \Wmu } . 
	\end{equation}   
	For example, for the scalar conservation law with $G=[U_{\min}, U_{\max} ]$, the simplified BP limiter is given by \eqref{eq:BPlimiter} 
	with 
	\begin{align}\label{eq:sBPlimiter1}
		p_{ij}^{\max} = \max \left \{ p_{ij}(x_{i \pm \frac12}, y_{j,q}^{\tt G}), p_{ij} (x_{i,q}^{\tt G}, y_{j\pm \frac12}), \Pi_{ij}  \right\},
		\\  \label{eq:sBPlimiter2}
		p_{ij}^{\min} = \min \left \{ p_{ij}(x_{i \pm \frac12}, y_{j,q}^{\tt G}), p_{ij} (x_{i,q}^{\tt G}, y_{j\pm \frac12}), \Pi_{ij} \right \},
	\end{align}
	where $\Pi_{ij}$ is computed by \eqref{eq:478a3}. 
	{\em A remarkable advantage of using 
	the simplified BP limiter is that it 
	only involves the boundary weight of CAD and 
	 does not require the information of internal CAD nodes in the resulting BP schemes. 
	 As we have seen, finding the internal CAD nodes is difficult. Therefore, using 
	  the simplified BP limiter to construct BP schemes effectively avoids such  difficulty. 
}
\end{remark}

As direct consequences of \Cref{thm:CFL}, we have the following conclusions. 

\begin{theorem}[BP via OCAD]\label{thm:CFL0}
	Consider the OCAD for the $\mathbb P^k$-based DG scheme. 
	If for all $i$ and $j$ the (limited) DG solution polynomial $\tilde p_{ij}(x,y)$ satisfies 
	\begin{equation}\label{eq:478a-L}
		\tilde u_{i-\frac12,q}^{+}  \in G, \quad
		\tilde u_{i+\frac12,q}^{-}  \in G, \quad
		\tilde u_{q,j-\frac12}^{+}  \in G, \quad
		\tilde u_{q,j+\frac12}^{-}  \in G, \quad
		q = 1,\dots,Q, 
	\end{equation}
	and 
	\begin{equation}\label{eq:478a2-L}
		\Pi_{ij} := \frac{ \bar u_{ij}^{n}  -  \Wmu_\star(\theta, \mathbb P^k) \left[ \frac{ 1 + \theta }2    \sum_{q=1}^Q   \omega_q^{{\tt G}}  
			\left(  \tilde u_{i-\frac12,q}^+ +  \tilde u_{i+\frac12,q}^- \right) 
			+  \frac{ 1 - \theta }2  \sum_{q=1}^Q   \omega_q^{{\tt G}} 
			\left( \tilde u_{q,j-\frac12}^+ +  \tilde u_{q,j+\frac12}^- \right)  \right]  }{1-2  \Wmu_\star(\theta, \mathbb P^k) } \in G, 
	\end{equation}  
	then the high-order scheme \eqref{eq:171} with the BP limiter preserves 
	$\bar u_{ij}^{n+1}\in G$ under the BP CFL condition
	\begin{equation}\label{key3ddo}
		\Delta t \left( \frac{a_1 } {\Delta x} + \frac{ a_2 }{ \Delta y} \right) \le    \Wmu_\star c_0. 
	\end{equation}
\end{theorem}

\begin{theorem}[BP via quasi-optimal CAD]\label{thm:CFL1}
	Consider the quasi-optimal CAD for the $\mathbb P^k$-based DG scheme. 
	If for all $i$ and $j$ the (limited) DG solution polynomial $\tilde p_{ij}(x,y)$ satisfies   
	\eqref{eq:478a-L} 
	and 
	\begin{equation}\label{eq:478a2-q}
		\Pi_{ij} := \frac{ \bar u_{ij}^{n}  -  \Wmu_\star^{\tt Q} \left[  \frac{ 1 + \theta }2    \sum_{q=1}^Q   \omega_q^{{\tt G}}  
			\left( \tilde u_{i-\frac12,q}^+ + \tilde  u_{i+\frac12,q}^- \right) 
			+  \frac{ 1 - \theta }2  \sum_{q=1}^Q   \omega_q^{{\tt G}} 
			\left( \tilde u_{q,j-\frac12}^+ +  \tilde u_{q,j+\frac12}^- \right)  \right]  }{1-2  \Wmu_\star^{\tt Q} } \in G, 
	\end{equation}  
	then the high-order scheme \eqref{eq:171} with the BP limiter preserves 
	$\bar u_{ij}^{n+1}\in G$ under the BP CFL condition
	\begin{equation}\label{key3ddq}
		\Delta t \left( \frac{a_1 } {\Delta x} + \frac{ a_2 }{ \Delta y} \right) \le    \Wmu_\star^{\tt Q} c_0
	\end{equation}
	with $\Wmu^{\tt Q}_\star = \frac{\Wmu_{\star,0} \omega_1^{{\tt GL}} }{\Wmu_{\star,0} |\theta|+ \omega_1^{{\tt GL}}(1-|\theta|)}$ and $\Wmu_{\star,0}:= \Wmu_{\star}(0, \mathbb P^k)$.
\end{theorem}

\begin{remark}[Comparison of CFL conditions]
		The standard CFL condition for linear stability of the $\mathbb P^k$-based DG method with a $(k+1)$-stage $(k+1)$-order Runge--Kutta (RK) time discretization \cite{cockburn2001runge} is given by the following empirical formula   
	\begin{equation}\label{eq:LS}
	\Delta t	\left(   \frac{a_1}{\dx}+ \frac{a_2}{\dy} \right)  \le \frac{1}{2k+1}.  
	\end{equation} 
Table \ref{tab:CFL} gives a comparison of different CFL numbers in the special case of $a_1 \Delta y = a_2 \Delta x$ (i.e.~$\theta =0$) and $c_0=1$. 
It indicates that 
the optimal BP CFL condition \eqref{key3ddo} of the DG schemes (with the BP limiter) is 
much weaker than the classic BP CFL condition \eqref{key3dd-c} via the Zhang--Shu classic CAD. 
Moreover, if $c_0=1$, 
the optimal BP CFL condition \eqref{key3ddo} is 
typically weaker than the standard CFL condition \eqref{eq:LS} except for $k=8$. In practice, one may need to consider both 
the BP CFL and linearly stable CFL conditions to fully ensure the stability of the DG method. 
Under this consideration, the standard CFL condition with OCAD is sufficient to guarantee the BP property when $c_0=1$, while if $c_0=\frac12$ the optimal BP CFL condition \eqref{key3ddo}  dominates but it is still much weaker than classic BP CFL condition \eqref{key3dd-c}; see \Cref{fig:199} for further illustration.  
\end{remark}

\begin{figure}[htbp]
	\centering
	\begin{subfigure}[t][][t]{0.48\textwidth}
		\centering
		\includegraphics[width=0.99\textwidth]{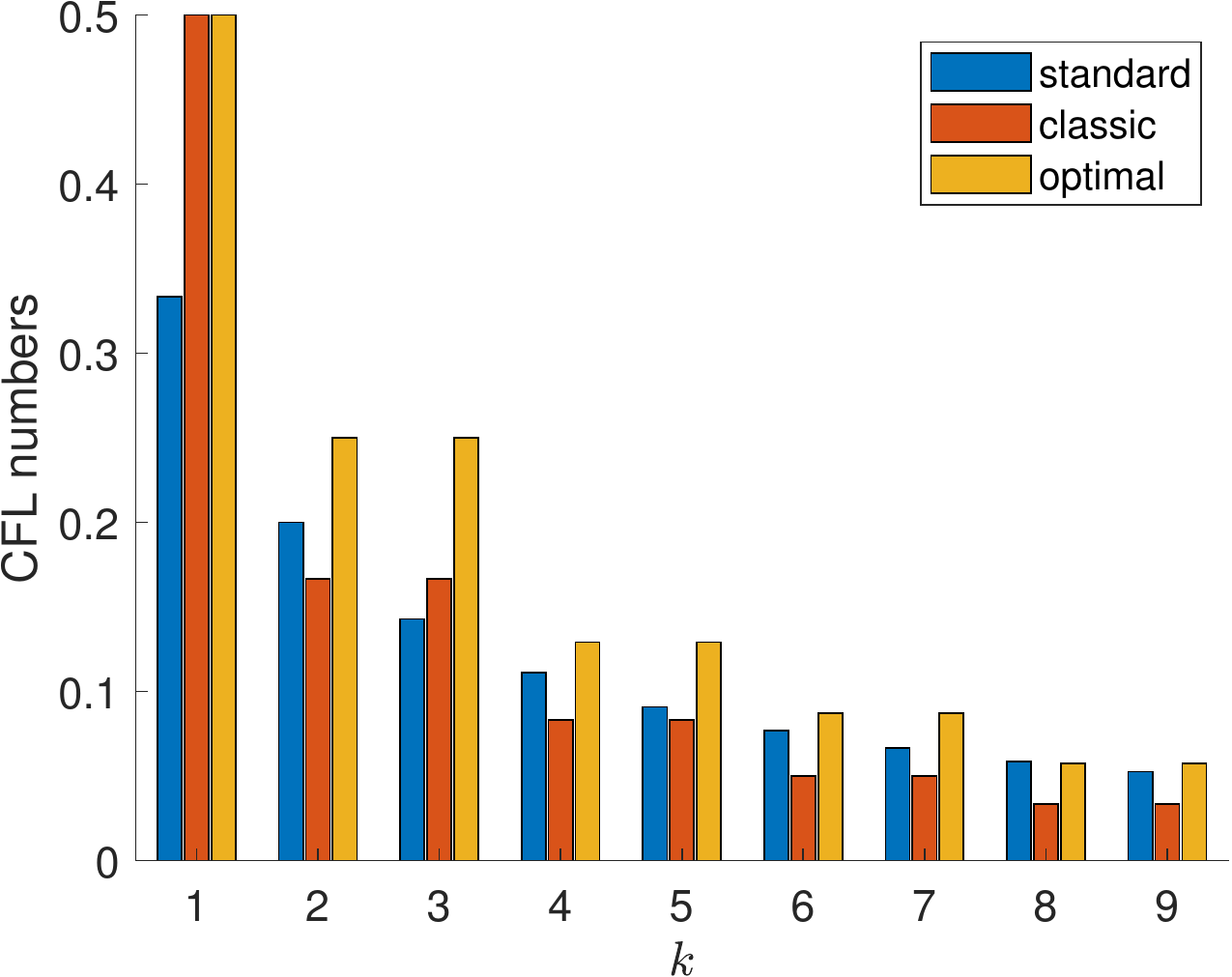}
		\caption{$c_0=1$}
	\end{subfigure}
	\begin{subfigure}[t][][t]{0.48\textwidth}
		\centering
		\includegraphics[width=0.99\textwidth]{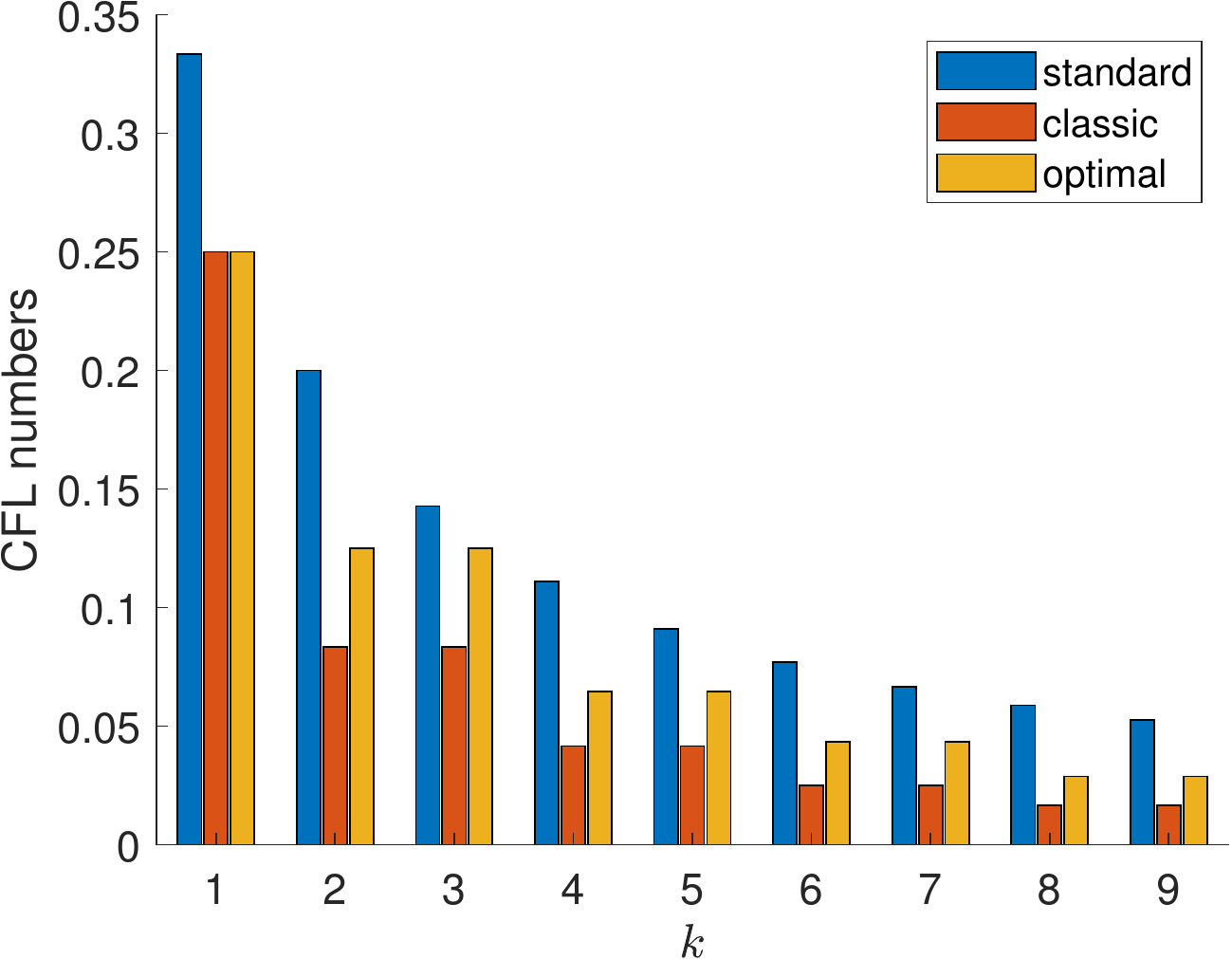}
		\caption{$c_0=\frac12$}
	\end{subfigure}
	\caption{Comparison of standard linearly stable CFL number $\frac{1}{2k+1}$,  classic BP CFL number $\omega_1^{{\tt GL}} c_0$, and 
		optimal BP CFL number $\Wmu_\star c_0$,  
		in the case of $a_1 \Delta y = a_2 \Delta x$ for $\mathbb P^k$-based DG methods with $1\le k \le 9$. 
	}\label{fig:199}
\end{figure}

\begin{remark}[Easy implementation]
	It is worth emphasizing that one 
	only requires a slight and local modification to an existing code to enjoy the above-mentioned advantages of our OCADs or quasi-optimal CADs. Specifically, one only needs to slightly modify the BP limiting procedure, and then the theoretical BP CFL condition is notably improved. 
\end{remark}

\begin{remark}
	The above analysis only considered the forward Euler time discretization. Because a high-order SSP time discretization can be viewed as a convex combination of the forward Euler method, our analysis and conclusions are also valid if the  high-order SSP time discretization is employed. 
\end{remark}

%
%
%

\section{Numerical experiments}\label{sec:examples}

This section tests the accuracy, efficiency, and robustness of the 2D high-order BP DG schemes designed via the proposed OCAD and quasi-optimal CAD, which are respectively referred to as the ``{\tt optimal} approach'' and 
``{\tt quasi-optimal} approach'' for short. For comparison, we also present the results of the 2D high-order BP DG schemes designed via the Zhang--Shu classic CAD \eqref{eq:U2Dsplit}, which are referred to as the ``{\tt classic} approach'' for short. 
We employ the three-stage third-order SSP Runge--Kutta method \cite{GottliebKetchesonShu2011} for time discretization, except for the accuracy tests in Example 1 where the $(k+1)$th-order SSP multi-step method is used    
for the $\mathbb P^k$-based DG scheme to match the temporal and spatial accuracy.  
The time step-size is taken as that indicated by the theoretical BP CFL condition, or that indicated by linear stability, whichever is smaller, namely, 
$$
\Delta t= C_{\tt SSP} \frac{ \min\{ \Wmu c_0, \frac{1}{2k+1} \} }{ \frac{a_1}{\dx}+ \frac{a_2}{\dy} },
$$
where $C_{\tt SSP}$ denotes the SSP coefficient of the adopted time discretization method, $\Wmu = \Wmu_\star$ for the {\tt optimal} approach, $\Wmu = \Wmu_\star^{\tt Q}$ for the {\tt quasi-optimal} approach, and 
$\Wmu = \omega_1^{\tt GL}$ for the {\tt classic} approach, respectively. 
While the CAD is independent of the choice of BP numerical fluxes, we adopt 
the global Lax--Friedrichs flux with $c_0=1$ in all our numerical tests. 
All  the  schemes  are  implemented using C++ language with double precision on a Linux server with 
Intel(R) Xeon(R) Platinum 8268 CPU @ 2.90GHz 2TB RAM. 


\subsection{Example 1: Linear convection equation}
In order to examine the convergence, 
we first consider the 2D linear convection equation 
$$
u_t+u_x + u_y=0, \qquad (x,y,t)\in[-1,1]\times[-1,1]\times\mathbb{R}^+
$$
with periodic boundary conditions and initial data  $u(x,y,0) = \sin(\pi(x+y))$. 
The exact solution $u(x,y,t)=\sin( \pi(x+y-2t) )$ satisfies a maximum principle with the invariant region $G=[-1,1]$. 
We simulate this problem until $t=0.5$. \Cref{tab:ex1_2,tab:ex1_3,tab:ex1_4,tab:ex1_5} list the $\ell^2$ numerical errors and the corresponding convergence rates for the $\mathbb{P}^k$-based BP DG method with $k\in \{2,4,6,8\}$ 
at different grid resolutions of $N \times N$ cells with $\Delta x=\Delta y=2/N$. 
The CPU time of all these simulations is presented in these tables. 
The results show the expected $(k+1)$th-order convergence is achieved by the $\mathbb{P}^k$ BP DG method, while the BP limiter does not destroy the accuracy. 
Moreover, the {\tt optimal} approach admits a larger time step and thus uses much less CPU time than the {\tt classic} approach, while the numerical errors of the two approaches are very close.  
This confirms the advantage in efficiency of using OCAD over the classic CAD.

\begin{table}[htbp] 
	\centering
	\caption{Example 1: $\ell^2$ errors at $t=0.5$ and corresponding convergence rates for $\mathbb{P}^2$-based BP DG methods with a four-step third-order multi-step method whose SSP coefficient is $\frac{1}3$. The CPU time is measured and shown in seconds.}
	\label{tab:ex1_2}
	\setlength{\tabcolsep}{2mm}{
		\begin{tabular}{ccllcll}
			\toprule[1.5pt]
			\multirow{2}{*}{$ N $} &
			\multicolumn{3}{c}{{\tt optimal} approach} & 
			\multicolumn{3}{c}{{\tt classic} approach} \\
			\cmidrule(r){2-4} \cmidrule(l){5-7}
			& $ \ell^{2} $ error & Order & CPU(s)
			& $ \ell^{2} $ error & Order & CPU(s) \\
			
			\midrule[1.5pt]
			& \multicolumn{3}{c}{ $\Delta t=\frac{1}{30}\Delta x$} & \multicolumn{3}{c}{ $\Delta t=\frac{1}{36}\Delta x$}\\
			\cmidrule(r){2-4}  \cmidrule(l){5-7}
			
			20  & 5.37e-4 & -    & 0.52   & 5.39e-4 & -    & 0.67    \\
			40  & 5.94e-5 & 3.18 & 2.50   & 5.96e-5 & 3.18 & 3.37    \\
			80  & 7.30e-6 & 3.02 & 14.62  & 7.31e-6 & 3.03 & 21.88   \\
			160 & 9.11e-7 & 3.00 & 103.20 & 9.11e-7 & 3.00 & 132.39  \\
			320 & 1.14e-7 & 3.00 & 517.60 & 1.14e-7 & 3.00 & 695.42  \\
			640 & 1.42e-8 & 3.00 &2870.80 & 1.42e-8 & 3.00 & 3858.67 \\
			
			\bottomrule[1.5pt]
		\end{tabular}
	}
\end{table}

\begin{table}[htbp] 
	\centering
	\caption{Same as \Cref{tab:ex1_2} except for $\mathbb{P}^4$-based BP DG methods with a ten-step fifth-order multi-step method whose SSP coefficient is about $0.282$.}
	\label{tab:ex1_3}
	\setlength{\tabcolsep}{2mm}{
		\begin{tabular}{ccllcll}
			\toprule[1.5pt]
			\multirow{2}{*}{$ N $} &
			\multicolumn{3}{c}{{\tt optimal} approach} & 
			\multicolumn{3}{c}{{\tt classic} approach} \\
			\cmidrule(r){2-4} \cmidrule(l){5-7}
			& $ \ell^{2} $ error & Order & CPU(s)
			& $ \ell^{2} $ error & Order & CPU(s) \\
			
			\midrule[1.5pt]
			& \multicolumn{3}{c}{ $\Delta t=\frac{0.282}{18}\Delta x$} & \multicolumn{3}{c}{ $\Delta t=\frac{0.282}{24}\Delta x$}\\
			\cmidrule(r){2-4} \cmidrule(l){5-7}
			
			20 &6.98e{-}7 &{-} &4.30  &6.78e{-}7 &{-} &6.95\\%
			40 &2.11e{-}8 &5.05&25.62 &2.11e{-}8 &5.00&35.66\\%
			60 &2.78e{-}9 &5.00&70.72 &2.78e{-}9 &5.00&108.96\\%
			80 &6.61e{-}10&5.00&142.74&6.61e{-}10&5.00&211.88\\%
			100&2.17e{-}10&5.00&225.72&2.17e{-}10&5.00&343.83\\%
			120&8.70e{-}11&5.00&354.57&8.70e{-}11&5.00&542.14\\%
			
			\bottomrule[1.5pt]
		\end{tabular}
	}
\end{table}

\begin{table}[htbp] 
	\centering
	\caption{Same as \Cref{tab:ex1_2} except for $\mathbb{P}^6$-based BP DG methods with an eighteen-step seventh-order multi-step method whose SSP coefficient is about $0.217$.}
	\label{tab:ex1_4}
	\setlength{\tabcolsep}{2mm}{
		\begin{tabular}{ccllcll}
			\toprule[1.5pt]
			\multirow{2}{*}{$ N $} &
			\multicolumn{3}{c}{{\tt optimal} approach} & 
			\multicolumn{3}{c}{{\tt classic} approach} \\
			\cmidrule(r){2-4} \cmidrule(l){5-7}
			& $\ell^{2} $ error & Order & CPU(s)
			& $\ell^{2} $ error & Order & CPU(s) \\
			
			\midrule[1.5pt]
			& \multicolumn{3}{c}{ $\Delta t=\frac{0.217}{26}\Delta x$} & \multicolumn{3}{c}{ $\Delta t=\frac{0.217}{40}\Delta x$}\\
			\cmidrule(r){2-4}  \cmidrule(l){5-7}
			
			4 &1.37e{-}4 &{-} &0.58 &1.57e{-}4 &{-} &1.12\\%
			8 &1.98e{-}7 &9.44&2.35 &3.12e{-}7 &8.98&4.07\\%
			12&1.17e{-}8 &6.97&5.18 &1.37e{-}8 &7.72&9.49\\%
			16&1.59e{-}9 &6.94&10.21&1.59e{-}9 &7.48&18.97\\%
			20&3.40e{-}10&6.91&16.91&3.40e{-}10&6.91&30.92\\%
			24&9.69e{-}11&6.88&24.99&9.69e{-}11&6.88&44.33\\%
			
			\bottomrule[1.5pt]
		\end{tabular}
	}
\end{table}

\begin{table}[htbp] 
	\centering
	\caption{Same as \Cref{tab:ex1_2} except for $\mathbb{P}^8$-based BP DG methods with a 22-step ninth-order multi-step method whose SSP coefficient is $0.1$.}
	\label{tab:ex1_5}
	\setlength{\tabcolsep}{2mm}{
		\begin{tabular}{ccllcll}
			\toprule[1.5pt]
			\multirow{2}{*}{$ N $} &
			\multicolumn{3}{c}{{\tt optimal} approach} & 
			\multicolumn{3}{c}{{\tt classic} approach} \\
			\cmidrule(r){2-4} \cmidrule(l){5-7}
			& $\ell^{2} $ error & Order & CPU(s)
			& $\ell^{2} $ error & Order & CPU(s) \\
			
			\midrule[1.5pt]
			& \multicolumn{3}{c}{ $\Delta t=\frac{0.005767}{2}\Delta x  \approx \frac{0.1}{34.68}\Delta x$} & \multicolumn{3}{c}{ $\Delta t=\frac{0.1}{60}\Delta x$}\\
			\cmidrule(r){2-4} \cmidrule(l){5-7}
			
			4 &2.12e{-}7 &{-} &3.38 &2.29e{-}7 &{-} &7.78\\%
			6 &5.36e{-}9 &9.08&7.18 &5.38e{-}9 &9.25&15.17\\%
			8 &4.00e{-}10&9.02&12.91&4.00e{-}10&9.04&28.41\\%
			10&5.41e{-}11&8.97&22.99&5.41e{-}11&8.96&47.77\\%
			12&1.04e{-}11&9.03&32.06&1.04e{-}11&9.03&68.85\\%
			14&2.61e{-}12&9.00&44.91&2.62e{-}12&8.97&92.61\\%
			
			\bottomrule[1.5pt]
		\end{tabular}
	}
\end{table}

\subsection{Example 2: Inviscid Burgers' equation}

In this example \cite{zhang2010}, we consider the inviscid Burgers' equation 
\begin{align}\label{eq:Burgers}
	u_t+ \left(\frac{u^2}{2} \right)_x + \left(\frac{u^2}{2}\right)_y=0, \qquad (x,y,t)\in[-1,1]\times[-1,1]\times\mathbb{R}^+. 
\end{align}
The initial condition is taken as $u(x,y,0) = \sin(\pi(x+y))$, and the periodic boundary conditions are adopted. 
The exact solution also obeys the maximum principle with $G=[-1,1]$. 
The exact solution is smooth up to $t = \frac{1}{2\pi} \approx 0.159$, and later a stationary shock wave develops.  
\Cref{fig:ex2-1} shows the numerical solutions at $t=0.23$ and the snapshots cut along $y=x$, obtained by using  
$\mathbb{P}^2, \mathbb{P}^4, \mathbb{P}^6$ -based BP DG methods. 
We see the shock is well captured, with only the BP limiter and without using any non-oscillatory limiters.  
One can observe that the {\tt optimal} approach  and the {\tt classic} approach give very similar results, and 
the shock is equally well resolved by both approaches. 
However, the CPU time of the {\tt optimal} approach is much less than that of the 
{\tt classic} approach, as shown in Table \ref{tab:ex2}.

\begin{figure}[htbp]
	\centering
	\begin{subfigure}[t][][t]{0.31\textwidth}
		\centering
		\includegraphics[width=\textwidth]{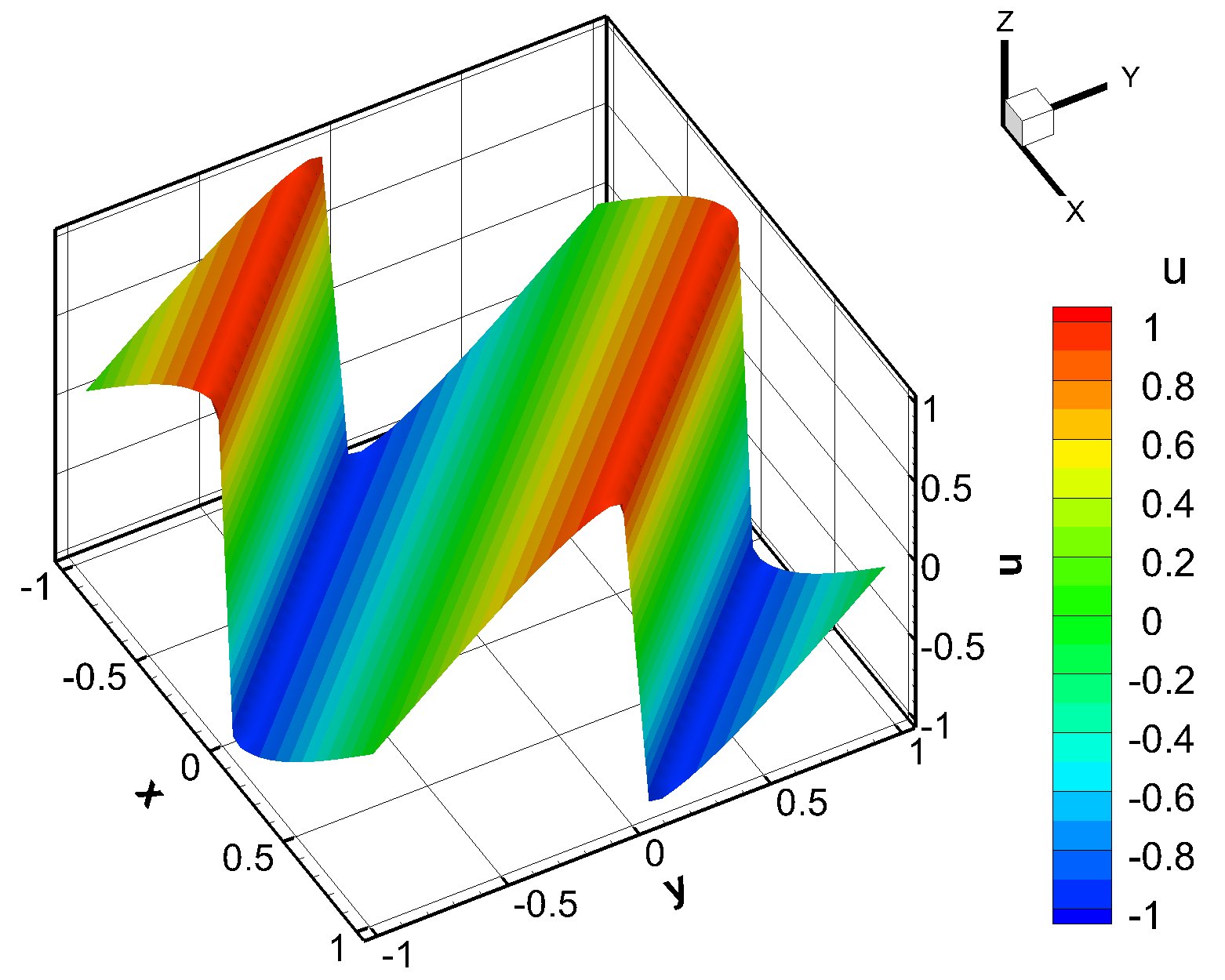}
		\caption{$\mathbb{P}^2$}
	\end{subfigure}
	\quad
	\begin{subfigure}[t][][t]{0.31\textwidth}
		\centering
		\includegraphics[width=\textwidth]{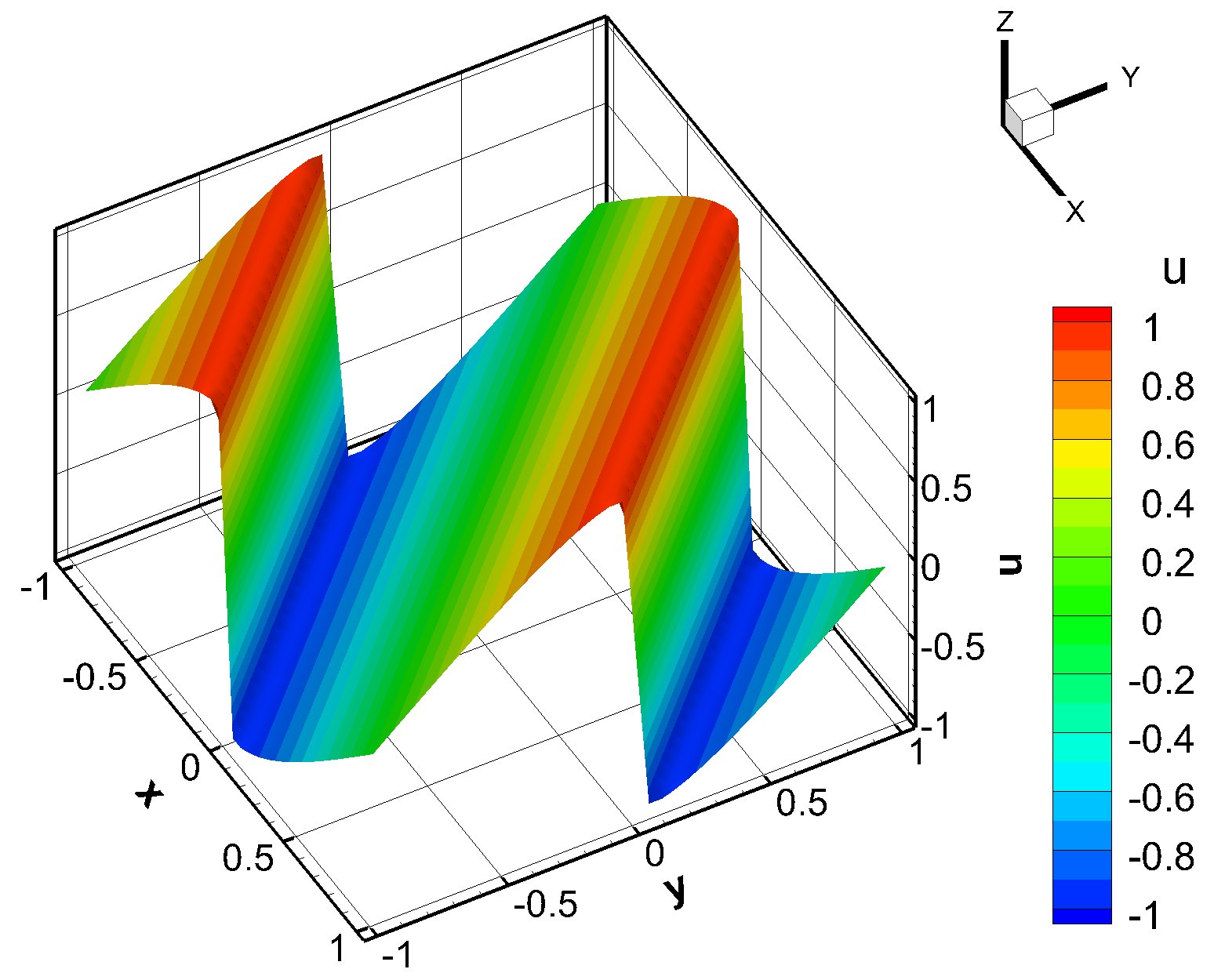}
		\caption{$\mathbb{P}^4$}
	\end{subfigure}
	\quad
	\begin{subfigure}[t][][t]{0.31\textwidth}
		\centering
		\includegraphics[width=\textwidth]{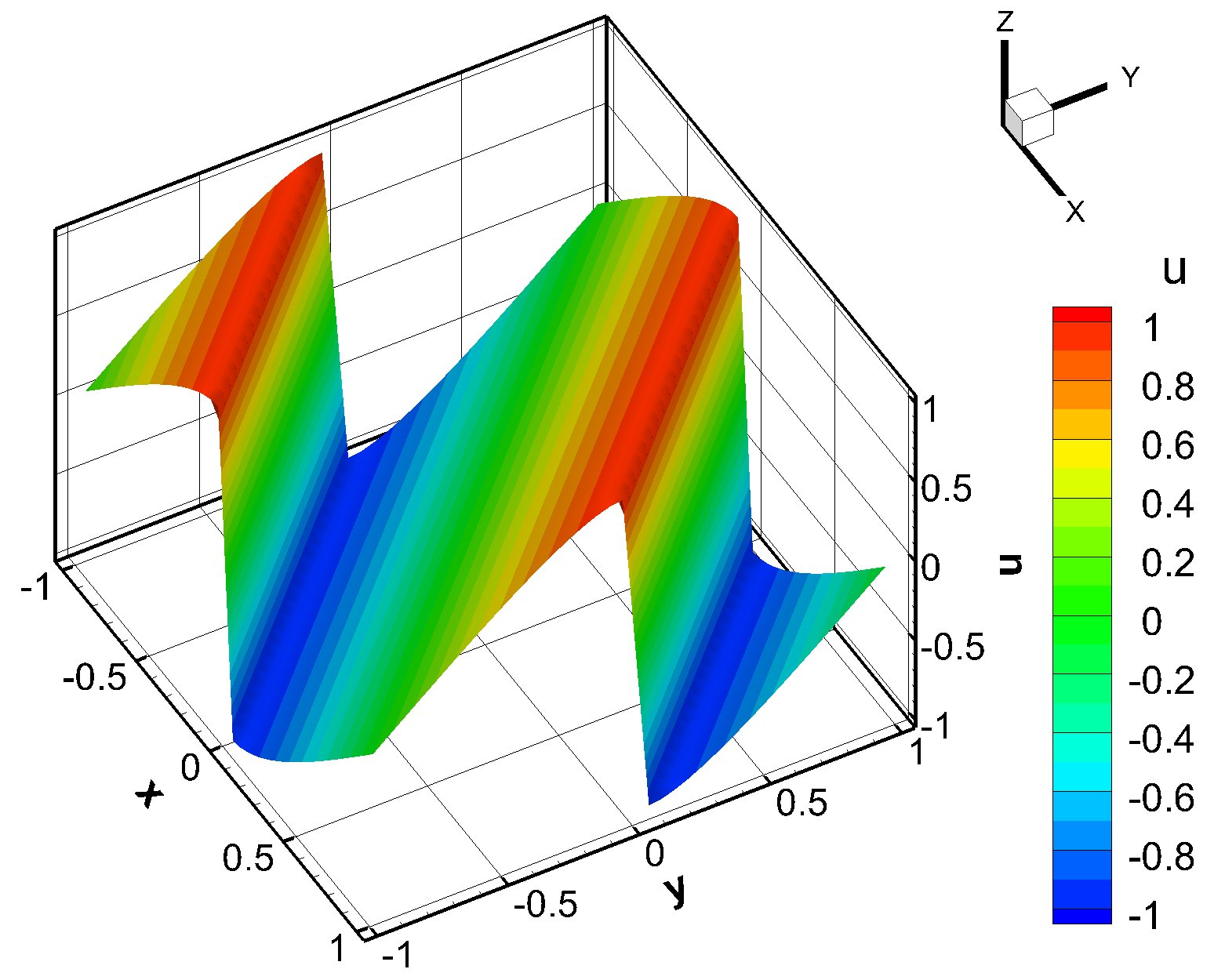}
		\caption{$\mathbb{P}^6$}
	\end{subfigure}
	
	\begin{subfigure}[t][][t]{0.31\textwidth}
		\centering
		\includegraphics[width=\textwidth]{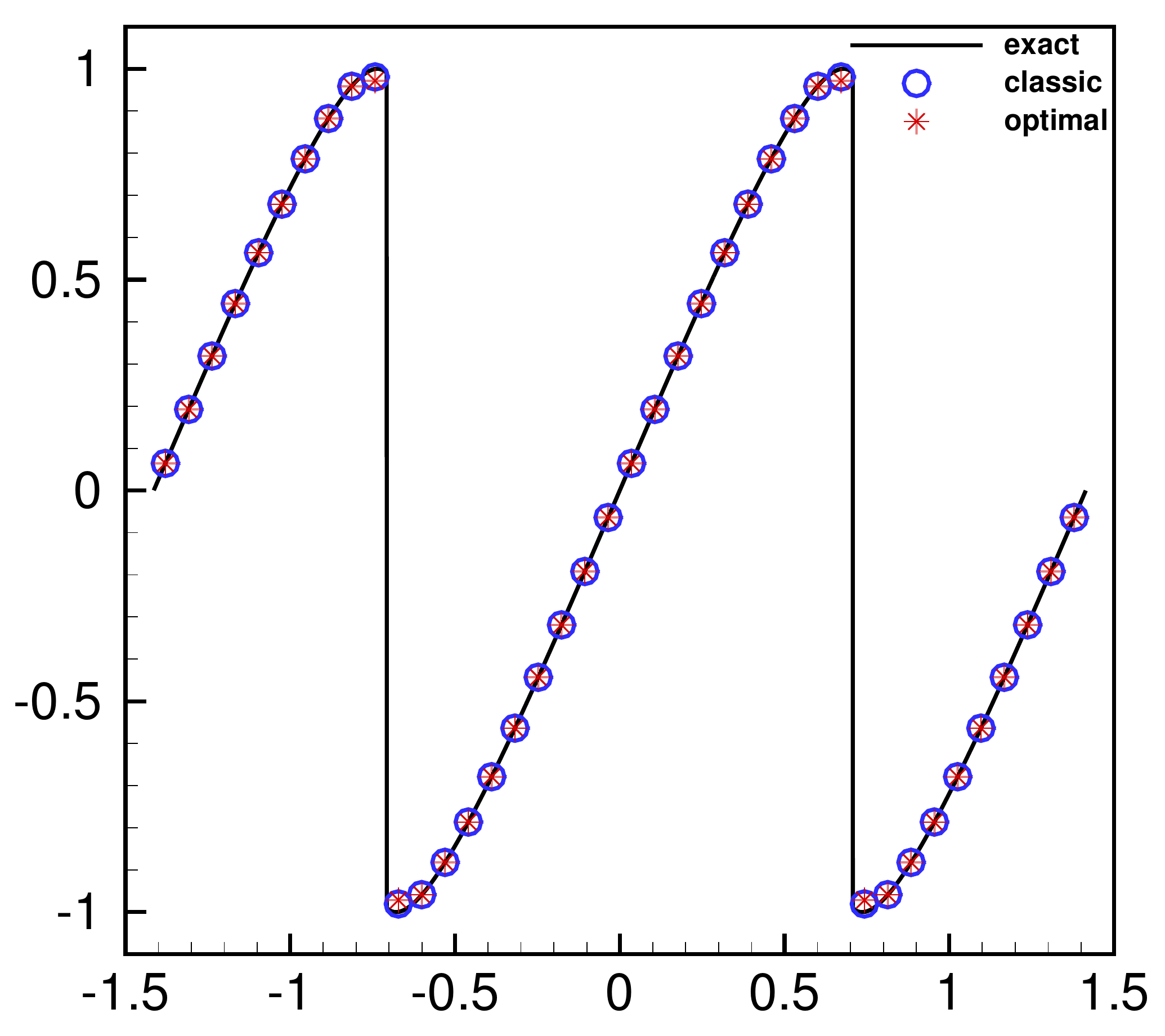}
		\caption{$\mathbb{P}^2$}
	\end{subfigure}
	\quad
	\begin{subfigure}[t][][t]{0.31\textwidth}
		\centering
		\includegraphics[width=\textwidth]{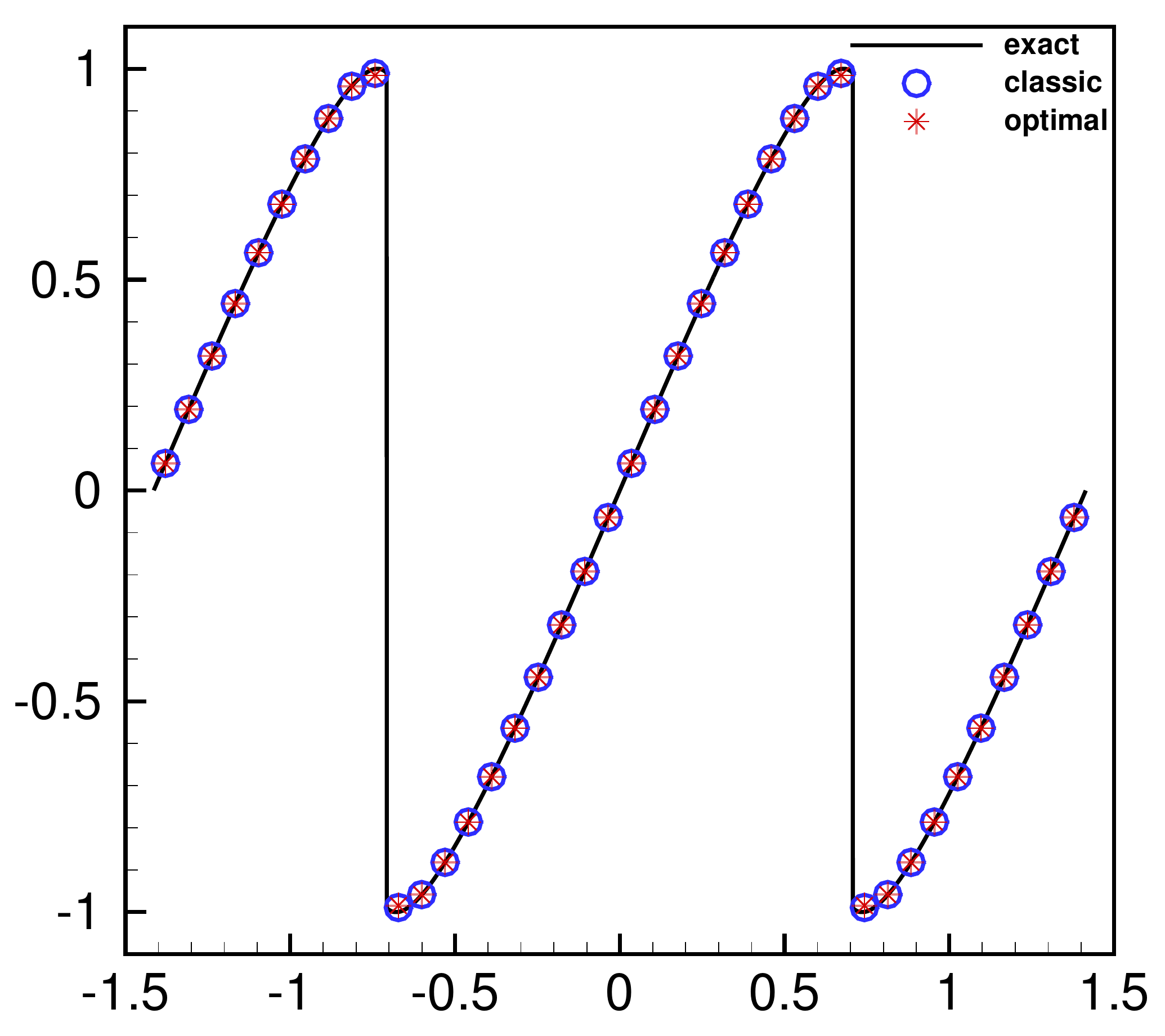}
		\caption{$\mathbb{P}^4$}
	\end{subfigure}
	\quad
	\begin{subfigure}[t][][t]{0.31\textwidth}
		\centering
		\includegraphics[width=\textwidth]{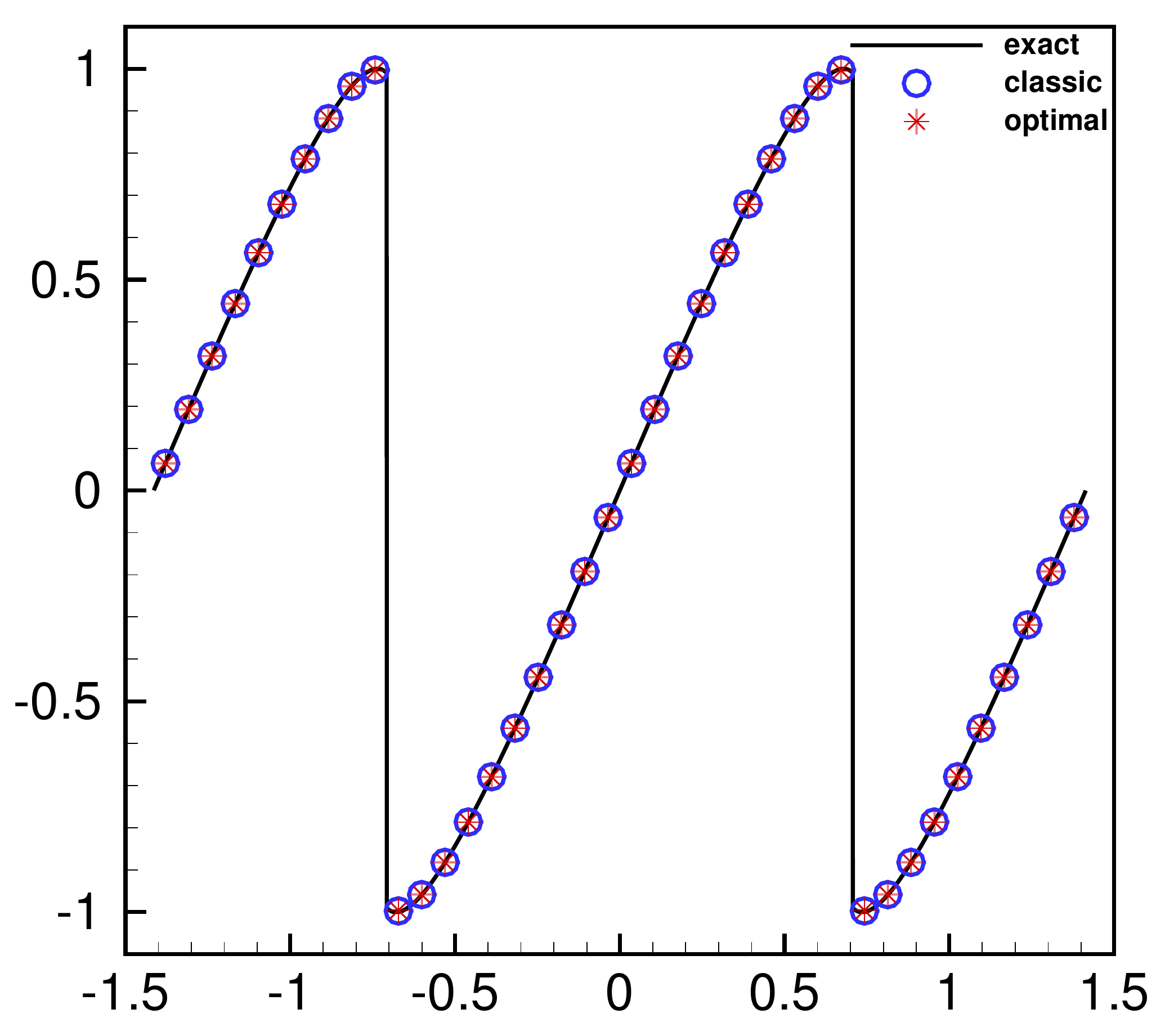}
		\caption{$\mathbb{P}^6$}
	\end{subfigure}
	\caption{Example 2: the numerical solutions at $t=0.23$ obtained by the $\mathbb{P}^2, \mathbb{P}^4, \mathbb{P}^6$-based BP DG methods. Top: the surface of the solutions with {\tt optimal} approach. Bottom: 
		Comparison of the solutions cut along $y=x$.
		The spatial mesh size is $\Delta x = \Delta y = \frac{2}{40}$.
	}
	\label{fig:ex2-1}
\end{figure}

\begin{table}[htbp] 
	\centering
	\caption{CPU time in seconds for simulating Example 2 up to $t=0.23$.}
	\label{tab:ex2}
	\setlength{\tabcolsep}{3.8mm}{
		\begin{tabular}{llccc}
			\toprule[1.5pt]
			
			Mesh & CAD & $\mathbb{P}^2$ & $\mathbb{P}^4$ & $\mathbb{P}^6$ \\
			
			\midrule[1.5pt]
			\multirow{2}{*}{ $40\times 40$} 
			& {\tt optimal} approach       & 1.098 & 1.988 & 39.015 \\
			& {\tt classic} approach      & 1.350 & 2.598 & 59.180 \\
			
			\midrule[1.5pt]
			\multirow{2}{*}{ $80\times 80$} 
			& {\tt optimal} approach       & 7.081 & 12.945 & 214.307 \\
			& {\tt classic} approach      & 8.321 & 16.272 & 323.648 \\
			
			\bottomrule[1.5pt]
		\end{tabular}
	}
\end{table}

\subsection{Example 3: Compressible Euler equations} 

In this example, we simulate the interaction of a shock and a vortex with low density and low pressure, by 
solving the two-dimensional compressible Euler equations, 
which can be 
formulated in the form of \eqref{2DCL} with 
\begin{equation} \label{Eq:Euler}
	u=\begin{pmatrix}
		\rho \\
		m_1 \\
		m_2 \\
		E
	\end{pmatrix}, \qquad {f_1}(u)=\begin{pmatrix}
		m_1 \\
		m_1 {v_1}+P \\
		m_2 v_1 \\
		(E+p) v_1
	\end{pmatrix}, \qquad {f_2}(u)=\begin{pmatrix}
		m_2 \\
		m_1 v_2 \\
		m_2 {v_2}+P \\
		(E+p) v_2
	\end{pmatrix}.
\end{equation}
Here $\rho$ is the density, $(m_1,m_2)=\rho(v_1,v_2)$ denotes the momentum vector with $(v_1,v_2)$ being the velocity field, $P$ is the pressure, and 
$E=\frac{1}{2} \rho ( {v_1}^{2} + {v_2}^{2})+\frac{P}{\gamma -1}$ denotes the total energy. 
The adiabatic index $\gamma$ is taken as $1.4$. 
The density and the internal energy should be positive, yielding the invariant region 
$$
G= \left\{ u=(\rho,m_1,m_2,E)^\top: \rho(u)>0,~ \rho e(u) := E-\frac{m_1^2+m_2^2}{2 \rho}>0 \right \}, 
$$ 
which is a convex set \cite{zhang2010b} because $\rho e(u)$ is a concave function of $u$.

The setup of our the shock-vortex interaction problem is similar to \cite{jiang1996efficient} except for that the present case involves very low density and low pressure. 
The computational domain is taken as $[0,2]\times[0,1]$.
A shock of Mach number $M=1.1$ is positioned at $x=0.5$ plane and perpendicular to the $x$-axis. 
Its left state is $(\rho_l, v_{1,l},v_{2,l},P_l) = (1,1.1\sqrt{\gamma},0,1)$ while the right state can be obtained through the Rankine-Hugoniot condition: 
\begin{align*}
	\frac{\rho_r }{\rho_l}=\frac{(\gamma+1){M}^2}{2+(\gamma-1){M}^2}, \quad
	\frac{P_r}{P_l} = 1+\frac{2\gamma}{\gamma+1}({M}^2+1), \quad
	\frac{v_{1,r}}{v_{1,l}} = \frac{2+(\gamma-1){M}^2}{(\gamma+1){M}^2}, \quad
	v_{2,r}=0.
\end{align*}
Initially, an isentropic vortex is imposed and 
centered at $(x_c,y_c) = (0.25,0.5)$ on the mean flow left to the shock. 
The perturbations to the velocity $(\delta v_1, \delta v_2)$,  temperature $T_l=P_l/\rho_l$, and entropy $S_l=\ln(P_l/\rho_l^{\gamma})$ associated with the vortex are denoted by
\begin{align*}
	(\delta v_1, \delta v_2) = \frac{\varepsilon}{r_c} e^{\alpha(1-\tau^2)}(\bar{y}, -\bar{x}), \quad
	\delta T = -\frac{(\gamma-1)\varepsilon^2}{4\alpha\gamma} e^{2\alpha(1-\tau^2)}, \quad
	\delta S =0,
\end{align*}
where $\tau = \frac{r}{r_c}, (\bar{x}, \bar{y}) = (x-x_c, y-y_c), r^2= \bar{x}^2 + \bar{y}^2$. 
Here $\epsilon$ is the strength of the vortex, $\alpha = 0.204$ controls the decay rate of the vortex, and $r_c=0.4$ is the critical radius. Different from \cite{jiang1996efficient}, we take the 
vortex 
strength as $\varepsilon=1.378106$, so that the lowest density and lowest pressure are $4.7\times 10^{-15}$ and  $8.8\times 10^{-21}$, respectively. 
\Cref{fig:ex3_2,fig:ex3_3} give the numerical solutions obtained by the $\mathbb P^2$-based and $\mathbb P^4$-based DG methods, respectively, on the uniform mesh of $450 \times 225$ cells. 
We can see that flow structures are well captured by all the BP DG schemes. 
The results of the {\tt optimal} and {\tt quasi-optimal} approaches are comparable to those of the {\tt classic} approach. 
However, the {\tt optimal} and {\tt quasi-optimal} approaches allow larger time steps, with which the CPU time is much less, as shown in Table \ref{tab:ex3Euler}.

For all the three approaches, we use the simplified BP limiter \cite{zhang2011b} as detailed in \Cref{rem:SBPEuler}. 
Without the BP limiter, the DG code would break down because of nonphysical solutions. 
Due to the presence of strong shocks in this and next examples, the WENO limiter \cite{Qiu2005} is also used, right before the BP limiter, within some adaptively detected troubled cells to suppress potential numerical oscillations.

\begin{figure}[htbp]
	\centering
	\begin{subfigure}[t][][t]{0.3\textwidth}
		\centering
		\includegraphics[width=\textwidth]{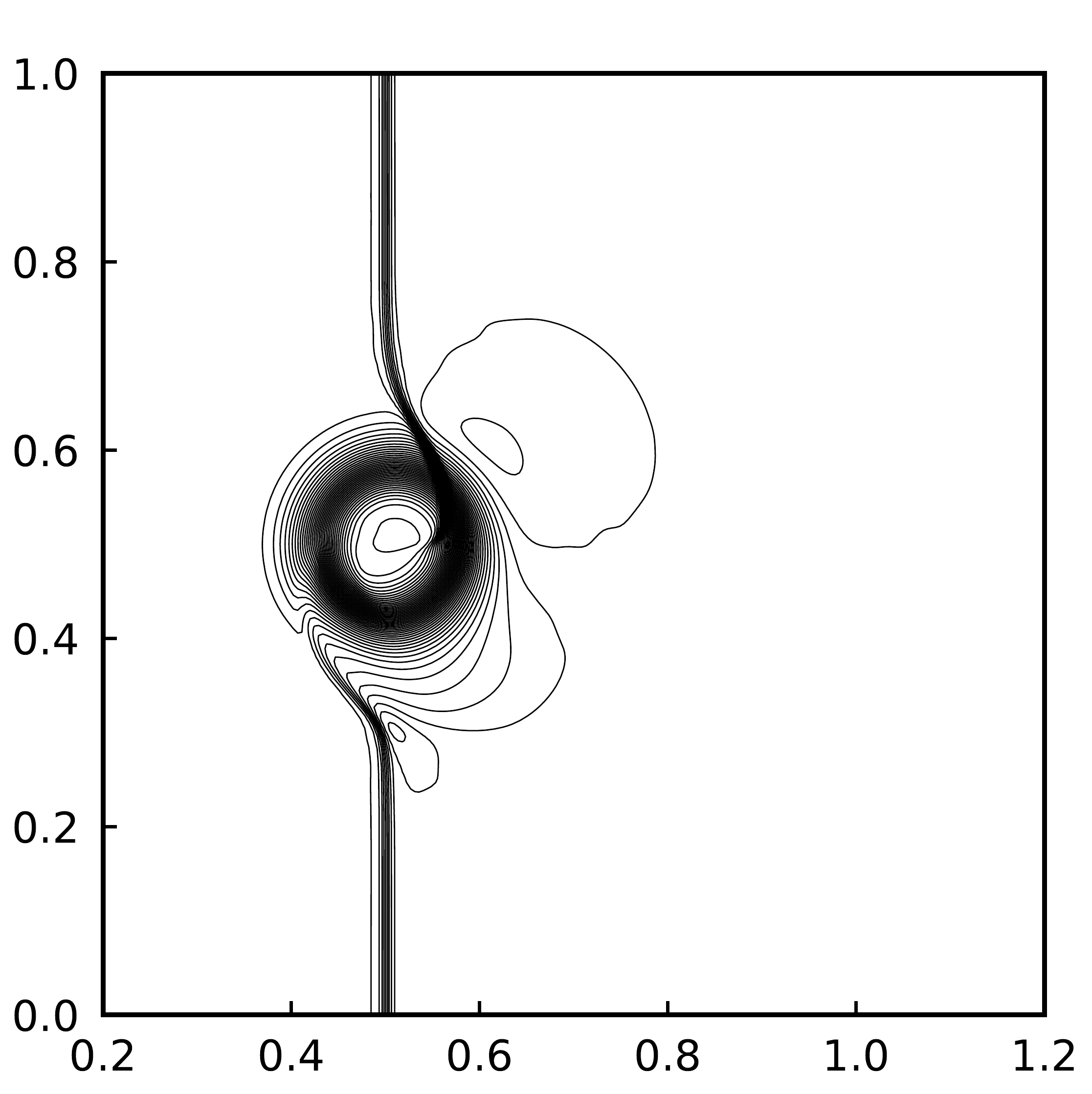}
		\caption{$t=0.2$, {\tt classic}}
		\label{fig:ex3_2-a}
	\end{subfigure}
	\quad
	\begin{subfigure}[t][][t]{0.3\textwidth}
		\centering
		\includegraphics[width=\textwidth]{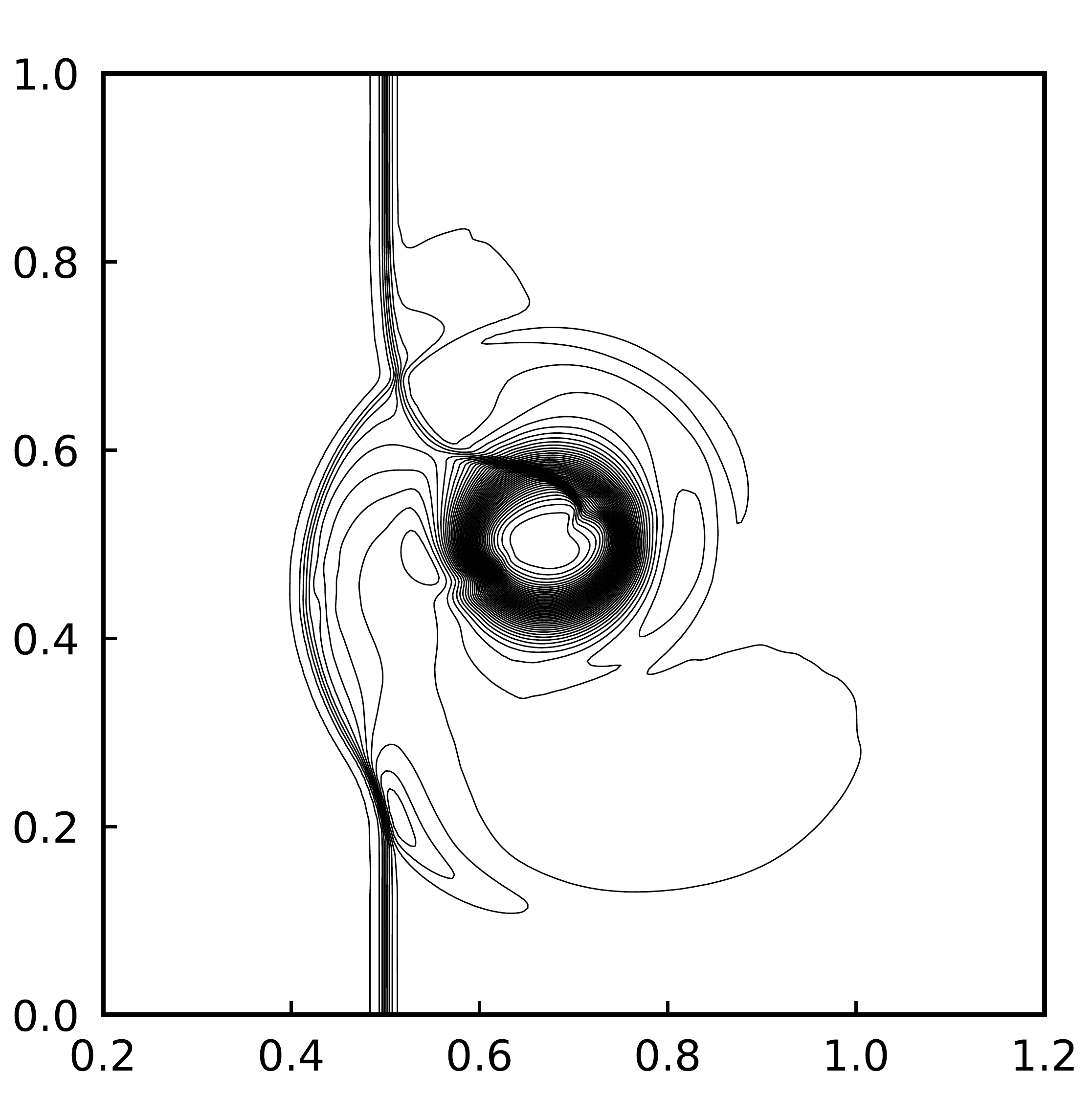}
		\caption{$t=0.35$, {\tt classic}}
		\label{fig:ex3_2-b}
	\end{subfigure}
	\quad
	\begin{subfigure}[t][][t]{0.33\textwidth}
		\centering
		\includegraphics[width=\textwidth]{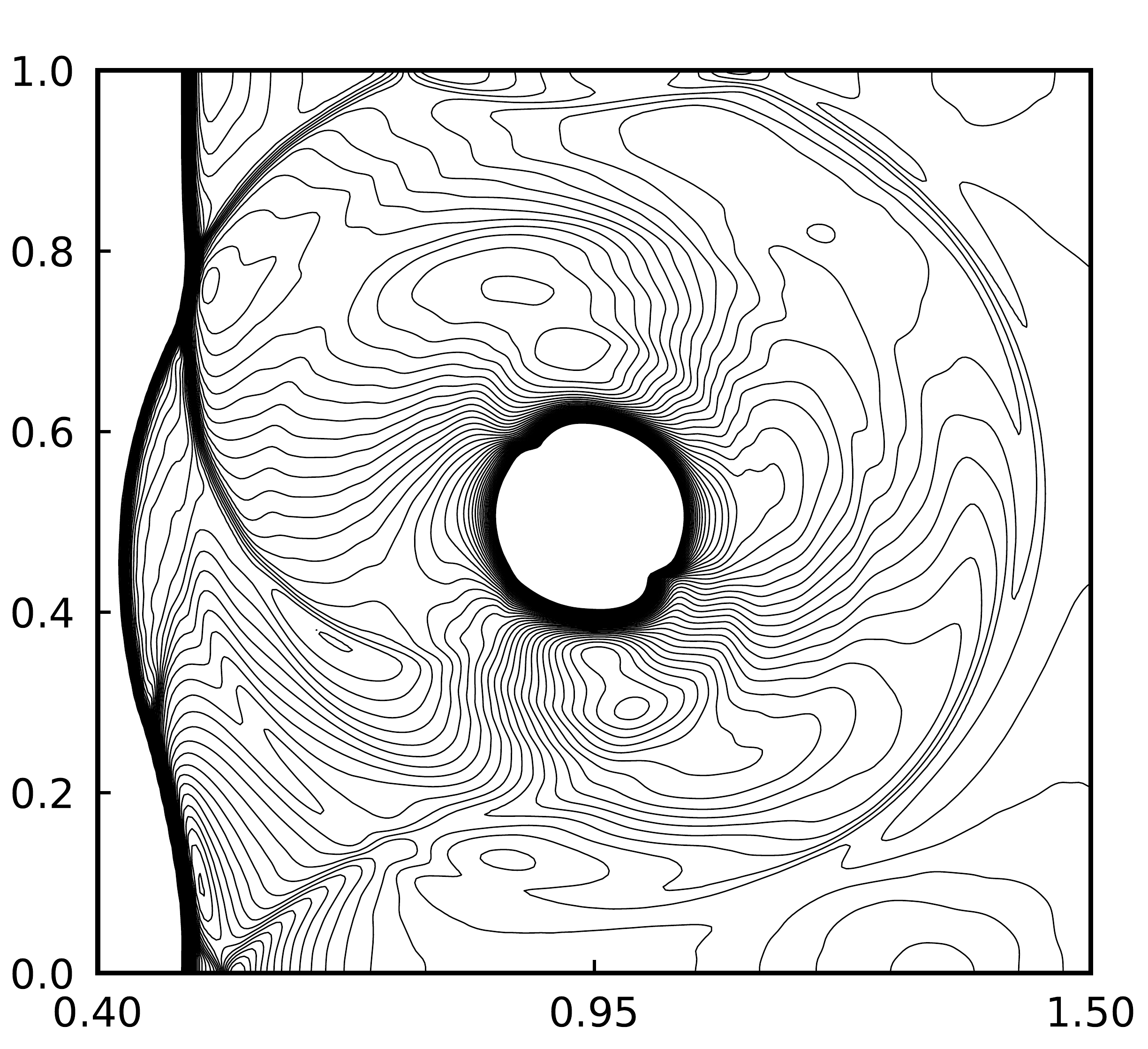}
		\caption{$t=0.6$ {\tt classic}}
		\label{fig:ex4_2-c}
	\end{subfigure}
	
	\begin{subfigure}[t][][t]{0.3\textwidth}
		\centering
		\includegraphics[width=\textwidth]{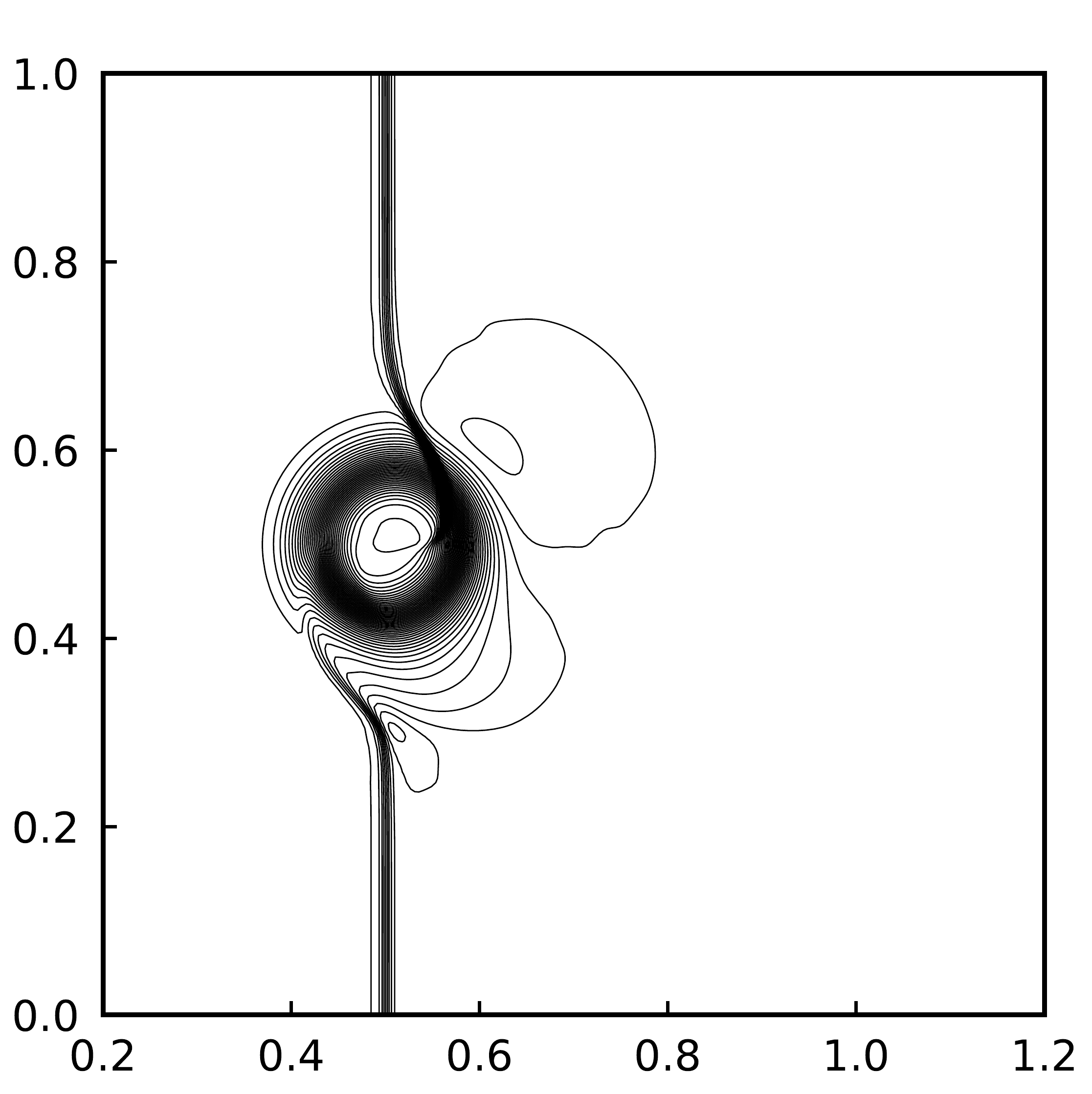}
		\caption{$t=0.2$, {\tt optimal}}
		\label{fig:ex3_2-d}
	\end{subfigure}
	\quad
	\begin{subfigure}[t][][t]{0.3\textwidth}
		\centering
		\includegraphics[width=\textwidth]{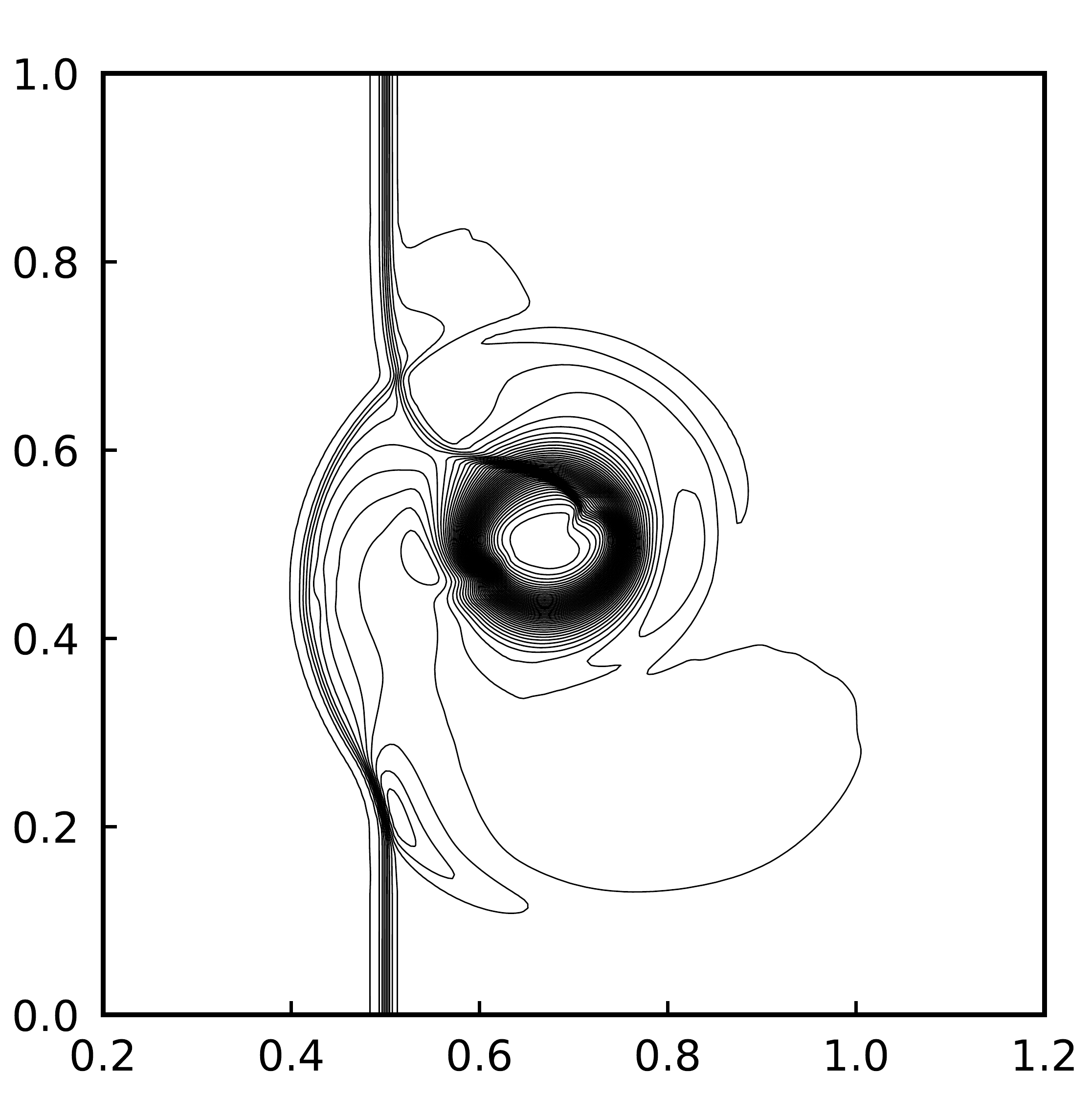}
		\caption{$t=0.35$, {\tt optimal}}
		\label{fig:ex3_2-e}
	\end{subfigure}
	\quad
	\begin{subfigure}[t][][t]{0.33\textwidth}
		\centering
		\includegraphics[width=\textwidth]{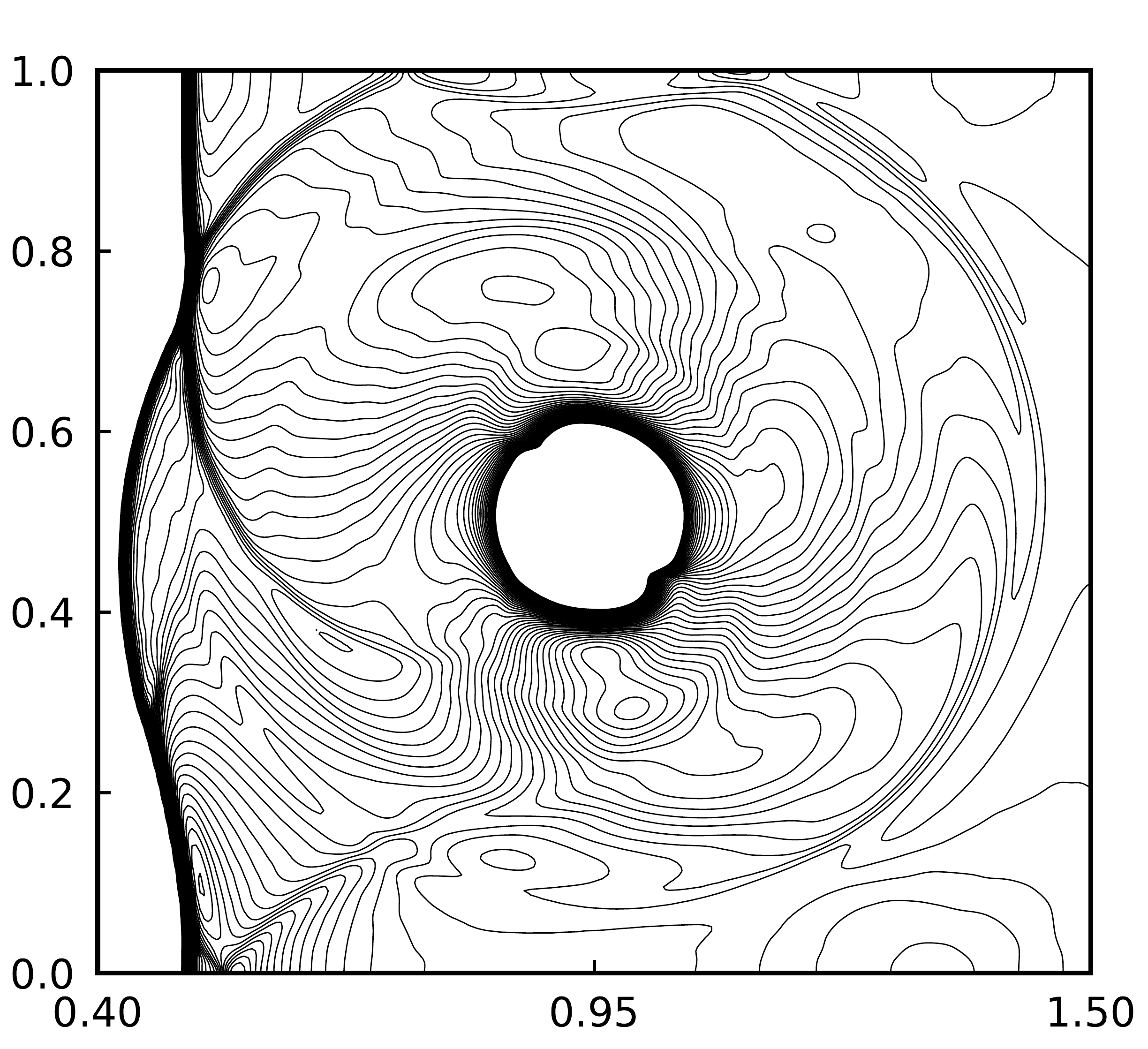}
		\caption{$t=0.6$, {\tt optimal}}
		\label{fig:ex3_2-f}
	\end{subfigure}
	
	\begin{subfigure}[t][][t]{0.3\textwidth}
		\centering
		\includegraphics[width=\textwidth]{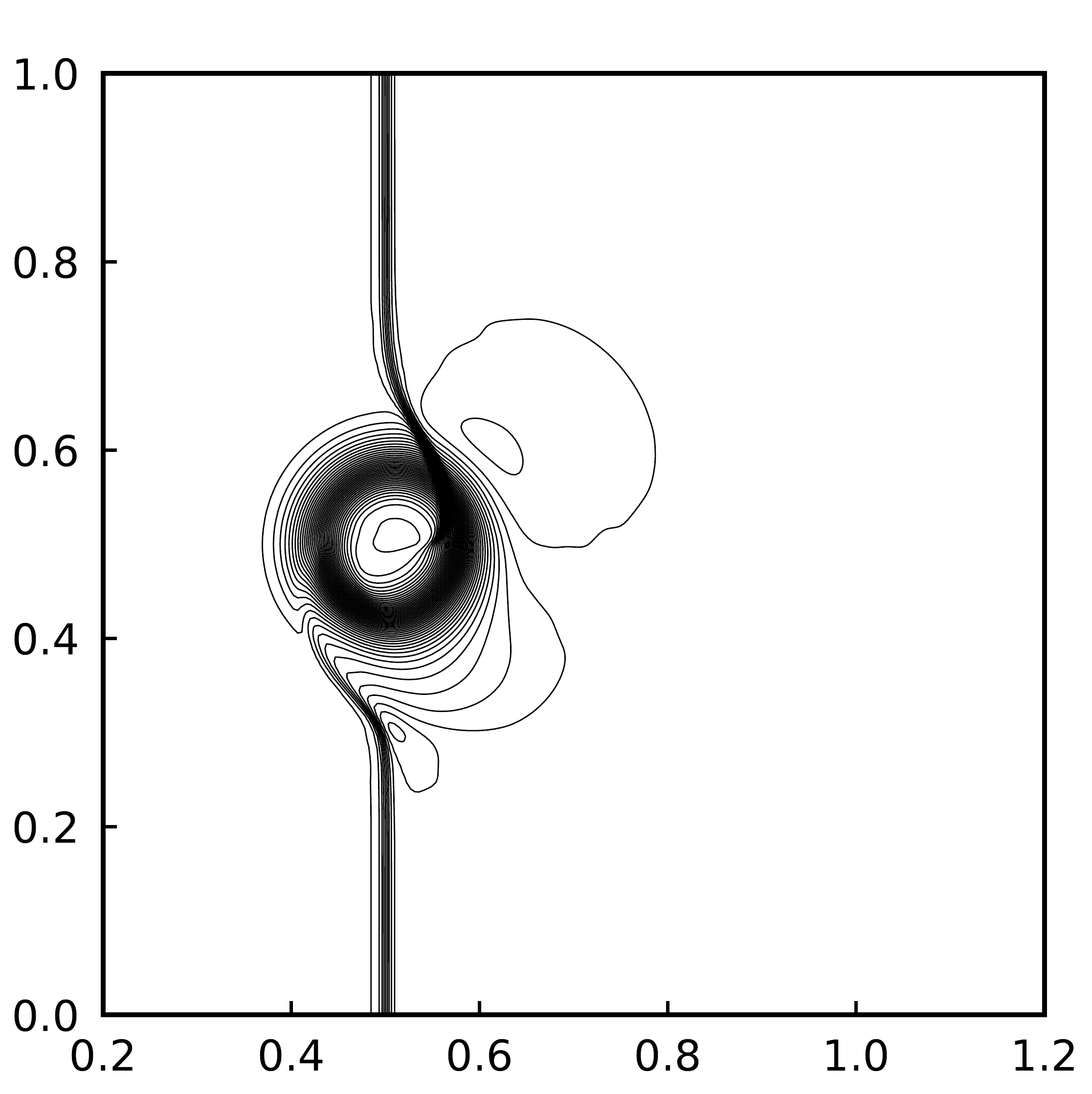}
		\caption{$t=0.2$, {\tt quasi-optimal}}
		\label{fig:ex3_2-g}
	\end{subfigure}
	\quad
	\begin{subfigure}[t][][t]{0.3\textwidth}
		\centering
		\includegraphics[width=\textwidth]{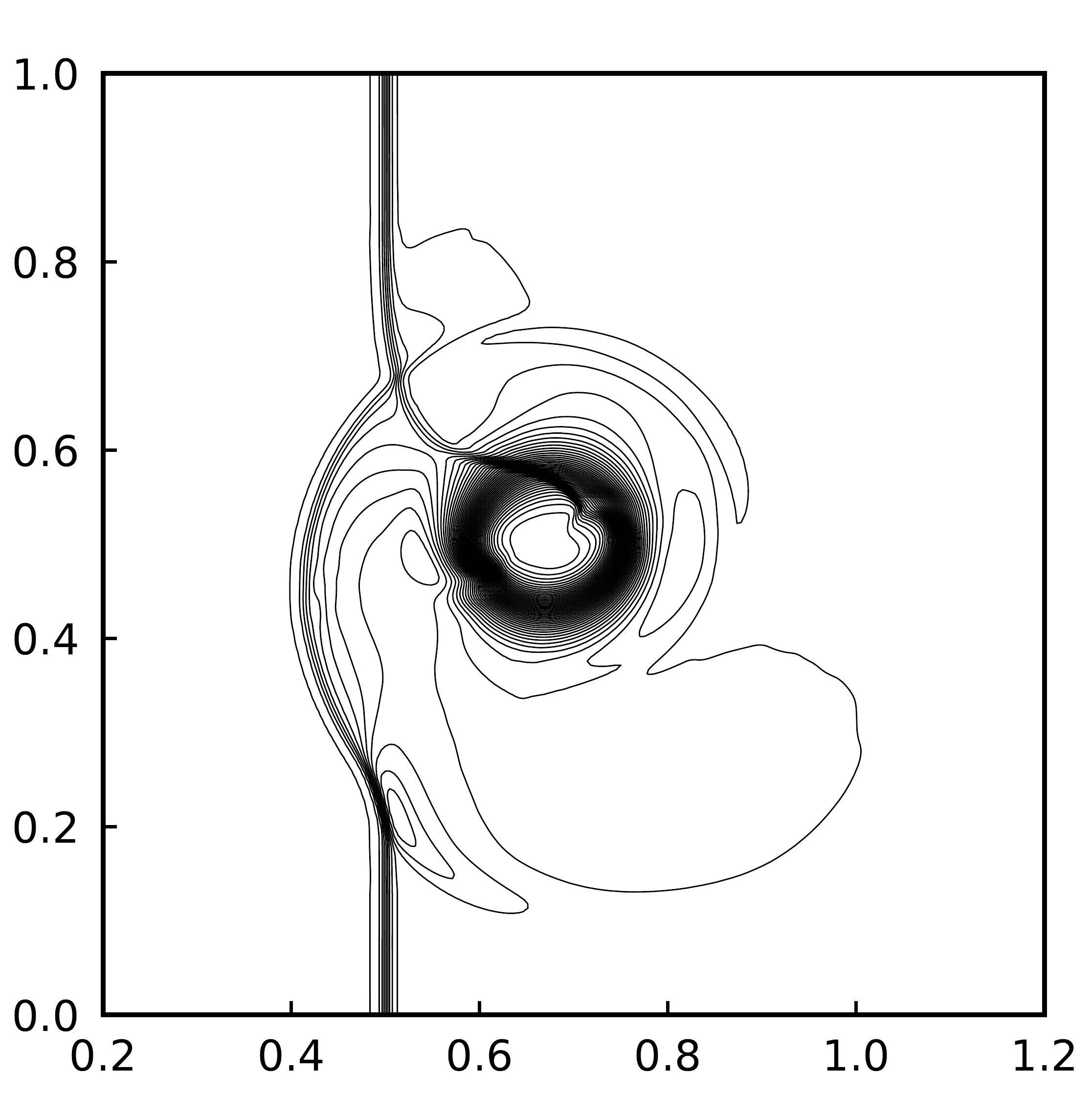}
		\caption{$t=0.35$, {\tt quasi-optimal}}
		\label{fig:ex3_2-h}
	\end{subfigure}
	\quad
	\begin{subfigure}[t][][t]{0.33\textwidth}
		\centering
		\includegraphics[width=\textwidth]{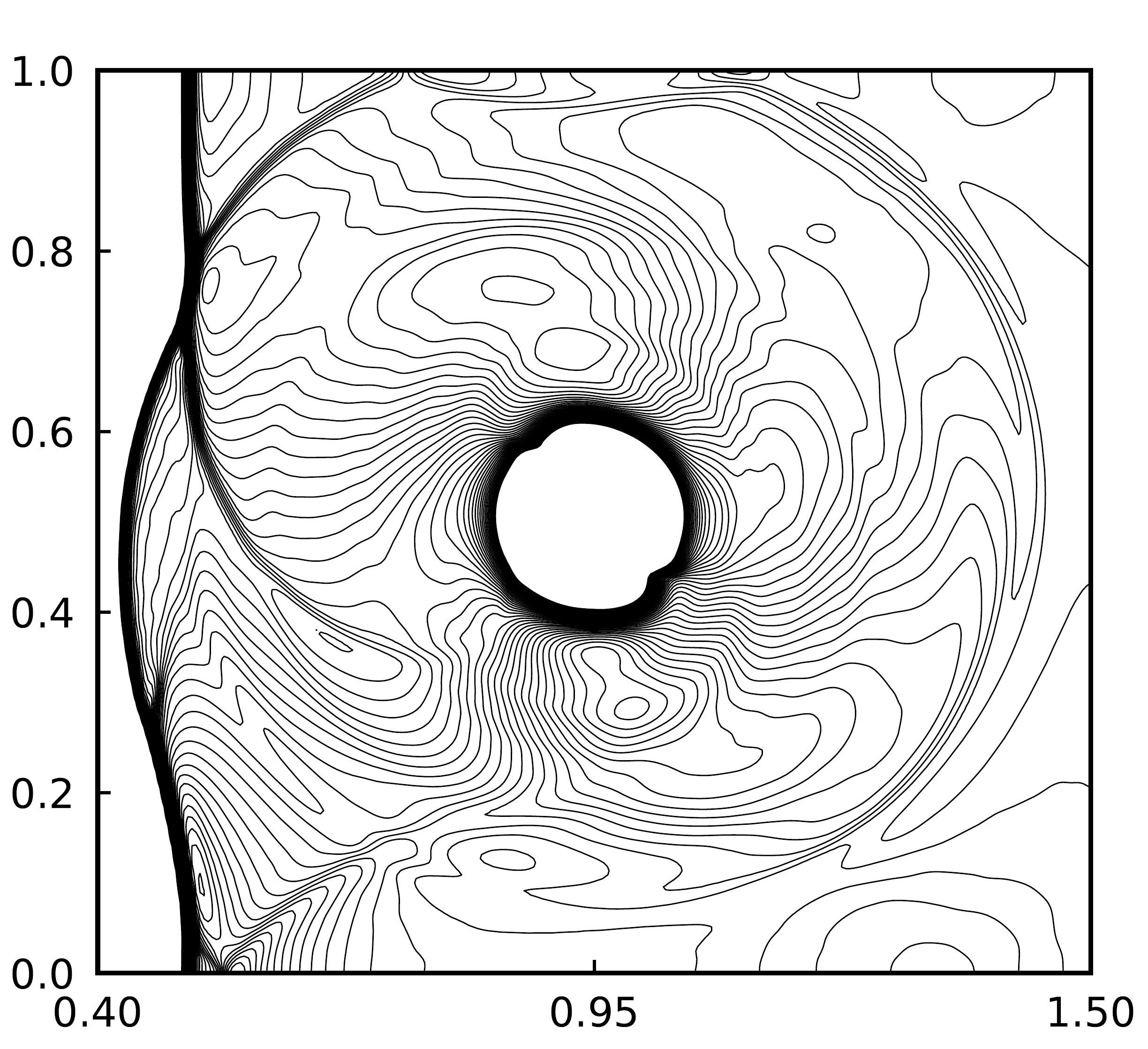}
		\caption{$t=0.6$, {\tt quasi-optimal}}
		\label{fig:ex3_2-l}
	\end{subfigure}
	\caption{Example 3: 
		The contour plots of pressure obtained by the $\mathbb P^2$-based BP DG methods (designed with three different CADs)
		at $t=0.2$, $t=0.35$, and $t=0.6$ (from let to right). 
		50 contour lines: from 0.005 to 1.33 for $t = 0.2$;
		 from 0.01 to 1.402 for $t = 0.35$; 
		 from 1.03 to 1.39: $t = 0.6$. 
	}
	\label{fig:ex3_2}
\end{figure}

\begin{figure}[htbp]
	\centering
	\begin{subfigure}[t][][t]{0.3\textwidth}
		\centering
		\includegraphics[width=\textwidth]{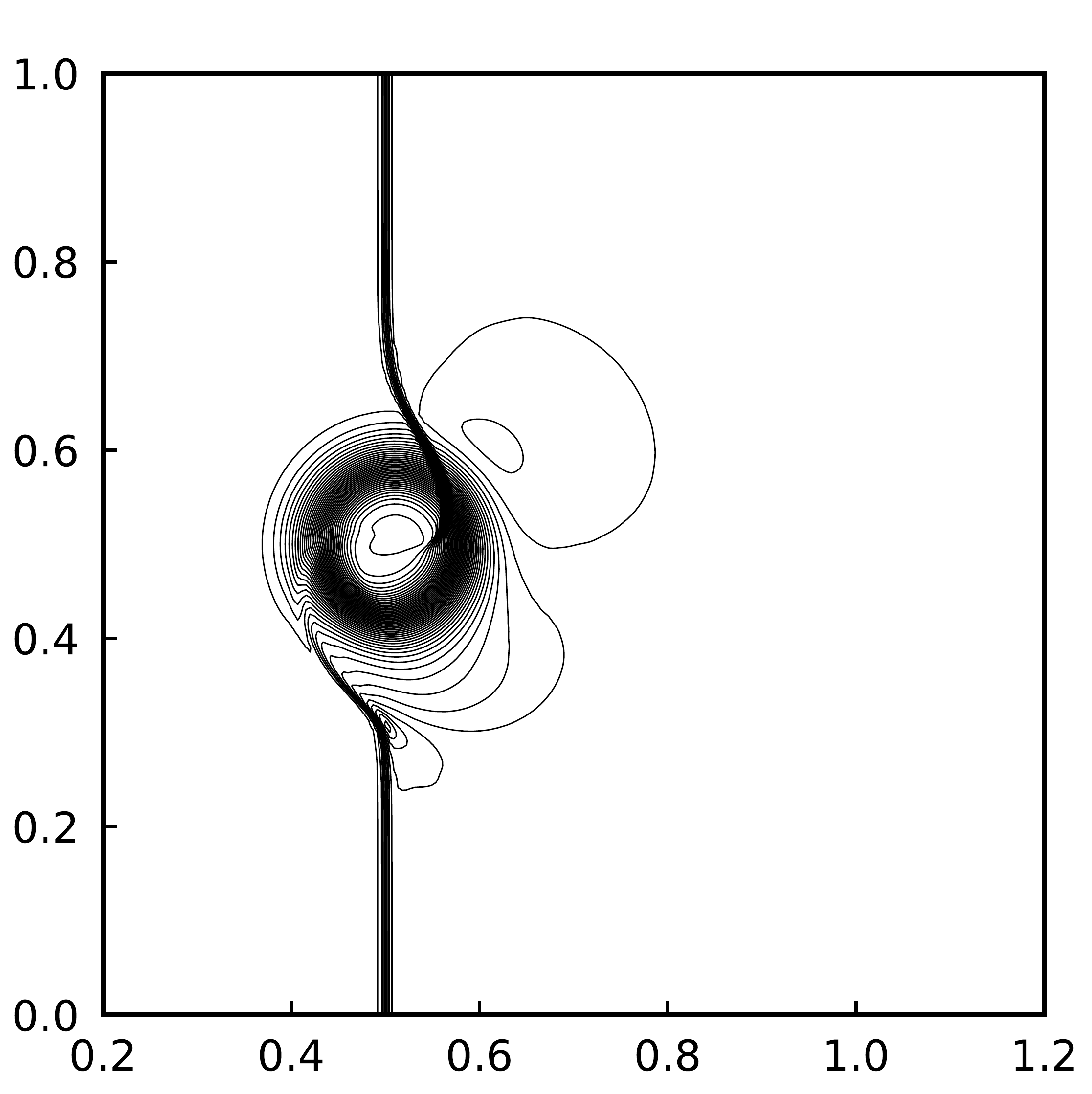}
		\caption{$t=0.2$, {\tt classic}}
		\label{fig:ex3_3-a}
	\end{subfigure}
	\quad
	\begin{subfigure}[t][][t]{0.3\textwidth}
		\centering
		\includegraphics[width=\textwidth]{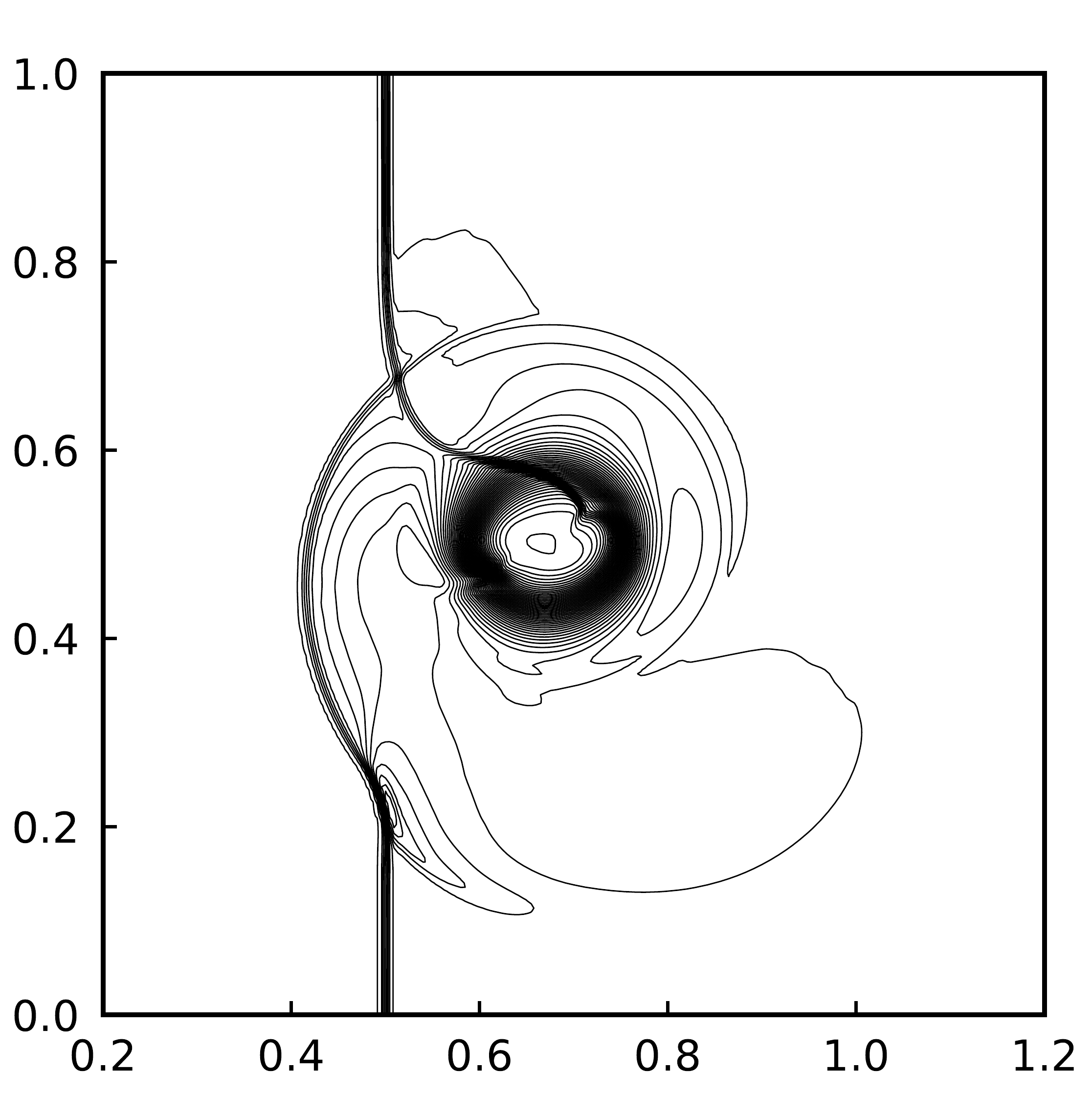}
		\caption{$t=0.35$, {\tt classic}}
		\label{fig:ex3_3-b}
	\end{subfigure}
	\quad
	\begin{subfigure}[t][][t]{0.33\textwidth}
		\centering
		\includegraphics[width=\textwidth]{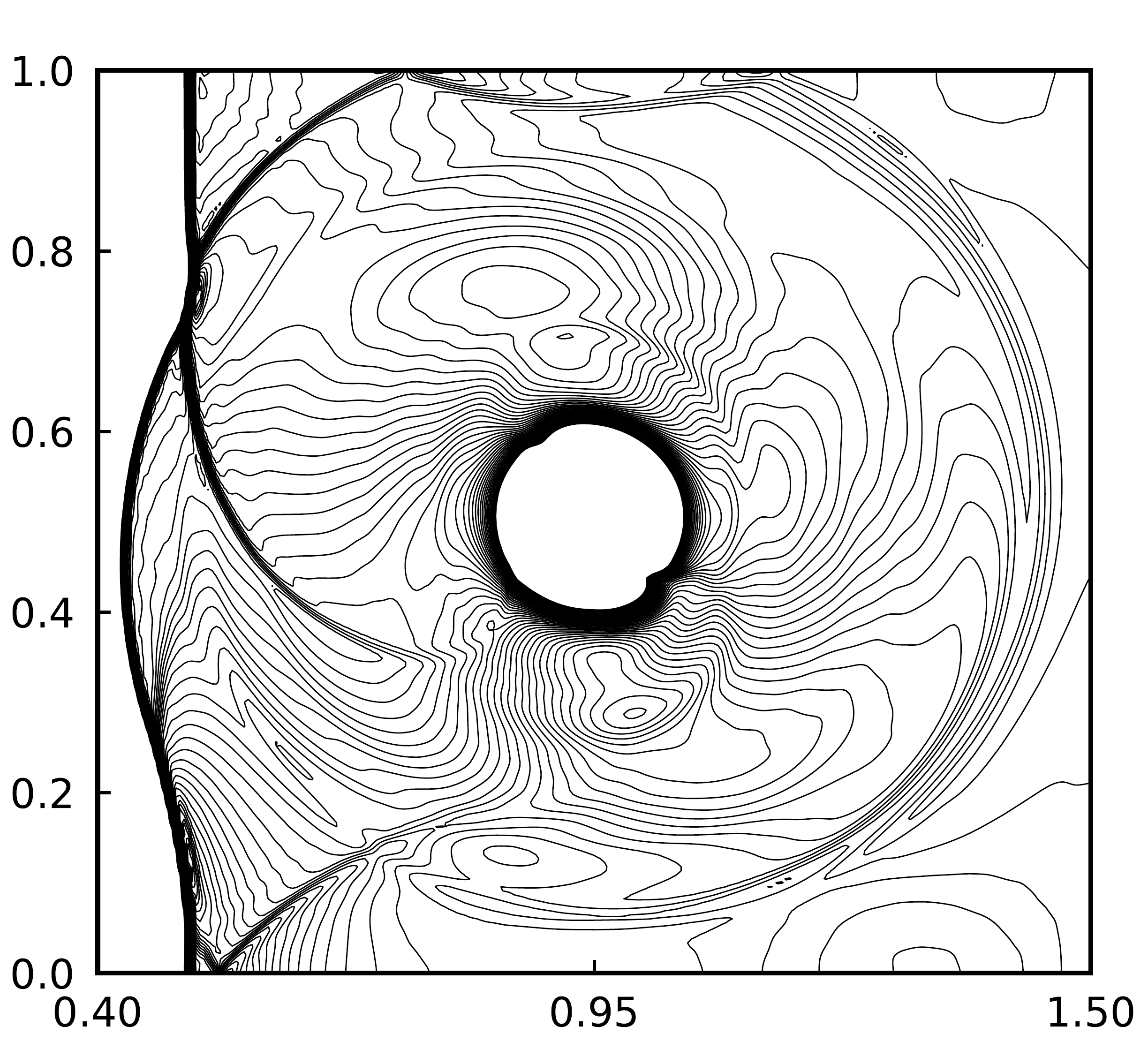}
		\caption{$t=0.6$, {\tt classic}}
		\label{fig:ex3_3-c}
	\end{subfigure}
	
	\begin{subfigure}[t][][t]{0.3\textwidth}
		\centering
		\includegraphics[width=\textwidth]{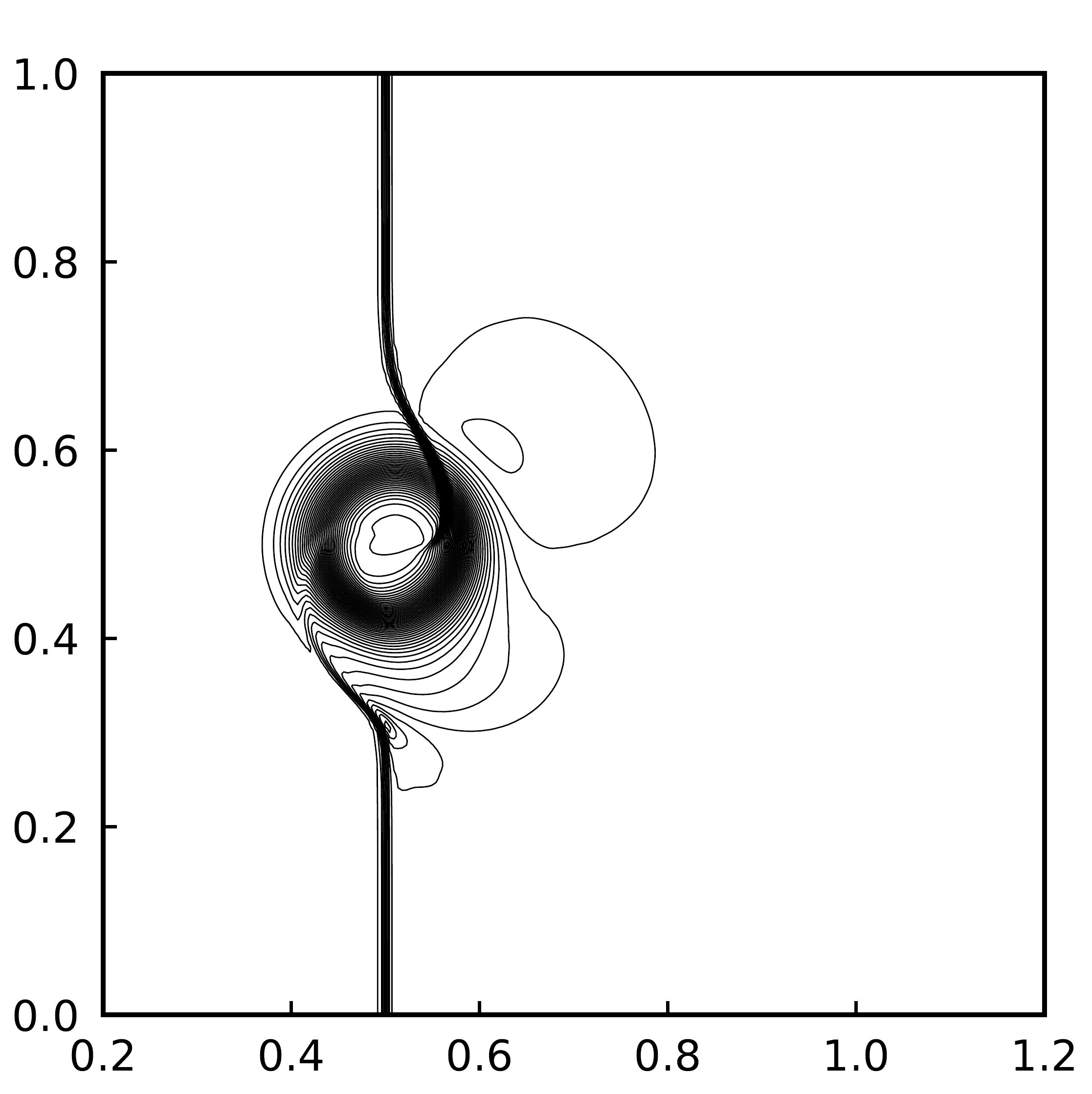}
		\caption{$t=0.2$, {\tt optimal}}
		\label{fig:ex3_3-d}
	\end{subfigure}
	\quad
	\begin{subfigure}[t][][t]{0.3\textwidth}
		\centering
		\includegraphics[width=\textwidth]{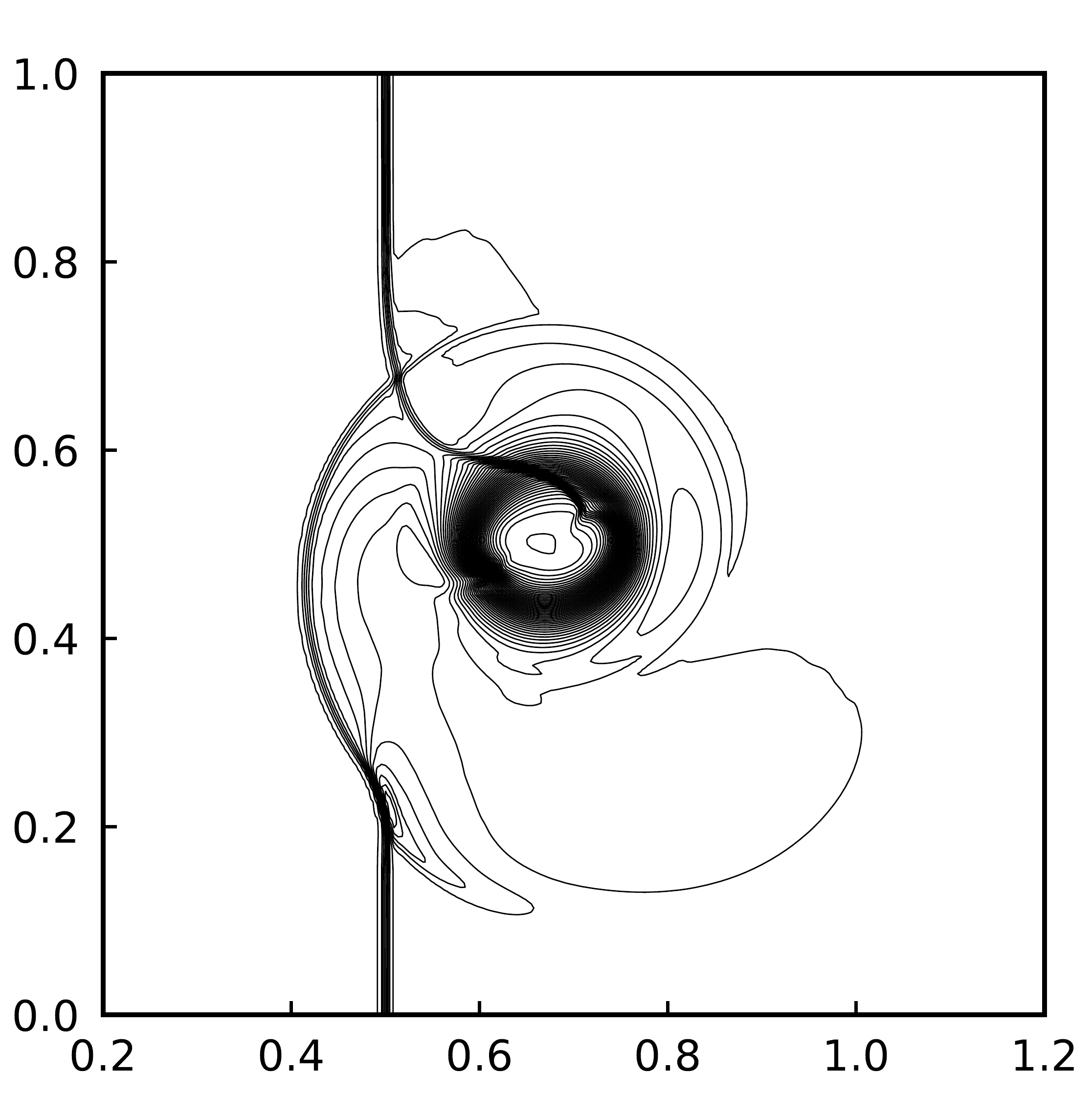}
		\caption{$t=0.35$, {\tt optimal}}
		\label{fig:ex3_3-e}
	\end{subfigure}
	\quad
	\begin{subfigure}[t][][t]{0.33\textwidth}
		\centering
		\includegraphics[width=\textwidth]{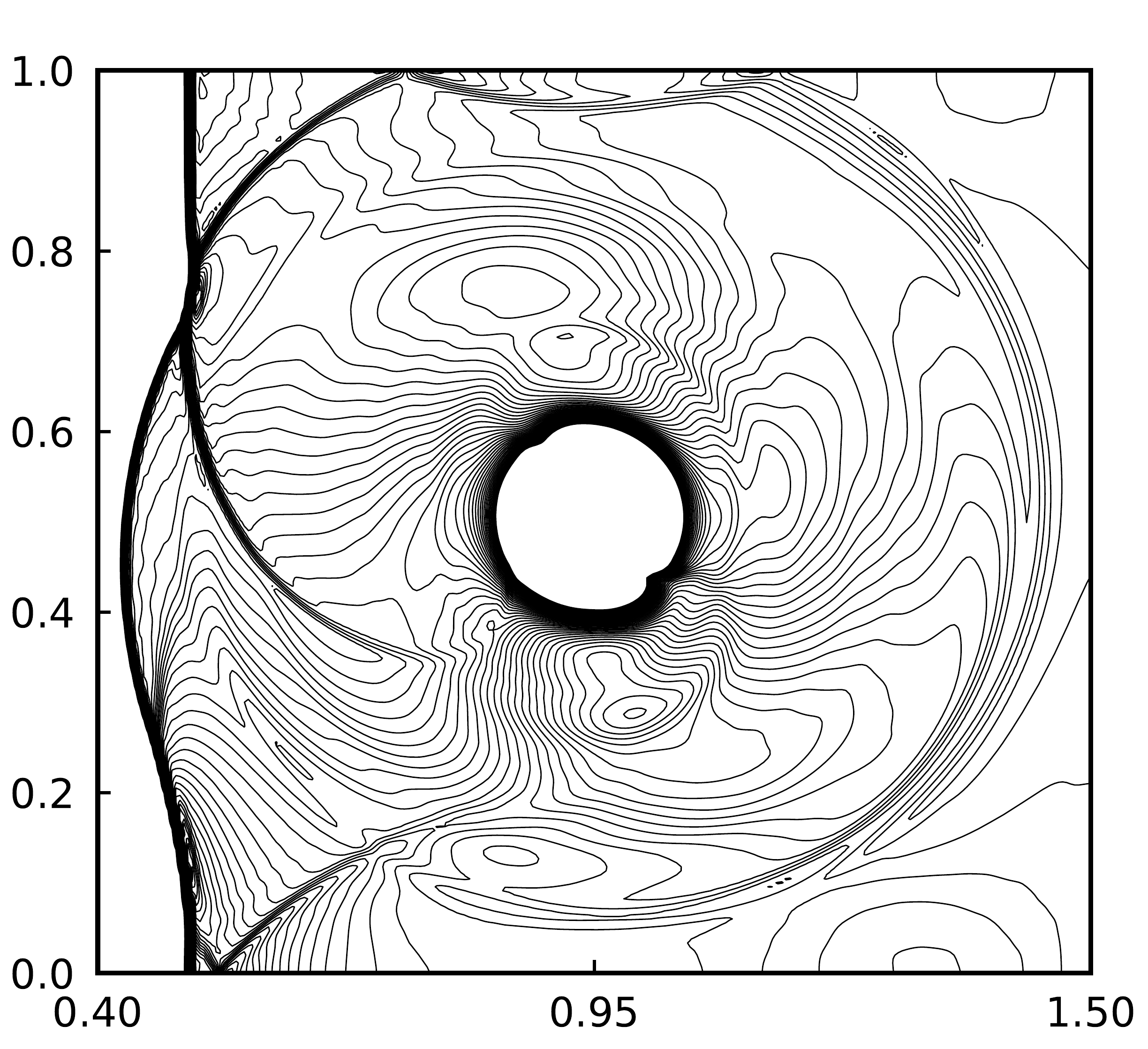}
		\caption{$t=0.6$, {\tt optimal}}
		\label{fig:ex3_3-f}
	\end{subfigure}
	
	\begin{subfigure}[t][][t]{0.3\textwidth}
		\centering
		\includegraphics[width=\textwidth]{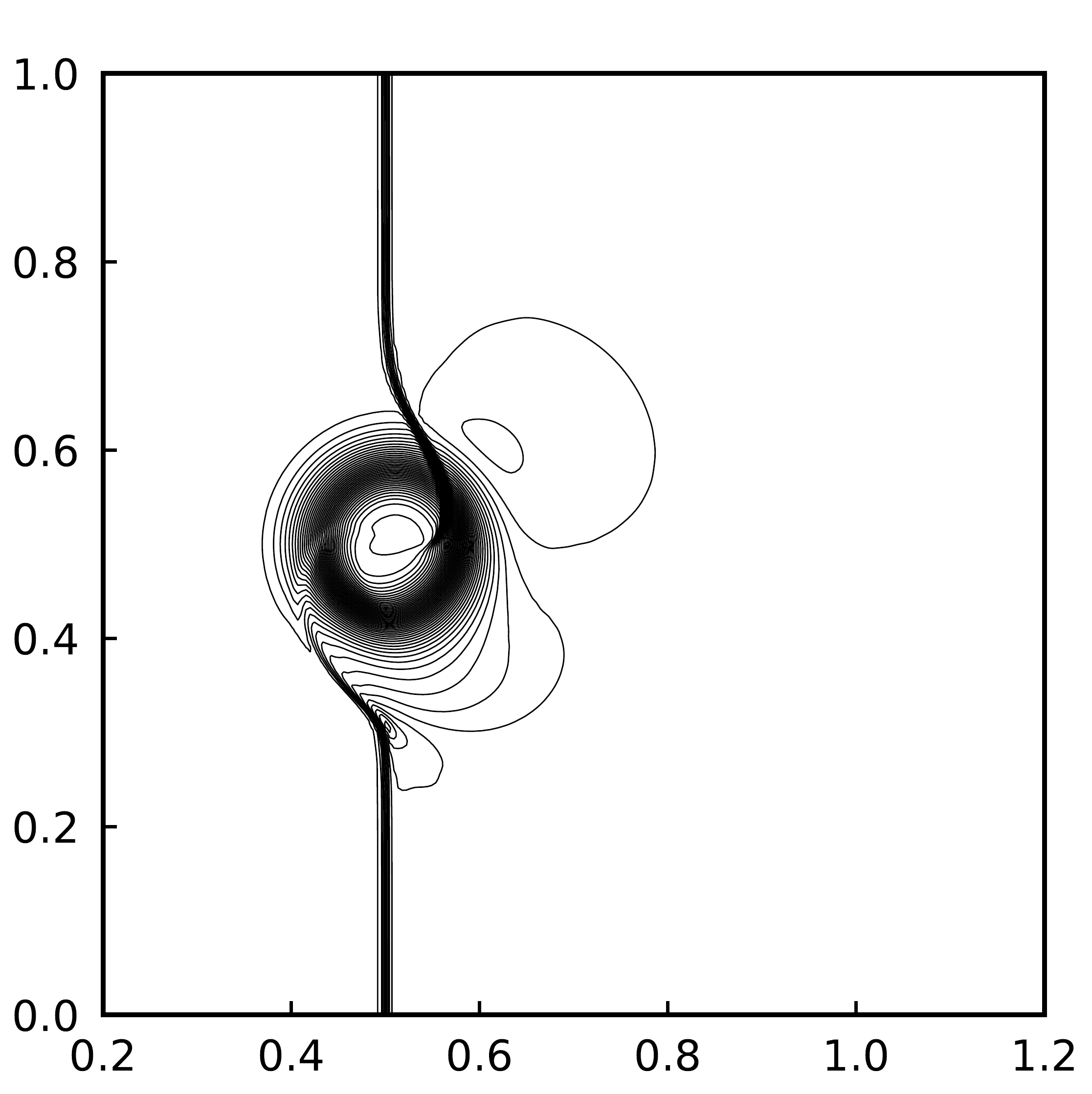}
		\caption{$t=0.2$, {\tt quasi-optimal}}
		\label{fig:ex3_3-g}
	\end{subfigure}
	\quad
	\begin{subfigure}[t][][t]{0.3\textwidth}
		\centering
		\includegraphics[width=\textwidth]{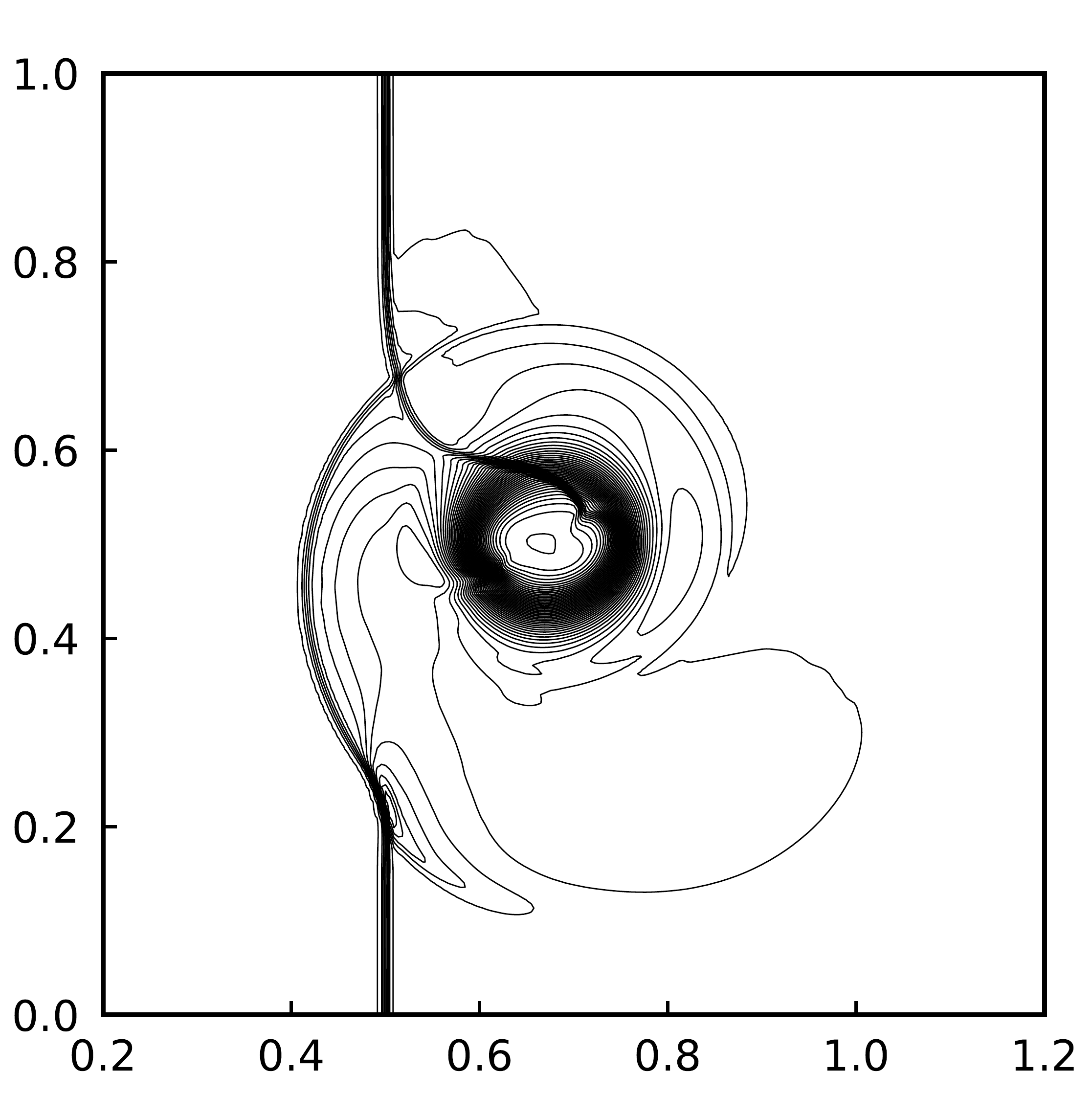}
		\caption{$t=0.35$, {\tt quasi-optimal}}
		\label{fig:ex3_3-h}
	\end{subfigure}
	\quad
	\begin{subfigure}[t][][t]{0.33\textwidth}
		\centering
		\includegraphics[width=\textwidth]{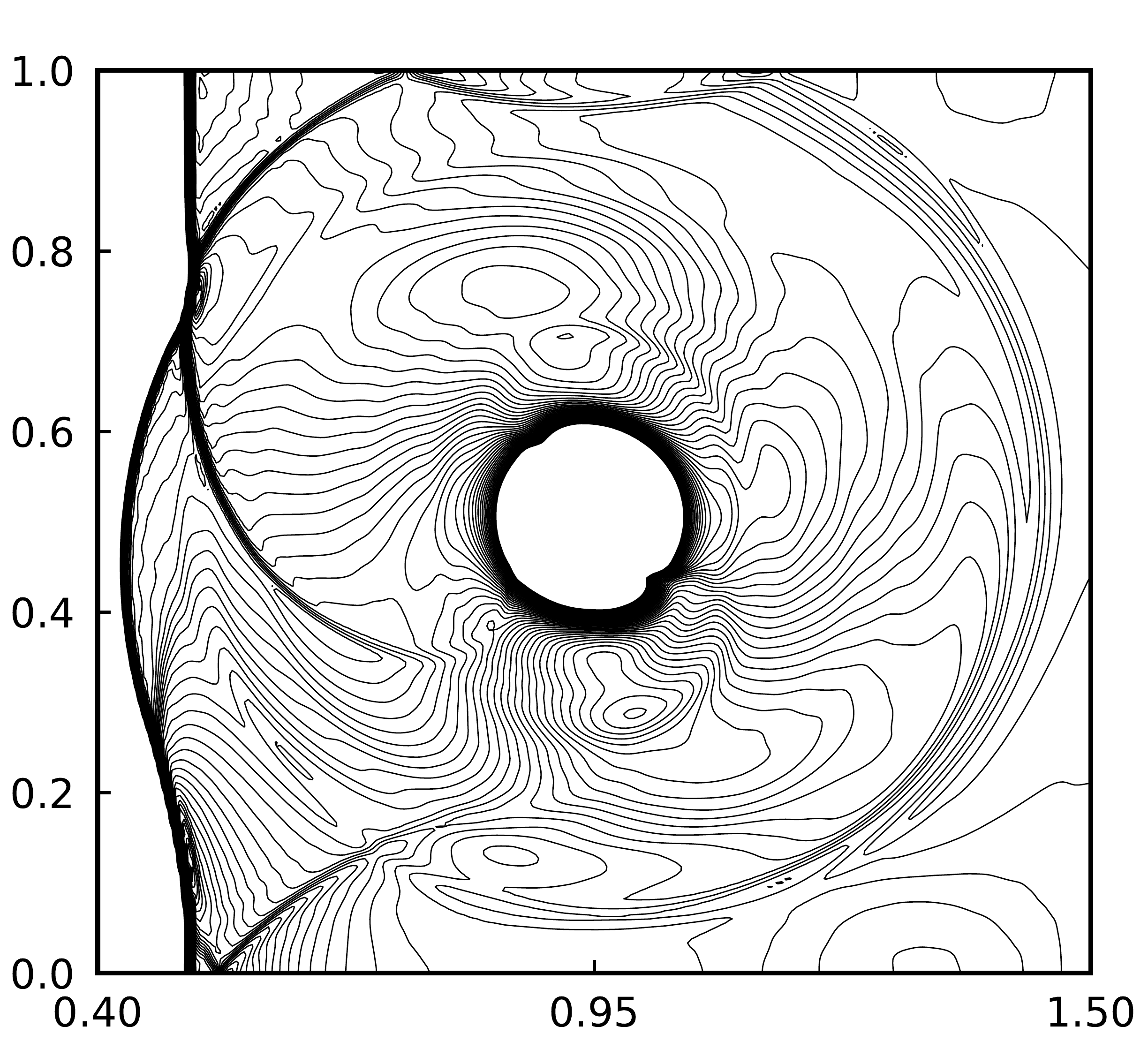}
		\caption{$t=0.6$, {\tt quasi-optimal}}
		\label{fig:ex3_3-l}
	\end{subfigure}
	\caption{
		Same as \Cref{fig:ex3_2} except for $\mathbb P^4$-based BP DG methods, which are designed with three different CADs. 
	}
	\label{fig:ex3_3}
\end{figure}

\begin{table}[htbp] 
	\centering
	\caption{CPU time in minutes of simulating Example 3 up to $t=0.6$.}
	\label{tab:ex3Euler}
	\setlength{\tabcolsep}{2mm}{
		\begin{tabular}{lccc}
			\toprule[1.5pt]
			
			& {\tt optimal} approach & {\tt quasi-optimal} approach & {\tt classic} approach \\
			
			\midrule[1.5pt]
			 $\mathbb{P}^2$     & 169.29 & 169.45 &  201.45 \\  
			 $\mathbb{P}^4$       & 1299.55 & 1327.45 & 1772.26 \\
			
			\bottomrule[1.5pt]
		\end{tabular}
	}
\end{table}

\begin{remark}[Simplified BP limiter for Euler equations]\label{rem:SBPEuler}
	As discussed in \Cref{rem:SBP} for the scalar conservation laws, in this example we use the simplified BP limiter, which modifies the DG polynomial $p_{ij}(x,y)=:(\rho_{ij} (x,y),{\bm m}_{ij}(x,y),E_{ij} (x,y))$ to $\tilde p_{ij}(x,y)$ as follows. 
	\begin{itemize}
		\item First, modify the density to enforce its positivity via 
		$$
		\widehat \rho_{ij} (x,y) = \theta_1 ( \rho_{ij} (x,y) - \bar \rho_{ij}  ) +  \bar \rho_{ij}  \quad \mbox{with} \quad 
		\theta_1 := \min\left\{  \left|
		\frac{ \bar \rho_{ij} - \epsilon_1 }{ \bar \rho_{ij} - \rho_{\min,ij} } \right|, 1
		\right\},
		$$
		where $\epsilon_1$ is a small positive number introduced for avoiding the effect of round-off error and can be taken as $\epsilon_1=\min\{ 10^{-13}, \bar \rho_{ij} \}$. 
		$\rho_{\min,ij} = \min\{ \rho( u_{i\pm\frac12,j}^\mp ), \rho( u_{i,j\pm\frac12}^\mp ), \rho( \Pi_{ij} ) \}$. 
		Here $\Pi_{ij}$ is defined by \eqref{eq:478a3} with $\Wmu$ taken as $\Wmu_\star$ for the {\tt optimal} approach, taken as $\Wmu_\star^{\tt Q}$ for the {\tt quasi-optimal} approach, and 
		taken as $\omega_1^{\tt GL}$ for the {\tt classic} approach, respectively.
		\item Modify $\widehat p_{ij}(x,y)=:(\widehat \rho_{ij} (x,y),{\bm m}_{ij}(x,y),E_{ij} (x,y))$ to $\tilde p_{ij}(x,y)$ to enforce the positivity of $\rho e$ by 
		$$
		\tilde p_{ij}(x,y) = \theta_2 ( \widehat p_{ij} (x,y) - \bar u_{ij}  ) +  \bar u_{ij}  \quad \mbox{with} \quad 
		\theta_2 := \min\left\{  \left|
		\frac{ \rho e(\bar u_{ij}) - \epsilon_2 }{ \rho e( \bar u_{ij} ) - \widehat{ \rho e}_{\min,ij} } \right|, 1
		\right\},
		$$	
		where $\epsilon_2$ is a small positive number introduced for avoiding the effect of round-off error and can be taken as $\epsilon_1=\min\{ 10^{-13}, \rho e( \bar u_{ij} ) \}$, and 
		$$
		\widehat{ \rho e}_{\min,ij} = \min \{  \rho e( \widehat u_{i\pm\frac12,j}^\mp ), \rho e( \widehat u_{i,j\pm\frac12}^\mp ), \rho e( \widehat \Pi_{ij} )  \},
		$$
		which are computed via \eqref{eq:u1}, \eqref{eq:u2}, and \eqref{eq:478a3} but based on $\tilde p_{ij}(x,y)$. 
	\end{itemize} 	
\end{remark}

\subsection{Example 4: 2D Riemann problem of relativistic hydrodynamics}
In this example, we simulate a 2D Riemann problem \cite{WuTang2015} with large Lorentz factor, low density, and low pressure for special relativistic hydrodynamics, 
whose governing equations take the form of \eqref{2DCL} with 
\begin{equation} \label{Eq:RHD}
	u=\begin{pmatrix}
		D \\
		m_1 \\
		m_2 \\
		E
	\end{pmatrix}, \qquad {f_1}(u)=\begin{pmatrix}
		Dv_1 \\
		m_1 {v_1}+P \\
		m_2 v_1 \\
		m_1
	\end{pmatrix}, \qquad {f_2}(u)=\begin{pmatrix}
		D v_2 \\
		m_1 v_2 \\
		m_2 {v_2}+P \\
		m_2
	\end{pmatrix}.
\end{equation}
Here $D=\rho W$ is the density with $\rho$ being the rest-mass density, $(m_1,m_2)=\rho h W^2 (v_1,v_2)$ denotes the momentum vector with $(v_1,v_2)$ being the velocity, $E=\rho h W^2 -P$ is the total energy with $P$ being the pressure, 
$W = (1-v_1^2-v_2^2)^{\frac12}$ denotes the Lorentz factor, and $h=1+\frac{\gamma P }{(\gamma -1)\rho}$ is the specific enthalpy with $\gamma$ being the adiabatic index. Normalized units are used here such that the speed of light equals one.  
For the relativistic hydrodynamic system \eqref{Eq:RHD}, 
the density and pressure should be positive, and the magnitude of velocity should be smaller than the speed of light. As proved in \cite{WuTang2015}, these constraints form an invariant region which can be equivalently expressed as 
$$
G= \left\{ u=(D,m_1,m_2,E)^\top: D(u)>0,~ q(u) := E-\sqrt{D^2+m_1^2+m_2^2}>0 \right \}, 
$$ 
where $q(u)$ is a concave function of $u$ so that $G$ is a convex set \cite{WuTang2015}.

The initial conditions of this 2D Riemann problem \cite{WuTang2015} are given by 
\begin{align*}
	(\rho, u, v, p) =
	\begin{cases}
		(0.1,0,0,20), & x>0.5, y>0.5,  \\
		(0.00414329639576,0.9946418833556542,0,0.05), & x<0.5, y>0.5,  \\
		(0.01,0,0,0.05), & x<0.5, y<0.5, \\
		(0.00414329639576,0,0.9946418833556542,0.05), & x>0.5, y<0.5. 
	\end{cases}
\end{align*}
The adiabatic index $\gamma$ is taken as $5/3$, and outflow conditions are specified on the boundary of the computational domain $[0,1]^2$. 
Figure \ref{fig:Riemann2} displays the contours of density logarithm $\ln \rho$, obtained by using the BP DG methods designed with three different CADs, on the mesh of $400\times 400$ uniform cells at $t=0.4$. 
It is seen that two moving shocks and two stationary contact discontinuities interact each other, forming a mushroom-like structure expanding to the left-bottom region. 
The results of three different approaches are consistent and 
agree well with those computed in \cite{WuTang2015}. 
Table \ref{tab:Riemann2} presents the CPU time in minutes for optimal, quasi-optimal and classic approaches, demonstrating the efficiency of using optimal and quasi-optimal approaches with larger time steps. 
We also observe that the maximum wave speeds in $x$- and $y$-directions are very close so that the quasi-optimal CAD is very close to OCAD in the simulation, since the exact solution of this problem is symmetric with respect to $y=x$.


\begin{figure}[htbp]
	\centering
	\begin{subfigure}[t][][t]{0.31\textwidth}
		\centering
		\includegraphics[width=\textwidth]{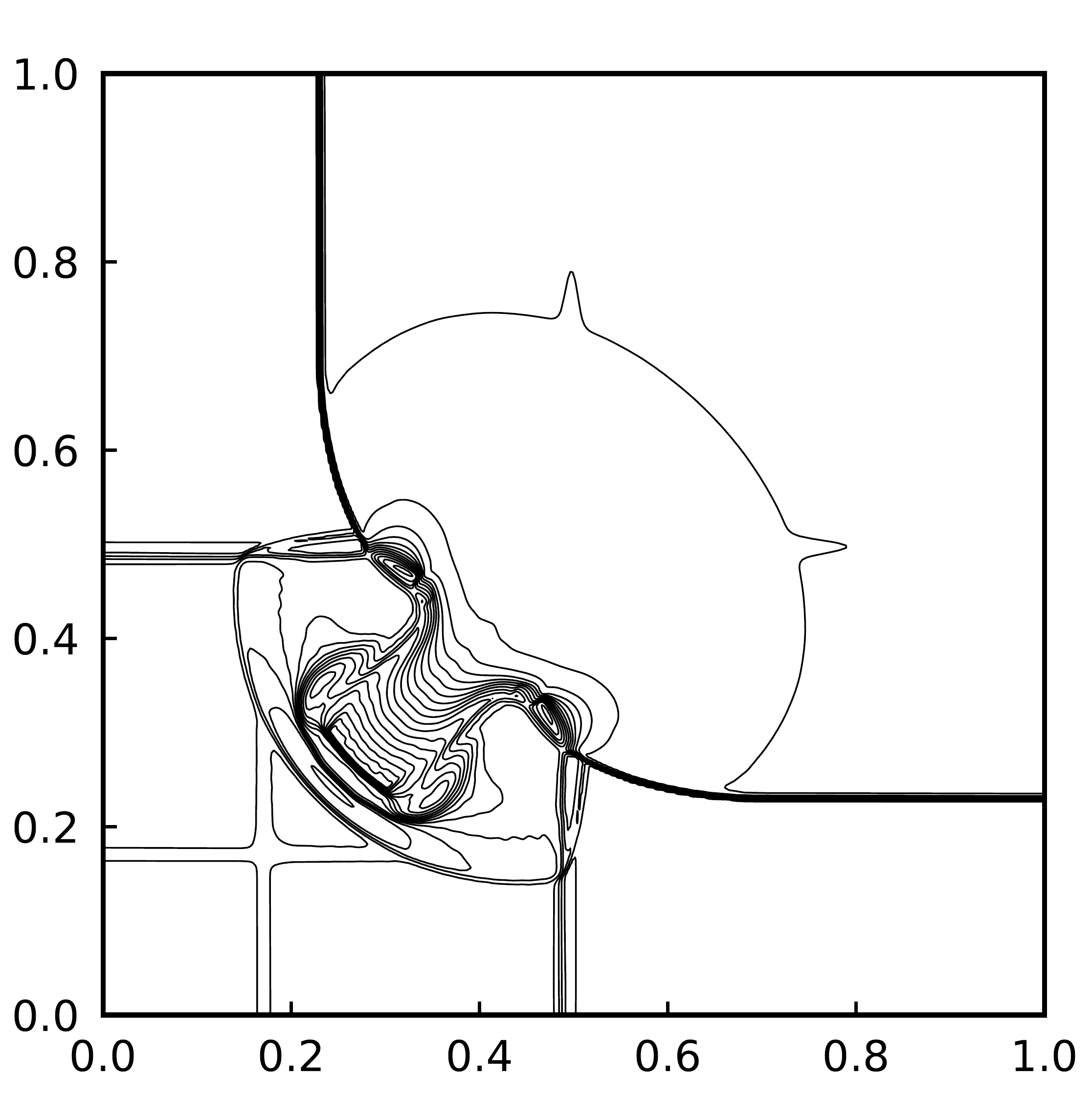}
		\caption{$\mathbb{P}^2$-based: {\tt optimal}}
		\label{fig:Riemann2-a}
	\end{subfigure}
	\quad
	\begin{subfigure}[t][][t]{0.31\textwidth}
		\centering
		\includegraphics[width=\textwidth]{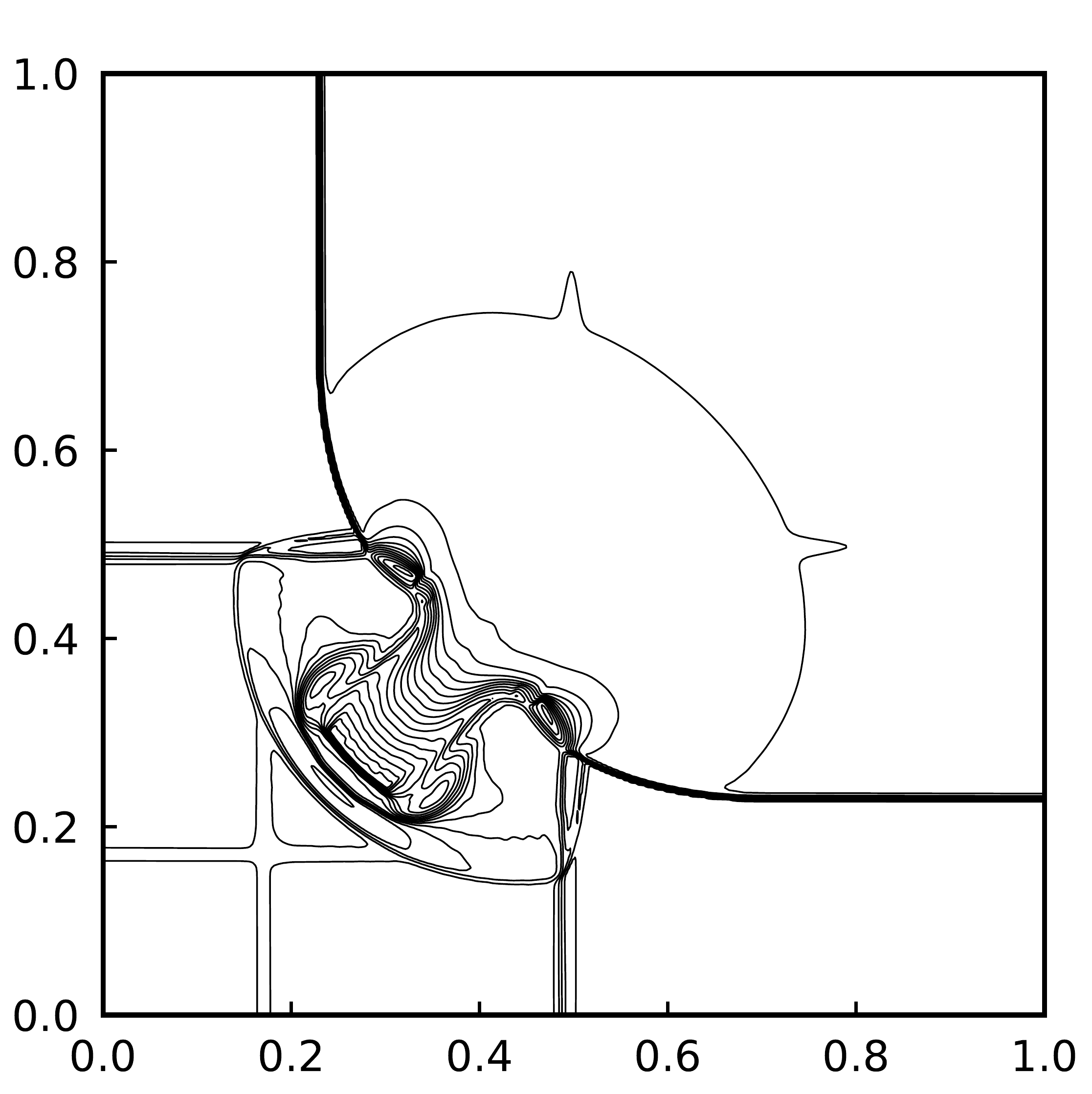}
		\caption{$\mathbb{P}^2$-based: {\tt quasi-optimal}}
		\label{fig:Riemann2-b}
	\end{subfigure}
	\quad
	\begin{subfigure}[t][][t]{0.31\textwidth}
		\centering
		\includegraphics[width=\textwidth]{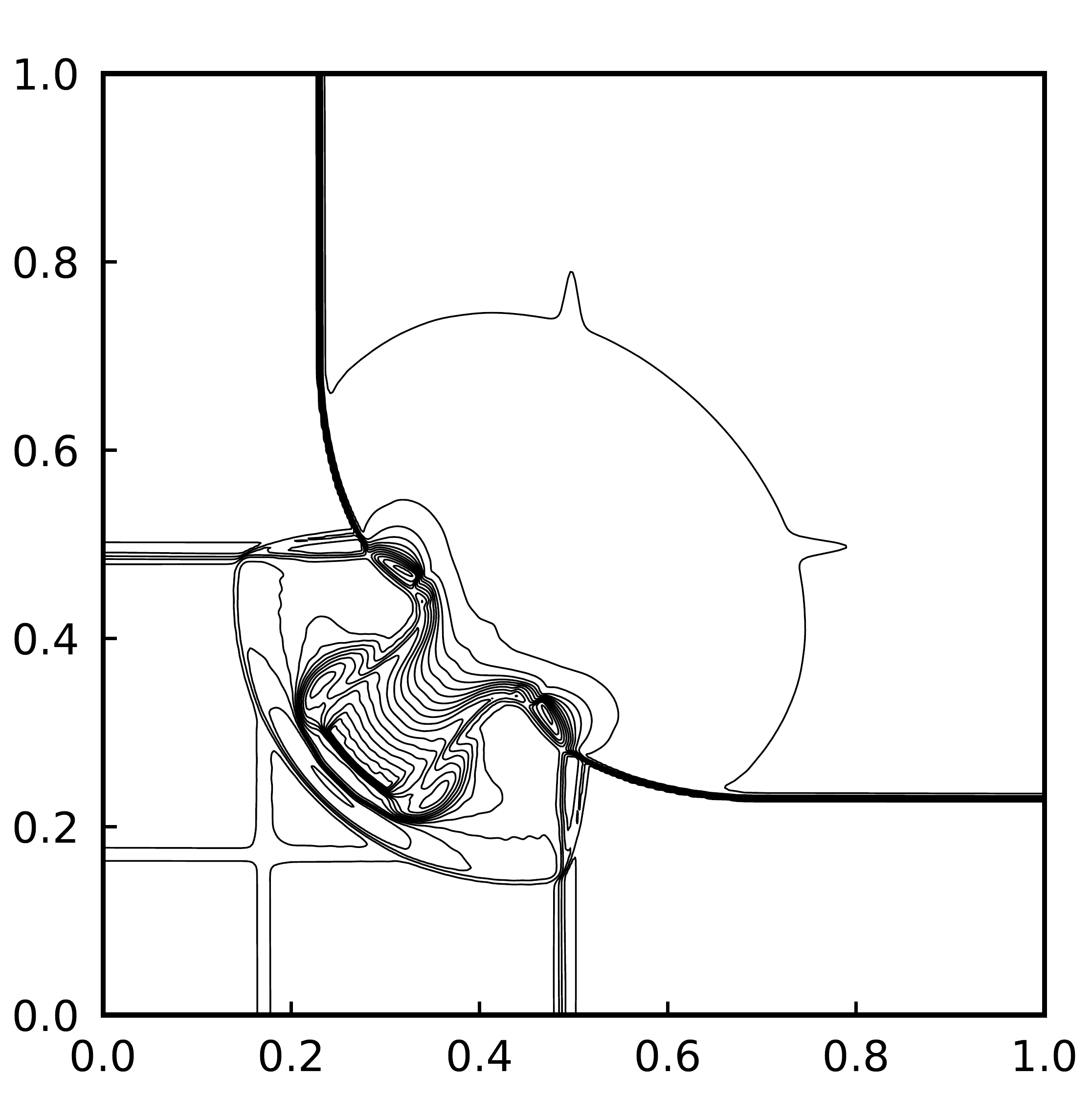}
		\caption{$\mathbb{P}^2$-based: {\tt classic}}
		\label{fig:Riemann2-c}
	\end{subfigure}
	
	\begin{subfigure}[t][][t]{0.31\textwidth}
		\centering
		\includegraphics[width=\textwidth]{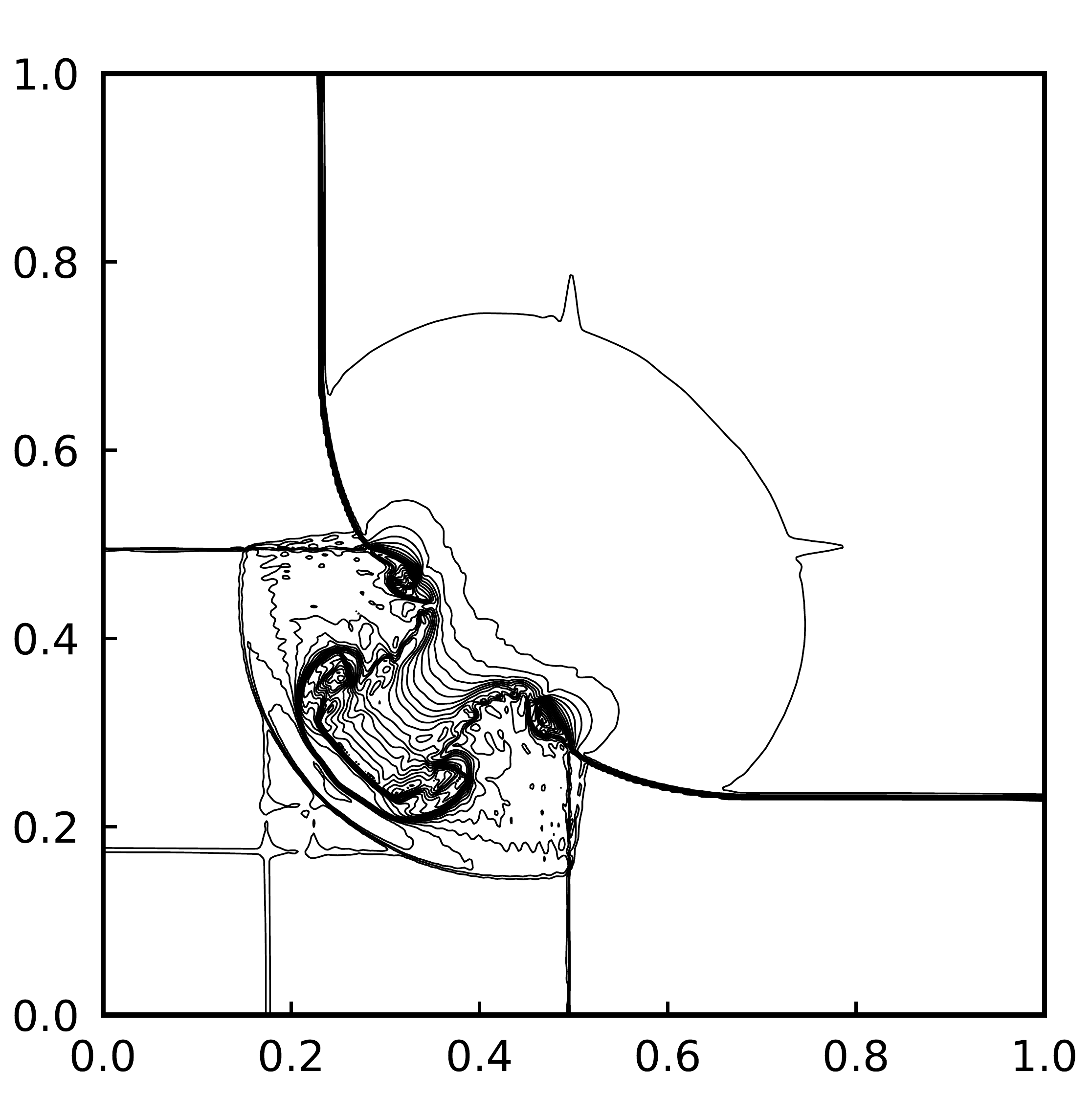}
		\caption{$\mathbb{P}^4$-based: {\tt optimal}}
		\label{fig:Riemann2-d}
	\end{subfigure}
	\quad
	\begin{subfigure}[t][][t]{0.31\textwidth}
		\centering
		\includegraphics[width=\textwidth]{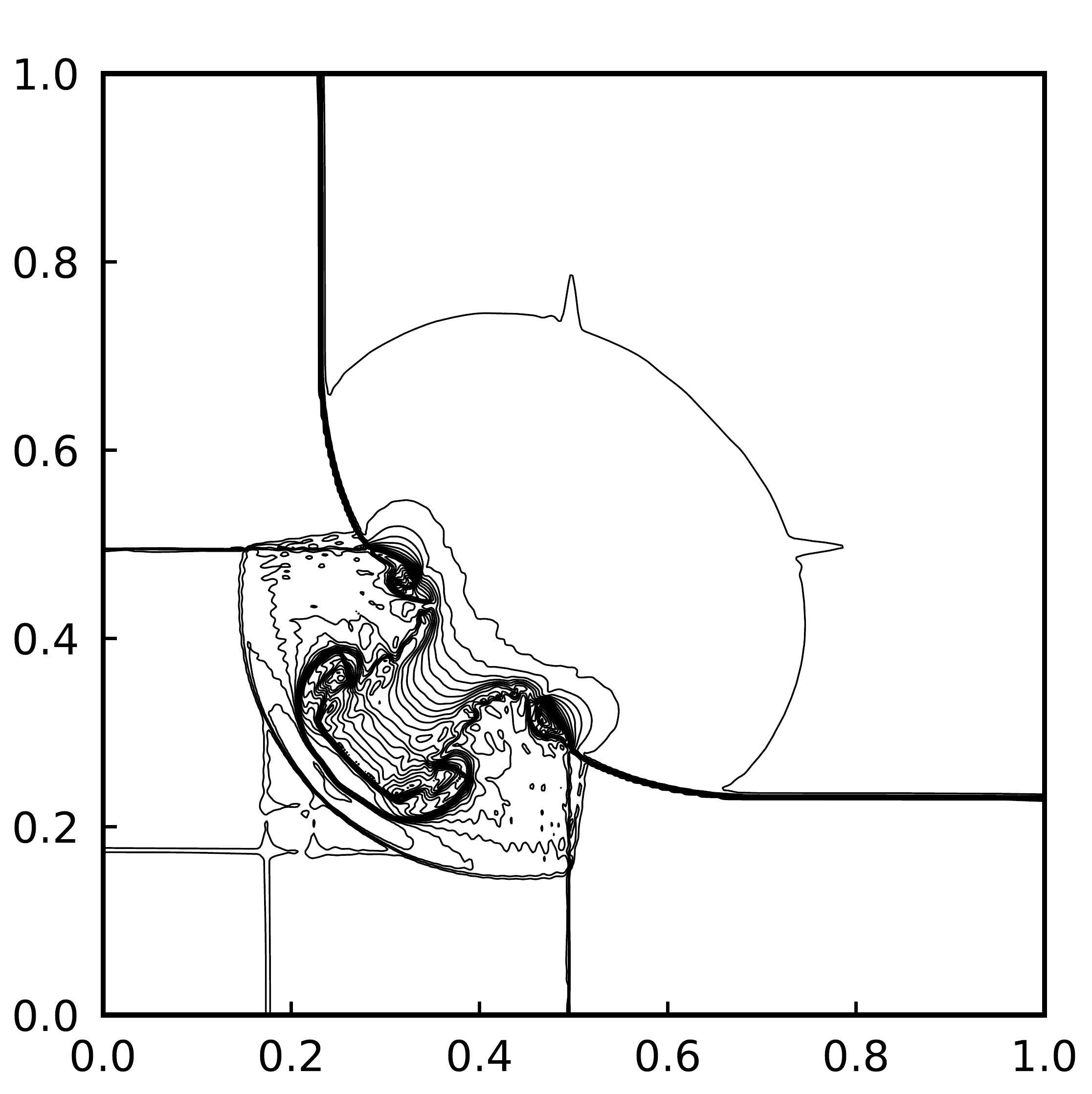}
		\caption{$\mathbb{P}^4$-based: {\tt quasi-optimal}}
		\label{fig:Riemann2-e}
	\end{subfigure}
	\quad
	\begin{subfigure}[t][][t]{0.31\textwidth}
		\centering
		\includegraphics[width=\textwidth]{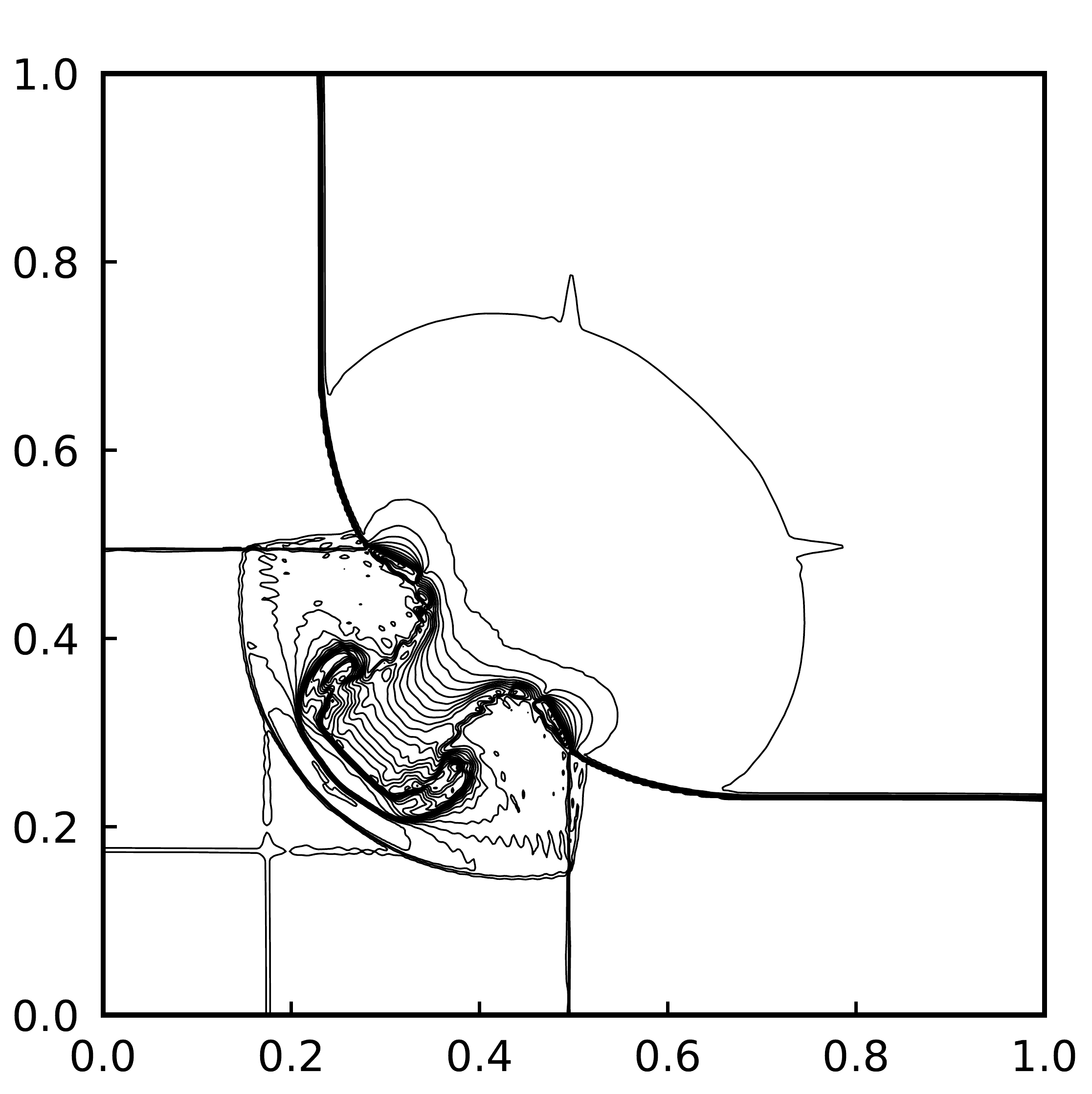}
		\caption{$\mathbb{P}^4$-based: {\tt classic}}
		\label{fig:Riemann2-f}
	\end{subfigure}
	\caption{Example 4: The density logarithm $\ln \rho$ with 25 equally spaced contour lines  from $-9.5$ to $-2.38$ at $t=0.4$.
	}
	\label{fig:Riemann2}
\end{figure}

\begin{table}[htbp] 
	\centering
	\caption{ CPU time in minutes of simulating Example 4 up to $t=0.4$.}
	\label{tab:Riemann2}
	\setlength{\tabcolsep}{2mm}{
		\begin{tabular}{lccc}
			\toprule[1.5pt]
			
			& {\tt optimal} approach & {\tt quasi-optimal} approach & {\tt classic} approach \\
			
			\midrule[1.5pt] 
			$\mathbb{P}^2$       & 292.28 & 298.75 & 358.83 \\
			$\mathbb{P}^4$       & 1623.67& 1597.97& 2103.50\\
			\bottomrule[1.5pt]
		\end{tabular}
	}
\end{table}

\subsection{Example 5: Axisymmetric jet of relativistic hydrodynamics}

In the last example, we simulate a challenging astrophysical jet problem \cite{zhang2006ram,WuTang2015} by solving the 
axisymmetric relativistic hydrodynamic equations, 
which can be written in the cylindrical coordinates $(r,z)$ as 
\begin{equation}\label{2DjetCL}
	\frac{\partial }{\partial t} u+ \frac{\partial }{\partial r} f_1(u) + \frac{\partial }{\partial z} f_2(u)  =  s(u) 
\end{equation}
with $u$, $f_1(u)$, and $f_2(u)$ defined as \eqref{Eq:RHD}, and
$$
s(u)=-\frac{1}{r} ( D v_1, m_1 v_1, m_2 v_1, m_1 ). 
$$
The adiabatic index is taken as $5/3$. 
The computational domain is set as $[0,15]\times[0,45]$, which is divided into $240\times720$ uniform cells. 
Initially, the domain is full of the static uniform medium with 
$$
(\rho,~ v_1,~ v_2,~ P) = (0.01,~ 0.0,~ 0.0,~ 0.000170304823218172071).
$$ 
A high-speed relativistic jet 
with state 
$$(\rho_b,~ v_{1,b},~ v_{2,b},~ P_b) = (0.01,~ 0.0,~ 0.99,~ 0.000170304823218172071)$$
is injected in $z$-direction through nozzle $(r\leq1)$ of the bottom boundary $(z=0)$. 
In other words, the fixed inflow condition $(\rho_b,v_{1,b},v_{2,b},P_b)$ is applied on $\{r\leq 1, z=0 \}$ of the bottom boundary. 
The symmetrical condition is specified on the left boundary $r=0$, outflow conditions are applied on other boundaries. 
For this jet, the classical Mach number is $6$, and the corresponding relativistic Mach number is about $41.95$. 

\begin{table}[htbp] 
	\centering
	\caption{CPU time in hours of simulating Example 5 up to $t=100$.}
	\label{tab:RHDJet1}
	\setlength{\tabcolsep}{2mm}{
		\begin{tabular}{lccc}
			\toprule[1.5pt]
			
			& {\tt optimal} approach & {\tt quasi-optimal} approach & {\tt classic} approach \\
			
			\midrule[1.5pt] 
			$\mathbb{P}^2$   & 45.8 & 50.12 &  59.72 \\   
			$\mathbb{P}^4$  & 297.26 &  334.59 &  414.19 \\ 
			
			\bottomrule[1.5pt]
		\end{tabular}
	}
\end{table}

The high speed and low pressure make this test very challenging. 
Without using the BP technique, the simulation with a DG code would break down quickly. 
Figure \ref{fig:RHDJet1} displays the schlieren images of rest-mass density logarithm $\ln \rho$ in the domain $[-15, 15]\times [0,45]$ by respectively using the {\tt optimal}, {\tt quasi-optimal}, and {\tt classic} approaches at $t=100$. 
The results demonstrate the excellent robustness for all the three approaches. 
One can see the turbulent structures are produced, and the jet dynamics and morphology agree well with those simulated in \cite{zhang2006ram,WuTang2015}. 
Table \ref{tab:RHDJet1} also displays the CPU time in this test, further confirming the notable advantage of the {\tt optimal} and {\tt quasi-optimal} approaches in efficiency.


\begin{figure}[htbp]
	\centering
	\begin{subfigure}[t][][t]{0.31\textwidth}
		\centering
		\includegraphics[width=\textwidth]{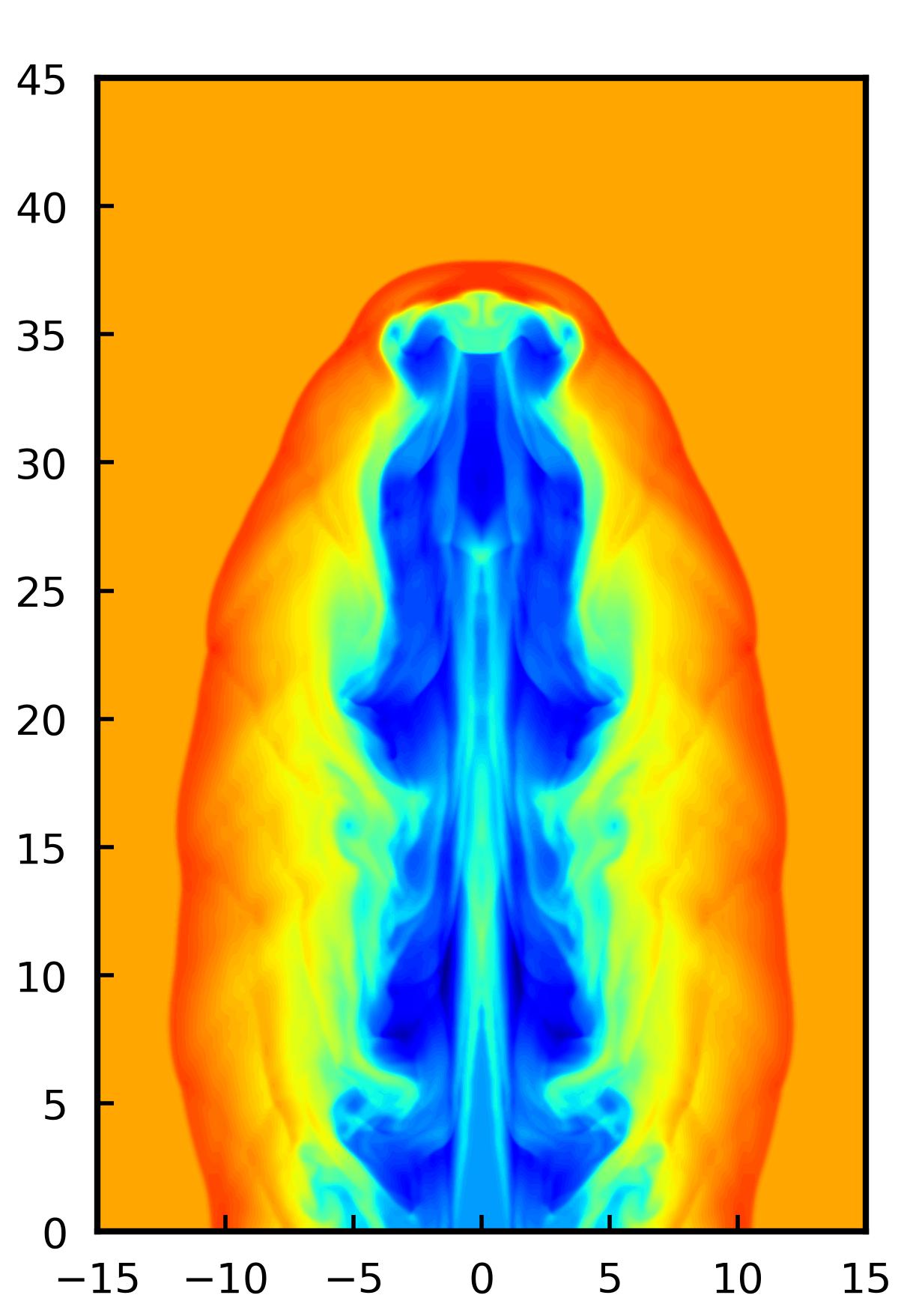}
		\caption{$\mathbb{P}^2$-based: {\tt optimal}}
		\label{fig:RHDJet1-a}
	\end{subfigure}
	\quad
	\begin{subfigure}[t][][t]{0.31\textwidth}
		\centering
		\includegraphics[width=\textwidth]{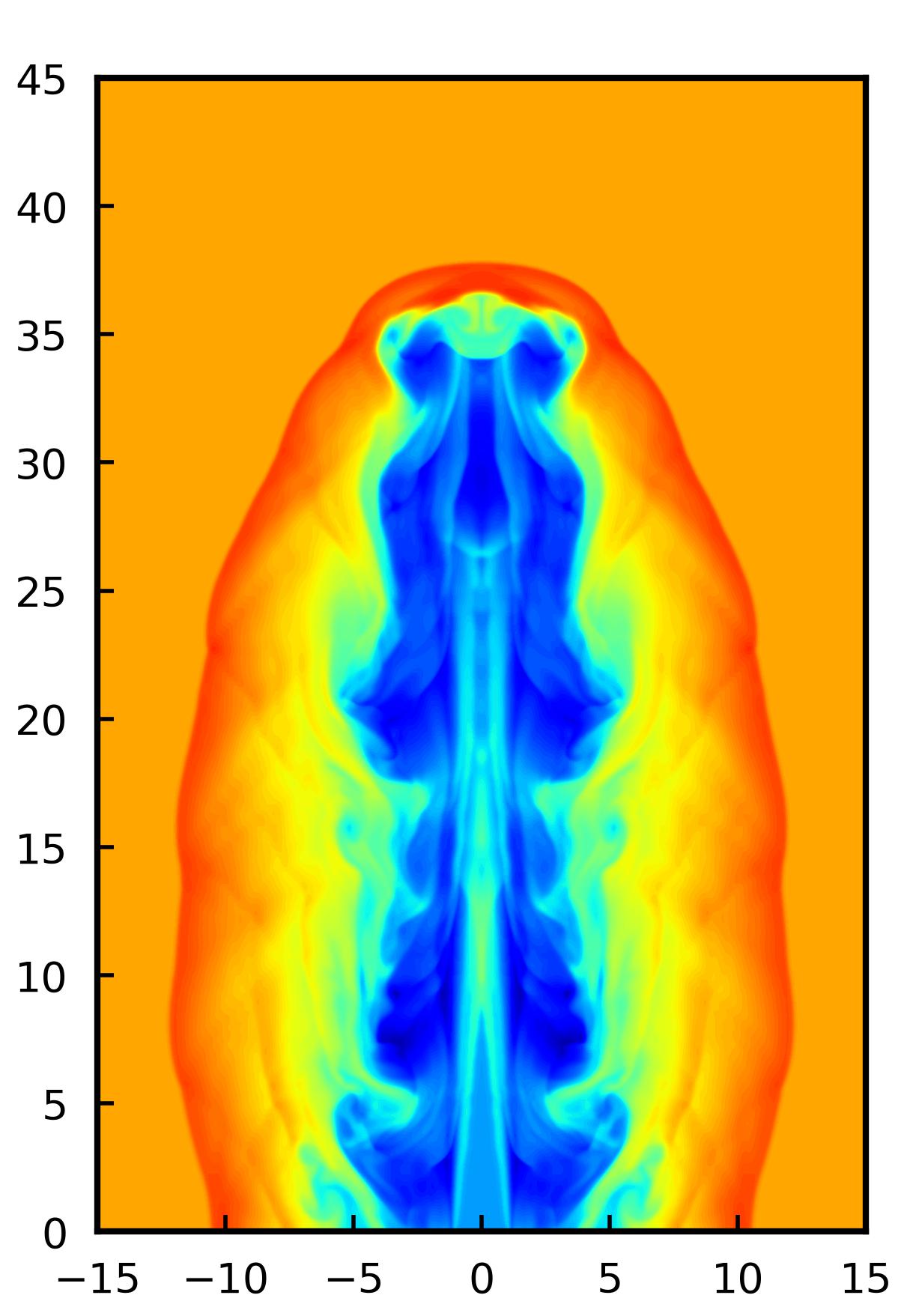}
		\caption{$\mathbb{P}^2$-based: {\tt quasi-optimal}}
		\label{fig:RHDJet1-b}
	\end{subfigure}
	\quad
	\begin{subfigure}[t][][t]{0.31\textwidth}
		\centering
		\includegraphics[width=\textwidth]{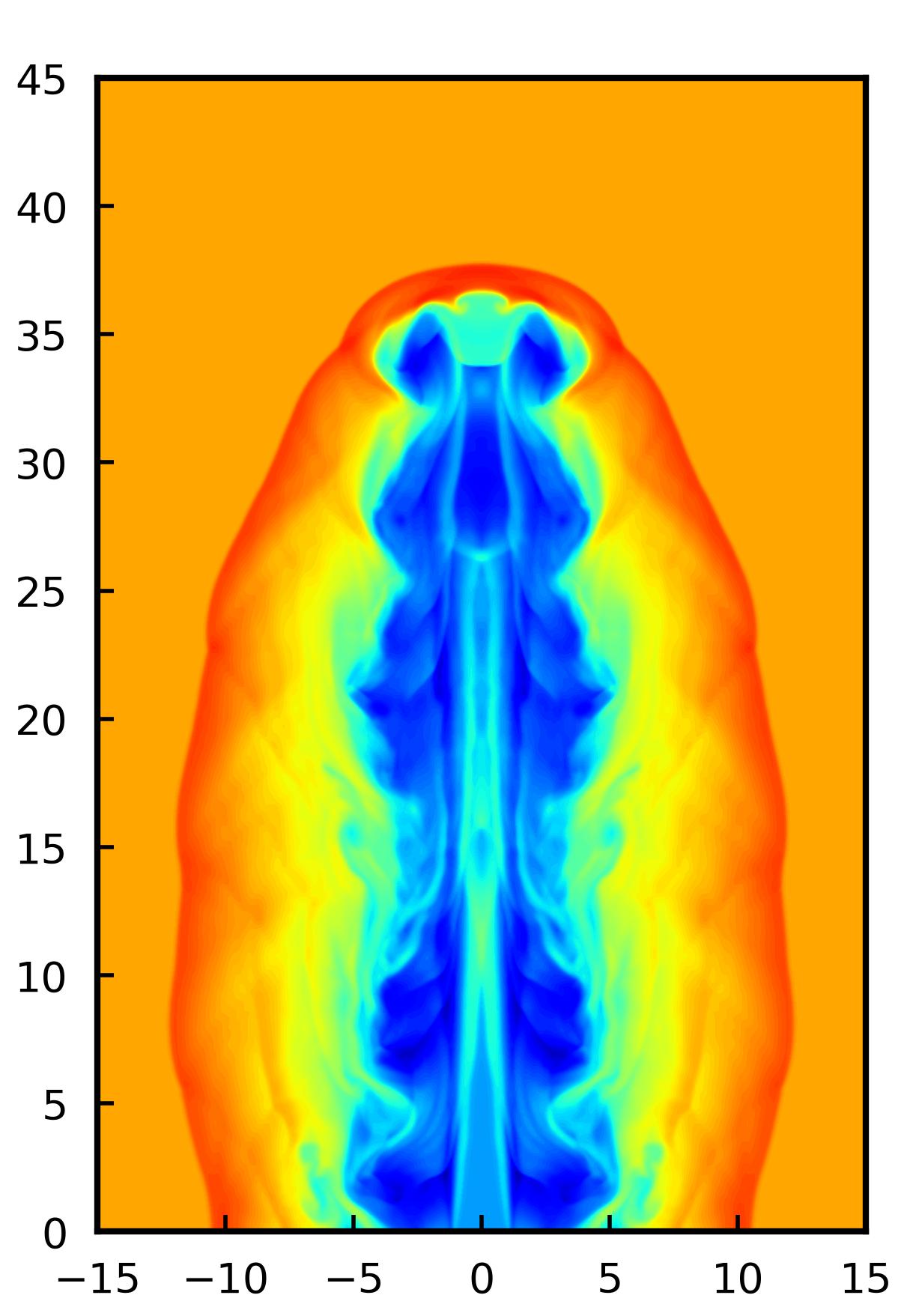}
		\caption{$\mathbb{P}^2$-based: {\tt classic}}
		\label{fig:RHDJet1-c}
	\end{subfigure}
	
	\begin{subfigure}[t][][t]{0.31\textwidth}
		\centering
		\includegraphics[width=\textwidth]{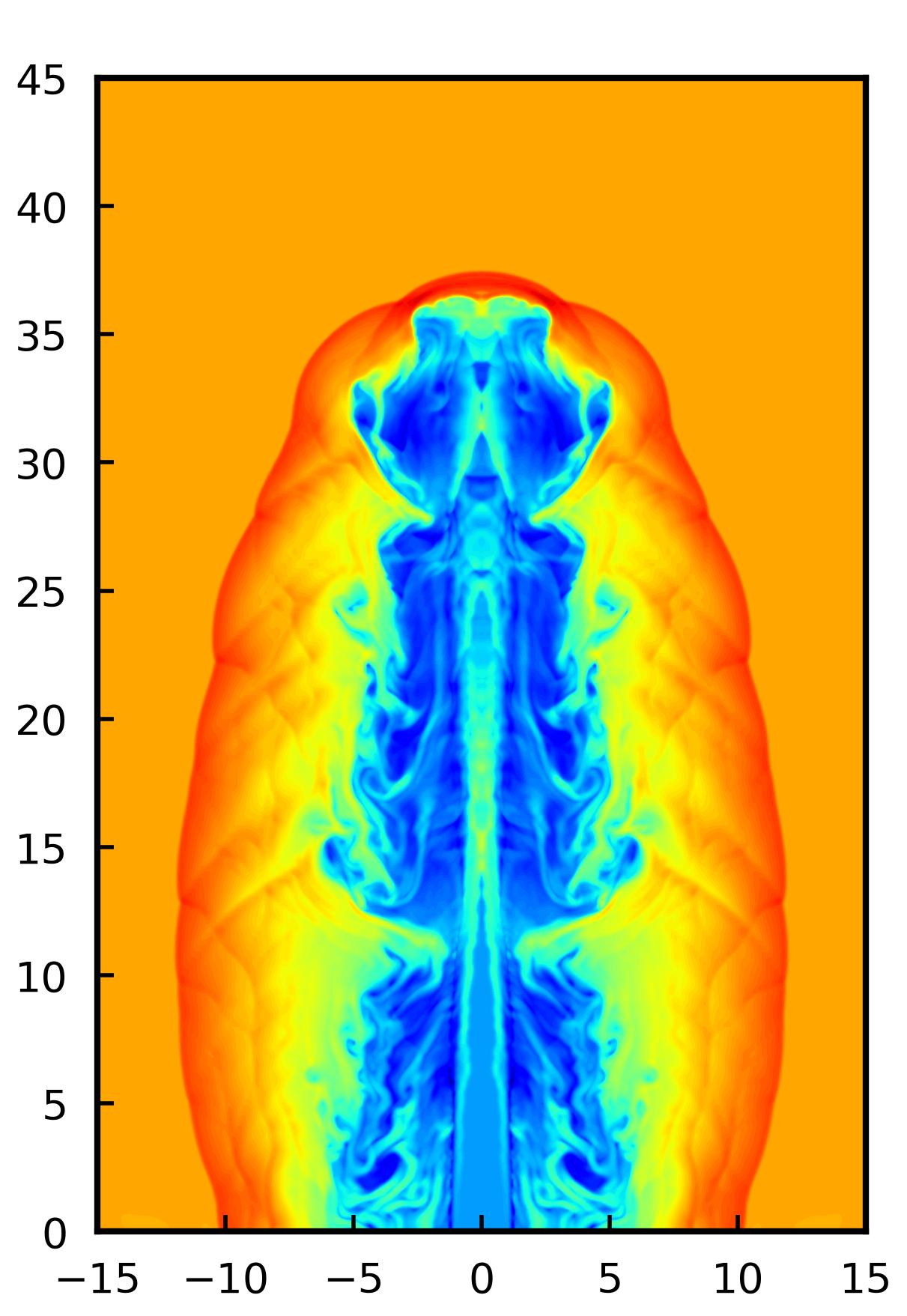}
		\caption{$\mathbb{P}^4$-based: {\tt optimal}}
		\label{fig:RHDJet1-d}
	\end{subfigure}
	\quad
	\begin{subfigure}[t][][t]{0.31\textwidth}
		\centering
		\includegraphics[width=\textwidth]{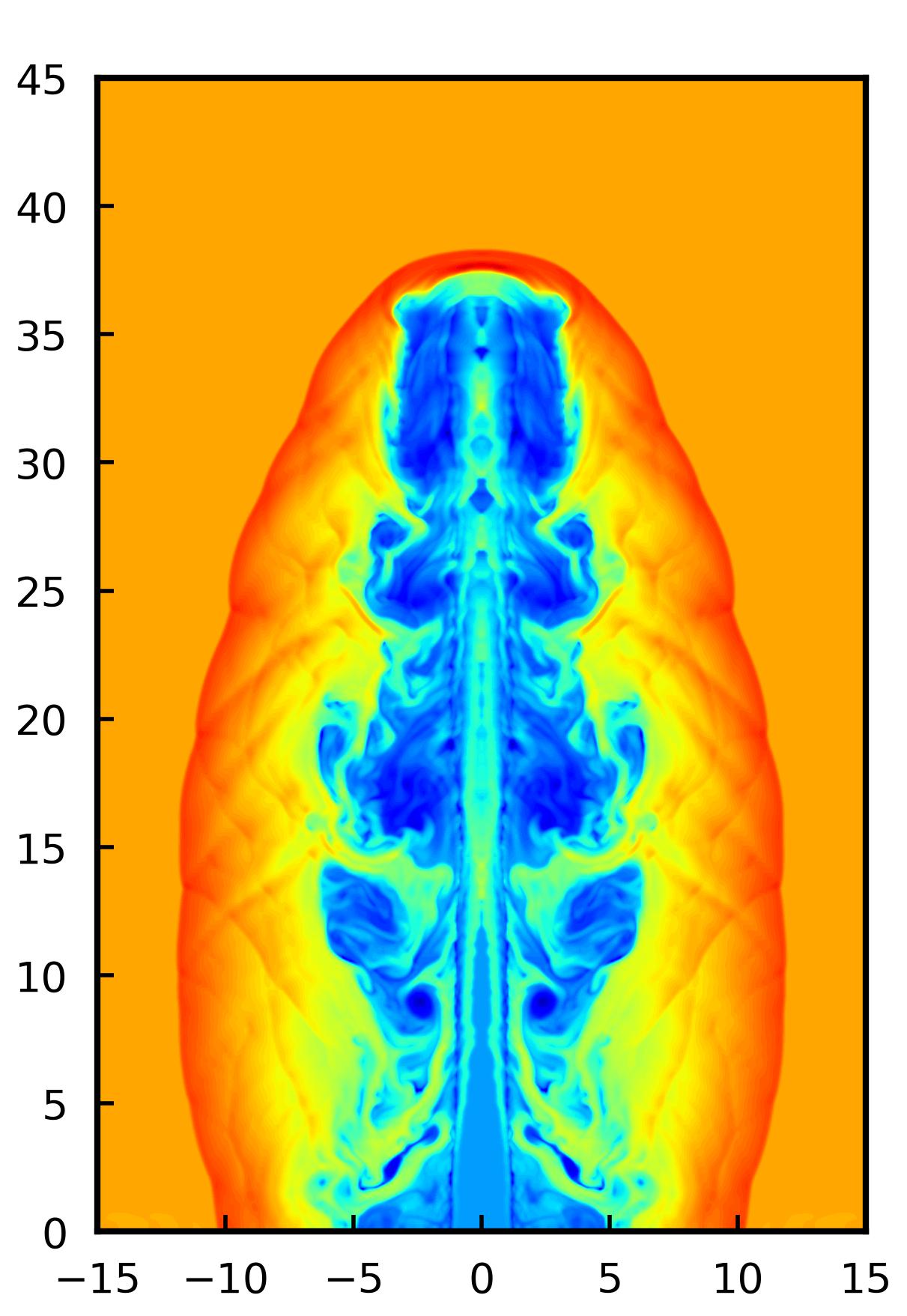}
		\caption{$\mathbb{P}^4$-based: quasi-optimal}
		\label{fig:RHDJet1-e}
	\end{subfigure}
	\quad
	\begin{subfigure}[t][][t]{0.31\textwidth}
		\centering
		\includegraphics[width=\textwidth]{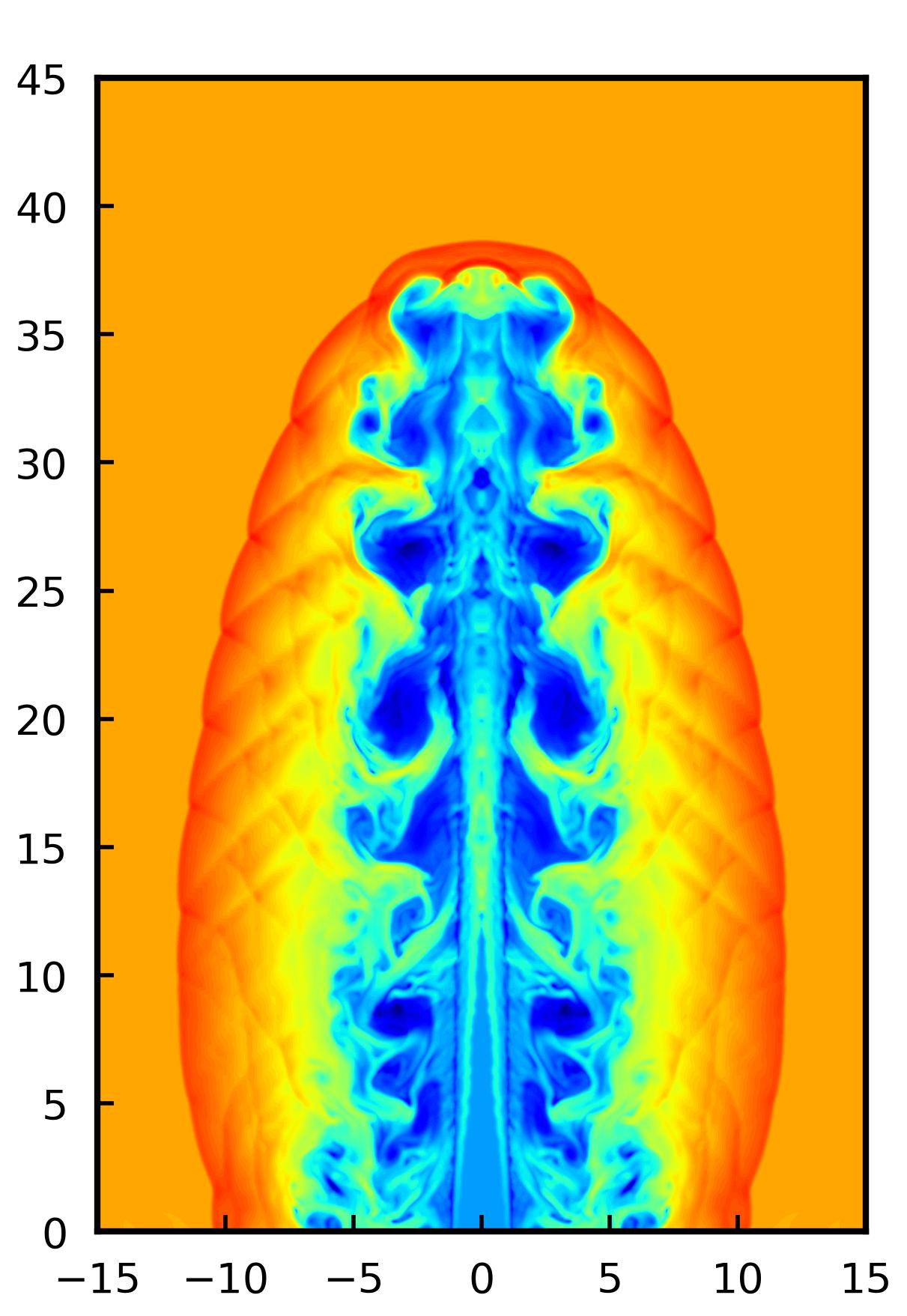}
		\caption{$\mathbb{P}^4$-based: {\tt classic}}
		\label{fig:RHDJet1-f}
	\end{subfigure}
	\caption{Example 5: The rest-mass density logarithm $\ln \rho$ at $t=100$. 
	}
	\label{fig:RHDJet1}
\end{figure}

\section{Conclusions}\label{sec:conclusion}

In this paper, we have presented the first systematic analysis of the 
OCAD problem for constructing efficient high-order BP numerical methods within Zhang--Shu framework. We have proved that the classic 1D CAD originally proposed by Zhang and Shu is optimal for general $\mathbb P^k$ spaces of an arbitrary $k \in \mathbb N_+$. 
We have also established the general theory for studying the 2D OCAD problem on Cartesian meshes. Based on the proposed theory, 
we have proved that the classic 2D CAD \eqref{eq:U2Dsplit} is optimal for general $\mathbb Q^k$ spaces of an arbitrary $k \in \mathbb N_+$.

Unfortunately, the classic CAD is not optimal for the widely used 2D $\mathbb P^k$ spaces. 
As the polynomial degree $k$ increases, seeking the genuine OCAD is more and more difficult. 
We have developed a systematic approach to 
find the genuinely optimal CADs for the 2D $\mathbb P^k$ spaces. 
We have derived the analytical formulas of OCADs for $\mathbb P^k$ spaces with $k \le 7$. 
A general algorithm has also been proposed to construct the OCADs for $\mathbb P^k$ spaces with $k\ge 8$. 
Based on some geometric insights, we have also proposed a more practical quasi-optimal CAD, which can be easily constructed via a convex combination of the OCADs in three special cases. We have demonstrated that our quasi-optimal CAD can achieve a near-optimal BP CFL condition, which is very close (at least 95\%) to the optimal one. 
The discovery of OCADs and quasi-optimal CADs is highly nontrivial yet meaningful, as it leads to an improvement of high-order BP schemes for a large class of hyperbolic or convection-dominated equations, at the little cost of only a slight and local modification to the implementation code. 
The remarkable advantages in efficiency 
have been confirmed by 
several numerical examples covering four hyperbolic partial differential equations.

The presented theory on OCAD is highly nontrivial and involves novel techniques from several branches of mathematics. For example, we have proved  
the existence of OCAD by using Carath\'eodory's theorem from {\em convex geometry}, and we have simplified the 2D OCAD problem to a symmetric OCAD problem based on {\em the invariant theory of symmetric group}. 
Most notably, we have discovered that the symmetric OCAD problem is closely related to 
polynomial optimization of a positive linear functional on the positive polynomial cone, by which we have established four useful criteria for examining the optimality of a feasible CAD. Some geometric insights have also been provided to interpret our critical findings. 
Our future work will include exploring OCADs and quasi-optimal CADs on unstructured meshes and 3D meshes.

\appendix

\section{Proof of \Cref{thm:3161}}\label{sec:A0}

The proof of \Cref{thm:3161} is given as follows.

\begin{proof}
	We only need to prove the result for $\theta \in [-1,0]$, as the conclusion for $\theta \in [0,1]$ then directly follows from \Cref{thm:sym_theta} and \Cref{cor:1575}. 	
	
	The proof for $\theta \in [-1,0]$ consists of the following three steps.  
	
	{\tt Step 1: Verify the feasibility condition (i) in \Cref{def:2D_FCAD}.} 
	Thanks to \Cref{lem:Gfs-invarint}, we only need to verify 
	that the symmetric CAD \eqref{eq:2879} is feasible 
	for the ${\mathscr G}_s$-invariant subspace $\mathbb{P}^4({\mathscr G}_{s}) = \mathbb{P}^5({\mathscr G}_{s})={\rm span} 
	\{1,x^2,y^2,x^4,x^2y^2,y^4\}$. 
	In other words, it suffices to verify the correctness of the following equations:
	\begin{align*}
		1 = & \frac{1}{1} \cdot \Wmu_\star (1+\theta) + \frac{1}{1} \cdot \Wmu_\star (1-\theta) + \omega_1 + \omega_2, \\
		\frac{1}{3} = & \frac{1}{1} \cdot \Wmu_\star (1+\theta) + \frac{1}{3} \cdot \Wmu_\star (1-\theta) + \omega_1 (x^{(1)})^2, \\
		\frac{1}{3} = & \frac{1}{3} \cdot \Wmu_\star (1+\theta) + \frac{1}{1} \cdot \Wmu_\star (1-\theta) + \omega_1 (y^{(1)})^2 + \omega_2 (y^{(2)})^2, \\
		\frac{1}{5} = & \frac{1}{1} \cdot \Wmu_\star (1+\theta) + \frac{1}{5} \cdot \Wmu_\star (1-\theta) + \omega_1 (x^{(1)})^4, \\
		\frac{1}{9} = & \frac{1}{3} \cdot \Wmu_\star (1+\theta) + \frac{1}{3} \cdot \Wmu_\star (1-\theta) + \omega_1 (x^{(1)})^2 (y^{(1)})^2, \\
		\frac{1}{5} = & \frac{1}{5} \cdot \Wmu_\star (1+\theta) + \frac{1}{1} \cdot \Wmu_\star (1-\theta) + \omega_1 (y^{(1)})^4 + \omega_2 (y^{(2)})^4,
	\end{align*}
	which are equivalent to
	\begin{subequations}\label{eq:2378}
		\begin{align}
			\label{eq:2378a} 1-2\Wmu_\star = & \omega_1 + \omega_2, \\
			\label{eq:2378b} \frac{1}{3} - \frac{2}{3} \Wmu_\star \theta - \frac{4}{3} \Wmu_\star = & \omega_1 (x^{(1)})^2, \\
			\label{eq:2378c} \frac{1}{3} + \frac{2}{3} \Wmu_\star \theta - \frac{4}{3} \Wmu_\star = & \omega_1 (y^{(1)})^2 + \omega_2 (y^{(2)})^2, \\
			\label{eq:2378d} \frac{1}{5} - \frac{4}{5} \Wmu_\star \theta - \frac{6}{5} \Wmu_\star = & \omega_1 (x^{(1)})^4, \\
			\label{eq:2378e} \frac{1}{9} - \frac{2}{3} \Wmu_\star = & \omega_1 (x^{(1)})^2 (y^{(1)})^2, \\
			\label{eq:2378f} \frac{1}{5} + \frac{4}{5} \Wmu_\star \theta - \frac{6}{5} \Wmu_\star = & \omega_1 (y^{(1)})^4 + \omega_2 (y^{(2)})^4.
		\end{align}
	\end{subequations}
	It is easy to check that the formulas of $\omega_1$, $\omega_2$, $(x^{(1)},y^{(1)})$ and $(x^{(2)},y^{(2)})$, given in \eqref{eq:2890b}--\eqref{eq:2890d}, always automatically satisfy the equations \eqref{eq:2378a}--\eqref{eq:2378e}.  
	The remaining task is to verify \eqref{eq:2378f}. 
	Submitting the formulas of $\omega_1$, $\omega_2$, $y^{(1)}$ and $y^{(2)}$ into \eqref{eq:2378f}, we can equivalently 
	reformulate \eqref{eq:2378f} into  
	\[
	\frac{1}{5} + \frac{4}{5} \Wmu_\star \theta - \frac{6}{5} \Wmu_\star =
	\frac{
		36\,\theta^3 \overline{\omega}_\star^3 
		-30\,\theta^2 \overline{\omega}_\star^3
		+17\,\theta^2 \overline{\omega}_\star^2
		-72\,\theta \overline{\omega}_\star^3
		+18\,\theta \overline{\omega}_\star^2
		+66\, \overline{\omega}_\star^3
		-59\, \overline{\omega}_\star^2
		+14\, \overline{\omega}_\star 
		-1
	}{
		45\,\theta^2 \overline{\omega}_\star^2
		+18\,\theta \overline{\omega}_\star^2
		+36\,\theta \overline{\omega}_\star 
		-63\, \overline{\omega}_\star^2
		+72 \overline{\omega}_\star
		- 9},
	\]
	which is equivalent to a cubic equation of $\overline{\omega}_\star$:
	\begin{equation}\label{eq:2978}
		12(1-\theta^2) \overline{\omega}_\star^3 + (26 \theta^2-50) \overline{\omega}_\star^2 + 14 \overline{\omega}_\star - 1 = 0.
	\end{equation}
	Now we only need to verify that $\overline{\omega}_\star$ defined \eqref{eq:2890a} satisfies \eqref{eq:2978}. 
	In fact, the following cubic equation 
	\begin{equation}\label{eq:2982}
		12(1-\theta^2) + (26 \theta^2-50) \mu + 14 \mu^2- \mu^3 = 0
	\end{equation}
	has the following positive zero
	$$ \mu = \frac{14}{3}+\frac{2}{3}\sqrt{78\,\theta^2+46} \cos \left[ \frac{1}{3} \arccos\frac{1476\,\theta^2-244}{(78\,\theta^2+46)^{\frac{3}{2}}} \right] = \frac1 {\overline{\omega}_\star }.$$
	This implies that $\overline{\omega}_\star$ defined \eqref{eq:2890a} satisfies \eqref{eq:2978}. 
	In summary, we have completed the verification of the feasibility condition (i) for the symmetric CAD \eqref{eq:2879}.
	
	{\tt Step 2: Verify the feasibility conditions (ii) and (iii) in \Cref{def:2D_FCAD}.} 
	The plots of $\Wmu_\star$, $x^{(1)}$, $x^{(2)}$, $y^{(1)}$, $y^{(2)}$, $\omega^{(1)}$, $\omega^{(2)}$ are displayed in  \Cref{fig:2417}, which clearly shows that 
	$$
	\Wmu_\star > 0, \quad \omega^{(1)}>0, \quad \omega^{(2)}>0, \quad (x^{(s)},y^{(s)}) \in [0,1]^2,~ s=1,2. 
	$$
	
	{\tt Step 3: Verify the optimality of the symmetric CAD \eqref{eq:2879}.} 
	Consider 
	\begin{equation}\label{eq:2407}
		p^\star(x,y) = q_\star^2(x,y) \quad \mbox{with} \quad q_\star(x,y) := \left(
		(y^{(2)})^2 - (y^{(1)})^2 \right)x^2 + (x^{(1)})^2 y^2 - (x^{(1)})^2(y^{(2)})^2, 
	\end{equation}
	which vanishes at all the internal nodes of symmetric CAD \eqref{eq:2879}. 
	Moreover, $q_\star(x,y) \in \mathbb{P}^2$. Thus, $p^\star(x,y)$ is an element of $(\mathbb{P}^2)^2$ and belongs to $\mathbb{P}^4_{+}$ and $\mathbb{P}^5_{+}$. 
	According to \Cref{lem:512}, 
	$p^\star(x,y)$ is the critical positive polynomial for both $\phi^\star (\theta,\mathbb{P}^4_{+})$ and 
	$\phi^\star (\theta,\mathbb{P}^5_{+})$.  
	Moreover, 
	Conjectures \ref{con:2070} and \ref{con:2071} hold true for all $\theta\in [-1,1]$ and $4\le k \le 5$ with
	$$
	\Wmu_\star(\theta,\mathbb{P}^k) = \phi^\star (\theta,\mathbb{P}^k_{+})= \phi^\star(\theta,(\mathbb{P}^2)^2), \qquad k=4,5.
	$$
	
	In summary, we have proved that \eqref{eq:2879} is an OCAD for $\mathbb{P}^4$ and $\mathbb{P}^5$ in the case of $\theta \in [-1,0]$. By \Cref{lem:1525,thm:sym_theta}, it can be proved that \eqref{eq:2879} is also an OCAD in case of $\theta \in [0,1]$.
	The proof is completed. 
\end{proof}

\section{Proof of \Cref{thm:3650}}\label{sec:A1}

The proof of \Cref{thm:3650} is given as follows.

\begin{proof}
	We only need to prove the result for $\theta \in [-1,0]$, as the conclusion for $\theta \in [0,1]$ then directly follows from \Cref{thm:sym_theta} and \Cref{cor:1575}.

	In the case of $\theta \in [-1,0]$,  we have $y_3=y_4=0$ and the formulas in \eqref{eq:3733}, \eqref{eq:3737} and \eqref{eq:3752} imply that the following moment equations are true:
	\begin{equation}\label{eq:3831}
		\begin{aligned}
			m_{02} & = \sum_{s = 1}^{4} \omega_s y_s^2 ,&
			m_{04} & = \sum_{s = 1}^{4} \omega_s y_s^4 ,&
			m_{06} & = \sum_{s = 1}^{4} \omega_s y_s^6 , \\
			m_{22} & = \sum_{s = 1}^{4} \omega_s x_s^2 y_s^2 ,&
			m_{24} & = \sum_{s = 1}^{4} \omega_s x_s^2 y_s^4 ,&
			m_{42} & = \sum_{s = 1}^{4} \omega_s x_s^4 y_s^2 .
		\end{aligned}
	\end{equation}
	The formulas in \eqref{eq:3791} and \eqref{eq:3796} lead to
	\[
	m_0 = \sum_{s = 3}^{4} \omega_s,       \quad
	m_2 = \sum_{s = 3}^{4} \omega_s x_s^2, \quad
	m_4 = \sum_{s = 3}^{4} \omega_s x_s^4, \quad
	m_6 = \sum_{s = 3}^{4} \omega_s x_s^6,
	\]
	implying that the following moment equations are true:
	\begin{equation}\label{eq:3849}
		m_{00} = \sum_{s = 1}^{4} \omega_s      , \quad
		m_{02} = \sum_{s = 1}^{4} \omega_s x_s^2, \quad
		m_{04} = \sum_{s = 1}^{4} \omega_s x_s^4, \quad
		m_{06} = \sum_{s = 1}^{4} \omega_s x_s^6.
	\end{equation}
	The moment equations in \eqref{eq:3831} and \eqref{eq:3849} imply the CAD \eqref{eq:3724} satisfies the feasibility condition (i). 
	The weights and coordinates are plotted in \Cref{fig:3862}, which demonstrates that $x_s, y_s \in [0,1]$ and $\omega_s \ge 0$.
	Thus, the feasibiilty condition (ii) and (iii) are satisfied and the CAD \eqref{eq:3724} is feasible.

	Next, we will show that the boundary weight $\overline{\omega}_\star$ given in \eqref{eq:3733} is equal to $\phi^\star(\theta,(\mathbb{P}^3)^2)$.
	According to \Cref{lem:2080}, $\phi^\star(\theta,(\mathbb{P}^3)^2)$ is the smallest real root of polynomial $\mathcal{F}_\theta(\phi)$ in the following form
	\[
	\mathcal{F}_\theta(\phi) = \frac{3^{14} \, 5^{8} \, 7^{2}}{4096} 
	\mathcal{F}^{(1)}_\theta(\phi) 
	\mathcal{F}^{(2)}_\theta(\phi) 
	\mathcal{F}^{(3)}_\theta(\phi) 
	\mathcal{F}^{(4)}_\theta(\phi)
	\]
	with
	\begin{equation}\label{eq:3324}
		\begin{aligned}
			\mathcal{F}^{(1)}_\theta(\phi) & = 6 \phi-1,\\\
			\mathcal{F}^{(2)}_\theta(\phi) & = (12\theta^2 - 12)\phi^3 + (50 - 26\theta^2)\phi^2 - 14\phi + 1 \\
			\mathcal{F}^{(3)}_\theta(\phi) & = (12\theta^3 - 48\theta^2 - 12\theta + 48)\phi^3 + (48\theta + 30\theta^2 - 102)\phi^2 + (- 6\theta+20)\phi - 1, \\
			\mathcal{F}^{(4)}_\theta(\phi) & = (12\theta^3+48\theta^2 - 12\theta - 48)\phi^3 + (48\theta - 30\theta^2 + 102)\phi^2 + (- 6\theta - 20)\phi + 1.
		\end{aligned}
	\end{equation}
	
	\begin{figure}
		\centering
		\includegraphics[width = 0.6\textwidth]{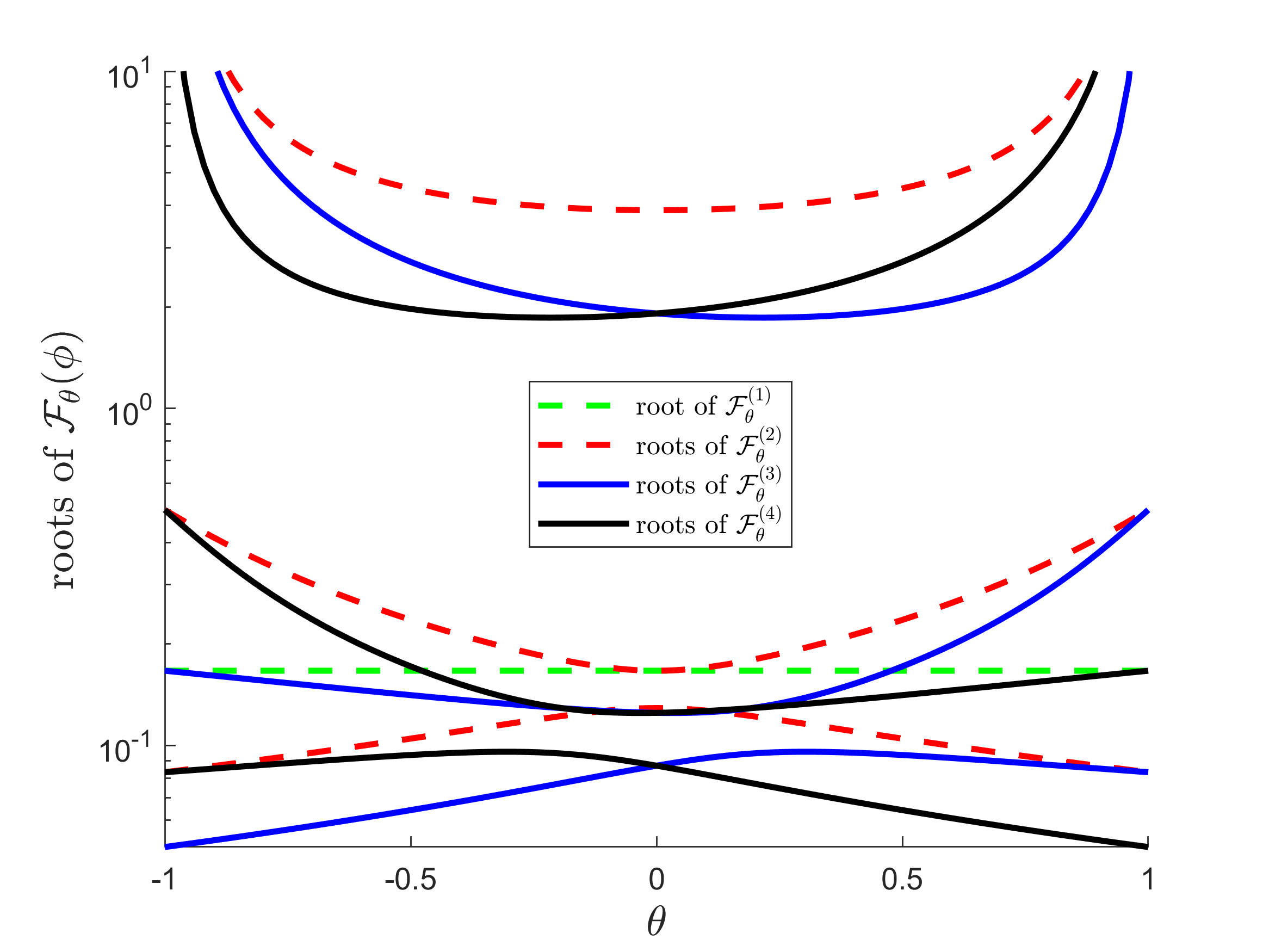}
		\caption{Roots of $\mathcal{F}^{(1)}_\theta(\phi)$, 
			$\mathcal{F}^{(2)}_\theta(\phi)$, 
			$\mathcal{F}^{(3)}_\theta(\phi)$ and 
			$\mathcal{F}^{(4)}_\theta(\phi)$ in \eqref{eq:3324} as functions of $\theta$.}\label{fig:3346}
	\end{figure}
	
	The real roots of 
	$\mathcal{F}^{(1)}_\theta(\phi)$, 
	$\mathcal{F}^{(2)}_\theta(\phi)$, 
	$\mathcal{F}^{(3)}_\theta(\phi)$ and 
	$\mathcal{F}^{(4)}_\theta(\phi)$ are plotted in \Cref{fig:3346}. It is clear that the smallest root of $\mathcal{F}_\theta(\phi)$ is the root of either $\mathcal{F}^{(3)}_\theta(\phi)$ (for $\theta \in [-1,0]$) or $\mathcal{F}^{(4)}_\theta(\phi)$ (for $\theta \in [0,1]$).
	With the help of the cubic formula from \cite{weisstein2002cubic}, the smallest real root of $\mathcal{F}^{(3)}_\theta(\phi)$ has the following explicit formula
	\[
	\phi^{(3)}_{1}(\theta) = \left[-2\theta+\frac{20}{3}+\frac{2}{3}\sqrt{126\,\theta^2-96\theta+94}\cos\left(\frac{1}{3}\arccos\frac{-864\theta^3+2916\,\theta^2-288\theta-532}{(126\,\theta^2-96\theta+94)^{\frac{3}{2}}}\right)\right]^{-1},
	\]
	and the smallest real root $\mathcal{F}^{(4)}_\theta(\phi)$ has the following explicit formula
	\[
	\phi^{(4)}_{1}(\theta) = \left[2\theta+\frac{20}{3}+\frac{2}{3}\sqrt{126\,\theta^2+96\theta+94}\cos\left(\frac{1}{3}\arccos\frac{864\theta^3+2916\,\theta^2+288\theta-532}{(126\,\theta^2+96\theta+94)^{\frac{3}{2}}}\right)\right]^{-1}.
	\]
	Thus,
	\[
	\begin{aligned}
		\phi^\star(\theta,(\mathbb{P}^3)^2) & = 
		\begin{dcases}
			\phi^{(3)}_{1}(\theta) & {\rm if }~~ \theta \in [-1,0], \\
			\phi^{(4)}_{1}(\theta) & {\rm if }~~ \theta \in [0,1], \\
		\end{dcases} \\
		& =  \left[2|\theta|+\frac{20}{3}+\frac{2}{3}\sqrt{126\,\theta^2+96|\theta|+94}\cos\left(\frac{1}{3}\arccos\frac{864|\theta|^3+2916\,\theta^2+288|\theta|-532}{(126\,\theta^2+96|\theta|+94)^{\frac{3}{2}}}\right)\right]^{-1}.
	\end{aligned}
	\]
	So far, we have proved that \eqref{eq:3724} is a feasible CAD and $\overline{\omega}_\star = \phi^\star(\theta,(\mathbb{P}^3)^2)$. By \Cref{lem:513}, \eqref{eq:3724} is the symmetric OCAD for $\mathbb{P}^6$ and $\mathbb{P}^7$ spaces. The proof is completed.
	
\end{proof}

\section{Fully Symmetric OCADs for $\theta = 0$ and $\mathbb{P}^8$ to $\mathbb{P}^{14}$ Spaces} \label{sec:A2}

This appendix gives the  fully symmetric OCADs for $\mathbb{P}^{k}$ spaces with $8 \le k \le 15$ in the case of $\theta =0$. 
See \Cref{fig:P8P9P10P11,fig:P12P13P14P15}.

\begin{sidewaysfigure}[b!]
	\vspace{15cm}
	\centering
	{
		\begin{minipage}[t]{0.45\textwidth}
			\centering
			\includegraphics[width=\textwidth]{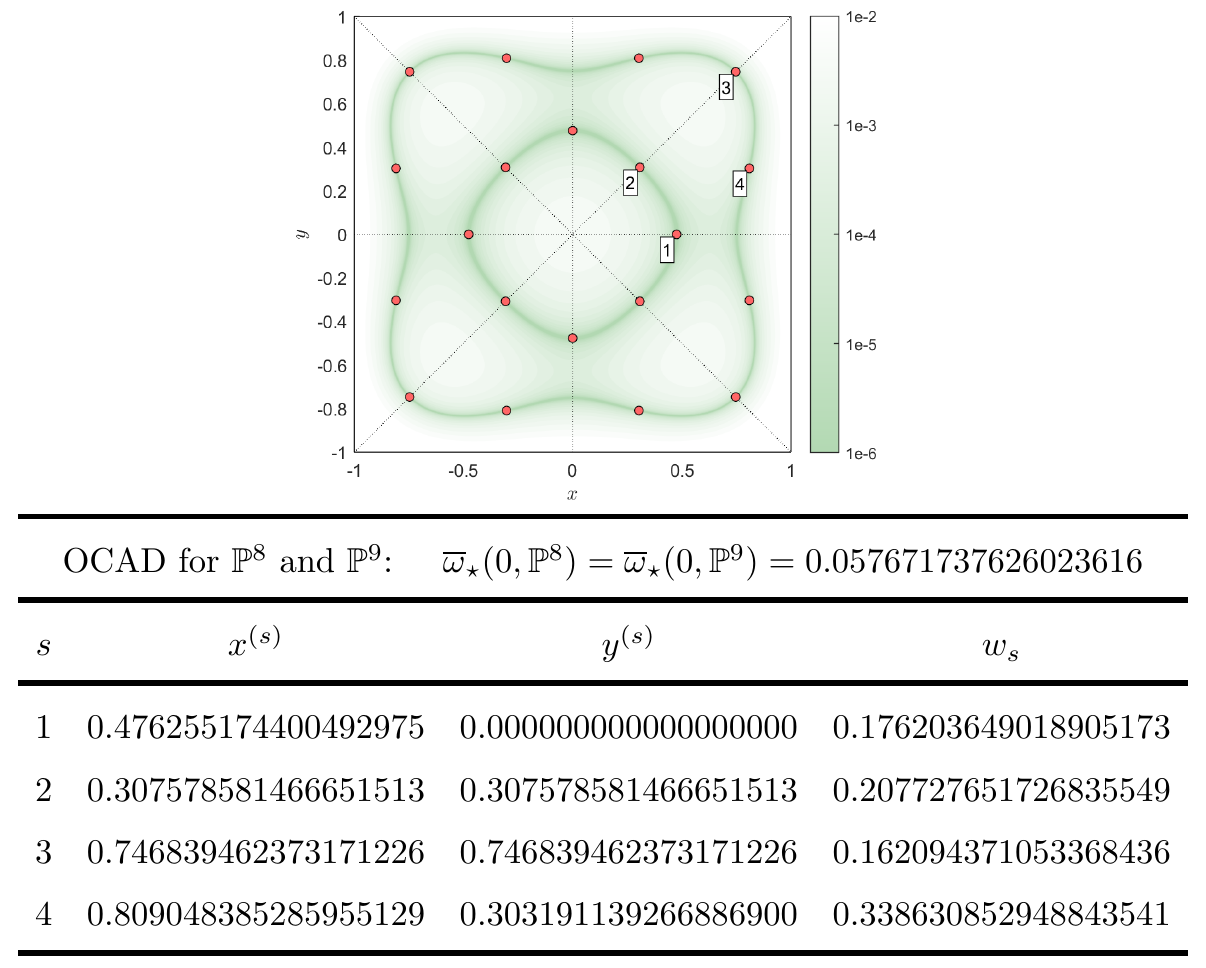}
		\end{minipage}
	}
	\quad 
	{
		\begin{minipage}[t]{0.45\textwidth}
			\centering
			\includegraphics[width=\textwidth]{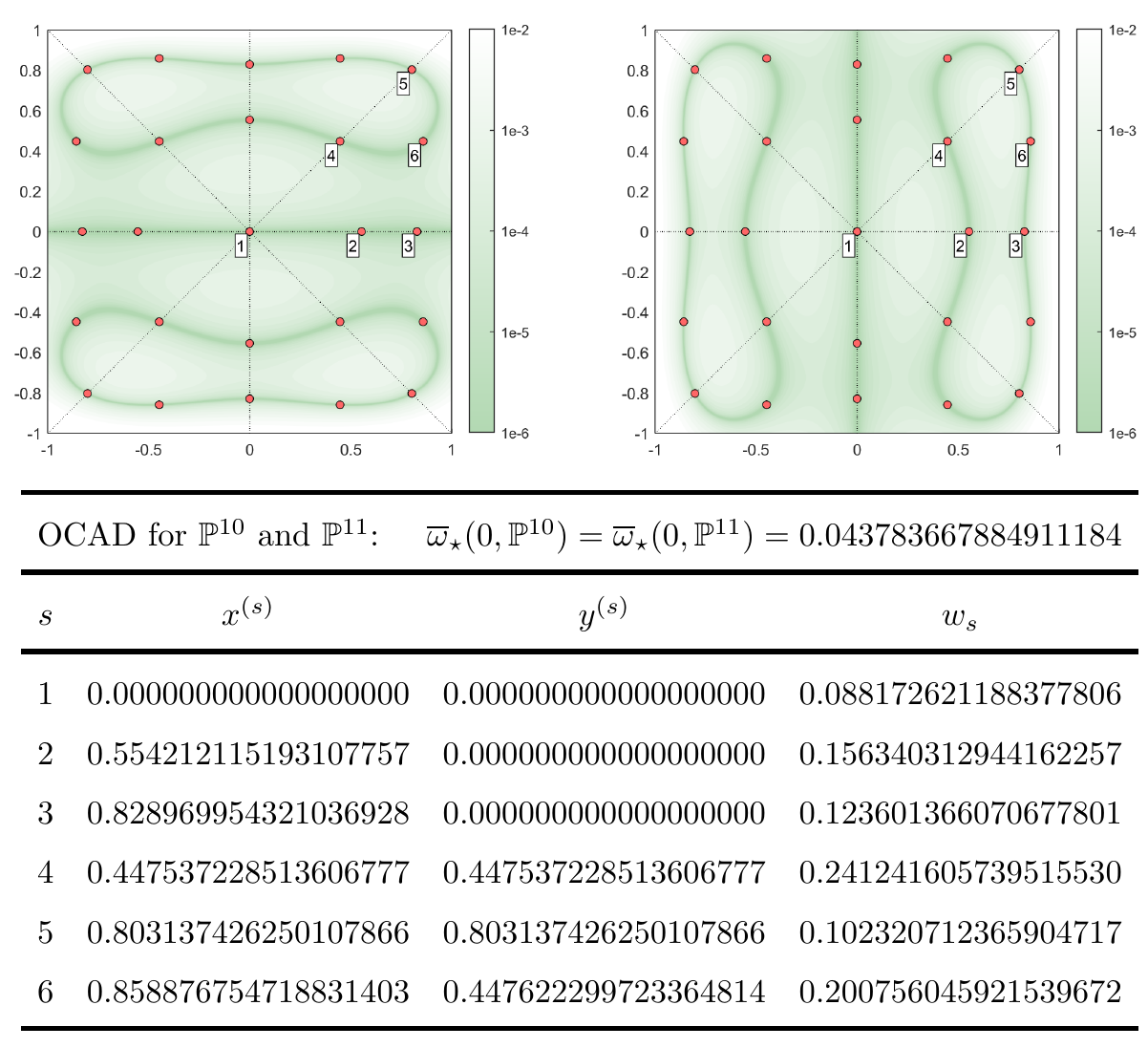}
		\end{minipage} 
	}
	\caption{
		Fully symmetric OCADs in the case of $\theta =0$. Left: $\mathbb P^8$ and $\mathbb P^9$; Right: $\mathbb P^{10}$ and $\mathbb P^{11}$. 
		The background shows the critical positive polynomials, which vanish at all the internal nodes. Note that there are two independent critical positive polynomials for both $\mathbb P^{10}$ and $\mathbb P^{11}$.
	}
	\label{fig:P8P9P10P11}
\end{sidewaysfigure}

\begin{sidewaysfigure}[b!]
	\vspace{136mm}
	\centering
	{
		\begin{minipage}[t]{0.45\textwidth}
			\centering
			\includegraphics[width=\textwidth]{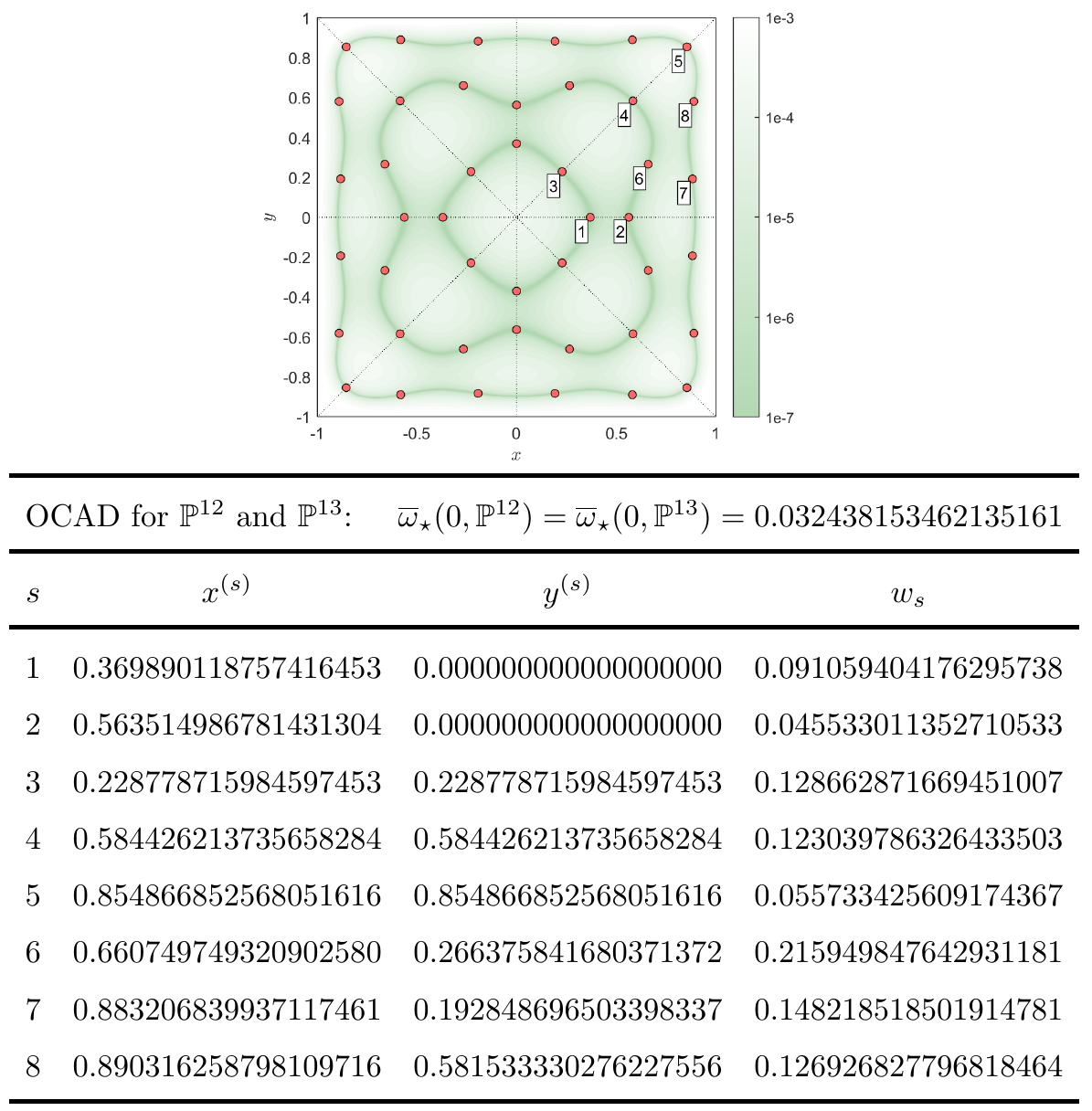}
		\end{minipage}
	}
	\quad 
	{
		\begin{minipage}[t]{0.45\textwidth}
			\centering
			\includegraphics[width=\textwidth]{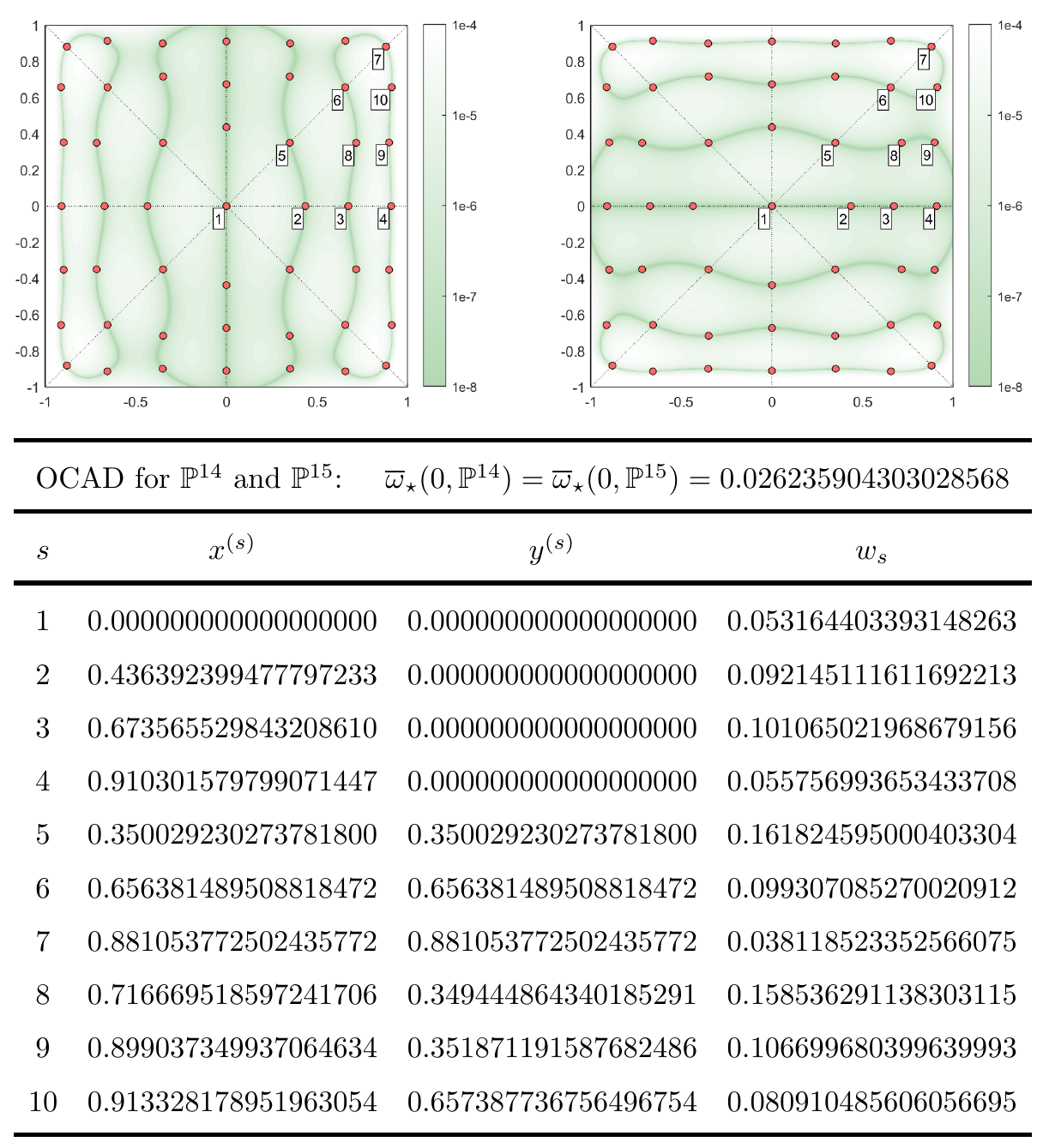}
		\end{minipage} 
	}
	\caption{
		Fully symmetric OCADs in the case of $\theta =0$. Left: $\mathbb P^{12}$ and $\mathbb P^{13}$; Right: $\mathbb P^{14}$ and $\mathbb P^{15}$. 
		The background shows the critical positive polynomials, which vanish at all the internal nodes. Note that there are two independent critical positive polynomials for both $\mathbb P^{14}$ and $\mathbb P^{15}$.
	}
	\label{fig:P12P13P14P15}
\end{sidewaysfigure}


\bibliographystyle{amsplain}

\bibliography{references_short}

%
%
%
%




\end{document}